\documentclass{article}

\usepackage[backend=biber, style=numeric,maxbibnames=99,maxalphanames=99]{biblatex}
\usepackage{graphicx}
\usepackage{float}
\usepackage[T1]{fontenc}
\usepackage[dvipsnames]{xcolor}
\usepackage[bookmarksopen,bookmarksdepth=2]{hyperref}
\hypersetup{colorlinks=true,citecolor=NavyBlue,linkcolor=NavyBlue,urlcolor=Green,hypertexnames=false} 
\usepackage{amsmath,amsthm,amssymb,amscd,mathtools}
\usepackage{amsfonts}
\usepackage{lastpage}
\usepackage{fancyhdr}
\usepackage{mathrsfs}
\usepackage{tikz-cd}
\usepackage{microtype}
\usepackage{stmaryrd}
\usepackage{tocbibind}
\usepackage{thmtools}

\newcommand{\mbb}[1]{\mathbb{#1}}
\newcommand{\mcal}[1]{\mathcal{#1}}

\newcommand{\NN}{\mathbb{N}}
\newcommand{\RR}{\mathbb{R}}
\newcommand{\ZZ}{\mathbb{Z}}
\newcommand{\QQ}{\mathbb{Q}}
\newcommand{\CC}{\mathbb{C}}

\newcommand{\GG}{\mathbb{G}}

\newcommand{\Gal}{\mathrm{Gal}}

\DeclareMathOperator{\im}{Im}

\DeclareMathOperator{\Hom}{Hom}

\DeclareMathOperator{\Spf}{Spf}
\DeclareMathOperator{\Spec}{Spec}

\DeclareMathOperator{\Lie}{Lie}
\DeclareMathOperator{\Mod}{Mod}

\newcommand{\GL}{\mathrm{GL}}

\newcommand{\Partial}[2]{\frac{\partial {#1}}{\partial {#2}}}

\DeclareMathOperator{\Fil}{Fil}

\newcommand{\p}{\mathfrak{p}}

\DeclareMathOperator{\pr}{pr}
\newcommand{\mO}{\mathcal{O}}
\newcommand{\mP}{\mathbb{P}}
\newcommand{\mA}{\mathbb{A}}

\newcommand{\mm}{\mathfrak{m}}

\DeclareMathOperator{\SL}{SL}

\DeclareMathOperator{\End}{End}

\newcommand{\id}{\mathrm{id}}
\newcommand{\mcorner}{\arrow[dr,phantom,"\ulcorner",very near start]}
\newcommand{\mF}{\mathcal{F}}

\DeclareMathOperator{\Ker}{Ker}

\newcommand{\mX}{\mathcal{X}}

\newcommand{\et}{\mathrm{\acute{e}t}}

\newcommand{\proket}{\mathrm{prok\acute{e}t}}

\newcommand{\dR}{\mathrm{dR}}

\newcommand{\OBdR}{\mO\mbb{B}_\dR}
\newcommand{\OBdRlog}{\mO\mbb{B}_{\dR,\log}}

\newcommand{\OBdRKp}{\mO\mbb{B}_{\dR,K^{p}}}
\newcommand{\OBdRlogKp}{\mO\mbb{B}_{\dR,\log,K^{p}}}
\newcommand{\BdR}{\mbb{B}_\dR}

\newcommand{\BdRKp}{\mbb{B}_{\dR,K^{p}}}
\newcommand{\BdRX}[1]{\mbb{B}_{\dR,#1}}
\DeclareMathOperator{\gr}{gr}
\newcommand{\HT}{\mathrm{HT}}
\newcommand{\GM}{\mathrm{GM}}

\newcommand{\Fl}{{\mathscr{F}\!l}}

\DeclareMathOperator{\Spa}{Spa}
\newcommand{\TT}{\mbb{T}}

\newcommand{\fn}{\mathfrak{n}}

\newcommand{\fg}{\mathfrak{g}}
\newcommand{\sm}{{\mathrm{sm}}}
\newcommand{\la}{{\mathrm{la}}}
\newcommand{\an}{{\mathrm{an}}}

\newcommand{\Kpp}{K^pK_p}

\newcommand{\mrm}[1]{\mathrm{#1}}
\newcommand{\mfk}[1]{\mathfrak{#1}}

\newcommand{\RHom}{R\!\Hom}

\newcommand{\mscr}[1]{\mathscr{#1}}
\newcommand{\fib}{\mrm{fib}}

\newcommand{\sen}{\mrm{sen}}

\newcommand{\Sym}{\mrm{Sym}}

\newcommand{\Rlim}{R\!\varprojlim}
\newcommand{\Fib}{\mrm{Fib}}

\DeclareMathOperator{\Fr}{Fr}
\newcommand{\QCoh}{\mrm{QCoh}}
\newcommand{\Vect}{\mrm{Vect}}

\newcommand{\rla}{\mrm{rla}}
\newcommand{\tor}{\mrm{tor}}
\DeclareMathOperator{\Rep}{Rep}

\newcommand{\VB}{V\!B}

\newcommand{\tVBn}{\tilde{V\!B}^{\text{naïve}}_{c,K^{p}}}
\newcommand{\tVB}{\tilde{\VB}_{c,K^{p}}}
\newcommand{\VBn}{\VB^{\text{naïve}}_{K^{p}}}

\newcommand{\tVBfu}{\tilde{V\!B}_{K^{p}}}
\newcommand{\tVBfun}{\tilde{V\!B}^{\text{naïve}}_{K^{p}}}

\DeclareMathOperator{\Spd}{Spd}
\DeclareMathOperator{\AnSpec}{AnSpec}

\usepackage{mathtools}

\newcommand{\ti}[1]{\tilde{#1}}
\newcommand{\Perf}{\mrm{Perf}}

\newcommand{\rF}{\mrm{FF}}

\newcommand{\FLT}[1]{\Fl_{G_{#1},\mu_{#1}}}

\DeclareMathOperator{\alg}{alg}

\newcommand{\pdR}{\mrm{alg}.\dR}

\newcommand{\hatotimes}{\mathbin{\hat{\otimes}}}
\newcommand{\unl}[1]{\underline{#1}}
\newcommand{\std}{\mrm{std}}
\newcommand{\Per}{\mscr{P}\!er}
\newcommand{\mFl}{\mrm{Fl}}
\newcommand{\hpp}{\hat{+}\hat{+}}
\newcommand{\hp}{\hat{+}}
\newcommand{\ket}{\mrm{k\acute{e}t}}

\usepackage{subfiles}
\usepackage{xr} 
\usepackage{cleveref}

\newtheorem{thm}{Theorem}[subsection]
\newtheorem{prop}[thm]{Proposition}

\newtheorem{lem}[thm]{Lemma}
\newtheorem{cor}[thm]{Corollary}

\theoremstyle{remark}
\newtheorem{rmk}[thm]{Remark}
\newtheorem{eg}[thm]{Example}
\theoremstyle{definition}
\newtheorem{dfn}[thm]{Definition}
\newtheorem{notation}[thm]{Notation}
\newtheorem{construction}[thm]{Construction}

\theoremstyle{plain}

\theoremstyle{remark}

\theoremstyle{definition}

\crefname{thm}{Theorem}{Theorems}
\crefname{prop}{Proposition}{Propositions}
\crefname{fact}{Fact}{Facts}
\crefname{lem}{Lemma}{Lemmas}
\crefname{cor}{Corollary}{Corollaries}
\crefname{conj}{Conjecture}{Conjectures}
\crefname{question}{Question}{Questions}
\crefname{rmk}{Remark}{Remarks}
\crefname{eg}{Example}{Examples}
\crefname{dfn}{Definition}{Definitions}
\crefname{notation}{Notation}{Notations}
\crefname{construction}{Construction}{Constructions}

\Crefname{thm}{Theorem}{Theorems}
\Crefname{prop}{Proposition}{Propositions}
\Crefname{fact}{Fact}{Facts}
\Crefname{lem}{Lemma}{Lemmas}
\Crefname{cor}{Corollary}{Corollaries}
\Crefname{conj}{Conjecture}{Conjectures}
\Crefname{question}{Question}{Questions}
\Crefname{rmk}{Remark}{Remarks}
\Crefname{eg}{Example}{Examples}
\Crefname{dfn}{Definition}{Definitions}
\Crefname{notation}{Notation}{Notations}
\Crefname{construction}{Construction}{Constructions}

\begin{document}
\title{Differential operators on locally analytic Shimura varieties}
\author{Yuanyang Jiang}
\date{}

\maketitle
\begin{abstract}
We investigate infinite-level Shimura varieties within the framework of analytic stacks of Clausen-Scholze, developing their smooth, completed, locally analytic, and de Rham realizations. 
We formulate a Grothendieck-Messing-Hodge-Tate period map, and establish a Grothendieck-Messing theory for locally analytic infinite-level Shimura varieties. This theory, combined with a reformulation of Riemann-Hilbert correspondence, implies that the locally analytic infinite-level Shimura variety can be fully reconstructed purely from its perfectoid counterpart and its \( \BdR^{+} \)-thickening. 
Building upon this geometric structure, we systematically construct differential operators generalizing those of Pan.
These constructions will be used to establish a locally analytic Jacquet-Langlands correspondence in a companion paper (\cite{Jiang2025HMF}).
\end{abstract}
\vspace{0.5cm} 
\noindent {\small 
}
\tableofcontents

\section{Introduction}\label{sectionIntro}

This work originated from the author's attempt to understand the work of Lue Pan (\cite{Pan22}, \cite{Pan2209.06II}) in a more geometric framework.
In this article, we will discuss various realizations of the infinite-level Shimura varieties in the category of \emph{analytic stacks} developed by Clausen-Scholze.

In a companion paper (\cite{Jiang2025HMF}), we apply the results here to construct a \emph{locally analytic Jacquet-Langlands correspondence} (\cite[Theorem \ref{2-thmJacquetLanglandsLA}]{Jiang2025HMF}). This correspondence relates cohomology of different quaternionic Shimura varieties, and serves as a crucial ingredient for showing a classicality theorem in the setting of Hilbert modular forms.

Given a Shimura datum \( (G,X) \) with its Hodge cocharacter \( \mu \),
we have the associated tower of Shimura varieties \( \{X_{K}\}_{K\subset G(\mA_{f})} \) over the reflex field \( E\subset\CC \).
We fix a prime \( p \) and \( \iota:\CC\cong\bar{\QQ}_{p} \), 
and let \( E_{\p} \)
be the closure of \( E\subset^{\iota}\bar{\QQ}_{p} \).
Let \( \mX_{K,E_{\p}} \)
be the analytification of \( X_{K}\otimes_{E}E_{\p} \), and let \( \mX_{K}:=\mX_{K,E_{\p}}\otimes_{E_{\p}}\CC_{p} \). Let \( \mFl_{G,\mu} \) be the partial flag variety defined over \( E \), let \( \Fl_{G,\mu,E_{\p}}:=(\mFl_{G,\mu}\otimes_{E}E_{\p})^{\an} \), and \( \Fl_{G,\mu}:=\Fl_{G,\mu,E_{\p}}\otimes_{E_{\p}} \CC_{p} \).
See \S \ref{subsecSetUpSV} for details.

Perfectoid infinite-level Shimura varieties were first introduced by \cite{Scholze15}, and subsequently generalized by Scholze in \cite{CS17}, \cite{Shen2017perfectoid}, \cite{HansenJohansson2020perfectoid}. 
For any neat compact open subgroup \( K^{p}\subset G(\mA_{f}^{p}) \), one define \[
\mX_{K^{p}}^{\diamond}:=\varprojlim_{K_{p}\subset G(\QQ_{p})}\mX_{\Kpp}^{\diamond},
\] where the limit is taken in the category of diamonds,
which admits a Hodge-Tate period map \[ \pi_{\HT}:\mX_{K^{p}}^{\diamond}\to \Fl_{G,\mu}^{\diamond}. \]
The aforementioned works show that \( \mX_{K^{p}}^{\diamond} \)
is representable by a perfectoid space \( \mX_{K^{p}} \), when \( (G,X) \)
is of pre-abelian type (\cite{HansenJohansson2020perfectoid}).

Since the test objects of the category of diamonds are perfectoid spaces, 
there is an implicit \( p \)-adic completion in the limit process above. Prior to the introduction of the analytic stacks by Clausen-Scholze, this seemed to be the only plausible construction, which is well-suited for the study of \'etale or pro\'etale cohomology. 

On the other hand, after the works of Lue Pan (\cite{Pan22}, \cite{Pan2209.06II}) and Rodr{\'\i}guez Camargo (\cite{Juan2022.09locallyShi}), the subspace of \( G(\QQ_{p}) \)-locally analytic vectors in the completed cohomology of Shimura varieties (\cite{Emerton06})
has drawn considerable attention. Over this subspace with its finer topology, it is possible to realize several \( p \)-adic Hodge theoretical constructions, such as arithmetic Sen operators and Fontaine operators, which provides a general framework for proving classicality theorems.

Moreover, these constructions admit an enhancement on the level of sheaves on \( \mX_{K^{p}} \). A central role is played by \( \mO^{\la}_{K^{p}} \),
the subsheaf of the structure sheaf \( \mO_{\mX_{K^{p}}} \) over \( \mX_{K^{p}} \) consisting of \( G(\QQ_{p}) \)-locally analytic vectors. 
In \cite{Pan2209.06II}, Pan further defines several  (logarithmic) differential operators on certain subsheaves of \( \mO^{\la}_{K^{p}} \) that describe the Fontaine operator. 

This shows that \( \mO^{\la}_{K^{p}} \) has quite interesting and non-trivial differential structures, which can not be seen directly on the perfectoid Shimura varieties, since perfectoid spaces always have trivial sheaves of differentials (relative to \( \CC_{p} \)). 
Our primary goal is to give a geometrical reformulation of these differential operators using a more refined locally analytic version of infinite-level Shimura varieties in the category of analytic stacks. 

\begin{rmk}
It is worth noting the recent work of Sean Howe (\cite{Howe2025inscription}), which provides an alternative framework for studying differential operators on infinite-level Shimura varieties.
\end{rmk}

\subsection{Shimura varieties as analytic stacks}\label{subsecSVasAnStkIntro}
It is natural to consider a ringed space \( \mX^{\la}_{K^{p}} \) whose underlying topological space is \( |\mX_{K^{p}}| \) and whose structure sheaf is \( \mO^{\la}_{K^{p}} \). Because it is crucial to remember the topological structure of \( \mO^{\la}_{K^{p}} \), the theory of Clausen-Scholze analytic stacks provides the ideal framework. 

In broad terms, Clausen and Scholze define analytic rings as (animated) topological rings with specified convergence conditions on their modules. These rings are quite general and, notably, need not be \( p \)-adically complete. Analytic stacks are then defined as sheaves (of anima) on the site of analytic rings equipped with the \( ! \)-topology (see \S \ref{subsectionTateStack} for details). For an analytic ring \( R \), we write \( \AnSpec(R) \) for the associated analytic stack via the Yoneda lemma.

As we will explain in \S \ref{subsecEgSolidStack}, schemes and rigid varieties can be realized as analytic stacks. In particular, we have \( X_{\Kpp} \) and \( \mX_{\Kpp} \) as analytic stacks.

We utilize the flexibility of analytic stacks to define various versions of infinite-level Shimura varieties:
\begin{itemize}
\item (Smooth version, Definition \ref{dfnNonCompletedInfiniteLevel})
Define \( \mX_{K^{p}}^{\sm}:=\varprojlim_{K_{p}}\mX_{\Kpp} \), where the limit is taken in the category of analytic stacks.
\item (Completed version, Lemma \ref{lemTorCompacAsTateStack})
We define an analytic stack \( \mX_{K^{p}} \) over \( \CC_{p} \). When \( (G,X) \) is of pre-abelian type, \( \mX_{K^{p}} \) coincides with the perfectoid space representing the diamond \( \mX_{K^{p}}^{\diamond} \).
\item (Locally analytic version, Definition \ref{dfnLocallyAnSV})
The sheaf \( \mO^{\la}_{K^{p}}\in D_{\an}(\mX_{K^{p}},\CC_{p}) \) naturally lifts to a quasi-coherent sheaf \( \mO^{\la}_{K^{p}}\in\QCoh(\mX_{K^{p}}^{\sm}) \), and we define \( \mX^{\la}_{K^{p}}:=\unl{\AnSpec}_{\mX^{\sm}_{K^{p}}}(\mO^{\la}_{K^{p}}) \), where \( \unl{\AnSpec} \) denotes the relative spectrum (see \S \ref{subsecAffineProperMorph}).
\item (\( \BdR^{+} \)-version, Lemma \ref{lemBdRXtorKp})
We define an analytic stack \( \BdR^{+}(\mX_{K^{p}}) \) over \( B_{\dR}^{+} \), which is a thickening of \( \mX_{K^{p}} \). When \( (G,X) \) is of pre-abelian type, \( \BdR^{+}(\mX_{K^{p}}) \) is given by gluing \( \Spf(\BdR^{+}(A)) \) for open affinoid subspaces \( \Spa(A,A^{+})\subset \mX_{K^{p}} \).
\end{itemize}

These spaces are closely related to the completed cohomology introduced by Emerton in \cite{Emerton06}. For a completed extension \( L/\QQ_{p} \), let
\[
\ti{H}^{i}(K^{p},L):=\left(\varprojlim_{n}\varinjlim_{K_{p}}H^{i}_{\et}(X_{K,\bar{E}},\ZZ/p^{n})\right)\hatotimes_{\ZZ_{p}}L,
\] 
and let \( \ti{H}^{i}(K^{p},L)^{\la} \) denote the subspace of \( G(\QQ_{p}) \)-locally analytic vectors.

If the varieties \( X_{K} \) are proper, the primitive comparison theorem (\cite{Scholze13}) implies that:
\begin{align*}
H^{i}(\mX^{\sm}_{K^{p}})&\cong \varinjlim_{K_{p}}H^{i}(\mX_{\Kpp},\mO_{\mX_{\Kpp}})\in \mrm{Rep}^{\sm}(G(\QQ_{p})),\\
H^{i}(\mX_{K^{p}})&\cong \ti{H}^{i}(K^{p},\CC_{p})\in \mrm{Rep}^{\mrm{Ban}}(G(\QQ_{p})),
\\
H^{i}(\mX_{K^{p}}^{\la})& \cong \ti{H}^{i}(K^{p},\CC_{p})^{\la}\in \mrm{Rep}^{\la}(G(\QQ_{p})),
\\
H^{i}(\BdR^{+}(\mX_{K^{p}}))&\cong \ti{H}^{i}(K^{p},\QQ_{p})\hatotimes B_{\dR}^{+},
\end{align*}
When \( X_{K} \) are not proper, we define logarithmic analogues \( \mX^{\tor}_{K^{p}} \), \( \mX^{\tor,\la}_{K^{p}} \), \( \mX^{\tor,\sm}_{K^{p}} \), and \( \BdR^{+}(\mX^{\tor}_{K^{p}}) \) for the toroidal compactifications \( X_{K}^{\tor} \), which satisfy analogous comparison theorems with completed cohomology.

\subsection{de Rham stacks}\label{subsecDRstkIntro}
The constructions in \S \ref{subsecSVasAnStkIntro} are largely formal. The first non-trivial input for our study is the theory of de Rham stacks, which unifies various aspects of the Fontaine operators. 

The concept of de Rham stacks first appeared in the work of Simpson (\cite{Simpson1996algebraicgeometricnstacks}) and has seen renewed interest following the works of Drinfeld (\cite{Drinfeld2022stacky}, \cite{Drinfeld2024prismatization}) and Bhatt-Lurie (\cite{BhattLurie2022prismatization}, \cite{BhattLurie2022prismatic}). Rodr{\'\i}guez Camargo (\cite{Juan2024analyticdeRham}) generalized this construction to the setting of analytic stacks.
  He introduces two variants: algebraic de Rham stacks \( X^{\pdR} \) and analytic de Rham stacks \( X^{\dR} \) (see \S \ref{subsecAndRStack}). 
For a map \( f:X\to Y \), there are also relative versions \( X^{\pdR/Y} \)
and \( X^{\dR/Y} \).
The analytic variant is further developed in \cite{AnschützBoscoLeBrasCamargoScholze2025analyticrhamstacksfarguesfontaine}.

\begin{rmk}
To work with toroidal compactifications of Shimura varieties, 
we will define in \S \ref{subsecProperAlgLogDRstack} a logarithmic version of algebraic de Rham stacks \( X^{\pdR/L}_{\log} \)
for a log smooth scheme or rigid variety \( X \) over \( L \), and its filtered (resp. filtered complete) variant \( X^{\pdR/L,+}_{\log} \) (resp. \( X^{\pdR/L,\hp}_{\log} \)). 
\end{rmk}

The de Rham stacks possess at least the following useful features:

(A) The quasicoherent sheaves over (logarithmic) algebraic de Rham stacks are equivalent to (logarithmic) \( D \)-modules (Lemma \ref{lemDRcomplexVSpushforward}).

(A)'  (Lemma \ref{lemDiffOpeToLogStack}) For \( X \) log smooth over \( L \), consider the natural map \( h^{+}:X\times [\mA^{1}/\GG_{m}]\to X^{\pdR/L,+}_{\log} \). For vector bundles  \( V_{1},V_{2} \) on \( X \), and \( n\in\NN \), \[ \Hom_{X^{\pdR/L,+}_{\log}}(h^{+}_{*}(p_{1}^{*}V_{1}\{-n\}),h^{+}_{*}(p_{1}^{*}V_{2}))\cong \mrm{Diff}_{\log}^{\le n}(V_{1},V_{2}), \]
where the RHS is the space of \emph{logarithmic differential operators} from \( V_{1} \)
to \( V_{2} \) of degree \( \le n \). 
Here \( \{i\} \)
denotes twisting by the \( i \)-th power of the tautological line bundle on \( [\mA^{1}/\GG_{m}]\).

(B) Trivial deformation theory: if \( Y\hookrightarrow Y' \) is a nilpotent thickening, then \( \Hom(Y,X^{\pdR})\cong \Hom(Y',X^{\pdR}) \). 
A similar statement holds for \( X^{\dR} \) (Definition \ref{dfnAnDRStack}).

(C) The analytic de Rham stacks ``decomplete'' diamonds (\cite{AnschützBoscoLeBrasCamargoScholze2025analyticrhamstacksfarguesfontaine}): for example, \( (\AnSpec(\CC_{p}))^{\dR}\cong \AnSpec(\bar{\QQ}_{p}) \), and when \( \mX_{K^{p}}^{\diamond }\) is perfectoid, \( (\mX_{K^{p}})^{\dR}\cong \varprojlim_{K_{p}}\mX_{\Kpp}^{\dR} \) (Theorem \ref{thmAnDRStackGeneral}).

Using (B), the natural morphisms \[ 
\mX_{\Kpp,E_{\p}}^{\pdR}
\leftarrow
\mX_{\Kpp,E_{\p}}
\leftarrow
\mX_{K^{p}}\to \Fl_{G,\mu,E_{\p}} \xrightarrow{\pi_{\HT}} \Fl_{G,\mu,E_{\p}}^{\pdR} \] uniquely extend to \[  \mX_{\Kpp,E_{\p}}^{\pdR}
\leftarrow \BdR^{+}(\mX_{K^{p}})\to\Fl_{G,\mu,E_{\p}}^{\pdR}. \]
Moreover, the maps admit filtered enhancements, whose (filtered completed) pull-back along \( \mX_{\Kpp,E_{\p}}\to \mX_{\Kpp,E_{\p}}^{\pdR} \)
recovers the period sheaf \( \OBdR^{+} \) of \cite{Scholze12} and \cite{DLLZ2022logarithmicJAMS}.
This interpretation of \( \BdR^{+} \) and \( \OBdR^{+} \)
is strongly motivated by \cite{Beilinson2012p} and \cite{GuoLi2021period}.

Similarly,
we obtain a natural morphism \( \beta^{\BdR^{+}}:\BdR^{+}(\mX_{K^{p}})\to \mX_{K^{p}}^{\dR} \), and we define \[ \BdRKp^{+}:=\beta^{\BdR^{+}}_{*}(\mO_{\BdR^{+}(\mX_{K^{p}})})\in \QCoh(\mX_{K^{p}}^{\dR}). \] This gives a natural lift of the de Rham period sheaf from the topological space \( |\mX_{K^{p}}| \) to the de Rham stack \( \mX_{K^{p}}^{\dR} \), which, combined with (A)', partially explains why Fontaine operators can be described as differential operators in \cite{Pan2209.06II}. 

\subsection{Grothendieck-Messing-Hodge-Tate period maps}\label{subsecGMtheoryIntro}
In this part, we explain a Grothendieck-Messing theory for locally analytic Shimura varieties, showing that its infinitesimal structure is equivalent to what we term the ``Grothendieck-Messing-Hodge-Tate period domain''. This aligns with the expectation that \( p \)-adic Shimura varieties parametrize variations of \( p \)-adic Hodge structures.  

The classical Grothendieck-Messing theory and Serre-Tate theory show that the deformation theory of abelian varieties is controlled by flag varieties. For Shimura varieties, let \( \mFl_{G,\mu}^{\std} \) be the standard partial flag variety over \( E \), and \( \Fl_{G,\mu,E_{\p}}^{\std}:=(\mFl_{G,\mu}^{\std}\otimes_{E}E_{\p})^{\an} \) (\S \ref{subsecSetUpSV}).  
Then infinitesimally, \( X_{K} \)
and \( \mFl_{G,\mu}^{\std} \) are expected to be isomorphic. 
Using de Rham stacks, it can be reformulated as follows: 
\begin{prop}[Propositions \ref{propAlgGMtheory} \& \ref{propAnalyticGMTheory}]\label{propFiniteLevelGMtheoryIntro}---

(1) (Algebraic Grothendieck-Messing theory) We have a Cartesian diagram over \( E \) \[
\begin{tikzcd}
X^{\tor}_{K}\times [\mA^{1}/\GG_{m}]\arrow[r,"\pi_{\GM,K}"] \arrow[d,"h_{K}"]& 
{\left[\mrm{Fl}^{\std}_{G,\mu}/G^{c}\right]\times [\mA^{1}/\GG_{m}]}\arrow[d,"h_{\mrm{Fl}^{\std}}"]
\\
(X^{\tor}_{K})^{\pdR/E,+}_{\log}\arrow[r,"\pi_{\GM,K}^{\pdR,+}"] &
{\left[(\mrm{Fl}^{\std}_{G,\mu})^{\pdR/E,+}/G^{c}\right]}.
\end{tikzcd}
\]

(2) (Analytic Grothendieck-Messing theory) We have a Cartesian diagram over \( E_{\p} \) \[
\begin{tikzcd}
\mX_{K,E_{\p}}\arrow[r,"\pi_{\GM,K}"] \arrow[d,"h_{K}^{\an}"]& 
{[\Fl^{\std}_{G,\mu,E_{\p}}/G^{c,\an}_{E_{\p}}]}\arrow[d,"h_{\Fl^{\std}}^{\an}"]
\\
\mX_{K,E_{\p}}^{\dR}\arrow[r,"\pi_{\GM,K}^{\dR}"] &
{[\Fl^{\std,\dR}_{G,\mu,E_{\p}}/G^{c,\an}_{E_{\p}}]}.
\end{tikzcd}
\]
\end{prop}
A crucial subtlety is the usage of \( (\mFl^{\std}_{G,\mu})^{\pdR/E}/G^{c} \), as it is easier to construct a period map towards \( (\mFl^{\std}_{G,\mu}/G^{c})^{\pdR/E} \).
The construction of \( \pi^{\pdR,+}_{\GM,K} \) uses the feature (A) to descend the de Rham \( G^{c} \)-torsor on \( X_{K}^{\tor} \) to \( (X_{K}^{\tor})^{\pdR/E}_{\log} \).
This construction was first suggested to us by Qixiang Wang. Once we have constructed the diagram, verifying that it is Cartesian boils down to the classical Kodaira-Spencer isomorphism. 

Using Proposition \ref{propFiniteLevelGMtheoryIntro}, we will give a simple reconstruction of BGG complexes on Shimura varieties in \S \ref{subsubsecBGGSV} using BGG complexes on \( \Fl^{\std}_{G,\mu,E} \), where the latter is constructed in \S \ref{subsubsecBGGComplex}.

Our main result establishes a similar theorem for the locally analytic infinite-level Shimura varieties \( \mX^{\la}_{K^{p}} \). 
We first learned the following period domain from \cite{DospinescuJuan2024jacquet} in the setting of local Shimura varieties.
\begin{dfn}[Definition \ref{dfnGMHTperDomain}]
We define the \emph{Grothendieck-Messing-Hodge-Tate period domain} as the following rigid variety over \( \CC_{p} \) \[
\Per_{G,\mu}:=M^{c,\an}_{\mu}\backslash
\left((N_{\mu}^{\an}\backslash G^{c,\an}_{\CC_{p}})\times (N_{\mu}^{\std,\an}\backslash G^{c,\an}_{\CC_{p}})\right).
\] 
See \S \ref{subsecSetUpSV} for undefined notation.
We denote the right action of \(G^{c,\an}_{\CC_{p}}\) on the first (resp. second) factor of \(\Per_{G,\mu}\) as \(*_{\HT}\) (resp. \(*_{\GM}\)). Moreover,
we have natural maps \[
\Fl_{G,\mu}\xleftarrow{\pi^{per}_{\HT}}\Per_{G,\mu}
\xrightarrow{\pi^{per}_{\GM}}\Fl^{\std}_{G,\mu},\]
where for \(\mrm{XX}\in\{\GM,\HT\}\),
 \(\pi_{\mrm{XX}}^{per}\)
is \(G^{c,\an}_{\CC_{p}}\)-equivariant for the \(*_{\mrm{XX}}\)-action on \(\Per_{G,\mu}\).
\end{dfn}
\begin{thm}[Theorem \ref{thmLAstrucGMHTper}]
    \label{thmLAstrucGMHTperIntro}
(1) 
We have the following commutative diagram \[
\begin{tikzcd}
\mX^{\tor,\la}_{K^{p}}
\arrow[d,"\pi^{\la}_{\sm}"]
\arrow[r,"\pi^{\la}_{\GM\HT}"] & 
\Per_{G,\mu}/G^{c,\an}_{\CC_{p}} 
\arrow[r,"\pi_{\HT}^{per}"] \arrow[d,"h_{\Per/\Fl^{\std}}"] &[-0.8cm] 
\Fl_{G,\mu}\arrow[d,"\ti{h}_{\Fl_{G,\mu}}"]\\
\mX^{\tor,\sm}_{K^{p}}
\arrow[d,"h^{\sm}"] \arrow[r,"\pi^{\sm}_{\GM\HT}"] & \Per_{G,\mu}^{\dR/\Fl^{\std}_{G,\mu}}/G^{c,\an}_{\CC_{p}} \arrow[r,"\pi_{\HT}^{per,\dR}"]\arrow[d,"\ti{h}_{\Fl^{\std}}"]\arrow[dr,"\pi_{\GM}^{per}"] & \Fl_{G,\mu}/(\fg^{0}/\fn^{0})^{\dagger}\\
(\mX^{\tor}_{K^{p}})_{\log}^{\pdR/\CC_{p},\sm}\arrow[rrd,bend right=10,"\pi_{\GM}^{\pdR}"]
\arrow[r,"\pi^{\pdR}_{\GM\HT}"] & (\Per_{G,\mu}^{\dR/(\Fl^{\std,\pdR/\CC_{p}}_{G,\mu})})/G^{c,\an}_{\CC_{p}}\arrow[dr,"\pi_{\GM}^{per,\pdR}"] & \Fl_{G,\mu}^{\std}/G^{c,\an}_{\CC_{p}}\arrow[d,"h_{\Fl^{\std}}"]\\
& & \Fl_{G,\mu}^{\std,\pdR/\CC_{p}}/G^{c,\an}_{\CC_{p}},
\end{tikzcd}
\] where \((\mX^{\tor}_{K^{p}})_{\log}^{\pdR/\CC_{p},\sm}:=\varprojlim_{K_{p}}(\mX^{\tor}_{\Kpp})_{\log}^{\pdR/\CC_{p}}\), the actions of \( G^{c,\an}_{\CC_{p}} \) on the middle column are induced by \( *_{\GM} \), and all the squares are \emph{Cartesian}. 

(2) 
We have the following Cartesian diagram \[
\begin{tikzcd}
\mX^{\la}_{K^{p}}\arrow[r,"\pi^{\la}_{\GM\HT}"]\arrow[d,"\beta^{\la}"]
& \Per_{G,\mu}/(G^{c,\an}_{\CC_{p}},*_{\GM})\arrow[d]
\\
\mX^{\dR,\sm}_{K^{p}}\arrow[r,"\pi^{\dR}_{\GM\HT}"]
& \Per_{G,\mu}^{\dR}/(G^{c,\an}_{\bar{\QQ}_{p}},*_{\GM}),
\end{tikzcd}
\] for \(\mX^{\dR,\sm}_{K^{p}}:=\varprojlim_{K_{p}}\mX^{\dR}_{\Kpp}\).
If \( \mX_{K^{p}}^{\diamond} \) is represented by a perfectoid space, then \( \mX^{\dR,\sm}_{K^{p}}\cong (\mX_{K^{p}})^{\dR} \). 
\end{thm}
The proof relies on Proposition \ref{propFiniteLevelGMtheoryIntro} alongside the calculation of \( \mO^{\la}_{K^{p}} \) using the geometric Sen theory developed by \cite{Juan2022.05GeoSen} and \cite{Juan2022.09locallyShi}. 

When \( \mX_{K^{p}}^{\diamond} \)
is represented by a perfectoid space,
Theorem \ref{thmLAstrucGMHTperIntro} shows that \( \mX^{\la}_{K^{p}} \)
is determined by \( \mX_{K^{p}} \)
and \( \pi_{\GM\HT}^{\dR} \), 
and the latter is determined by the descent of the de Rham \( G^{c} \)-torsor on \( \mX_{K^{p}} \) to \( \mX^{\dR}_{K^{p}} \).
The following theorem demonstrates how to recover this last missing piece, and provides a reformulation of the Riemann-Hilbert correspondence of \cite{LiuZhu2017rigidity} and \cite{DLLZ2022logarithmicJAMS}.
\begin{thm}[Corollary \ref{corReformulateRiemmanHilbPerfdCase}]\label{thmReformulateRiemmanHilbPerfdCaseIntro}
Assume that \( \mX_{K^{p}}^{\diamond} \)
is represented by a perfectoid space. 
 Then for any \( V\in\Rep(G^{c}) \), 
we have 
\begin{align*}
(\pi^{\dR}_{K_{p}})^{*}
\mcal{\bar{G}}^{c,\an}_{\dR,\Kpp}(V)&\cong \RHom_{\fg^{c}}(V^{\vee},E_{0}(D_{\sen}(\BdRKp^{\la})))\in \QCoh(\mX^{\dR}_{K^{p}}),
\end{align*} where 
\begin{itemize}
\item \( \pi^{\dR}_{K_{p}}:\mX_{K^{p}}^{\dR}\to \mX_{\Kpp,E_{\p}}^{\dR} \), and \( \mcal{\bar{G}}^{c,\an}_{\dR,\Kpp} \) 
is the de Rham \( G^{c} \)-torsor on \( \mX_{\Kpp,E_{\p}}^{\dR} \) (Proposition \ref{propAnalyticGMTheory});
\item \( \BdRKp^{\la}:=(\BdRKp^{+}[1/t])^{\la} \), with \( \BdRKp^{+} \) as in \S \ref{subsecDRstkIntro}, and with \( (-)^{\la} \) (as in Notation \ref{notationRlaConvention}) taken in the filtered sense;
\item 
\( D_{\sen}(-):=\varinjlim_{L}(-)^{H_{L},\Gamma_{L}-\la} \) (Notation \ref{notationEpInftyHandGamma})
with the colimit taken over all the finite extensions \( E_{\p}\subset L\subset \bar{\QQ}_{p} \);
\item \( E_{0} \) denotes the generalized eigenspace for \( \Theta=0 \) (Definition \ref{dfnE0functor}), where \( \Theta \)
is a fixed generator of \( \Lie(\Gamma_{E_{\p}}) \).
\end{itemize}
\end{thm}
For the proof, we first establish a similar theorem over 
\( (\mX^{\tor}_{K^{p}})^{\pdR/\CC_{p},+,\sm}_{\log} \), which is a direct consequence of the logarithmic Riemann-Hilbert correspondence of \cite{DLLZ2022logarithmicJAMS}. We then use the equivalence between the category of vector bundles over algebraic and analytic de Rham stacks (Proposition \ref{propAnToAlgDRJuan}) to deduce Theorem \ref{thmReformulateRiemmanHilbPerfdCaseIntro}.

Theorem \ref{thmReformulateRiemmanHilbPerfdCaseIntro} shows that the \( G^{c,\an} \)-torsor \( (\pi^{\dR}_{K_{p}})^{*}\bar{\mcal{G}}^{c,\an}_{\dR,\Kpp} \) on \( \mX^{\dR}_{K^{p}} \) can be reconstructed from the RHS, namely, from \( \BdRKp^{\la} \) with the natural Galois action, without any additional information about finite-level Shimura varieties. 

Combined, Theorem \ref{thmLAstrucGMHTperIntro} and Theorem \ref{thmReformulateRiemmanHilbPerfdCaseIntro} illustrate how to reconstruct \( \mX^{\la}_{K^{p}} \) from the perfectoid data, i.e. \( \mX_{K^{p}} \) and \( \BdR^{+}(\mX_{K^{p}}) \) with the action of \( G(\QQ_{p})\times \Gal_{E_{\p}} \). This will be the key ingredient for constructing a locally analytic Jacquet-Langlands correspondence in \cite{Jiang2025HMF}. In fact, \cite{Jiang2025HMF} will first establish a perfectoid Jacquet-Langlands correspondence using Igusa stacks (\cite{Zhang2023pel}, \cite{DanielsHoftenKimZhang2024igusa}, \cite{Kim2025uniquenessfunctorialityigusastacks}, \cite{DanielVanHoftenKimZhangs2026igusastackscohomologyshimura}), and then deduce the locally analytic version using the idea above. 
\subsection{Differential operators}
We use 
Theorem \ref{thmLAstrucGMHTperIntro} to construct and generalize the differential operators \( d^{k} \) and \( \bar{d}^{k} \) in \cite{Pan2209.06II}. 

We briefly recall the construction in \cite{Pan2209.06II}. 
We consider the action of \( \fg:=\Lie(G(\QQ_{p})) \)
on \( \mO^{\la}_{K^{p}} \), and we linearly extend it to an action of \( \fg^{0}:=\mO_{\Fl}\otimes\fg  \). 
Since \( \Fl_{G,\mu} \)
parametrizes parabolic subgroups of \( G \) that are conjugate to \( P_{\mu} \), \( \fg^{0} \)
admits sub-bundles \( \fn_{\mu}^{0}\subset \p_{\mu}^{0}\subset \fg^{0} \), corresponding to the Lie algebra of the universal parabolic subgroup and its nilpotent radical. Then \cite{Pan22} and \cite{Juan2022.09locallyShi} show that the action of \( \fn_{\mu}^{0} \) is trivial on \( \mO^{\la}_{K^{p}} \), and induces an action of \( \mm_{\mu}^{0}:=\p_{\mu}^{0}/\fn_{\mu}^{0} \) on \( \mO^{\la}_{K^{p}} \). 

Let \( \mO^{\la,0}_{K^{p}} \)
denote the subsheaf of \( \mO^{\la}_{K^{p}} \)
that is the kernel of the \( \mm_{\mu}^{0} \)-action, which contains \( \pi_{\HT}^{-1}(\mO_{\Fl}) \)
and \( \pi_{K_{p}}^{-1}(\mO_{\mX_{\Kpp}^{\tor}}) \).
In the case of modular curves, \cite{Pan2209.06II} constructs differential operators between various twists of \( \mO^{\la,0}_{K^{p}} \) by line bundles over \( \mX^{\tor}_{\Kpp} \)
and \( \Fl_{G,\mu} \). 

\begin{prop}[Corollary \ref{corXLa0WithTwoFlag}, Lemma \ref{lemKunnFlSm}]\label{propHorizontalDRIntro}
Let \[ \mX^{\tor,\la,0}_{K^{p}}:=\unl{\AnSpec}_{\mX^{\sm}_{K^{p}}}(\mO^{\la,0}_{K^{p}}) \text{ (Definition \ref{dfnHorizontalLAsv}),} \] let \( (\mX^{\tor}_{K^{p}})^{\pdR/\CC_{p},\hp,\sm}_{\log}:=\varprojlim_{K_{p}}(\mX^{\tor}_{\Kpp})^{\pdR/\CC_{p},\hp}_{\log} \), and let \[
(\mX^{\tor,\la,0}_{K^{p}})^{\pdR/\CC_{p},\hpp}_{\log}:=\Fl^{\pdR/\CC_{p},\hp}_{G,\mu}\times_{\Fl^{\dR/\CC_{p}}_{G,\mu}}
(\mX^{\tor}_{K^{p}})^{\pdR/\CC_{p},\hp,\sm}_{\log}\text{}
\] as in Definition \ref{dfnBiFilAgLa0}, where the map \( (\mX^{\tor}_{K^{p}})^{\pdR/\CC_{p},\hp,\sm}_{\log}\to\Fl^{\dR/\CC_{p}}_{G,\mu} \)
is induced by \( \pi_{\HT} \). Then the following hold:

(1) We have the following Cartesian diagram \[
\begin{tikzcd}[column sep=40pt]
{\mX^{\tor,\la,0}_{K^{p}}\times [\hat{\mA}^{1}/\GG_{m}]^{2}}
\arrow[r,"\pi^{\la,0}_{\HT}\times \pi^{\la,0}_{\sm}"]
\arrow[d,"h^{\la,0,\hpp}"] &
{(\Fl_{G,\mu}\times [\hat{\mA}^{1}/\GG_{m}])}
\times  
{(\mX^{\tor,\sm}_{K^{p}}\times [\hat{\mA}^{1}/\GG_{m}])}
\arrow[d,"h_{\Fl}^{\hp}\times h^{\sm,\hp}"]
\\ 
(\mX^{\tor,\la,0}_{K^{p}})^{\pdR/\CC_{p},\hpp}_{\log}
\arrow[r,"\pi^{\pdR,\hpp}_{\GM\HT}"] 
& \Fl_{G,\mu}^{\pdR/\CC_{p},\hp}\times  (\mX^{\tor}_{K^{p}})^{\pdR/\CC_{p},{\hp},\sm}_{\log}.
\end{tikzcd}
\]  

(2) Let \( V \in \QCoh(
\Fl_{G,\mu})\) and let \( W \in \QCoh(\mX^{\tor,\sm}_{K^{p}}) \). Then the natural morphism in \( \QCoh((\mX^{\tor,\la,0}_{K^{p}})^{\pdR/\CC_{p},\hpp}_{\log}) \) 
\begin{align*}
(\pi^{\pdR,\hpp}_{\GM\HT})^{*}\left(
(h^{\hat{+}}_{\Fl})_{*}(V)\boxtimes (h^{\sm,\hat{+}})_{*}(W)\right)\to\\ 
(h^{\la,0,\hpp})_{*}\left(
	(\pi^{\la,0}_{\HT})^{*}(V)\otimes (\pi^{\la,0}_{\sm})^{*}(W)
\right)
\end{align*}
is an isomorphism.
\end{prop}
Statement (1) follows from Theorem \ref{thmLAstrucGMHTperIntro}. However, (2) is not a direct consequence of (1), since the vertical maps are not proper. Indeed, the statement is false without taking the filtered completion.

By Lemma \ref{lemDiffOpeToLogStack} and Proposition \ref{propHorizontalDRIntro} (2), we have constructed a natural map \begin{align}\label{alignIntroConstDiff}
&\mrm{Diff}_{\Fl_{G,\mu}}^{\le n}(V_{1},V_{2})\otimes 
\mrm{Diff}_{\mX^{\tor}_{\Kpp},\log}^{\le m}(W_{1},W_{2})
\\&\to \mrm{Diff}_{\mX^{\tor,\la,0}_{K^{p}},\log}^{\le (n,m)}\left((\pi^{\la,0}_{\HT})^{*}(V_{1})\otimes (\pi^{\la,0}_{K_{p}})^{*}(W_{1}),(\pi^{\la,0}_{\HT})^{*}(V_{2})\otimes (\pi^{\la,0}_{K_{p}})^{*}(W_{2})\right),
\end{align} where \( \pi_{K_{p}}^{\la,0}:\mX^{\tor,\la,0}_{K^{p}}\to \mX^{\tor}_{\Kpp} \), and the RHS 
denotes the space of logarithmic differential operators that are of degree \( \le m \) over \( \mO_{\mX_{\Kpp}^{\tor}} \), and of degree \( \le n \) over \( \mO_{\Fl_{G,\mu}} \).
\begin{eg}\label{egRecoverPanConstruction}
In the case of modular curves, if we take \[ V_{1}=V_{2}=\omega^{1-k}_{\Fl_{G,\mu}}, W_{1}=\omega^{1-k}_{\mX^{\tor}_{\Kpp}}, W_{2}=\omega^{1+k}_{\mX^{\tor}_{\Kpp}},k\in\ZZ_{\ge1},\] then \( 1\otimes \theta^{k}_{\mX^{\tor}_{\Kpp}} \) is sent to \( d^{k} \)
defined in \cite{Pan2209.06II}. Similarly, if we take \[ V_{1}=\omega^{1-k}_{\Fl_{G,\mu}},V_{2}=\omega^{1+k}_{\Fl_{G,\mu}}, W_{1}=W_{2}=\omega^{1-k}_{\mX^{\tor}_{\Kpp}},k\in\ZZ_{\ge1}, \]
then \( \theta^{k}_{\Fl_{G,\mu}}\otimes 1 \) is sent to \( \bar{d}^{k} \) loc. cit. Thus if we take \[ V_{1}=\omega^{1-k}_{\Fl_{G,\mu}},V_{2}=\omega^{1+k}_{\Fl_{G,\mu}}, W_{1}=\omega^{1-k}_{\mX^{\tor}_{\Kpp}}, W_{2}=\omega^{1+k}_{\mX^{\tor}_{\Kpp}},k\in\ZZ_{\ge1},\]
then \( \theta^{k}_{\Fl_{G,\mu}}\otimes \theta^{k}_{\mX^{\tor}_{\Kpp}} \)
is sent to the Fontaine operator by \cite[Theorem 1.2.10]{Pan2209.06II}. 
\end{eg}
The main source of differential operators on \( \Fl_{G,\mu} \) and \( \mX^{\tor}_{\Kpp} \) is provided by dual BGG complexes, which we will discuss in Corollary \ref{corFilterDualBGG} and Proposition \ref{propBGGFilSV} respectively. 
\subsection{Organization of the paper}
The paper consists of two parts. Section \ref{sectionAnDRStacks} provides the general prerequisites including general discussions about the theory of Clausen-Scholze analytic stacks and its solid and Tate variants (\cite{Juan2024analyticdeRham}), the construction of a logarithmic version of algebraic de Rham stacks, and the study of filtered BGG complexes on partial flag varieties. Section \ref{sectionLAVectorSV} applies the general results to the setting of Shimura varieties. We refer the readers to the introduction to each section for more details. 

\subsection{Acknowledgment}
I would like to express my gratitude to my advisor Vincent Pilloni for many invaluable discussions and for his constant encouragement and support, without which this work wouldn't be possible. 
I would like to thank Johannes Anschütz, Dustin Clausen, Arthur-C\'esar Le Bras, Lue Pan, Juan Esteban Rodr\'iguez Camargo, and Qixiang Wang, from whom I learn many important ideas that are presented in this work.  
I would like to thank Xingzhu Fang, Valentin Hernandez, Pol van Hoften, Dongryul Kim, 
Shizhang Li, Benchao Su, Matteo Tamiozzo, Longke Tang, Xiangqian Yang, Mingjia Zhang, and Daming Zhou for helpful discussions.
Part of this work was done during my visits to Princeton University and to Beijing International Center for Mathematical Research,
and I want to thank both institutes and my hosts Lue Pan, Yiwen Ding and Liang Xiao for their hospitality. 
The work was done while the author was a PhD student at Université Paris-Saclay, under the
specific doctoral contract for normaliens (CDSN), and I would like to thank Université Paris-Saclay for its support.

\subsection{Notation}
We will work with solid modules in the sense of \cite{CS19}. To ease the notation, all the tensor products will be taken in the solid sense, which we will simply denote as \(-\hatotimes-\).

We fix \(\bar{\QQ}\subset\CC\),
and we fix an isomorphism \(\iota:\CC\cong\bar{\QQ}_{p}\).
Denote by \(\CC_{p}\) the completion of \(\bar{\QQ}_{p}\) for the \(p\)-adic topology.
By Frobenius, we will always refer to the geometric Frobenius. 
Let \(\chi_{\mrm{cycl}}:\Gal_{\QQ}\to\ZZ_{p}^{\times}\)
denote the cyclotomic character, sending \(\Fr_{l}\) to \(l^{-1}\). By our convention, \(\chi_{\mrm{cycl}}\)
has Hodge-Tate weight \(-1\). Our convention for local class field theory takes the uniformizer to a lift of the geometric Frobenius.  

\section{Preliminary on analytic stacks}\label{sectionAnDRStacks}

This section collects some technical results on solid and Tate stacks, algebraic and analytic de Rham stacks, and a log version of algebraic de Rham stacks.
 The main references are \cite{Juan2024analyticdeRham} and \cite{AnschützBoscoLeBrasCamargoScholze2025analyticrhamstacksfarguesfontaine}. Many results 
presented here are due to Clausen-Scholze.


The section is structured as follows: 

Subsection \ref{subsectionTateStack} introduces the general definition of analytic stacks of Clausen-Scholze and \cite{AnschützBoscoLeBrasCamargoScholze2025analyticrhamstacksfarguesfontaine}. We also discuss general six functor formalisms, following \cite{Mann2022padic6functorformalismrigidanalytic}, \cite{Mann20226functorformalismmathbbzell}, \cite{Scholze2022sixFunctor}, \cite{HeyerMann20246functorformalismssmoothrepresentations}, and \cite{Juan2024analyticdeRham}.

Subsection \ref{subsecSolidTateStack} introduces the solid and Tate stacks of \cite{Juan2024analyticdeRham}. Since we need to work with both types of stacks, we also study the relation between the two categories, and introduce so-called ``essentially Tate'' stacks, for which the functor into Tate stacks is fully faithful.

In Subsection \ref{subsecAffineProperMorph},
we define so-called ``affine proper morphisms", and provide a solid stack version of the relative Spec construction. Many spaces that later arise will come from this construction.

In Subsection \ref{subsecFiltrationFormalSpec}, we recall the Rees construction realizing filtered objects as quasi-coherent sheaves on \( [\mA^{1}/\GG_{m}] \) as in \cite{MoulinosTasos2021Tgof}, \cite{BhattLurie2022prismatic}. We further define a (relative) Spf construction associate a solid stack over \( [\mA^{1}/\GG_{m}] \) to a filtered complete algebra.

In Subsection \ref{subsecEgSolidStack}, we give some examples of solid stacks, notably various geometric realizations of \( p \)-adic Lie groups and vector bundles.

In Subsection \ref{subsecAndRStack}, 
we give the definition of the algebraic and analytic de Rham stacks of \cite{Juan2024analyticdeRham}, and prove some basic properties. Many of the results in this subsection are essentially proven in \cite{AnschützBoscoLeBrasCamargoScholze2025analyticrhamstacksfarguesfontaine}, except that we work with Tate stacks where (compared with Gelfand stacks) all the test objects are not equipped with induced analytic structures. 
We do not claim any originality 
over these results. 

In Subsection \ref{subsecProperAlgLogDRstack}
we give a log version of filtered algebraic de Rham stack. Given a log smooth rigid variety \( X \) over \( L \), \( X^{\pdR/L,+}_{\log} \)
is a solid stack over \( [\mA^{1}/\GG_{m}] \). We also prove that the category of quasi-coherent sheaves on \( X^{\pdR/L,+}_{\log} \) is equivalent to that of filtered \( D_{X/L,\log} \)-modules.
We also show that the filtered complete version \( X^{\pdR/L,\hat{+}}_{\log} \) is affine when \( X=\Spa(A,A^{+}) \), in the sense that the category of quasi-coherent sheaves on \( X^{\pdR/L,\hat{+}}_{\log} \) is equivalent to that of filtered complete \( \dR_{\log}(A) \)-modules. 
The two interpretations have their own benefits. 

In Subsection \ref{subsecAnalyticVSAlgeDR}, we compare the relation between analytic and algebraic de Rham stacks. The main new result is a proof of an identification of the categories of vector bundles on analytic and algebraic de Rham stacks.

In Subsection \ref{subsecRepTheory}, we prove some results in representation theory which will be used later. In \S
\ref{subsecPartialFl}, we discuss the de Rham stacks of partial flag varieties. One key construction is \( \Fl_{P}/\widehat{(\fg^{0}/\fn^{0})} \) as in \cite{Scholze2024RealLanglands}, which should relate to twisted \( D \)-modules via Beilinson-Bernstein localization.
In \S \ref{subsubsecRepLieAlg}, we show how to realize the representations of Lie algebras on a suitable classifying stack, following \cite{Scholze2024RealLanglands}. We also construct the infinitesimal action using geometry.
In \S \ref{subsubsecParBGGCatO}, we define the parabolic BGG category as a subcategory of the category of equivariant \( D \)-modules on partial flag varieties, and we compare this definition to more classical definitions. Finally, in \S \ref{subsubsecBGGComplex}, we construct the (dual) BGG complexes on partial flag varieties following \cite{Faltings1983CohoOfLocSymmetric}.

In Subsection \ref{subsecDeRhamPeriodSheaves},
we realize the de Rham period sheaf of \cite{Scholze13} and \cite{DLLZ2022logarithmicJAMS}
as a quasicoherent sheaf on the filtered algebraic log de Rham stacks. We actually give two realizations: one depends on the period sheaf \( \OBdRlog \), while the other is by deformation. 
The key observation is that the two coincide.
Later in Subsections \ref{subsecBdRSV} and \ref{subsecRiemannHilbert}, we will reformulation the logarithmic Riemann-Hilbert correspondence over Shimura varieties over the filtered algebraic log de Rham stacks.

\subsection{Analytic stacks}\label{subsectionTateStack}
The main reference is the YouTube video of the course given by Clausen-Scholze,  
\cite{Juan2024analyticdeRham} and \cite{AnschützBoscoLeBrasCamargoScholze2025analyticrhamstacksfarguesfontaine}.
Later, we will work mainly with the solid and Tate stacks introduced in \cite{Juan2024analyticdeRham} which we will introduce in \S \ref{subsecSolidTateStack}.
\begin{dfn}[Clausen-Scholze]
	A \emph{pre-analytic ring}
	is a pair \(A:=(A^{\rhd},D_{\ge 0}(A))\)
	where \(A^{\rhd}\)
	is an animated condensed ring,
	and \(D_{\ge0}(A)\) is a full subcategory of \(D_{\ge0}(A^{\rhd})\)
	such that after stabilization,
	\(D(A)\)
	is a full stable subcategory closed under all limits, colimits, and internal Hom's, and \(D_{\ge0}(A)=D_{\ge 0}(A^{\rhd})\cap D(A)\). 

	We say \(A\) is an \emph{analytic ring}
	if the unit \(A^{\rhd}\in D_{\ge0}(A^{\rhd})\)
	lies in the subcategory \(D_{\ge0}(A)\). 

	We denote by \((-)^{\wedge}_{A}:D(A^{\rhd})\to D(A)\)
	the left adjunction to the inclusion functor. Note that \(D(A)\)
	has a symmetric monoidal structure.
\end{dfn}
\begin{dfn}
	A map between analytic rings \(A\to B\)
	is a map \(f:A^{\rhd}\to B^{\rhd}\)
	with the condition that the forgetful functor \(D(B^{\rhd})\to D(A^{\rhd})\)
	takes the full subcategory  \(D(B)\) to \(D(A)\).

	In this case, we denote by \(f_{*}:D(B)\to D(A)\)
	the forgetful functor,
	and by \(f^{*}\) its left adjoint, that is, the pull-back functor \(M\mapsto (B^{\rhd}\otimes_{A}M)^{\wedge}_{B}\).
\end{dfn}
\begin{dfn}\label{dfnPartSixFunctor}
	We say  \(f:A\to B\) is \emph{proper}	
	if \(B\) has the induced analytic structure, in the sense that \(D_{\ge0}(B)\subset \mrm{Mod}_{B^{\rhd}}(D_{\ge0}(A))\).
	In other word,
	\(M\in D_{\ge0}(B^{\rhd})\)
	lies in \(D_{\ge0}(B)\)
	if and only if 
	\(M|_{A^{\rhd}}\in D_{\ge0}(A)\).
	In this case, we write \(B:=(B^{\rhd},A)\), which we refer to as the \emph{analytic structure induced from \(A\)}.

	Given any \(f:A\to B\),
	there is a unique decomposition \[A\xrightarrow{p} (B^{\rhd},A)\xrightarrow{j}B,\]
	such that \(p\) is proper.

	We say \(f:A\to B\) is an \emph{open immersion} if \(f^{*}:D(A)\to D(B)\)
	admits a left adjoint that is fully faithful, which we denote as \[f_{!}:D(B)\to D(A).\]

	We say that \(f\) is \emph{!-able}
	if the induced map \(j:(B^{\rhd},A)\to B\) 
	is an open immersion. In this case, using the decomposition \(f=p\circ j\),
	we define \[f_{!}:=p_{*}\circ j_{!}:D(B)\to D(A).\]

	It commutes with colimits, and thus has a right adjoint, which we denote as \(f^{!}:D(A)\to D(B)\).
\end{dfn}
\begin{dfn}
	A map between analytic rings \(A\to B\)
	is a \(!\)-cover if it is \(!\)-able, and it is a \(\mcal{D}\)-cover for the six-functor formalism \(D(-)\) as in \cite[Definition 5.13]{Scholze2022sixFunctor}. In other words,
	we have \(D(A)\cong \varprojlim_{i}^{!}D(B^{\otimes_{i}A})\) in \(\mrm{Pr}^{R}\)
	and \(D(A)\cong \varprojlim_{i}^{*}D(B^{\otimes_{i}A})\) in \(\mrm{Pr}^{L}\).
	Here \(\varprojlim^{?}\)
	means that the transition maps are given by \(f^{?}\) for \(?\in\{*,!\}\),
	and \(\mrm{Pr}^{L}\) (resp. \(\mrm{Pr}^{R}\))
	is the category of presentable categories with morphisms given by left (resp. right) adjunctions as in \cite[Definition 5.5.3.1]{Lurie2009HTT}.
\end{dfn}


\begin{dfn}[{\cite[Definition 4.2.1]{AnschützBoscoLeBrasCamargoScholze2025analyticrhamstacksfarguesfontaine}}]\label{dfnAnStkOverC}
	Let \(\mscr{C}^{op}\) be a small subcategory of analytic rings, closed under push-outs and finite products, and let \( \mscr{C}^{!} \) be the non-full subcategory of \( \mscr{C} \) 
	 consisting of \( \mrm{Obj}(\mscr{C}) \) and \( ! \)-able maps. Denote by \(\mcal{P}(\mscr{C})\) the presheaf category, i.e. the category of functors from \(\mscr{C}^{op}\)
	to anima.

	For any \(X\in \mcal{P}(\mscr{C})\),
	define \[\QCoh(X):=
	\varprojlim_{(A\in\mscr{C}^{op},s\in X(A))}^{!} 
	D(A),
	\] where the transition maps are given by \((-)^{!}\).

	A map \(f:X\to Y\)
	in \(\mcal{P}(\mscr{C}^{!})\)
	is a \(!\)-equivalence if for 
	any iterated diagonal \(\Delta_{f}^{(n)}\) of \(f\), \(\Delta_{f}^{(n),!}\) induces an equivalence on \(\QCoh(-)\). 

	An \emph{analytic stack} over \(\mscr{C}\) is a presheaf \(F\in\mcal{P}(\mscr{C})\), such that for any \(!\)-equivalence \(f:X\to Y\) between \(\omega_{1}\)-compact objects in \(\mcal{P}(\mscr{C}^{!})\), \(\mrm{Map}(Y,F)\cong \mrm{Map}(X,F)\). 
	We denote the category of analytic stacks over \(\mscr{C}\) by \(\mrm{AnStk}(\mscr{C})\).

	By \cite[Remark 4.2.8]{AnschützBoscoLeBrasCamargoScholze2025analyticrhamstacksfarguesfontaine}, 
	the category of analytic stacks over \(\mscr{C}\) with \(\QCoh(-)\) can be endowed with a six-functor formalism, such that the functors \((f^{*},f_{*})\) and \((f_{!},f^{!})\) coincides with those in Definition \ref{dfnPartSixFunctor} in the setting there.
	We further extend the class of \(!\)-able morphisms by \cite[Theorem 5.19]{Scholze2022sixFunctor}. 
	
	Also note that \( \QCoh(X)\cong \varprojlim^{*}_{(A\in\mscr{C},s\in X(A))}D(A) \)  
	by \cite[Remark 4.2.3]{AnschützBoscoLeBrasCamargoScholze2025analyticrhamstacksfarguesfontaine}. In particular, \( \QCoh(X) \) has a symmetric monoidal structure with the unit \( 1_{X}:=\mO_{X} \). We will sometimes omit the subscript and write \( 1 \)  when it causes no confusion.

	As in \cite[Remark 4.2.4]{AnschützBoscoLeBrasCamargoScholze2025analyticrhamstacksfarguesfontaine},
	it follows from 
	\cite[Theorem 4.3.3]{AnelBiedermannFinsterJoyall2022left} that \(\mrm{AnStk}(\mscr{C})\)
	is an \(\infty\)-topos. In particular, the corresponding ``\(!\)-sheafification''
	commutes with finite limits by \cite[Corollary 6.2.1.6, Proposition 6.2.2.7]{Lurie2009HTT}. We will refer to epimorphisms in this \( \infty \)-topos as \emph{\(!\)-surjections}. 	
\end{dfn}
\begin{dfn}\label{dfnVariousPropertyMorphiDescent}
Let \( f:X\to Y \) in \( \mrm{AnStk}(\mscr{C}) \). We will often consider the following pull-back diagram: \begin{equation}\label{equationPullBack}
\begin{tikzcd}
X'\arrow[r,"f'"]\arrow[d,"g'"] & Y'\arrow[d,"g"]\\
X\arrow[r,"f"] & Y
\end{tikzcd}
\end{equation}

(1) We say that \( f \) satisfies \emph{\( * \)-descent} if \( D(Y)\cong \varprojlim^{*}_{[n]\in\Delta}D(X^{n/Y}) \).    We say that \( f \) satisfies \emph{universal \( * \)-descent} if for any pull-back diagram (\ref{equationPullBack}), \( f' \) satisfies \( * \)-descent.

(2)
We say that \( f \) satisfies \emph{\( ! \)-descent} if \( f \) is \( ! \)-able, and \[ D(Y)\cong \varprojlim^{!}_{[n]\in\Delta}D(X^{n/Y}). \]   We say that \( f \) satisfies \emph{universal \( ! \)-descent} if for any pull-back diagram (\ref{equationPullBack}), \( f' \) satisfies \( ! \)-descent.

(3) Assume that \( f \) is \( ! \)-able. 
We say that \( A\in\QCoh(X) \) is \emph{\( f \)-suave} (\cite[Definition 4.4.1 (a)]{HeyerMann20246functorformalismssmoothrepresentations}) if \( A\in\Hom_{\mrm{LZ}}(X,Y) \) is a left adjoint. 

Equivalently, by \cite[Proposition 6.5 
\& 6.6]{Scholze2022sixFunctor}, let \( \mbb{D}_{f}(A):=\unl{\Hom}(A,f^{!}(\mO_{Y})) \), then for any pull-back diagram (\ref{equationPullBack}), 
the natural morphism \[ g^{\prime,*}\mbb{D}_{f}(A)\otimes f^{\prime,*}(-)\to \unl{\Hom}(g^{\prime,*}A,f^{\prime,!}(-)) \] is an equivalence. 
Moreover, it suffices to verify the case \( g=f,\;Y'=X \). 

We will say that \( f \) is \emph{suave} if \( \mO_{X} \) is \( f \)-suave.  

(4)
We say that \( f \) satisfies \emph{suave  \(!\)-descent} and \( f \) is a \emph{suave \( ! \)-cover}  if \( f \) is \( ! \)-able and suave, 
and \( f^{*}:\QCoh(Y)\to\QCoh(X) \) is conservative.

(5) Assume that \( f \) is \( ! \)-able. 
 We say that \( A\in\QCoh(X) \) is \emph{\( f \)-prim} (\cite[Definition 4.4.1 (b)]{HeyerMann20246functorformalismssmoothrepresentations}) if \( A\in\Hom_{\mrm{LZ}}(Y,X) \) is a left adjoint. 

 Equivalently, by \cite[Proposition 6.8 
\& 6.9]{Scholze2022sixFunctor}, let \[ \mbb{P}_{f}(A):=p_{2,*}\unl{\Hom}(p_{1}^{*}A,\Delta_{f,!}(\mO_{X})), \] then for any pull-back diagram (\ref{equationPullBack}),
the natural morphism \[ f'_{!}(g^{\prime,*}\mbb{P}_{f}(A)\otimes -)\to f'_{*}\unl{\Hom}(g^{\prime,*}A,-) \] is an equivalence. 
Moreover, it suffices to verify the case \( g=\id,\;Y'=Y \). 

We will say that \( f \) is \emph{prim} if \( \mO_{X} \) is \( f \)-prim.  

(6)
We say that \( f \) satisfies \emph{prim \(!\)-descent} and \( f \) is a \emph{prim \( ! \)-cover} if \( f \) is \( ! \)-able and prim,
and \( f_{*}(\mO_{X})\in\mrm{CAlg}(\QCoh(Y)) \) is descendable (\cite[Definition 3.18]{Mathew2016galois}).

(7) 
We say \(f\) is a \emph{monomorphism} if \(\Delta_{f}:X\to X\times_{Y}X\)
is an isomorphism.

(8)
We say \(f\) is an \emph{open immersion} if \(f\) is a \( ! \)-able  monomorphism, 
and the canonical map \( f^{!}\to f^{*} \) (\cite[Proposition 6.13]{Scholze2022sixFunctor}) is an equivalence.

(9)
We say \(f\) is \emph{cohomologically smooth} if \(f\) is \(!\)-able, \( \mO_{X} \) is \( f \)-suave, and \( \mbb{D}_{f}(\mO_{X}) \) is invertible.

(10)
We say \(f\) is \emph{cohomologically co-smooth} if \(f\) is \(!\)-able, \( \mO_{X} \) is \( f \)-prim, and \( \mbb{P}_{f}(\mO_{X}) \) is invertible.
\end{dfn}
\begin{thm}\label{thmVariousDescent}

Let \( f:X\to Y \) in \( \mrm{AnStk}(\mscr{C}) \).

(1) If \( f \) satisfies {universal \( ! \)-descent}, then \( f \) satisfies {universal \( * \)-descent}.

(2) If \( f \) is a suave \( ! \)-cover,  then \( f \) satisfies {universal \( ! \)-descent}.

(3) If \( f \) is a prim \( ! \)-cover,  then \( f \) satisfies {universal \( ! \)-descent}.

(4) Consider the pull-back diagram (\ref{equationPullBack}). Let \( A\in \QCoh(X) \).   If \( g \) satisfies universal \( * \)-descent, and  \( g^{\prime,*}(A) \) is \( f' \)-suave (resp. \( f \)-prim), then \( A \) is \( f \)-suave (resp. \( f \)-prim).   

(5) Consider \( g:Z\to X \). If \( g \) is suave (resp. prim), then  
\( A\in\QCoh(X) \) is \( f \)-suave (resp. \( f \)-prim) implies that \( g^{*}(A) \) is \( (f\circ g) \)-suave (resp. \( f\circ g \)-prim). The inverse is true if \( g \) is a suave (resp. prim) \( ! \)-cover.  
In that case, \( \mbb{D}_{f\circ g}(g^{*}(A))\cong g^{*}\mbb{D}_{f}(A)\otimes \mbb{D}_{g}(\mO_{Z}) \)
(resp. \( \mbb{P}_{f\circ g}(g^{*}(A))\cong g^{*}\mbb{P}_{f}(A)\otimes \mbb{P}_{g}(\mO_{Z}) \)).

In particular, the class of cohomological smooth (resp. co-smooth) maps is closed under compositions.

(6) If \( f \) is a \( ! \)-surjection, then \( f \) satisfies universal \( * \)-descent.   

(7) If \( f \) satisfies {suave \( ! \)-descent}, and \( X \) and \( Y \) are \( \omega_{1} \)-compact,   then \( f \) 
is a \( ! \)-surjection. 
\end{thm}
\begin{proof}
(1)-(3) are the contents of \cite[Theorem 5.12 \& Propositions 6.19 \& 6.20]{Scholze2022sixFunctor} respectively. (4) is given by \cite[Corollary 7.8]{Mann20226functorformalismmathbbzell} and \cite[Proposition 3.1.4]{Juan2024analyticdeRham}.
(5) is given by \cite[Corollary 8.6]{Mann20226functorformalismmathbbzell} and \cite[Proposition 3.1.24]{Juan2024analyticdeRham}.
(6) follows from \cite[Remark 4.2.3]{AnschützBoscoLeBrasCamargoScholze2025analyticrhamstacksfarguesfontaine} since 
\( ``\varinjlim_{[n]\in\Delta^{op}}"X^{n/Y}\to Y \) is a \( ! \)-equivalence. 
Note that to prove \( ! \)-ability, in (4) we use ``local on the target'' 
and in (5) we use ``local on the source'' of \cite[Theorem 5.19]{Scholze2022sixFunctor}.

For (7), consider the \v Cech nerve \( \{X^{n/Y}\}_{n\in\Delta^{op}} \)  associated to \( f \) 
and let \( Y':=\varinjlim_{n\in\Delta^{op}}X^{n/Y} \). 
Then by (4) and (6), \( X\to Y' \) 
also satisfies suave \( ! \)-descent. 
In particular, the monomorphism \( i:Y'\to Y \) satisfies suave \( ! \)-descent by (5). This implies that \( Y'\to Y \) is a \( ! \)-equivalence. 
Since they are both \( \omega_{1} \)-compact, \( Y'\cong Y \)  
in \( \mrm{AnStk}(\mscr{C}) \). 
\end{proof}
\begin{dfn}\label{dfnPerfectComplex}
Let \( X \) 
be an analytic stack over \(\mscr{C}\). 

(1)
\(\mF\in\QCoh(X)\)
is defined to be a \emph{perfect complex}
if it is dualizable, and we denote the corresponding subcategory as \(\mrm{Perf}(X)\subset \QCoh(X)\). 

(2)
\( \mF\in \QCoh(X) \) is defined to be a \emph{vector bundle} (of rank \(d\)) if there exists a \( ! \)-surjection \( f:Y\to X \), such that \( f^{*}(\mF)\cong \mO_{Y}^{\oplus d} \).
We denote the corresponding subcategory as \(\Vect(X)\subset \QCoh(X)\). 
By Theorem \ref{thmVariousDescent} (6), we have \( \Vect(X)\subset \Perf(X) \). 
\end{dfn}
\begin{lem}[Projection formula for perfect complexes]\label{lemProjectionFormuaForPerf}
For any \( f:X\to Y \), \( \mF\in\QCoh(X) \), and \( \mcal{G}\in \Perf(Y) \), then the natural morphism \( f_{*}(\mcal{F})\otimes \mcal{G}\to
f_{*}(\mcal{F}\otimes f^{*}\mcal{G}) \) is an equivalence.   
\end{lem}
\begin{proof}
The natural morphism is induced by \[
f^{*}(f_{*}(\mcal{F})\otimes \mcal{G})\to 
\mF\otimes f^{*}(\mcal{G}).
\] Moreover, for any \( M\in \QCoh(Y) \),
\begin{align*}
\Hom(M,f_{*}(\mF\otimes f^{*}\mcal{G}))
&\cong \Hom(f^{*}M,\mF\otimes f^{*}\mcal{G})
\cong \Hom(f^{*}M\otimes f^{*}\mcal{G}^{\vee},\mF)\\
&\cong \Hom(M\otimes \mcal{G}^{\vee},f_{*}\mF)
\cong \Hom(M,\mcal{G}^{\vee}\otimes f_{*}\mF).
\end{align*}
Thus by Yoneda Lemma, \( f_{*}(\mcal{F})\otimes \mcal{G}\cong
f_{*}(\mcal{F}\otimes f^{*}\mcal{G}) \). 
\end{proof}
\begin{notation}\label{dfnFnaturalFunctor}
For any cohomologically smooth map \(h\), we write \(h_{\natural}(-)\cong h_{!}(-\otimes h^{!}1)\)
for the left adjoint of \(h^{*}\). 
\end{notation}
\begin{lem}\label{lemSmBaseChange}
Assume that we have a Cartesian diagram (\ref{equationPullBack}).

(1) (Smooth base change) If \(g\) is cohomologically smooth, then
the natural morphisms \(g^{*}f_{*}\to f'_{*}g^{\prime,*}\), \( g'_{\natural}f^{\prime,*}\to f^{*}g_{\natural} \), and \( g_{\natural}(\mF\otimes g^{*}(\mcal{G}))\to g_{\natural}(\mF)\otimes \mcal{G} \) 
are isomorphisms. 

(2) (Proper base change) If \( f \) is prim, then the natural morphisms \( g^{*}f_{*}\to f'_{*}g^{\prime,*} \) and \( f_{*}(\mF)\otimes \mcal{G}\to f_{*}(\mF\otimes f^{*}(\mcal{G})) \)
 are isomorphisms.  
If in addition \( g \) is \( ! \)-able, then the natural morphism \( g_{!}f'_{*}\to f_{*}g'_{!} \) is an isomorphism.
\end{lem}
\begin{proof}
For (1),
the natural morphism
is induced by \(f_{*}\to g_{*}f_{*}'g^{\prime,*}\cong f_{*}g'_{*}g^{\prime,*}.\)
Now note that 
\begin{align*}
f_{*}'g^{\prime,*}(\mF)&\cong f_{*}'(g^{\prime,!}(\mF)\otimes (g^{\prime,!}1)^{-1})
\cong f_{*}'(g^{\prime,!}(\mF)\otimes (f^{\prime,*}g^{!}1)^{-1})\\
&\cong
f_{*}'g^{\prime,!}(\mF)\otimes (g^{!}1)^{-1}
\cong
g^{!}f_{*}(\mF)\otimes (g^{!}1)^{-1}
\cong 
g^{*}f_{*}(\mF),
\end{align*}
where the third isomorphism is given by Lemma \ref{lemProjectionFormuaForPerf}.
The second isomorphism is given by taking left adjoint. 
The last isomorphism is given by projection formula.

For (2), we have 
\begin{align*}
&g^{*}f_{*}(\mF)
\cong g^{*}f_{!}(\mF\otimes \mbb{P}_{f}(\mO_{X}))
\cong f'_{!}g^{\prime,*}(\mF\otimes \mbb{P}_{f}(\mO_{X}))\\
&\cong f'_{!}(g^{\prime,*}(\mF)\otimes g^{\prime,*}(\mbb{P}_{f}(\mO_{X})))
\cong f'_{!}(g^{\prime,*}(\mF)\otimes \mbb{P}_{f'}(\mO_{X}))
\cong f'_{*}g^{\prime,*}(\mF).
\end{align*}
Moreover, 
\begin{align*}
f_{*}(\mF)\otimes \mcal{G}
\cong f_{!}(\mF\otimes \mbb{P}_{f}(\mO_{X}))
\otimes \mcal{G}
\cong f_{!}(\mF\otimes \mbb{P}_{f}(\mO_{X})\otimes f^{*}(\mcal{G}))
\cong f_{*}(\mF\otimes f^{*}(\mcal{G})).
\end{align*}
If \( g \) is \( ! \)-able, we have \[
g_{!}f'_{*}\cong g_{!}f'_{!}(-\otimes \mbb{P}_{f'}(\mO_{X'}))
\cong f_{!}g'_{!}(-\otimes (g')^{*}\mbb{P}_{f}(\mO_{X}))
\cong f_{!}(g'_{!}(-)\otimes \mbb{P}_{f}(\mO_{X}))
\cong f_{*}g'_{!}.
\]
\end{proof}
\begin{lem}[K\"unneth formula]
	\label{lemKunnethProper}
Let \( Z \) be an analytic stack over \( \mscr{C} \). 
Consider 
\( f:X\to X' \) and \( g:Y\to Y' \) in \( \mrm{AnStk}(\mscr{C})_{/Z} \). Let \( \mF\in \QCoh(X) \)
and \( \mcal{G}\in \QCoh(Y) \). Then \[
(f\times g)_{!}(\mF\boxtimes \mcal{G})\cong f_{!}(\mF)\boxtimes g_{!}(\mcal{G})\in\QCoh(X'\times_{Z}Y').
\]

If \( f \)
and \( g \) are prim, then \[
(f\times g)_{*}(\mF\boxtimes \mcal{G})\cong f_{*}(\mF)\boxtimes g_{*}(\mcal{G})\in\QCoh(X'\times_{Z}Y').
\]
\end{lem}
\begin{proof}
We can reduce to the case when \( f=\id:X\to X'=X \)
or \( g=\id \). Without loss of generality, let us assume that \( g=\id \). Then \[
(f\times 1)_{!}(\mF\boxtimes \mcal{G})
\cong 
(f\times 1)_{!}(p_{1}^{*}\mF\otimes (f\times 1)^{*}p_{2}^{*}\mcal{G})
\cong (f\times 1)_{!}p_{1}^{*}(\mF)\otimes p_{2}^{*}\mcal{G}
\cong p_{1}^{*}f_{!}(\mF)\otimes p_{2}^{*}\mcal{G},
\] where the second isomorphism is given by projection formula, and the last isomorphism is given by proper base change.
The proof of the other isomorphism is similar, using Lemma \ref{lemSmBaseChange} (2).
\end{proof}

We will also need the following notion:
\begin{dfn}\label{dfnFredholm}
An analytic ring \( A=(A^{\rhd},D_{\ge0}(A)) \) 
is \emph{Fredholm} if all perfect complexes over \( A \) (i.e. dualizable objects in \( D(A) \))
arises from (discrete) perfect complexes over \( A^{\rhd}(*) \). 
\end{dfn}
As an example:
\begin{thm}\label{thmSolidTateIsFredholm}
If \( A \) is a solid analytic ring, such that \( A^{\rhd} \) is Tate (i.e. \(  A^{\rhd} \) has a ring of definition  \( A^{\rhd,\circ} \) and a pseudo-uniformizer \( \varpi \) such that \( A^{\rhd,\circ} \) is \( \varpi \)-complete and \( A^{\rhd}=A^{\rhd,\circ}[1/\varpi]  \)),
then \( A \) is Fredholm.
\end{thm}
\begin{proof}
This is \cite[Theorem 5.50]{Andreychev21pseudocoherent}.
\end{proof}
\begin{lem}\label{lemFredholmVectorBundle}
If \( A \) is Fredholm, then any vector bundle in \( \QCoh(\AnSpec(A))=D(A) \) (Definition \ref{dfnPerfectComplex} (2)) arises from a finite projective module over \( A^{\rhd}(*) \)
\end{lem}
\begin{proof}
Assume that \( M \) arises from a perfect complex \( M^{\delta}\in D(A(*)) \), such that there is a \( ! \)-surjection \( \AnSpec(B)\to \AnSpec(A) \), \( M\otimes_{A}B\cong B^{\oplus d} \), 
then by \( * \)-descent (Theorem \ref{thmVariousDescent} (6)),  \( M\cong \varprojlim_{n}f_{n,*}f_{n}^{*}(M) \) for \( f_{n}:\AnSpec(B^{\otimes_{A}(n+1)})\to \AnSpec(A) \), and 
\( f_{n}^{*}(M)\cong (B^{\otimes_{A}(n+1)})^{\oplus d} \). 
Thus \( M^{\delta} \)     
is connective. Applying the same argument to \( M^{\vee} \), 
 \( (M^{\delta})^{\vee} \) is also connective. These implies that \( M^{\delta} \) is a finite projective module, as desired.
\end{proof}

\subsection{Solid and Tate stacks}\label{subsecSolidTateStack}
We will mainly work over the following classes of analytic rings. Restricting to these classes allows us to define algebraic and analytic de Rham stacks as in \cite{Juan2024analyticdeRham}.
\begin{dfn}[{\cite[Definition 2.6.1, 2.6.6, 2.6.10]{Juan2024analyticdeRham}}]\label{dfnBoundedRing}---

	(1)
	For any analytic ring \(A\) over \( \ZZ_{\square} \),
	define \(A^{+}:=\underline{\mrm{Map}}(\ZZ[T]_{\square},A)\),
	where 
	\( \ZZ[T]_{\square} \) 
	is the analytic ring defined in \cite[Example 7.3]{CS19}, and 
	the \(\underline{\mrm{Map}}\)
	is the internal Hom space of analytic rings, which has an inherited animated ring structure.

	(2) An analytic ring \( A=(A^{\rhd},D_{\ge0}(A)) \) over \( \ZZ_{\square} \)  is \emph{solid} if  
	natural map \( (A^{\rhd},\pi_{0}(A^{+}))_{\square}\to A \) is an equivalence of analytic rings. In other words, \( \forall M\in D_{\ge0}(A^{\rhd}) \), 
	\( M\in D_{\ge0}(A) \) 
	if and only if \( \forall f\in \pi_{0}(A^{+})(*) \) inducing \( \ZZ[T]\to A^{\rhd} \),  
	\( M|_{\ZZ[T]}\in D_{\ge0}(\ZZ[T]_{\square}) \).

	(3) 
	An {animated condensed ring} \(A^{\rhd}\) over \( \QQ_{p} \)  is \emph{bounded}  if \(A^{\rhd}\cong A^{\rhd,\circ}[1/p]\)
	with \(A^{\rhd,\circ}(S):=\mrm{Map}(\ZZ_{p}\langle \NN[S]\rangle_{\square},A)\) for \(S\) extremely disconnected. 
	
	(4)
	An analytic ring \(A=(A^{\rhd},D_{\ge0}(A))\) over \((\QQ_{p},\ZZ_{p})_{\square}\)
	is \emph{bounded} if \( A \) is solid and \( A^{\rhd} \) is bounded.

	(5) We fix an uncountable cutoff cardinal \( \kappa \)  (\cite[Definition 2.9.11]{Mann2022padic6functorformalismrigidanalytic}). A solid analytic ring \( A \) is \emph{\(\kappa\)-small} 
	if the underlying condensed ring \( A^{\rhd} \) 
	is \( \kappa \)-small. 
\end{dfn}
\begin{dfn}[Solid stacks, {\cite[Definition 3.2.7]{Juan2024analyticdeRham}}]\label{dfnSolidStk}
Let \(\mscr{C}:=\mrm{AnRing}_{\ZZ_{\square},\kappa}\) be
	the category of \( \kappa \)-small solid analytic rings  over \(\ZZ_{\square}\). A (\( \kappa \)-small) \emph{solid stack} is defined to be an analytic stack over \(\mrm{AnRing}_{\ZZ_{\square},\kappa}\) (Definition \ref{dfnAnStkOverC}), and we denote \(\mrm{SolidStk}_{\kappa}:=\mrm{AnStk}(\mrm{AnRing}_{\ZZ_{\square},\kappa})\). We will omit the subscript \( \kappa \) from now on. 

	For \( A\in \mrm{AnRing}_{\ZZ_{\square}} \), let \( \AnSpec(A) \) be the sheaf co-represented by \( A \) by Yoneda Lemma. 
\end{dfn}
\begin{dfn}[Tate stacks, {\cite[Definition 3.2.10]{Juan2024analyticdeRham}}]\label{dfnTateStk}
Let \(\mscr{C}:=\mrm{AnRing}^{b}_{\QQ_{p},\kappa}\) be
	the category of \( \kappa \)-small bounded analytic rings over \(\QQ_{p}\). A (\( \kappa \)-small) \emph{Tate stack} over \(\QQ_{p}\) is defined to be an analytic stack over \(\mrm{AnRing}^{b}_{\QQ_{p},\kappa}\) (Definition \ref{dfnAnStkOverC}), and we denote \(\mrm{TateStk}_{\QQ_{p}}:=\mrm{AnStk}(\mrm{AnRing}^{b,op}_{\QQ_{p},\kappa})\).
	We will omit the subscript \( \kappa \) from now on. 

	For \( A\in \mrm{AnRing}_{\QQ_{p}}^{b} \), let \( \AnSpec^{b}(A) \) be the sheaf co-represented by \( A \) by Yoneda Lemma. 
\end{dfn}

\begin{lem}\label{lemSolidToTateAnalyticficationFunctor}
(1)
There is a morphism of \( \infty \)-topoi whose right adjoint is given by \[(-)^{Tate}:\mrm{SolidStk}\to \mrm{TateStk}_{\QQ_{p}},\;X\mapsto X|_{\mrm{AnRing}^{b}_{\QQ_{p}}},
\] and whose left exact left adjoint is given by \[
(-)_{Solid}:\mrm{TateStk}_{\QQ_{p}}\to \mrm{SolidStk},\;Y\mapsto \varinjlim_{\AnSpec^{b}(A)\to Y} \AnSpec(A).
\]

(2)
For \( A\in \mrm{AnRing}_{\QQ_{p}}^{b} \), \[ (\AnSpec(A))^{Tate}\cong\AnSpec^{b}(A), (\AnSpec^{b}(A))_{Solid}\cong \AnSpec(A). \]

(3) For any \( X\in\mrm{TateStk}_{\QQ_{p}} \),  \(\QCoh((X_{Solid})^{Tate}) \cong \QCoh(X_{Solid})\cong\QCoh(X) \), 
which is compatible with six functor formalisms.
\end{lem}
\begin{proof}
The proof is the same as that of \cite[Lemma 4.2.7]{AnschützBoscoLeBrasCamargoScholze2025analyticrhamstacksfarguesfontaine}.
We have a full faithful colimit-preserving embedding \( \iota:\mrm{AnRing}^{b}_{\QQ_{p}}\hookrightarrow \mrm{AnRing}_{\ZZ_{\square}} \), 
which induces  \[ \iota_{*}: \mcal{P}(\mrm{AnRing}^{op}_{\ZZ_{\square}}) \to
\mcal{P}(\mrm{AnRing}^{b,op}_{\QQ_{p}}),\;F\mapsto F|_{\mrm{AnRing}^{b}_{\QQ_{p}}}.\]
This functor admits a left adjoint \[
\iota^{*}:\mcal{P}(\mrm{AnRing}^{b,op}_{\QQ_{p}})\to \mcal{P}(\mrm{AnRing}^{op}_{\ZZ_{\square}}),\;
F\mapsto \iota^{*}(F):= \varinjlim_{\AnSpec^{b}(A)\to F}\AnSpec(A).
\] It suffices to show that the functor commutes with fiber products, and \( \QCoh(-) \), and takes \( ! \)-equivalences to \( ! \)-equivalences, which is clear.
\end{proof}
\begin{dfn}\label{dfnEssentialTate}
We say that \( X\in\mrm{SolidStk} \) 
is \emph{essentially Tate} (over \( \QQ_{p} \))
if the natural morphism \( (X^{Tate})_{Solid}\to X \) is an isomorphism. We denote the full subcategory of essentially Tate solid stacks as \( \mrm{SolidStk}^{b}_{\QQ_{p}}\subset \mrm{SolidStk} \). 
\end{dfn}
\begin{lem}\label{lemBasicPropertyEsseTate}
(1) If \( X\in\mrm{SolidStk}^{b}_{\QQ_{p}} \), then \( \QCoh(X)\cong \QCoh(X^{Tate}) \), and \( \Perf(X)\cong \Perf(X^{Tate}) \) which is compatible with six functor formalisms. 

If, in addition,
there is a \( ! \)-surjection \( f:\coprod_{i}\AnSpec(A_{i})\to X \) such that \( A_{i} \)
are bounded and Fredholm, 
 then \( \Vect(X)\cong \Vect(X^{Tate}) \).

(2) The subcategory \( \mrm{SolidStk}^{b}_{\QQ_{p}}\subset \mrm{SolidStk} \) 
is closed under finite limits.

(3) For any \( A\in \mrm{AnRing}^{b}_{\QQ_{p}} \), then \( \AnSpec(A)\in \mrm{SolidStk}^{b}_{\QQ_{p}} \).

(4) \( (-)^{Tate}|_{\mrm{SolidStk}^{b}_{\QQ_{p}}} \) is fully faithful. More generally,
if \( X\in \mrm{SolidStk}^{b}_{\QQ_{p}} \) and \( Y\in \mrm{SolidStk} \),  
then \( \Hom_{\mrm{SolidStk}}(X,Y)\cong \Hom_{\mrm{TateStk}_{\QQ_{p}}}(X^{Tate},Y^{Tate})  \). 

(5) If \( f:X\to Y \) in \( \mrm{SolidStk}^{b}_{\QQ_{p}} \) 
satisfies suave (resp. prim) \( ! \)-descent (Definition \ref{dfnVariousPropertyMorphiDescent}), then \( f^{Tate}:X^{Tate}\to Y^{Tate} \)  
also satisfies suave (resp. prim) \( ! \)-descent.

(6) Assume that \( f:X\to Y \) is a suave \( ! \)-cover between \( \omega_{1} \)-compact Tate stacks. For any \( n\in \ZZ_{\ge1} \), assume that the \( n \)-th fiber product  \( X^{n/Y} \) satisfies the property that the natural map \( X^{n/Y}\to ((X^{n/Y})_{Solid})^{Tate} \)     
is an isomorphism. Then \( Y\to( Y_{Solid})^{Tate} \) is also an isomorphism. In particular, \( Y_{Solid}\in \mrm{SolidStk}^{b}_{\QQ_{p}} \). 
\end{lem}
\begin{rmk}\label{rmkEssTateOmitTate}
By Lemma \ref{lemBasicPropertyEsseTate} (4), we will identify \( \mrm{SolidStk}^{b}_{\QQ_{p}} \)  with its essential image under \( (-)^{Tate}|_{\mrm{SolidStk}^{b}_{\QQ_{p}}} \). 
For \( X\in \mrm{SolidStk}^{b}_{\QQ_{p}} \), by abuse of notation, we write \( X:=X^{Tate} \).  
\end{rmk}
\begin{proof}
For (1), \( \QCoh(X^{Tate})\cong \QCoh((X^{Tate})_{Solid}) \) by Lemma \ref{lemSolidToTateAnalyticficationFunctor}, which is symmetric monoidal by \( * \)-descent. So  \( \Perf(X^{Tate})\cong \Perf((X^{Tate})_{Solid}) \).

If there is a \( ! \)-surjection \( f:\coprod_{i}\AnSpec(A_{i})\to X \) such that \( A_{i} \)
are bounded and Fredholm, 
then for any \( \mF\in \Perf(X) \) (resp. \( \mF\in\Perf(X^{Tate}) \)), \( \mF \)
is a vector bundle if and only if \( f^{*}(\mF)\in \Vect(\coprod_{i}\AnSpec(A_{i})) \) is (resp. \( \Vect(\coprod_{i}\AnSpec^{b}(A_{i})) \)). However, by Lemma \ref{lemFredholmVectorBundle}, we have \( \Vect(\coprod_{i}\AnSpec(A_{i}))\cong \Vect(\coprod_{i}\AnSpec^{b}(A_{i}))\cong \Vect(A^{\rhd}(*)) \).

The statements (2)-(4) follow immediately from Lemma \ref{lemSolidToTateAnalyticficationFunctor}
since  and both \( (-)^{Tate} \) 
and \( (-)_{Solid} \) 
commute with finite limits.
(5) follows from (1) and (2).


For (6), we have \(\varinjlim_{[n]\in\Delta^{op}}(X^{n/Y})\cong Y \), \[  \varinjlim_{[n]\in\Delta^{op}}(X^{n/Y})_{Solid}\cong \varinjlim_{[n]\in\Delta^{op}}(X_{Solid})^{n/Y_{Solid}}
\cong
Y_{Solid} \]
and 
\begin{align*}
i:Y\cong 
\varinjlim_{[n]\in\Delta^{op}}((X^{n/Y})_{Solid})^{Tate}\cong \varinjlim_{[n]\in\Delta^{op}}((X_{Solid})^{Tate})^{n/(Y_{Solid})^{Tate}}\\\hookrightarrow (Y_{Solid})^{Tate},
\end{align*}
where the last morphism is a monomorphism.
Note that by Lemma \ref{lemSolidToTateAnalyticficationFunctor}, \[ i^{*}:\QCoh((Y_{Solid})^{Tate})\cong \QCoh(Y). \]

We claim that \( i \)  satisfies suave \( ! \)-descent. This will conclude the proof by Lemma \ref{thmVariousDescent} (7).
We know that \( p_{i}:X\times_{Y}X\to X \) satisfy suave \( ! \)-descent and are \( ! \)-surjections 
for \( i=1,2 \). 
By the equivalent definitions in Definition \ref{dfnVariousPropertyMorphiDescent}
and by Lemma \ref{lemSolidToTateAnalyticficationFunctor}, 
\( p_{i}:((X\times_{Y}X)_{Solid})^{Tate}\to (X_{Solid} )^{Tate}\) also satisfy suave \( ! \)-descent. 
By Lemma \ref{lemSolidToTateAnalyticficationFunctor}, \( X\to (Y_{Solid})^{Tate} \) 
also satisfies universal \( * \)-descent. 
Therefore, \( X\to (Y_{Solid})^{Tate} \) also satisfies suave \( ! \)-descent by Lemma \ref{thmVariousDescent} (4). By Lemma \ref{thmVariousDescent} (5), \( i \) 
is also a suave \( ! \)-cover, as desired. 
\end{proof}
\begin{eg}\label{egEssentialTate}
(1) If \( A\in\mrm{AnRing}^{b}_{\QQ_{p}} \), 
then \( \AnSpec(A)\)  is essentially Tate. 

(2) Consider \( \mA^{1}:=\AnSpec(\QQ_{p,\square}[t]) \), 
then \[ (\mA^{1})^{Tate}\cong \mA^{1,\an}:=\varinjlim_{n}\AnSpec^{b}(\QQ_{p}\langle p^{n}t\rangle), \]  and \(\mA^{1,\an} \) is essentially Tate.

Similarly, for \( \GG_{m}:=\AnSpec(\QQ_{p,\square}[t^{\pm1}]) \), 
\( \GG_{m}^{\an}:=((\GG_{m})^{Tate})_{Solid}\cong \varinjlim\AnSpec(\QQ_{p}\langle p^{n}t,\frac{p^{n}}{t}\rangle) \) is essentially Tate.

(3) Consider \( \hat{\mA}^{1}:=\varinjlim_{n}\AnSpec(\QQ_{p,\square}[t]/t^{n}) \). 
Then \( \hat{\mA}^{1}\to \mA^{1,\an} \)   
is a monomorphism, 
and \( 0=\AnSpec(\QQ_{p,\square})\to \hat{\mA}^{1} \) 
is a prim \( ! \)-cover, being colimits of descendable covers by \cite[Proposition 3.35]{Mathew2016galois}. 
Thus for any \( B\in\mrm{AnRing}^{b}_{\QQ_{p}} \), 
\( \hat{\mA}^{1}(B) \) 
is the full sub-anima of \( \mA^{1}(B) \), whose \( \pi_{0} \)  consists of \( f\in \pi_{0}(B^{\rhd}) \)  
such that \( B^{\rhd}\to B^{\rhd}/f \) 
is descendable. Note that \( B^{\rhd}/f \) 
is also bounded. Therefore, \( (\hat{\mA}^{1})^{Tate}\cong \varinjlim_{n} \AnSpec^{b}(\QQ_{p,\square}[t]/t^{n}) \), and \( \hat{\mA}^{1} \)  
is essentially Tate.

(4) Assume that there is a squence of bounded rings \( A_{n}\in\mrm{AnRing}^{b}_{\QQ_{p}} \) 
\[
\AnSpec(A_{0})\hookrightarrow \AnSpec(A_{1})\hookrightarrow \cdot \hookrightarrow\AnSpec(A_{n})\hookrightarrow\cdot
\] such that the transition maps are monomorphisms. Then \(X:=\varinjlim_{n}\AnSpec(A_{n})\)
is essentially Tate:

Given any \( B\in\mrm{AnRing}^{b}_{\QQ_{p}} \), 
and \( f:\AnSpec(B)\to X \), there exists a \( ! \)-cover \(g: \AnSpec(C) \to \AnSpec(B)\) and a factorization of \( f\circ g \) as \( \AnSpec(C)\to \AnSpec(A_{n}) \).   
Thus \( g \) 
factors through \( \AnSpec(B)\times_{X}\AnSpec(A_{n})\subset \AnSpec(B) \).
Since \( g \) is a \( ! \)-surjection and \( \AnSpec(A_{n})\to X \) is a monomorphism, \( f \) uniquely
factors through \( \AnSpec(A_{n}) \).  This implies that 
\( X^{Tate}\cong \varinjlim_{n}\AnSpec^{b}(A_{n}) \), and thus \( (X^{Tate})_{Solid}\cong X \). 

(5) By (2), (3) and Lemma \ref{lemBasicPropertyEsseTate} (7), we know that \( B\GG_{m}^{\an}\), \( \hat{\mA}^{1}/\GG_{m}^{\an} \)  and \( \mA^{1,\an}/\GG_{m}^{\an}\) are essentially Tate.
\end{eg}







\subsection{Affine proper morphisms}\label{subsecAffineProperMorph}
For our later purpose,
we need the following class of morphisms:
\begin{dfn}\label{dfnAffineProper}
Let \(f:X\to Y\)
be a morphism between solid stacks.
We say \(f\)
is \emph{affine proper}
if there exists a \(!\)-surjection \(Y'=\coprod_{i\in I}\AnSpec(A_{i})\to Y\),
such that \(X\times_{Y}\AnSpec(A_{i})\cong \AnSpec(B_{i})\),
where \(A_{i}\to B_{i}\) are 
proper morphisms between analytic rings.

We say that \(f\) is \emph{affine proper and descendable} (of index \(\le n\))
if in addition, one can take \(Y'\to Y\) to be an open covering, and for each \(i\), \(A_{i}\to B_{i}\)
is descendable (of index \(\le n\)) in the sense of \cite{Mathew2016galois}.
\end{dfn}
For later purpose,
we introduce the following notion:
\begin{dfn}
Let \(f:X\to Y\)
be a morphism between solid stacks. We say that \(f\) is \emph{essentially affine proper and descendable} (of index \(\le n\)), if there exists \(g:X'\to X\)
such that \(f\circ g:X'\to Y\)
is affine proper, and descendable (of index \(\le n\)).
\end{dfn}
\begin{prop}[Clausen-Scholze]
	\label{propAffineProperCover}
	If \(f\)
	is essentially affine proper and descendable,
	then \(f\) is a \(!\)-surjection and is cohomologically co-smooth.
\end{prop}
\begin{proof}
By definition,
we are reduced to the case of a proper and descendable morphism \(f:X=\AnSpec(B)\to Y= \AnSpec(A)\). It is a \(!\)-cover directly by the descendability assumption. Then we conclude by Theorem \ref{thmVariousDescent} (3).
The cohomologically co-smooth is given by Theorem \ref{thmVariousDescent} (4) and (6). 
\end{proof}

We will need the following results:
\begin{lem}\label{lemAffProperProp}
(1) The class of affine proper morphisms is closed under pull-back. Moreover,
if \(Y_{0}\to Y\)
is any \(!\)-surjection,
such that \(X\times_{Y}Y_{0}\to Y_{0}\)
is affine proper,
then so is \(X\to Y\). 
\cref{lemAffProperProp}

(2) If \(Y=\AnSpec(A)\),
and \(f:X\to Y\)
is affine proper,
such that \(B^{\rhd}=f_{*}\mO_{X}\in D(A)\)
lies in \(D_{\ge0}(A)\),
then \(X\cong \AnSpec(B^{\rhd},A)\),
where we recall that \((B^{\rhd},A)\)
denotes the analytic ring whose underlying ring is \(B^{\rhd}\)
and equipped with the analytic structure induced by \(A\).

(3) (Relative \(\unl{\AnSpec}\) construction) Let \(X\) be
a solid stack over \(\QQ\),
and \(B\in \mrm{CAlg}(\QCoh(X))\). 
Then the presheaf \[ \mcal{P}(\mrm{AnRing}_{\ZZ_{\square}})_{/X}\to \mrm{Ani},\;(Z\xrightarrow{s_{Z}}X)\mapsto \Hom_{\mrm{CAlg}(\QCoh(X))}(B,s_{Z,*}(\mO_{Z})) \] 
is a solid stack over \(X\), which we denote as \( \underline{\AnSpec}_{X}(B) \). 

For any \( f:Y\to X \), we have \( Y\times_{X}\unl{\AnSpec}_{X}(B)\cong \unl{\AnSpec}_{Y}(f^{*}(B)) \).  

Moreover, if \( X=\AnSpec(A) \) 
and \( B\in D_{\ge0}(A) \), then  
\( \unl{\AnSpec}_{X}(B)\cong \AnSpec(B,A) \), where the RHS is as in (2). 
\end{lem}

\begin{rmk}
	Note that we are working over \(\QQ\).
	Since \(f_{*}\)
	is lax symmetric monoidal,
	\(B^{\rhd}\)
	has a natural \(\mbb{E}_{\infty}\)-structure,
	which is the same as an animated \(A^{\rhd}\)-algebra structure in characteristic \(0\).
\end{rmk}
\begin{proof}
For (1), if \( X\to Y \) is affine proper with a \(!\)-surjection \( Y'\to Y \) as in Definition \ref{dfnAffineProper}.
Then 
we can find a \(!\)-surjection \(Y_{0}'\to Y_{0}\)
such that the morphism \(Y_{0}\to Y\)
lifts to \(Y_{0}'\to Y'\).
After further refinement of \(Y_{0}'\),
we can assume that \(Y_{0}'\cong \coprod_{i\in I,j\in J_{i}}\AnSpec(A_{i,j})\)
with \(\AnSpec(A_{i,j})\)
maps to \(\AnSpec(A_{i})\)
via \(Y_{0}'\to Y_{0}\).
Then the pull-back of \(X\times_{Y}Y_{0}\to Y_{0}\)
to \(\AnSpec(A_{i,j})\)
is given by 
\(\AnSpec(A_{i,j}\otimes_{A_{i}}B_{i})\).
This shows that \(X\times_{Y}Y_{0}\to Y_{0}\)
is affine proper. The other direction is obvious.

For
(2), since \(Y\) is quasi-compact,
we can find a \( ! \)-surjection  \(h:Y'=\AnSpec(A')\to Y\),
such that \(X\times_{Y}Y'\cong \AnSpec(B^{\prime,\rhd},A)\).
By proper base change,
we know \(h^{*}B^{\rhd}\cong B^{\prime,\rhd}\). Denote \(Z:=\AnSpec(B^{\rhd},A)\),
then there is a natural factorization \(X\to Z\to Y,\)
such that \(X\times_{Y}Y'\cong Z\times_{Y}Y'\). Since \( Y'\to Y \) is a \( ! \)-surjection, we know that \( X\cong Z \).  

For (3), note that \(
(Z\xrightarrow{s_{Z}}X)\mapsto \Hom_{\mrm{CAlg}(\QCoh(X))}(B,s_{Z,*}(\mO_{Z}))
\) takes \( ! \)-equivalences to equivalences: if \( g:Z\to Z' \)  
is a \( ! \)-equivalences over \( X \),  then \( g^{*}:\QCoh(Z')\to \QCoh(Z) \) 
is an equivalence (\cite[Remark 4.2.3]{AnschützBoscoLeBrasCamargoScholze2025analyticrhamstacksfarguesfontaine}), 
with the inverse given by \( g_{*} \), 
and thus \( \mO_{Z'}\to g_{*}\mO_{Z} \)  is an isomorphism.

The compatibility with pull-back's follows from the adjunction \( (f^{*},f_{*}) \). Finally, if \( X=\AnSpec(A) \) 
and \( B\in D_{\ge0}(A) \), there is a natural map \( \AnSpec(B,A)\to \unl{\AnSpec}_{X}(B) \). To verify that it is an equivalence, one can verify at each affinoid, which is clear. 
\end{proof}


\begin{dfn}\label{dfnUniformAffineProper}
We say a tower of solid stacks \[
\cdots\xrightarrow{f_{m}} X_{m}\xrightarrow{f_{m-1}} X_{m-1}\xrightarrow{f_{m-2}} \cdots \xrightarrow{f_{1}} X_{1}\xrightarrow{f_{0}} X_{0}.
\] is \emph{uniformly affine proper} if 
there exists \( m_{0}\in\NN \) 
and a map \[Y_{m_{0}}=\coprod_{i\in I}\AnSpec(A_{i}^{[m_{0}]})\to X_{m_{0}}\] that satisfies universal \( ! \)-descent,
such that for any \( m\ge m_{0} \), 
 \(X_{m}\times_{X_{m_{0}}}\AnSpec(A_{i}^{[m_{0}]})\cong \AnSpec(A_{i}^{[m]})\),
and \(A_{i}^{[m_{0}]}\to A_{i}^{[m]}\) is a 
proper morphism between analytic rings.
\end{dfn}
\begin{lem}\label{lemUniformAffineProper}
Let \((X_{m},f_{m})\) be a uniformly affine proper tower of solid stacks. Let \( X^{\sm}:=\varprojlim_{m}X_{m} \).  
Then we have an isomorphism in \( \mrm{Pr}^{L} \)  \[
\QCoh(X^{\sm})\cong \varinjlim_{m}\QCoh(X_{m}),
\] where the transition maps are given by \((-)^{*}\).
\end{lem}
\begin{proof}
(1)
Assume first that we can find a tower \( (Z_{m},g_{m}) \) such that \( Z_{m}\cong\AnSpec(B_{m}) \), \( B_{m}\to B_{m+1} \) is proper, with a map \( X_{m_{0}}\to Z_{m_{0}} \), such that 
\( X_{m}\cong X_{m_{0}}\times_{Z_{m_{0}}}Z_{m} \).
Then \( \QCoh(X_{m})\cong \mrm{CAlg}_{B_{m}}(\QCoh(X_{m_{0}})) \).

Taking \( \varprojlim_{i} \), \( Z^{\sm}\cong \AnSpec(B^{\sm}) \) with \( B^{\sm}\cong \varinjlim_{m}B_{m} \), and \( X^{\sm}\cong X_{m,0}\times_{Z_{m,0}}Z^{\sm} \). Thus \( \QCoh(X^{\sm})\cong\mrm{CAlg}_{B^{\sm}}(\QCoh(X_{m_{0}}))  \).  
Then 
\begin{align*}
\QCoh(X^{\sm})\cong \mrm{CAlg}_{B^{\sm}}(\QCoh(X_{m_{0}})) \cong \varinjlim_{m}\mrm{CAlg}_{B_{m}}(\QCoh(X_{m_{0}}))\\\cong \varinjlim_{m}\QCoh(X_{m}),
\end{align*}
where the second isomorphism follows from the equivalence \(\mrm{Pr}^{L}\cong(\mrm{Pr}^{R})^{op}\) (\cite[Corollary 5.5.3.4]{Lurie2009HTT}) and \cite[Proposition 7.1.2.7]{Lurie2017HA}.

(2) In general, we consider the \v Cech nerve associated to \( Y_{m_{0}}\to X_{m_{0}} \). 
Let \( Y_{m_{0}}^{[n]}:=Y_{m_{0}}^{n/X_{m_{0}}} \), 
 \( Y_{m}^{[n]}:=Y_{m_{0}}^{[n]}\times_{X_{m_{0}}}X_{m} \) and \(Y^{\sm,[n]}:=\varprojlim_{m}Y_{m}^{[n]}\cong Y_{m_{0}}^{[n]}\times_{X_{m_{0}}}X^{\sm}\). 

 In what follows, transition maps in \(\varinjlim^{!}\) are given by \((-)_{!}\) and those in \(\varinjlim^{*}\) are given by \((-)^{*}\).
 
 Then by \( ! \)-descent, \[
 \QCoh(X_{m})\cong \varinjlim^{!}_{[n]\in\Delta^{op}}\QCoh(Y_{m}^{[n]}).
 \] 

 On the other hand, by (1), \[
 \QCoh(Y^{\sm,[n]})\cong \varinjlim^{*}_{m}\QCoh(Y_{m}^{[n]}).
 \]
Thus by proper base change,
\begin{align*}
\QCoh(X^{\sm})
&\cong \varinjlim^{!}_{[n]\in\Delta^{op}}\QCoh(Y^{\sm,[n]})
\cong \varinjlim^{!}_{[n]\in\Delta^{op}}\varinjlim^{*}_{m}\QCoh(Y^{[n]}_{m})\\
&\cong
\varinjlim^{*}_{K_{p}}\varinjlim^{!}_{[n]\in\Delta^{op}}\QCoh(Y^{[n]}_{m})
\cong \varinjlim^{*}_{K_{p}}\QCoh(X_{m}),
\end{align*} as desired.
\end{proof}

\subsection{Filtration and formal spectrum}\label{subsecFiltrationFormalSpec}
We recall the following discussion about the stacky approach to the filtered object. The result below is in the algebraic setting.
\begin{dfn}[Category of filtered objects]\label{dfnFiltraObj}
Let \(N\in \ZZ_{\ge1}\), and \(\ZZ^{N}\) denote the nerve of the 1-category, corresponding to the partially ordered set \(\ZZ^{N}\), with \(\underline{m}=(m_{1},\cdots,m_{N})\ge \underline{m}'= (m_{1}',\cdots,m_{N}')\) if and only if \(m_{i}\ge m_{i}'\) for any \(1\le i\le N\).
Let 
\((\ZZ^{N})^{ds}\) be the discrete category corresponds to the set \(\ZZ^{N}\),

Let \(\mscr{C}\)
be any stable \(\infty\)-category. We define the category of \emph{\(\ZZ^{N}\)-filtered objects} (resp. \( \ZZ^{N} \)-graded objects) in \(\mscr{C}\) as the category of functors from \((\ZZ^{N})^{op}\) (resp. \( (\ZZ^{N})^{ds} \)) to \(\mscr{C}\), which we denote by \(\Fil^{\ZZ^{N}}(\mscr{C})\) (resp. \( \gr^{\ZZ^{N}}(\mscr{C}) \)). We will refer to ``\(\ZZ\)-filtered" simply as ``filtered". 

For \( M\in\Fil^{\ZZ^{N}}(\mscr{C})  \) (resp. \( M\in \gr^{\ZZ^{N}}(\mscr{C}) \)), we write \( \Fil^{\unl{i}}(M) \) (resp. \( \gr^{\unl{i}}(M) \))
for the image of \( \unl{i}\in\ZZ^{N} \).



We denote by \(\Fil^{\NN^{N}}(\mscr{C})\) the full subcategory consisting of \(C\in \Fil^{\ZZ^{N}}(\mscr{C})\) such that \(C\cong \Fil^{\underline{j}}(C)\) unless \(\underline{j}> \underline{0}\). 

Assume in addition that \(\mscr{C}\) is co-complete, and is equipped with a symmetric monoidal structure. 
Then we can define the Day convolution symmetric monoidal structure on \(\Fil^{\ZZ^{N}}(\mscr{C})\), given by \[\Fil^{\underline{n}}
(C\otimes D):=\varinjlim_{\underline{i}+\underline{j}\ge \underline{n}}\Fil^{\underline{i}}(C)\otimes \Fil^{\underline{j}}(D).
\] 
Note that \(\Fil^{\NN^{N}}(\mscr{C})\subset \Fil^{\ZZ^{N}}(\mscr{C})\) is closed under the Day convolution tensor product, and \(\gr^{\underline{0}}:\Fil^{\NN^{N}}(\mscr{C})\to \mscr{C}\) is symmetric monoidal. 

Given the symmetric monoidal structure, we can talk about filtered commutative algebras \(A\in \mrm{CAlg}(\Fil^{\ZZ^{N}}(\mscr{C}))\), and for such an algebra \(A\), we can talk about the category of filtered modules \(\Mod_{A}(\Fil^{\ZZ^{N}}(\mscr{C}))\) over \(A\).

For \(n=1\), by \cite[Proposition 3.2.1]{Lurie2015rotation}, we have a symmetric monoidal functor \[\gr^{\bullet}:\Fil^{\ZZ}(\mscr{C})\to \gr^{\ZZ}(\mscr{C}),M\mapsto (i\mapsto  \mrm{cofib}(\Fil^{i+1}M\to \Fil^{i}M)).\]
Note that \( \Fil^{\ZZ}(\Fil^{\ZZ^{N-1}}(\mscr{C}))\cong\Fil^{\ZZ^{N}}(\mscr{C}) \), so we can define inductively a symmetric monoidal functor \( \gr^{\bullet}:\Fil^{\ZZ^{N}}(\mscr{C})\to \gr^{\ZZ^{N}}(\mscr{C}) \). We will write \( \gr^{\unl{i}}(M):=\gr^{\unl{i}}(\gr^{\bullet}(M)) \)
for \( M\in \Fil^{\ZZ^{N}}(\mscr{C}) \).
\end{dfn}

\begin{thm}[{\cite{MoulinosTasos2021Tgof}, \cite[Proposition 2.2.6]{BhattLurie2022prismatic}}]
	\label{thmFilterdatA1modGm}
Let \(R\) be a commutative ring, and 
\(\mA^{1}:=\Spec(R[t])\) and \( \GG_{m}:=\Spec(R[t^{\pm1}]) \).
Consider the stack \(\mA^{1}/\GG_{m}\) over \(R\). 

(1) There is a symmetric monoidal equivalence (``Rees construction'')
\[\mrm{Rees}:\Fil^{\ZZ}(D(R))\cong \QCoh([\mA^{1}/\GG_{m}]),\;(\mF,\Fil^{\bullet})\mapsto \bigoplus_{i\in\ZZ} \Fil^{i}(\mF)\cdot t^{-i},
\] 
such that for the maps in \[
\begin{tikzcd}
i:B\GG_{m}\arrow[r,hook] & \mA^{1}/\GG_{m}\arrow[d,"f"] & \arrow[l,hook]\GG_{m}/\GG_{m}\cong \Spec(R):j,\\
& \Spec(R)
\end{tikzcd}
\] 

(a) The functor \(i^{*}:\QCoh([\mA^{1}/\GG_{m}])\to \QCoh(B\GG_{m})\) is compatible with \( \gr^{\bullet}:\Fil^{\ZZ}(D(R))\to \gr^{\ZZ}(D(R)) \) via a symmetric monoidal equivalence \[\gr^{\ZZ}(D(R)):=\mrm{Fun}(\ZZ^{ds},D(R))\cong \QCoh(B\GG_{m}).
\] 

(b) \(j^{*}\) corresponds to the forgetful functor \( \Fil^{\ZZ}(D(R))\to D(R) \). 

(c) \( f_{*}: \QCoh([\mA^{1}/\GG_{m}])\to D(R) \) corresponds to \( M\mapsto \Fil^{0}(M )\), and
\( f^{*}:D(R)\to \QCoh([\mA^{1}/\GG_{m}]) \) corresponds to \( M\mapsto  \mrm{triv}^{0}(M)\) with \[
\mrm{triv}^{0}(M):i\mapsto \begin{cases}
M,&i\ge 0;\\
0,&i<0.
\end{cases}
\]

(2) \((\mF,\Fil^{\bullet})\in \Fil^{\ZZ}(D(R))\) is filtered complete if and only if \(\mrm{Rees}(\mF,\Fil^{\bullet})\) 
is derived \(t\)-complete, and we have an equivalence \[
\mrm{Rees}:\widehat{\Fil}^{\ZZ}(D(R))\cong \QCoh([\hat{\mA}^{1}/\GG_{m}]),
\] where \(\widehat{\Fil}^{\ZZ}(D(R))\)
denotes the subcategory of \(\Fil^{\ZZ}(D(R))\)
consisting of filtered complete objects,
and \(\hat{\mA}^{1}\) denotes the formal completion of \(\mA^{1}\) along the origin.

The natural inclusion \(\widehat{\Fil}^{\ZZ}(D(R))\subset {\Fil}^{\ZZ}(D(R))\)
is induced by \[\widehat{j}_{*}:\QCoh([\hat{\mA}^{1}/\GG_{m}])\to \QCoh([\mA^{1}/\GG_{m}])\]
for \(\widehat{j}:\hat{\mA}^{1}/\GG_{m}\to \mA^{1}/\GG_{m}\).
\end{thm}
\begin{rmk}\label{rmkReesToTrivial}
For any \( (\mF,\Fil^{\bullet})\in \Fil^{\NN}(D(R)) \) (i.e. \( \Fil^{n}\mF \cong \mF\) for \( n\in\ZZ_{\le 0} \)), then
there is a natural \( \GG_{m} \)-equivariant morphism \[ 
\mF\otimes_{R} R[t]\to \mrm{Rees}(\mF)\cong\bigoplus_{i\in\ZZ}\Fil^{i}(\mF)\cdot t^{-i},\] 
that becomes an isomorphism after inverting \( t \).
\end{rmk}

We go back to the setting of solid stacks:
\begin{dfn}
Let \( \mA^{1}:=\AnSpec(\ZZ_{\square}[t]) \), \( \GG_{m}:=\AnSpec(\ZZ_{\square}[t^{\pm1}]) \),  
and \[ \hat{\mA}^{1}:=\varinjlim_{n}\AnSpec(\ZZ_{\square}[t]/t^{n}) .\]
\end{dfn}
\begin{cor}\label{corFilterdatA1modGm}
Let \( X \) be any solid stack.  

(1) There is an equivalence 
\[\mrm{Rees}:\Fil^{\ZZ}(\QCoh(X))\cong \QCoh(X\times [\mA^{1}/\GG_{m}]),\;(\mF,\Fil^{\bullet})\mapsto \bigoplus_{i\in\ZZ} \Fil^{i}(\mF)\cdot t^{-i},
\] 
such that for the maps in \[
i:B\GG_{m}\hookrightarrow \mA^{1}/\GG_{m}\hookleftarrow \GG_{m}/\GG_{m}\cong \Spec(R):j,
\] \(i^{*}\) is compatible with \( \gr^{\bullet} \) an equivalence 
\[\gr^{\ZZ}(\QCoh(X))\cong \QCoh(X\times B\GG_{m}),
\] 
and 
\(j^{*}\) correponds to forgetting the filtration. 

(2) \((\mF,\Fil^{\bullet})\in \Fil^{\ZZ}(\QCoh(X))\) is filtered complete if and only if \(\mrm{Rees}(\mF,\Fil^{\bullet})\) 
is derived \(t\)-complete, and we have an equivalence \[
\mrm{Rees}:\widehat{\Fil}^{\ZZ}(\QCoh(X))\cong \QCoh(X\times [\hat{\mA}^{1}/\GG_{m}]),
\] where \(\widehat{\Fil}^{\ZZ}(\QCoh(X))\)
denotes the subcategory of \(\Fil^{\ZZ}(\QCoh(X))\)
consisting of filtered complete objects.

The natural inclusion \(\widehat{\Fil}^{\ZZ}(\QCoh(X))\subset {\Fil}^{\ZZ}(\QCoh(X))\)
is induced by \(\widehat{j}_{*}:\QCoh(X\times [\hat{\mA}^{1}/\GG_{m}])\to \QCoh(X\times [\mA^{1}/\GG_{m}])\)
for \(\widehat{j}:\hat{\mA}^{1}/\GG_{m}\to \mA^{1}/\GG_{m}\).
\end{cor}
\begin{proof}
Note that \begin{align*}
\QCoh(\mA^{1})\cong D(\ZZ_{\square})\otimes_{D(\ZZ)}(\mA^{1}),\;
\QCoh(\GG_{m})\cong D(\ZZ_{\square})\otimes_{D(\ZZ)}D(\ZZ[t]),\\
\QCoh(\hat{\mA}^{1})\cong D(\ZZ_{\square})\otimes_{D(\ZZ)}D_{t^{\infty}-\mrm{torsion}}D(\ZZ[[t]]).
\end{align*}
Therefore, \begin{align*}
\QCoh(X\times[\mA^{1}/\GG_{m}])&\cong \varprojlim^{*}_{n}\QCoh(X\times \mA^{1}\times \GG_{m}^{n})
\\&\cong \varprojlim^{*}_{n}\QCoh(X)\otimes_{D(\ZZ)}(D(\ZZ[t])\otimes_{D(\ZZ)} D(\ZZ[t^{\pm1}])^{\otimes n}) 
\\&\cong \QCoh(X)\otimes_{D(\ZZ)}\Fil^{\ZZ}(D(\ZZ))\cong \Fil^{\ZZ}(\QCoh(X)),
\end{align*} where the second isomorphism follows from the dualizability of \( D(\ZZ[t]) \) 
and \( D(\ZZ[t^{\pm1}]) \) 
over \( D(\ZZ) \) 
and the \( * \)-descent of \( \QCoh(X) \),  
and the third isomorphism follows from Theorem \ref{thmFilterdatA1modGm}.
The calculations of \( \QCoh(X\times B\GG_{m}) \) 
and \( \QCoh(X\times \hat{\mA}^{1}/\GG_{m}) \) 
are similar, which we omit here.
\end{proof}
\begin{notation}\label{notationTwistByAGm}\label{rmkEmbedFilterComplete}
(1)
We write \(\mO\{-1\}\) be the tautological line bundle on \([\mA^{1}/\GG_{m}]\) with an embedding \(\mO\{-1\}\hookrightarrow \mO\). 
For any \(f:X\to [\mA^{1}/\GG_{m}]\), and \(M\in\QCoh(X)\), we write \(M\{i\}:=M\otimes f^{*}(\mO\{-1\}^{\otimes(-i)})\). 

(2)
Given a solid stack \( X\to [\mA^{1}/\GG_{m}] \), and \( M\in \QCoh(X\times_{\mA^{1}/\GG_{m}}\hat{\mA}^{1}/\GG_{m}) \), we will always regard \( M \)   
as an object in \( \QCoh(X) \) 
via \( \widehat{j}_{*} \) for \( \widehat{j}: X\times_{\mA^{1}/\GG_{m}}\hat{\mA}^{1}/\GG_{m}\to X\). 

(3) Given \( M\in\QCoh(X) \), we  will always regard \( M \) (equipped with the trivial filtration)  
as an object in \( \QCoh(X\times [\mA^{1}/\GG_{m}]) \) 
by \( p_{1}^{*} \)  
for \( p_{1}: X\times [\mA^{1}/\GG_{m}]\to X\). 
\end{notation}


\begin{lem}\label{lemPerfComplexA1Gm}
Let \(X\in\mrm{SolidStk}\), and \[
X\times B\GG_{m}\xrightarrow{i}
X\times \hat{\mA}^{1}/\GG_{m}\xrightarrow{\widehat{j}} X\times \mA^{1}/\GG_{m}.
\] 

(1) \(i\) is a \(!\)-surjection, and \(i^{*}\) is conservative.



(2)
The adjunction
\( (\widehat{j}^{*},\widehat{j}_{*}) \) 
induces equivalences \( \Perf(X\times \hat{\mA}^{1}/\GG_{m}) \cong \Perf(X\times \mA^{1}/\GG_{m})	 \) and \( \Vect(X\times \hat{\mA}^{1}/\GG_{m}) \cong \Vect(X\times \mA^{1}/\GG_{m})	 \).
In particular, the natural map \(\mO_{X\times \mA^{1}/\GG_{m}}\to \widehat{j}_{*}(\mO_{X\times \hat{\mA}^{1}/\GG_{m}})\) is an isomorphism.
\end{lem}
\begin{proof}
For 
(1), the first part  follows from the fact that \(\AnSpec(\QQ_{p})\to \hat{\mA}^{1}\)
is a \(!\)-surjection, being colimits of descendable covers by \cite[Proposition 3.35]{Mathew2016galois}. 



For (2), 
note that \( \mO_{X} \) with the trivial filtration is filtered complete, and by Corollary, \ref{corFilterdatA1modGm} thus \( \widehat{j}_{*}(\widehat{j}^{*}(\mO_{X}))\cong \mO_{X} \). 
By Lemma \ref{lemProjectionFormuaForPerf}, for any \( M\in\Perf(X\times \mA^{1}/\GG_{m}) \), 
the natural map \( M\to \widehat{j}_{*}(\widehat{j}^{*}(M)) \) is an isomorphism. Thus \( \widehat{j}^{*}:\Perf(X\times \mA^{1}/\GG_{m})\to \Perf(X\times \widehat{\mA}^{1}/\GG_{m}) \) is fully faithful. Moreover, by Corollary \ref{corFilterdatA1modGm}, \( \widehat{j}_{*} \) 
also preserves perfect complexes, and \( \widehat{j}^{*}\circ \widehat{j}_{*}\cong \id \) 
when restricted to perfect complexes. 
\end{proof}

\begin{dfn}\label{dfnSpfFunctor}

%

(1) (\( \Spf \) construction)

Let \( \mrm{AniAlg}(\widehat{\Fil}^{\NN}(D(\ZZ_{\square}))) \) be the category of animated ring objects \( R^{\rhd} \)  in \( \widehat{\Fil}^{\NN}(D(\ZZ_{\square})) \) (i.e. 
\( R^{\rhd}\in\widehat{\Fil}^{\ZZ}(D(\ZZ_{\square}))\) such that \(\gr^{i}\cong 0\) for \( i<0 \), and \( R^{\rhd}/\Fil^{i} \) is equipped with a condensed animated ring structure). 
Let  \( \mrm{AnCAlg}(\widehat{\Fil}^{\NN}(D(\ZZ_{\square}))) \) be the category of \( R^{\rhd}\in \mrm{AniAlg}(\widehat{\Fil}^{\NN}(D(\ZZ_{\square}))) \) 
equipped with a solid analytic ring structure  \( \gr^{0}(R) \) on \( \gr^{0}(R^{\rhd}) \) (Definition \ref{dfnBoundedRing} (2)), such that \( \gr^{i}(R^{\rhd})\in D(\gr^{0}(R)),\forall i\in\NN \). 
By \cite[Proposition 12.23]{CS20}, 
there is a unique analytic ring structure \(R/\Fil^{n}\) on \(R^{\rhd}/\Fil^{n}\) such that the natural map lifts to a map between analytic rings \(R/\Fil^{n}\to R/\Fil^{1}\). 

We define a functor  
\begin{align*}
&\AnSpec^{+}_{n}(R):\mrm{AnCAlg}(\widehat{\Fil}^{\NN}(D(\ZZ_{\square})))^{op}\to (\mrm{SolidStk})_{/[\mA^{1}/\GG_{m}]}\\
&\AnSpec^{+}_{n}(R):=
\AnSpec^{+}(R/\Fil^{n}):=
\left[\left(\AnSpec\left(\mrm{Rees}(R/\Fil^{n})\right)\right)/\;\GG_{m}\right], 
\end{align*}
whose fiber over \([\GG_{m}/\GG_{m}]\)
is isomorphic to \(\AnSpec(R/\Fil^{n})\). By Remark \ref{rmkReesToTrivial}, we have a natural map \( \AnSpec^{+}(R/\Fil^{n})\to \AnSpec(R/\Fil^{n})\times [\mA^{1}/\GG_{m}] \), which is an isomorphism when restricted to \( [\GG_{m}/\GG_{m}] \).

We further a functor \( \Spf^{+}: \mrm{AnCAlg}(\widehat{\Fil}^{\NN}(D(\ZZ_{\square})))^{op}\to (\mrm{SolidStk})_{/[\mA^{1}/\GG_{m}]} \) 
via \[\Spf^{+}(R):=\varinjlim_{n}\AnSpec^{+}(R/\Fil^{n})\to [\mA^{1}/\GG_{m}],
\] and its fiber over \([\GG_{m}/\GG_{m}]\) is
\[\Spf(R):=\varinjlim_{n}\AnSpec(R/\Fil^{n}).
\]

(2) (Relative \( \unl{\Spf} \) construction)
Let \(X\) be a solid stack over \((\QQ,\ZZ)_{\square}\). 


Let \( \mrm{AnCAlg}(\widehat{\Fil}^{\NN}(\QCoh(X))) \)
be the category of \( R\in \mrm{CAlg}(\widehat{\Fil}^{\NN}(\QCoh(X))) \) equipped with a filtration such that \( \gr^{i}\cong 0 \) for \( i<0 \), and \( \gr^{i}(R) \) are connective for \( i\ge 0 \).
We define a functor 
\begin{align*}
&\unl{\AnSpec}_{X,n}^{+}:\mrm{AnCAlg}(\widehat{\Fil}^{\NN}(\QCoh(X)))^{op}\to (\mrm{SolidStk})_{/X\times[\mA^{1}/\GG_{m}]}\\
&\unl{\AnSpec}_{X,n}^{+}(R):=
\unl{\AnSpec}^{+}_{X}(R/\Fil^{n}):=\left[
		\unl{\AnSpec}_{X\times \mA^{1}}\left(\mrm{Rees}(R/\Fil^{n})\right)
		/\;\GG_{m}\right]
\end{align*}
whose fiber over \([\GG_{m}/\GG_{m}]\)
is isomorphic to \(\unl{\AnSpec}_{X}(R/\Fil^{n})\to X\),
where \( \unl{\AnSpec} \) is defined as in Lemma \ref{lemAffProperProp} (3).

We further define a functor 
\begin{align*}
&\unl{\Spf}_{X}^{+}:\mrm{AnCAlg}(\widehat{\Fil}^{\NN}(\QCoh(X)))^{op}\to (\mrm{SolidStk})_{/X\times[\mA^{1}/\GG_{m}]}\\
&\unl{\Spf}_{X}^{+}(R):=\varinjlim_{n}\unl{\AnSpec}_{X}^{+}(R/\Fil^{n})\to X\times [\mA^{1}/\GG_{m}],
\end{align*}
and its fiber over \([\GG_{m}/\GG_{m}]\) is
\[\unl{\Spf}_{X}(R):=\varinjlim_{n}\unl{\AnSpec}_{X}(R/\Fil^{n})\to X.
\] 
\end{dfn}
\begin{rmk}\label{rmkFromFilterToNoFil}
By Remark \ref{rmkReesToTrivial}, we have a natural map \( \Spf^{+}(R)\to \Spf(R)\times [\mA^{1}/\GG_{m}] \)
that is an isomorphism over \( [\GG_{m}/\GG_{m}] \). Similar results hold for the relative \( \unl{\Spf} \).
\end{rmk}
\begin{lem}\label{lemSpfCommuteWithLimit}
The functors \( \Spf^{+} \) and \(\unl{\Spf}^{+}_{X}\) in Definition \ref{dfnSpfFunctor}
commute with fiber products.
\end{lem}
\begin{proof}
Since \( \mrm{SolidStk} \) 
is an \( \infty \)-topoi by Definition \ref{dfnAnStkOverC}, the colimits commute with fiber products by Giraud's axioms (\cite[Theorem 6.1.0.6]{Lurie2009HTT}). Given a pushout diagram \(
R=B\otimes_{A}C
\) in \( \mrm{AnCAlg}(\widehat{\Fil}^{\NN}(D(\ZZ_{\square})))^{b} \), the pro-systems
\( (R/\Fil^{i})_{i\in\NN} \) and \( B/\Fil^{i}\otimes_{A/\Fil^{i}}C/\Fil^{i} \) are cofinal to each other. Thus  \begin{align*}
\Spf^{+}(R)&\cong \varinjlim_{i}(\AnSpec^{+}(B/\Fil^{i})\times_{\AnSpec^{+}(A/\Fil^{i})}\AnSpec^{+}(C/\Fil^{i}))\\&\cong 
(\varinjlim_{i}\AnSpec^{+}(B/\Fil^{i}))\times_{\varinjlim_{i}\AnSpec^{+}(A/\Fil^{i})}(\varinjlim_{i}\AnSpec^{+}(C/\Fil^{i}))
\\&\cong \Spf^{+}(B)\times_{\Spf^{+}(A)}
\Spf^{+}(C),
\end{align*} as desired. The proof for \( \unl{\Spf}_{X}^{+} \) 
is similar.
\end{proof}

\subsection{Examples}\label{subsecEgSolidStack}
\begin{eg}[Schemes]\label{egScheme}
Let \( X \) be an algebraic scheme.
 If \( X=\Spec(A) \), then we can realize \( X \) as the solid stack \( \AnSpec((A,\ZZ)_{\square}) \). A Zariski open immersion \( \Spec(A[1/f])\to \Spec(A) \) induces a proper morphism 
\( \AnSpec((A[1/f],\ZZ)_{\square})\to \AnSpec((A,\ZZ)_{\square}) \). A Zariski open covering \( \{\Spec(A[1/f_{i}])\}_{i=1}^{n} \) of \( \Spec(A) \) induces a prim \( ! \)-cover by the \v Cech resolution \[
0\to A\to \prod_{i=1}^{n}A[1/f_{i}]\to \cdots A[1/(f_{1}\cdots f_{n})]\to 0.
\] Thus we can realize general schemes as solid stacks.
By construction, \( X\to \AnSpec(\ZZ_{\square}) \) is cohomologically co-smooth by Theorem \ref{thmVariousDescent} (5) and Proposition \ref{propAffineProperCover}.
\end{eg}
\begin{lem}\label{lemBasicCoverOfA1}
Let
\[
U_{0}:=\AnSpec(\ZZ[T]_{\square}),\;U_{\infty}:=\AnSpec((\ZZ[T^{\pm1}],\ZZ[T^{-1}]_{\square})).
\] For \( ?\in\{0,\infty\} \), \( j_{?}:U_{?}\to \mA^{1} \) is an open immersion, and \( U_{0}\coprod U_{\infty}\to \mA^{1}=\AnSpec(\ZZ_{\square}[T]) \) is a suave \( ! \)-cover.
\end{lem}
\begin{proof}
By definition (\cite[Observation 8.10-8.11]{CS19}), these open subspaces are
defined 
by the idempotent algebras \( A_{0}:=\ZZ((T^{-1})) \) and \( A_{\infty}:=\ZZ[[T]] \) in \( D(\ZZ_{\square}[T]) \) respectively. 
More precisely, by the results loc. cit. \( M\in D(\ZZ_{\square}[T]) \) lies in \( D(\ZZ[T]_{\square}) \) if and only if \( M\otimes A_{0}\cong 0 \), and \( j_{0}^{*} \) has a fully faithful left adjoint \( j_{0,!}:M\mapsto M\otimes( A_{0}/\ZZ[T])[-1]  \). Thus \( j_{0} \) is an open immersion. A similar argument works for \( j_{\infty} \).
Finally, to show it is a suave \( ! \)-cover, it suffices to show that \( (j_{0}\coprod j_{\infty})^{*} \) is conservative. If \((j_{0}\coprod j_{\infty})^{*}  M\cong 0 \), then \( M\otimes A_{?}/\ZZ[T]\cong 0 \),
so \( M\cong M\otimes A_{?}\cong M\otimes A_{0}\otimes A_{\infty} \). Thus \( M\cong 0 \) because by \cite[Proposition 6.3]{CS19} \[ A_{0}\otimes A_{\infty}\cong \ZZ((S,T))/(ST-1)\cong
0 \]
since \( 1-ST \) is invertible in \( \ZZ((S,T)) \).
\end{proof}
\begin{eg}[Analytic adic spaces]\label{egAnAdicSpace}
Let \( X \) be a sheafy analytic adic space over \( \Spa(\QQ_{p},\ZZ_{p}) \). If \( X=\Spa(A,A^{+}) \), we can realize \( X \) as the solid stack \( \AnSpec((A,A^{+})_{\square}) \). An analytic open immersion \( X\left(\frac{f_{1},\ldots,f_{n}}{g}\right)\hookrightarrow X \) induces an open immersion. 
For this, it suffices to consider \( X\left(\frac{1}{f}\right) \) or \( X\left(\frac{f}{1}\right) \). This follows from Lemma \ref{lemBasicCoverOfA1} because 
\( X\left(\frac{f}{1}\right)\hookrightarrow X \) (resp. \( X\left(\frac{1}{f}\right)\hookrightarrow X \)) is the pull-back of \( U_{0} \) (resp. \(  U_{\infty}  \)) along  \( \ZZ_{\square}[T]\to (A,A^{+})_{\square} \) induced by \( T\mapsto f \). 

An analytic open covering of \( \Spa(A,A^{+}) \) induces a suave \( ! \)-cover. For this, it remains to show that pulling back is conservative, which in showed in \cite[Proposition 4.12 (5)]{Andreychev21pseudocoherent}. Thus the construction globalizes to realizations of arbitrary sheafy analytic adic spaces as solid stacks. Since \( X \) is analytic, \( X \) is essentially Tate by Lemma \ref{lemBasicPropertyEsseTate} (3).
\end{eg}

Let \(L\) be a \(p\)-adic field
 (i.e. a Banach extension of \(\QQ_{p}\) equipped with the \(p\)-adic topology).

 \begin{prop}[Analytic GAGA, Clausen-Scholze]\label{propAnGAGA}
Let \( X \) be a scheme of finite type over \( L \), and let \( X^{\an} \) be its analytification. We regard \( X \) and \( X^{\an} \) as solid stacks in Examples \ref{egScheme} and \ref{egAnAdicSpace} respectively.
Then there is a monomorphism of solid stacks \( i:X^{\an}\to X \). 
Moreover,

(1)
If \( X \) is projective over \( L \), then \( i \) is an isomorphism.

(2)
If \( X \) is smooth over \( L \), then \( X^{\an}\cong (X^{Tate})_{Solid} \), where the functors on the RHS are as in Lemma \ref{lemSolidToTateAnalyticficationFunctor}.
\end{prop} 
\begin{rmk}
For (1), it suffices to assume that \( X \) is proper rather than projective. See \cite[Theorem 1.2.2]{Wang2026relative}.
\end{rmk}
\begin{proof}
For \( X=\Spec(L[f_{1},\ldots,f_{n}]/I) \), \[ X^{\an}=V(I)\subset \varinjlim_{m} \Spa(L\langle p^{m}f_{1},\ldots,p^{m}f_{n}\rangle,\mO_{L}\langle p^{m}f_{1},\ldots,p^{m}f_{n}\rangle). \]
Thus realized as a solid stack as in Example \ref{egAnAdicSpace}, 
\begin{align*}
X^{\an}&\cong \varinjlim_{m}V(I)\cap
\Spa(L\langle p^{m}f_{1},\ldots,p^{m}f_{n}\rangle,\mO_{L}\langle p^{m}f_{1},\ldots,p^{m}f_{n}\rangle)\\
&\cong \varinjlim_{m}\AnSpec((L\langle p^{m}f_{1},\ldots,p^{m}f_{n}\rangle,\mO_{L})_{\square}),
\end{align*} 
which admits a natural map to \( X \), and is isomorphic to \( X\times_{\mA^{n}}\mA^{n,\an} \). 

We claim that \( \mA^{n,\an}\to \mA^{n} \) is a monomorphism.
We reduce easily to the case when \( n=1 \).
Note that \( \AnSpec(L\langle T\rangle, \mO\langle T\rangle)\to \AnSpec(L_{\square}[T]) \) is an open immersion.
Thus \( \mA^{1,\an} \) is a filtered colimit of open immersions, so it is a monomorphism by Giraud's axioms (\cite[Theorem 6.1.0.6]{Lurie2009HTT}).

If we have
\( U=\Spec(A[1/f])\hookrightarrow X=\Spec(A) \), then we clearly have a commutative diagram \[
\begin{tikzcd}
U^{\an}\arrow[r]\arrow[d] & X^{\an}\arrow[d]\\
U\arrow[r] & X.
\end{tikzcd}
\] Moroever, \( U^{\an}\to X^{\an} \)
is an analytic open subspace, so by Example \ref{egAnAdicSpace}, \( U^{\an}\to X^{\an} \) is an open immersion of solid stacks, and so is the composition \( U^{\an}\to X \).
Thus the construction \( X^{\an}\to X \) globalizes to any finite type scheme. Moreover, it is a monomorphism, being a colimit of monomorphisms.

If \( X\subset Y \) is a closed immersion, then \( X^{\an}\cong Y^{\an}\times_{Y}X \). Thus if \( X \) is projective, to show \( X^{\an}\cong X \), we can reduce to the case when \( X=\mbb{P}^{n} \). Since \( X^{\an}\to X \) is a monomorphism, it suffices to show that it is a \( ! \)-surjection. Consider \( \mbb{P}^{n}=\mrm{Proj}(L[S_{0},\ldots,S_{n}]) \), and consider the open \( \mA^{n}=\Spec(L[T_{1},\ldots,T_{n}]) \)
for \( T_{i}=S_{i}/S_{0} \). Without loss of generality, it suffices to show that \( \mA^{n}\times_{\mP^{n}}\mP^{n,\an}\to \mA^{n} \) is a \( ! \)-surjection. 

Let \[V_{0}:=
\AnSpec((L\langle T_{1},\ldots,T_{n}\rangle,\mO_{L}\langle T_{1},\ldots,T_{n}\rangle)_{\square})
\] and for \( 1\le i\le n \)
\[
V_{i}:=\AnSpec\left(\left(
L\left\langle\frac{1}{T_{i}},
	\frac{T_{1}}{T_{i}},\ldots,\frac{T_{n}}{T_{i}}
	\right\rangle[T_{i}],
\mO_{L}\left\langle\frac{1}{T_{i}},
	\frac{T_{1}}{T_{i}},\ldots,\frac{T_{n}}{T_{i}}
	\right\rangle[T_{i}]\right)_{\square}
\right).
\] Then \( V_{0} \) is contained in \( \mA^{n,\an} \), and for \( 1\le i\le n \), \( V_{i} \) is 
contained in \( (\Spec(L[S_{0}/S_{i},\ldots,S_{n}/S_{i}]))^{\an} \).
Thus the claim boils down to proving that \( \{V_{i}\}_{i=0}^{n} \)
is a suave \( ! \)-cover of \( \mA^{n} \).

For this, let \( f_{i}:\mA^{n}\to \mA^{1},(T_{1},\ldots,T_{n})\mapsto T_{i} \). We will use the notation of Lemma \ref{lemBasicCoverOfA1}.
For any \( I\subset \{1,\ldots,n\} \), let \(U_{I}\subset \mA^{n}\) be the intersection of \( f_{i}^{-1}(U_{0}) \)
for \( i\in I \) and \( f_{i}^{-1}(U_{\infty}) \)
for \( i\notin I \). Then by Lemma \ref{lemBasicCoverOfA1}, \( \{U_{I}\}_{I\subset \{1,\ldots,n\}} \) forms a suave \( ! \)-cover of \( \mA^{n} \).
Note that \( U_{\emptyset}=V_{0} \), and for \( i\notin I \), \( 1/T_{i}\in \mO(U_{I}) \), and for any \( (i,j)\in (I^{c} )^{2}\), define \( f_{i,j}:U_{I}\to \mA^{1},(T_{1},\ldots,T_{n})\mapsto T_{i}/T_{j} \). 
For fixed \( I \), and \( J\subset (I^{c})^{2} \), we define \( W_{I,J} \) to be the intersection of 
\( f_{i,j}^{-1}(U_{0}) \) for \( (i,j)\in J \) and \( f_{i,j}^{-1}(U_{\infty}) \) for \( (i,j)\notin J \).
Then \( \{W_{I,J}\}_{J\subset (I^{c})^{2}} \) forms a suave \( ! \)-cover of \( U_{I} \) by Lemma \ref{lemBasicCoverOfA1}. But such \( J \) will define a partial order in \( I^{c} \), and if \( i\in I^{c} \) is a maximal element for this partial order, then \( W_{I,J}\subset V_{i} \).

For the last part, 
by Example \ref{egAnAdicSpace}, \( ((X^{\an})^{Tate})_{Solid}\cong X^{\an} \). By Lemma \ref{lemSolidToTateAnalyticficationFunctor}, \( ((-)^{Tate})_{Solid} \) commutes with finite limits, and thus \( X^{\an}\cong ((X^{\an})^{Tate})_{Solid}\to (X^{Tate})_{Solid}  \) is also a monomorphism. We claim that \( X^{\an}\to (X^{Tate})_{Solid}  \) is a \( ! \)-surjection. By Lemma \ref{lemBasicPropertyEsseTate} (5), we can localize to the case where \( X \) is affine.
Since \( X \) is smooth, after further Zariski localization, we can embed \( X \) into \( \mA^{n} \) via a regular sequence. Thus using commutativity with finite limits, we can reduce to the case where \( X=\mA^{n} \), and further to \( X=\mA^{1} \). 
Then \( X^{\an}\to (X^{Tate})_{Solid}  \) is an isomorphism by Example \ref{egEssentialTate} (1).
\end{proof}



 \begin{eg}[Dagger neighborhood]\label{egDaggerNBHD}
Let \(X\) be an analytic adic space over \(L\) regarded as a solid stack, and \(Z\subset |X|\)
is a closed set. Then we define \((Z\subset X)^{\dagger}:=\varprojlim_{Z\subset U\subset X}U\)
where the limit is taken in \(\mrm{TateStk}_{\QQ_{p}}\)
and is taken over all the analytic open subspaces \(U\) of \(X\) which
contain \(Z\). We regard \(X\backslash Z\) as an analytic open subspace of \(X\).
Then \[i:(Z\subset X)^{\dagger}\hookrightarrow X\hookleftarrow
X\backslash Z:j
\] satisfies the excision: \( j \) is \( ! \)-able, and   \[j_{!}j^{!}(\mF)\to \mF\to 
i_{*}i^{*}(\mF)
\] is a fiber sequence for any \(\mF\in\QCoh(X)\). For this, 
we note that \( X\backslash Z\cong \varinjlim_{U'\subset X,|U'|\cap Z=\emptyset}U' \), 
where the colimits are taken over all the open subspaces \( U' \) 
of \( X \) that is disjoint from \( Z \). Since \( U'\to X \)   
is \( ! \)-able, so \( X\backslash Z\to X \) is also \( ! \)-able by \cite[Theorem 5.19]{Scholze2022sixFunctor}. 
Moreover, \( U'\hookrightarrow X\hookleftarrow (X\backslash U')^{\dagger} \)   
satisfies the excision, as 
one can prove by reducing to the affine case, and using the relation of open and closed subspaces with idempotent algebras. See for example \cite[Definition 2.2.2, Remark 2.2.3]{Juan2024analyticdeRham}. 
We will refer to \((Z\subset X)^{\dagger}\)
and \(X\backslash Z\) as complements of each other in \(
X\). 
\end{eg}

\begin{dfn}\label{dfnIncarnationGroup}
	Let \(G\) be a locally profinite group.
	
	(1) Define \emph{the continuous realization} of \( G \)  over \(L\) \[\underline{G}_{L}:=\varinjlim_{K}\AnSpec(\mrm{Cont}(K,L),\mrm{Cont}(K,\mO_{L})),\]
	where \(\mrm{Cont}(-)\)
	stands for continuous functions, and the colimit is taken along all the open compact subspaces \(K\) of \(G\).

	(2)
	Define \emph{the smooth realization} of \( G \) 
	over \( L \) 
	as \[G^{\sm}_{L}:=\varinjlim_{K}\AnSpec(\mrm{LocConst}(K,L),\mrm{LocConst}(K,\mO_{L})),\]
	where \(\mrm{LocConst}(-)\)
	stands for locally constant functions, and the colimit is as in (1).

	(3)
	We define \((1\subset \underline{G}_{L})^{\dagger}\)
	to be the kernel of the natural map \(\underline{G}_{L}\to G^{\sm}_{L}\). Note that \(\underline{G}\to G^{\sm}\) is a descendable cover by Theorem \ref{thmAnDRStackGeneral} (5) below.

	(4) Assume in addition that \(G\) is a \(p\)-adic Lie group. Then we define \emph{the locally analytic realization} of \( G \) 
	over \( L \) as \[
	G^{\la}_{L}:=\varinjlim_{K}\AnSpec(C^{\la}(K,L),C^{\la}(K,\mO_{L})),
	\] 
	where \(C^{\la}(-)\)
	stands for locally analytic functions, and the colimit is as in (1).

	We have natural maps \(\unl{G}_{L}\to G^{\la}_{L}\to G^{\sm}_{L}\).
	We define \((1\subset G_{L}^{\la})^{\dagger}\)
	to be the kernel of \({G}_{L}^{\la}\to G^{\sm}_{L}\). 
	We denote $\mO_{G_{L},1}:=\varinjlim_{K\subset G}C^{\an}(K,L)$ with the colimit as in (1).
	Then \(\mO_{G_{L},1}\cong R\Gamma((1\subset G_{L}^{\la})^{\dagger})\).

	Note that these spaces are essentially Tate by Example \ref{egEssentialTate} (4),
	and we will follow the convention of Remark \ref{rmkEssTateOmitTate}.

	We will omit the subscript \(L\) when it causes no confusion.
\end{dfn}

\begin{dfn}\label{dfnVBvariousNBHD}
 Let \(V\) be a vector bundle on an analytic adic space \(X\) over \(\Spa(L,\mO_{L})\). 

 (1)
 We define \(\mO_{X}[[V]]\in \mrm{Fil}^{\ZZ}(\QCoh(X))\) as the filtered completion of \(\bigoplus_{n\in\NN}\Sym^{n}_{\mO_{X}}V\)
with respect to the (descending) degree filtration. We define \[
\hat{V}:=\unl{\Spf}_{X}(\mO_{X}[[V^{\vee}]]),
\] where the RHS is defined in Definition \ref{dfnSpfFunctor}.


(2) Let \(V_{0}\) be a locally finite free \(\mO^{+}_{X}\)-module. We define \(\mO_{X}\langle V_{0}\rangle\in \QCoh(X)\) via \[\mO_{X}\langle V_{0}\rangle:=\left(
	\bigoplus_{i\in\NN}\Sym^{i}_{\mO^{+}_{X}}(V_{0})
\right)^{\wedge}_{p}\left[\frac{1}{p}\right].
\]

(3) 
We define \(\mO_{X}\{V\}^{\dagger}\in \QCoh(X)
\) via \[\mO_{X}\{V\}^{\dagger}:=\varinjlim_{n}\mO_{X}\left\langle \frac{V_{0}}{p^{n}}\right\rangle,
\] where \(V_{0}\subset V\) is a locally chosen lattice. Note that the colimit is independent of the choice of the lattice, so the construction glues to a global construction. We define \[
V^{\dagger}:=\unl{\AnSpec}_{X}(\mO_{X}\{V^{\vee}\}^{\dagger}),
\] where the RHS is defined in Lemma \ref{lemAffProperProp} (3).

(4)  We define \(V^{\an}\) to be the analytification of the algebraic  total space of \(V\). Concretely, \[V^{\an}:=\varinjlim_{n}\unl{\AnSpec}_{X}(\mO_{X}\langle p^{n}V_{0}\rangle),\]
where \(V_{0}\subset V\) is a locally chosen lattice. Note that the colimit is independent of the choice of the lattice, so the construction glues to a global construction. 

We have natural inclusions \(\mO_{X}\langle V_{0}\rangle \subset \mO_{X}\{V\}^{\dagger}\subset \mO_{X}[[V]]\), which induces natural inclusions \(\hat{V}\hookrightarrow V^{\dagger}\hookrightarrow V^{\an}\). 

Also note that these solid stacks are essentially Tate: after localizing to an affinoid \( \Spa(A,A^{+})\subset X \) by Lemma \ref{lemBasicPropertyEsseTate} (6),  \( A\{V\}^{\dagger} \) is bounded by construction, and \( V^{\an} \) is essentially Tate by Example \ref{egEssentialTate}.
 \end{dfn}

 \begin{dfn}\label{dfnDaggerGroupObject}
Assume that \(\fg\) is a finite dimensional Lie algebra over \(L\).
We can consider \(L\{\fg^{\vee}\}^{\dagger}\) and \(L[[\fg^{\vee}]]\),
which has a Hopf algebra structure given by the Baker-Campbell-Hausdorff formula (see for example \cite[Theorem 5.3]{Hall2015LgLaGTM}). Thus \[\fg^{\dagger}:=\AnSpec(L\{\fg^{\vee}\}^{\dagger}),\;\hat{\fg}^{+}:=\Spf^{+}(L[[\fg^{\vee}]])\] have group structures. Here \( \Spf^{+} \) is as in Definition \ref{dfnSpfFunctor}. Let \( \hat{\fg}:=\hat{\fg}^{+}|_{\GG_{m}/\GG_{m}} \).

If there is a \(p\)-adic Lie group \(G\) such that \(\Lie(G)\otimes L=\fg\),
then we have an isomorphism \(\mO_{G,1}\cong L\{\fg^{\vee}\}^{\dagger}\),
as Hopf algebras, where the group structure of \(G\)
gives a Hopf algebra structure on \(\mO_{G,1}\). Therefore, \(\fg^{\dagger}\cong (1\subset G^{\la})^{\dagger}\times L\), where the RHS is defined in Definition \ref{dfnIncarnationGroup}.
\end{dfn}

\subsection{de Rham stacks}\label{subsecAndRStack}
We will recall the definition of the filtered de Rham stacks from \cite{Juan2024analyticdeRham}. 
\begin{dfn}[Algebraic de Rham stack, {\cite[Definition 5.1.1]{Juan2024analyticdeRham}}]\label{dfnAlgDRJuan}
Given a solid stack \(X\) (over \(\ZZ_{\square}\)), 
	we define \(X^{\pdR,+,pre}\)
	as the presheaf sending \(A\) over \(\mA^{1}/\GG_{m}\) to \(\varinjlim_{I}X(A/I\{-1\})\), where the colimit is taken over all the uniformly nilpotent ideals \(I\) of \(A\) (\cite[Definition 2.5.8]{Juan2024analyticdeRham}), and \( \{-1\} \) is as in Notation \ref{notationTwistByAGm}. We define the \emph{filtered algebraic de Rham stack}
	\(X^{\pdR,+}\to [\mA^{1}/\GG_{m}]\)
	as the \emph{solid stack} associated to \(
	X^{\pdR,+,pre}\). 

	We define the \emph{algebraic de Rham stack} \(X^{\pdR}\) (resp. \emph{algebraic Hodge stack} \(X^{\alg-\mrm{Hodge}}\)) as the restriction of \(X^{\pdR,+}\)
to \([\GG_{m}/\GG_{m}]\subset [\mA^{1}/\GG_{m}]\) (resp. \(B\GG_{m}\subset [\mA^{1}/\GG_{m}]\)). 

	Given a morphism \(f:X\to Y\),
	we define the \emph{relative filtered algebraic de Rham stack over \(Y\)}
	as \(X^{\pdR/Y,+}:=Y\times_{Y^{\pdR,+}}X^{\pdR,+}\). We define similarly \(X^{\pdR/Y}\) and \(X^{\alg-\mrm{Hodge}/Y}\). 
\end{dfn}
\begin{rmk}\label{rmkAlgDRStackOFFilter}
We can define a functor \( (-)^{\pdR,+,\mA^{1}/\GG_{m}}:\mrm{SolidStk}_{/[\mA^{1}/\GG_{m}]}\to \mrm{SolidStk}_{/[\mA^{1}/\GG_{m}]} \), sending \( X \)
to the sheafification of the presheaf \[
(\AnSpec(A)\to \mA^{1}/\GG_{m})\mapsto \varinjlim_{I}\Hom_{[\mA^{1}/\GG_{m}]}(\AnSpec(A/I\{-1\}),X)
\] where \( I \)
and \( A/I\{-1\} \)
are as in Definition \ref{dfnAlgDRJuan}. There is a natural transformation \( \id\to (-)^{\pdR,+,\mA^{1}/\GG_{m}} \) over \( [\mA^{1}/\GG_{m}] \).

By definition, for \( X\in\mrm{SolidStk} \), \( (X\times [\mA^{1}/\GG_{m}])^{\pdR,+,\mA^{1}/\GG_{m}}\cong X^{\pdR,+} \).
Many properties of \( (-)^{\pdR,+} \)
also hold for \( (-)^{\pdR,+,\mA^{1}/\GG_{m}} \). 
\end{rmk}

\begin{dfn}[Analytic de Rham stack, {\cite[Definition 2.6.1, 5.2.2]{Juan2024analyticdeRham}}]\label{dfnAnDRStack}
Given \(A\in\mrm{AnRing}^{b}_{\QQ_{p}}\), we define
\(\mrm{Rad}^{\dagger}(A)\) as \[(\mrm{Rad}^{\dagger}(A))(S):=\Hom(\QQ_{p}\{\NN[S]\}^{\dagger},A),\]
for \(S\) extremely disconnected.

	Given a Tate stack \(X\) over \(\QQ_{p}\), 
	we define 
	as the presheaf 
	\[ X^{\dR,+,pre}:(\AnSpec^{b}(A)\to \mA^{1,\an}/\GG_{m}^{\an})\mapsto X(A/\mrm{Rad}^{\dagger}(A)\{-1\})
	\] where \( \{-1\} \) is as in Notation \ref{notationTwistByAGm}. We define the \emph{filtered analytic de Rham stack}
	\(X^{\dR,+}\)
	as the \emph{Tate stack} associated to \(
	X^{\dR,+,pre}\).

	We define the \emph{analytic de Rham stack} \(X^{\dR}\) (resp. \emph{analytic Hodge stack} \(X^{\mrm{Hodge}}\)) as the restriction of \(X^{\dR,+}\)
to \([\GG_{m}^{\an}/\GG_{m}^{\an}]\subset [\mA^{1,\an}/\GG_{m}^{\an}]\) (resp. \(B\GG_{m}^{\an}\subset [\mA^{1,\an}/\GG_{m}^{\an}]\)). 

	Given a morphism \(f:X\to Y\),
	we define the \emph{relative filtered analytic de Rham stack over \(Y\)}
	as \(X^{\dR/Y,+}:=Y\times_{Y^{\dR,+}}X^{\dR,+}\).
	We define similarly \(X^{\dR/Y}\) and \(X^{\mrm{Hodge}/Y}\). 
\end{dfn}

We will need the following results:

\begin{dfn}
For any pro-objects \(C=``\varprojlim"_{i\in I}C_{i}\),
we say \(C\)
is \emph{light}
if the index category \(I\)
is countable. 
\end{dfn}
\begin{thm}[\cite{Juan2024analyticdeRham},
\cite{AnschützBoscoLeBrasCamargoScholze2025analyticrhamstacksfarguesfontaine}]---\label{thmAnDRStackGeneral}

(1) The functors \(X\mapsto X^{\dR,+}\), \(X\mapsto X^{\pdR,+}\) and  \( (-)^{\pdR,+,\mA^{1}/\GG_{m}} \) in Remark \ref{rmkAlgDRStackOFFilter}
commute with finite limits. 

(2) If \(i:X\to Y\) 
is a monomorphism, in the sense that \(\Delta:X\to X\times_{Y}X\)
is an isomorphism,
then \(X\to X^{\dR/Y}\) is a monomorphism. If \(i:X\to Y\)
is moreover \(\dagger\)-formally smooth (\cite[Definition 3.7.2]{Juan2024analyticdeRham}),
then \(X\cong X^{\dR/Y}\).

(3) For \(f:X\to Y\) is a smooth morphism between rigid varieties,
then \(X\to X^{\dR/Y}\)
is a \(!\)-surjection and affine proper, and in particular, it is cohomogically co-smooth.
Moreover, \(X\times_{X^{\dR/Y}}X\cong \Delta(X)^{\dagger}:=(X\xhookrightarrow{\Delta}X\times_{Y} X)^{\dagger}\). 

Similarly, \(X\to X^{\pdR/Y}\)
is a \(!\)-surjection, and \(X\times_{X^{\pdR/Y}}X\cong \hat{\Delta}_{X}:=(X\xhookrightarrow{\Delta}X\times_{Y} X)^{\wedge}\).

(4) For \(f:X\to Y\)
a light pro-finite-\'etale cover or an \'etale cover between perfectoid spaces,
then \(f\)
induces a \(!\)-surjection for the associated Tate stacks. More precisely,
if \(f\) is light pro-finite \'etale,
then \(f\) is affine proper,
and descendable of index \(\le 4\).

In particular,
if \(G\) is a locally light profinite group,
and \(\ti{X}\to X\) is a pro\'etale \(G\)-torsor between perfectoid spaces over \(\Spa(\CC_{p},\mO_{\CC_{p}})\),
then \(X\cong [\ti{X}/\underline{G}]\)
as Tate stacks.

(5) If \(X_{i}=\Spa(A_{i},A_{i}^{+})\) are smooth rigid varieties over \(\CC_{p}\) for 
\(i\in\NN\),
such that the transition maps \(X_{i}\to X_{i-1}\)
are finite \'etale, and if \(X=\Spa(A,A^{+})\) is perfectoid,
for \(A^{+}:=(\varinjlim_{i}A_{i})^{\wedge}_{p}\)
and \(A:=A^{+}[1/p]\),
then the natural maps \(X\to X^{\sm}\to X_{i}\)
are both proper affine and descendable,
and
we have \[X^{\dR}\cong (X^{\sm})^{\dR}\cong \varprojlim_{i}X_{i}^{\dR},\]
where \(X^{\sm}:=\varprojlim_{i}X_{i}\cong \AnSpec^{b}(A^{\sm},A^{\sm,+})\)
with \(A^{\sm,+}:=\varinjlim_{i}A_{i}\)
and \(A^{\sm}:=A^{\sm,+}[1/p]\). 
We also have \(X^{\sm}\cong X^{\sm,\dR}\times_{X^{\dR}_{i}}X_{i}\).

(6) We have \((\AnSpec^{b}(\CC_{p},\mO_{\CC_{p}}))^{\dR}\cong \AnSpec^{b}(\bar{\QQ}_{p},\bar{\ZZ}_{p})\). In particular,
\(\AnSpec^{b}(\CC_{p},\mO_{\CC_{p}})\to \AnSpec^{b}(\CC_{p},\mO_{\CC_{p}})^{\dR}\)
is a \(!\)-surjection.
\end{thm}
\begin{proof}
	(1) follows 
	from 
	 \cite[Proposition 5.1.2 (3) \& 5.2.3 (3)]{Juan2024analyticdeRham}, as the sheafification functor commutes with finite limits by \cite[Corollary 6.2.1.6, Proposition 6.2.2.7]{Lurie2009HTT}. Note that the same proof also works for \( X^{\pdR,+,\mA^{1}/\GG_{m}} \).

	For (2), by (1),
	we know that \(X^{\dR}\to Y^{\dR}\)
	is also a monomorphism,
	and thus so is \(X^{\dR/Y}\to Y\).
	We have a factorization \(X\to X^{\dR/Y}\to Y\),
	and thus \(X\to X^{\dR/Y}\)
	is also a monomorphism.
	If \(X\to Y\) is \(\dagger\)-formally smooth, \(X\to X^{\dR/Y}\) is \(!\)-surjective,
	and thus an isomorphism.

	For (3), the surjection is given by \cite[Proposition 5.2.3 (2)]{Juan2024analyticdeRham}. Now we know that as Tate stacks, 
	\begin{align*}
	&X\times_{X^{\dR/Y}}X\cong X^{\dR/Y}\times_{X^{\dR/Y}\times X^{\dR/Y}}(X\times X)\\&\cong (\Delta(X)^{\dagger})^{\dR/Y}\times_{(X\times X)^{\dR/Y}}X\times X\cong (\Delta(X)^{\dagger})^{\dR/X\times_{Y} X}.
	\end{align*}
	Now by (2), we need to prove that \(\Delta(X)^{\dagger}\hookrightarrow X\times_{Y} X\) is a \(\dagger\)-formally smooth monomorphism. This follows from the fact that \(\Delta(X)^{\dagger}\)
	is the inverse limit of all 
	its open neighborhood in \(X\times_{Y} X\),
	and by \cite[Proposition 3.7.4(2)]{Juan2024analyticdeRham},
	inverse limits of \(\dagger\)-formally smooth morphisms are still \(\dagger\)-formally smooth. The same proof works for \(X^{\pdR/Y}\) using \cite[Proposition 5.1.4]{Juan2024analyticdeRham}.

We are left to prove that the morphism is affine proper.
By Lemma \ref{lemAffProperProp} (1), it suffices to show that \(X\times_{X^{\dR/Y}}X\xrightarrow{p_{1}} X \)
	is affine proper.
We know \(X\times_{X^{\dR/Y}}X\cong \Delta(X)^{\dagger}:=(X\xhookrightarrow{\Delta}X\times_{Y}X)^{\dagger}\).
	So we want to prove that \(p_{1}:\Delta(X)^{\dagger}\to X\)
	is affine proper.
	
	By localizing at 
	\(X\) and \(Y\),
	we can assume that \(X=\Spa(A,A^{+})\)
	and \(Y=\Spa(B,B^{+})\).
	Then \(X\times_{Y}X=\AnSpec^{b}(A\otimes_{B}A,A^{+}\otimes_{B^{+}}A^{+})\).
	Let \(I:=\Ker(A\otimes_{B}A\to A)\).
	Then \[\Delta(X)^{\dagger}=\AnSpec^{b}(A\otimes_{B}A\{I\}^{\dagger},B),\]
	with \[A\otimes_{B}A\{I\}^{\dagger}:=\varinjlim_{n}A\otimes_{B}A\left\langle \frac{I}{p^{n}}\right\rangle. \]
	Note that \(\Delta(X)^{\dagger}\to X\)
	induces an isomorphism after taking \(\dagger\)-reduction,
	and thus the (solid) analytic structure is automatically induced from \(X\), by \cite[Proposition 2.6.16 (1)]{Juan2024analyticdeRham}.
	Thus, \(\Delta(X)^{\dagger}\to X\)
	is affine proper, as desired.

For (4), by localizing, we can assume that \(Y\) is perfectoid affinoid. In particular, \(Y\) is qcqs,
and by further localizing on \(X\),
we can assume that \(X\) is also affinoid perfectoid.

Assume first that \(X\to Y\) is a pro-finite-\'etale cover. We can then write \(X\) as the inverse limit (in the category of diamonds) of a finite \'etale tower \(Y_{i}=\Spa(A_{i},A_{i}^{+})\),
with \(Y=Y_{0}\),
and all the transition maps are finite \'etale. By \cite[Theorem 1.10]{Scholze12},
we that know \(A_{j}^{+}/p\) 
is almost finite \'etale over \(A_{i}^{+}/p\)
for any \(j>i\in\NN\). Then we know that the map \(A_{i}^{+}/p^{n}\to A_{j}^{+}/p^{n}\)
almost splits, for any \(j>i\),
for example, by \cite[Remark 4.3.5]{Bhatt17}. In particular, the map is almost descendable of index \(\le 1\).
Here, ``almost descendable" means that \(A_{j}^{+}/p^{n}\) is descendable in the category of almost \(A_{i}^{+}/p^{n}\)-modules in the sense of \cite{Mathew2016galois}.
Then by \cite[Lemma 11.22]{BhattScholze2017projectivity},
\(A_{i}^{+}/p^{n}\to A^{\sm,+}/p^{n}\)
is almost descendable of index \(\le 2\).
By the exact same proof as that of  \cite[Lemma 11.22]{BhattScholze2017projectivity},
we know that after taking limits along \(n\), \(A_{i}^{+}\to A^{+}\)
is almost descendable of index \(\le 4\). Inverting \(p\),
we know that \(A_{i}\to A\)
is descendable of index \(\le 4\).
Therefore, we know that the map is a \(!\)-surjection.

Now if \(X\to Y\) is \'etale, by localizing on \(X\),
we can assume that it is of the form \(\coprod_{i=1}^{m}U_{i}\)
and \(U_{i}=\Spa(B_{i},B_{i}^{+})\)
is affinoid perfectoid, and \'etale over \(Y=\Spa(A,A^{+})\).
If all the morphisms \(U_{i}\to Y\) are open immersions,
we are done.
Now we do induction on the number of \(i\)
such that \(U_{i}\to Y\) is \emph{not} an open immersion. Fix one such \(i_{0}\),
then we can find a finite open affinoid covering \(V_{j}\)
of \(Y\) such that \(U_{i_{0}}|_{V_{j}}\to V_{j}\)
admits a finite \'etale
compactification. By replacing \(Y\) with \(V_{j}\),
we assume that \(Y=V_{j}\),
and \(U_{i_{0}}\xrightarrow{j} \bar{U}_{i_{0}}\xrightarrow{\bar{f}_{i_{0}}} Y\),
with \(j\) an open immersion,
and \(\bar{f}_{i_{0}}\)
finite \'etale.
Since we have proven above that finite \'etale morphisms give \(!\)-surjection,
we can replace \(Y\) by \(\bar{U}_{i_{0}}\),
and thus assume that \(U_{i_{0}}\to Y\) is an open immersion. Then we are done by induction.

Now if \(\ti{X}\to X\)
is a pro\'etale \(G\)-torsor,
fixing \(K\subset G\) an open compact subgroup,
then \(\ti{X}\to \ti{X}/K\)
is pro-finite \'etale,
and \(\ti{X}/K\to X\) is \'etale, by \cite[Corollary 9.11]{Scholze2017etaleDiamond}. Then we know from the above that \(\ti{X}\to X\)
is a \(!\)-surjection.
Now we finish by noting that the \(i\)-th fiber product of \(\ti{X}\) over \(X\) \(\ti{X}^{i/X}\cong \ti{X}\times G^{i-1}\) for \(
i\in \ZZ_{\ge1}\).
	
For (5), we first prove that
\(X^{\sm}\to X_{i}\)
is a \(!\)-surjection. 
It is clearly affine proper, 
since \(A^{\sm,+}\)
is the integral closure of \(A_{i}^{+}\)
in \(A^{\sm}\). We know that \(A_{i}\to A_{j}\)
is descendable of index 
\(1\),
since it has a splitting,
given by, for example, \cite[Construction 4.3.3]{Bhatt17}. Therefore, again by \cite[Lemma 11.22]{BhattScholze2017projectivity},
\(A_{i}\to A^{\sm}\)
is descendable of index \(\le 2\).
This finishes the proof that \(X^{\sm}\to X_{i}\) is a \(!\)-surjection.

We now continue to prove that \(X\to X^{\sm}\) is a \(!\)-surjection. 
Since  \(A^{\sm,+}\) is dense in \(A^{+}\),
we know that the map is affine proper.
To prove it is a \(!\)-surjection,
upon localizing,
we can assume that there exists an \'etale morphism 
\(X_{i}\to \mbb{T}^{n}\),
with \[\mbb{T}^{n}:=\Spa(\CC_{p}\langle T_{1}^{\pm1},\ldots,T_{n}^{\pm1}\rangle,\mO_{\CC_{p}}\langle T_{1}^{\pm1},\ldots,T_{n}^{\pm1}\rangle).\]
Consider the finite \'etale tower given by \[\mbb{T}^{n}_{i}:=\Spa(\CC_{p}\langle T_{1}^{\pm\frac{1}{p^{i}}},\ldots,T_{n}^{\pm\frac{1}{p^{i}}}\rangle,\mO_{\CC_{p}}\langle T_{1}^{\pm\frac{1}{p^{i}}},\ldots,T_{n}^{\pm\frac{1}{p^{i}}}\rangle),\]
and \[\mbb{T}^{n}_{\infty}:=\Spa(\CC_{p}\langle T_{1}^{\pm\frac{1}{p^{\infty}}},\ldots,T_{n}^{\pm\frac{1}{p^{\infty}}}\rangle,\mO_{\CC_{p}}\langle T_{1}^{\pm\frac{1}{p^{\infty}}},\ldots,T_{n}^{\pm\frac{1}{p^{\infty}}}\rangle).\]
Then we know that \(\mbb{T}^{n}_{\infty}\to \mbb{T}^{n}_{i}\)
is affine proper, and of descendable index \(\le 1\),
since we can construct an obvious splitting. 
Pulling back the tower along \(X_{i}\to \TT^{n}\),
we have a profinite \'etale map \(g_{i}:X_{i,\infty}\to X_{i}\),
such that \(X_{i,\infty}\)
is perfectoid,
and \(g_{i}\) is an affine proper morphism of descendable index \(\le 1\). 
Consider the diagram \[\begin{tikzcd}
	X\times_{X_{i}}X_{i,\infty}
	\arrow[r,"f_{i}'"]\arrow[d] & X_{i,\infty}
	\arrow[d,"g_{i}"]\\
	X \arrow[r,"f_{i}"] & X_{i},
\end{tikzcd}\]
where by (4), \(f_{i}'\) (resp. \(g_{i}\))
is affine proper and descendable 
of index \(\le 4\) (resp. \(\le 1\)).

Taking limit along \(i\), we have 
\[\begin{tikzcd}
	X\times_{X^{\sm}}X_{\infty}^{\sm}
	\arrow[r,"f'"]\arrow[d] & X^{\sm}_{\infty}
	\arrow[d,"g"]\\
	X \arrow[r,"f"] & X^{\sm},
\end{tikzcd}\]
where by \cite[Lemma 11.22]{BhattScholze2017projectivity},
\(f'\) (resp. \(g\))
is affine proper and descendable 
of index \(\le 8\) (resp. \(\le 2\)).
In particular, \(g\)
and \(f'\)
are both \(!\)-surjections,
from which we see that \(f\) is also a \(!\)-surjection.

Now we go on to prove the statement about analytic de Rham stacks.
we have natural morphisms \[X^{\dR}\to (X^{\sm})^{\dR}\to \varprojlim_{i}X_{i}^{\dR}.\]
We want to prove that both morphisms are isomorphisms.

First, since \(X\mapsto X^{\dR}\)
commutes with finite limits by (1),
we know \(X^{\dR}\times_{X^{\sm,\dR}}X^{\dR}\cong (X\times_{X^{\sm}}X)^{\dR}\).
Now \(X\times_{X^{\sm}}X\cong \AnSpec^{b}(A\otimes_{A^{\sm}}A,A^{+}\otimes_{A^{\sm,+}}A^{+})\).
We claim that 
\((A\otimes_{A^{\sm}} A)^{\dagger-red}\cong A\).
In other words, the kernel of the natural map \(A\otimes_{A^{\sm}}A\to A \)
lies in the \(\dagger\)-nilradical.
The kernel is generated by elements of the form \(a\otimes 1-1\otimes a\)
for \(a\in A(S)\) for \(S\) profinite. For any 
\(n\),
we can find \(a_{n}\in A^{\sm}(S)\),
such that \(a-a_{n}\in p^{n}A^{+}(S)\): consider \(a:S\to A\), and as \(a\) is continuous and \(S\) is compact,
we can find a finite partition \(S=\coprod_{i}S_{i}\) such that \(f(S_{i})\in c_{i}+p^{n}A^{+}\)
for some \(c_{i}\in A\),
and then by density, we can assume that 
\(c_{i}\in A_{m}\),
for a fixed \(m\).
Then we can define \(a_{n}\in A_{m}(S)\), mapping \(S_{i}\) to \(\{c_{i}\}\).
As a result, \(a\otimes 1-1\otimes a\in p^{n}(A^{+}\otimes_{A^{\sm,+}}A^{+})\)(S).
Since this is true for any \(n\),
we know it lies in the \(\dagger\)-nilradical.
Now by the definition of \(X\mapsto X^{\dR}\),
the functor only depends on the \(\dagger\)-reduction,
and thus \(\Delta:X\to X\times_{X^{\sm}}X\) induces an isomorphism  \[(X\times_{X^{\sm}}X)^{\dR}\cong X^{\dR},\]
which implies that \(X^{\dR}\to X^{\sm,\dR}\)
is a monomorphism.

We take a digression to prove (6).
The argument about being monomorphism also works in the special case of \(X_{i}=\Spa(K_{i},\mO_{K_{i}})\)
where \(K_{i}\)
are finite extensions of \(\QQ_{p}\),
and \(\bigcup_{i}K_{i}=\bar{\QQ}_{p}\). In particular,
we have proven that \(\AnSpec^{b}(\CC_{p},\mO_{\CC_{p}})^{\dR}\to \AnSpec^{b}(\bar{\QQ}_{p},\bar{\ZZ}_{p})^{\dR}\cong \AnSpec^{b}(\bar{\QQ}_{p},\bar{\ZZ}_{p})\)
is a monomorphism.
But since \(K_{i}\to \CC_{p}\)
splits,
it is affine proper and descendable of index \(\le 1\),
so \(\bar{\QQ}_{p}\to \CC_{p}\)
is descendable of index \(\le 2\). Therefore,
\(\AnSpec^{b}(\CC_{p},\mO_{\CC_{p}})\to \AnSpec^{b}(\bar{\QQ}_{p},\bar{\ZZ}_{p})^{\dR}\cong \AnSpec^{b}(\bar{\QQ}_{p},\bar{\ZZ}_{p})\)
is a surjection,
which implies that \[\AnSpec^{b}(\CC_{p},\mO_{\CC_{p}})^{\dR}\cong \AnSpec^{b}(\bar{\QQ}_{p},\bar{\ZZ}_{p}),\]
finishing the proof of (6). As a corollary, we know that \(\AnSpec^{b}(\CC_{p},\mO_{\CC_{p}})\to \AnSpec^{b}(\CC_{p},\mO_{\CC_{p}})^{\dR}\)
is a \(!\)-surjection,
and thus for any \(X\)
over \(\CC_{p}\),
\(X^{\dR/\CC_{p}}\to X^{\dR}\)
is a \(!\)-surjection.

We now continue to prove (5).
We have shown that \(X\to X^{\sm}\)
is a \(!\)-surjection. Moreover, we know that \(X_{i}\) are \(\dagger\)-formally smooth over \(\CC_{p}\), and thus by \cite[Proposition 3.7.4(2)]{Juan2024analyticdeRham},
\(X^{\sm}\) is also \(\dagger\)-formally smooth over \(\CC_{p}\). In particular,
by \cite[Proposition 5.2.3(2)]{Juan2024analyticdeRham},
\(X^{\sm}\to X^{\sm,\dR/\CC_{p}}\)
is a \(!\)-surjection.
Combined with (6), this implies that \(X^{\dR}\to X^{\sm,\dR}\)
is also a surjection,
and thus an isomorphism.

Moreover, since transition maps are \'etale,
we know \(X_{j}\cong X_{j}^{\dR}\times_{X_{i}^{\dR}}X_{i}\)
for \(j>i\),
and thus taking limits along \(j\), we have 
\(X^{\sm}\cong (\varprojlim_{i}X_{i}^{\dR})\times_{X_{0}^{\dR}}X_{0} \). Taking de Rham stacks on both sides, and using (1),
we know \[X^{\sm,\dR}\cong (\varprojlim_{i}X_{i}^{\dR})^{\dR}\times_{X_{0}^{\dR}}X_{0}^{\dR}\cong (\varprojlim_{i}X_{i}^{\dR})^{\dR} \]
Since \(X_{0}\to X_{0}^{\dR}\)
is a \(!\)-surjection,
we know that \(X^{\sm}\to \varprojlim_{i}X_{i}^{\dR}\)
is also a \(!\)-surjection.
In particular, \(X^{\sm,\dR}\to \varprojlim_{i}X_{i}^{\dR}\)
is a \(!\)-surjection. Moreover, it gives a retract for the natural map  \[\varprojlim_{i}X_{i}^{\dR}\to(\varprojlim_{i}X_{i}^{\dR})^{\dR}\cong X^{\sm,\dR}, \]
so it gives an isomorphism,
as desired.
\end{proof}

\begin{cor}\label{corTorsorDRSmTorsor}
Let \(G\) be a locally light profinite group, and \(X\to Y\)
be a pro\'etale \(G\)-torsor, where \(X\) is a  perfectoid space, and \(Y\) is either a smooth rigid variety or a perfectoid space over \(\CC_{p}\). Assume that \(Y\to Y^{\dR/\CC_{p}}\) is a \(!\)-surjection.

(1) We have \(G^{\sm}\cong \underline{G}^{\dR/\CC_{p}}\).

(2) \(X^{\dR/\CC_{p}}\to Y^{\dR/\CC_{p}}\)
is a \(G^{\sm}\)-torsor.
\end{cor}
\begin{proof}
Let \(K\subset G\)
be an open compact subset.
Then \[\underline{K}=\AnSpec^{b}(\mrm{Cont}(K,\CC_{p}),\mrm{Cont}(K,\mO_{\CC_{p}})).\]
Assume that as a light profinite group, \(K=\varprojlim_{i\in\NN}K_{i}\)
with \(K_{i}\)
a finite group,
then \(\underline{K_{i}}=\AnSpec^{b}(\mrm{Cont}(K_{i},\CC_{p}),\mrm{Cont}(K_{i},\mO_{\CC_{p}}))\)
and \[\mrm{Cont}(K,\mO_{\CC_{p}})\cong \left(\varinjlim_{i}\mrm{Cont}(K_{i},\mO_{\CC_{p}})\right)^{\wedge}_{p}.\]
Thus we can apply Theorem \ref{thmAnDRStackGeneral} (5)
to conclude that \[\underline{K}^{\dR/\CC_{p}}\cong (K^{\sm})^{\dR/\CC_{p}}\cong \varprojlim_{i}K_{i}^{\dR/\CC_{p}}\cong \varprojlim_{i}K_{i}\cong K^{\sm}.\]

Now for (2), we first consider the case where \(Y\) is perfectoid.
Since \(X\to Y\)
is a pro\'etale \(G\)-torsor,
we have an isomorphism \[X\times_{Y}X\cong \underline{G}\times_{\CC_{p}}X,\]
and by Theorem \ref{thmAnDRStackGeneral} (1),
we know \[X^{\dR/\CC_{p}}\times_{Y^{\dR/\CC_{p}}}X^{\dR/\CC_{p}}\cong G^{\sm}\times_{\CC_{p}}X^{\sm}.\]
Since 
\(Y\to Y^{\dR/\CC_{p}}\)
is a \(!\)-surjection by assumption,
and \(X\to Y\)
is also a \(!\)-surjection by Theorem \ref{thmAnDRStackGeneral} (4),
we know \(X^{\dR/\CC_{p}}\to Y^{\dR/\CC_{p}}\)
is also a \(!\)-surjection,
from which we see that \(X^{\dR/\CC_{p}}\to Y^{\dR/\CC_{p}}\)
is a \(G^{\sm}\)-torsor.

Now if \(Y\) is smooth over \(\CC_{p}\), by
localizing at \(Y\) and finding a toric chart, we can assume that there exists a pro-finite-\'etale \(K\)-torsor \(\ti{Y}\to Y\)
for \(K\cong \ZZ_{p}^{n}\).
Then we know from the perfectoid case that \((X\times_{Y}\ti{Y})^{\dR/\CC_{p}}\to \ti{Y}^{\dR/\CC_{p}} \) is a \(G^{\sm}\)-torsor,
where the fiber product is taken in the category of diamonds. Moreover,
\((X\times_{Y}\ti{Y})^{\dR/\CC_{p}}\to X^{\dR/\CC_{p}}\) is a \(K^{\sm}\)-torsor.
It suffices to show finally that \(\ti{Y}^{\dR/\CC_{p}}\to Y^{\dR/\CC_{p}}\) is a \(K^{\sm}\)-torsor.
We are then done by Theorem \ref{thmAnDRStackGeneral} (5).
\end{proof}
The following is a filtered version of \cite[Proposition 5.1.4]{Juan2024analyticdeRham}. It will be used in the next subsection for the comparison of two definitions of algebraic de Rham stacks (Lemma \ref{lemTwoAlgDRSame}):
\begin{lem}\label{lemClosedImmeAlgDRFil}
Let \( \Spec(A/I)\hookrightarrow \Spec(A) \) be a Zariski closed immersion defined by a finite regular sequence. 
Then we have a natural isomorphism \[
\Spf^{+}(A,I^{\bullet})\cong \AnSpec((A/I,\ZZ)_{\square})^{\pdR/(\AnSpec((A,\ZZ)_{\square})),+},
\] where the LHS is defined in Definition \ref{dfnSpfFunctor}, and \( I^{\bullet} \) 
denotes the \( I \)-adic  filtration on \( A \). 
\end{lem}
\begin{proof}
By the construction in Example \ref{egScheme}, the association \( \Spec(A)\mapsto \AnSpec((A,\ZZ)_{\square}) \) commutes with fiber product. Thus by Theorem \ref{thmAnDRStackGeneral} (1), we can reduce to the case where \( A=\ZZ[x_{1},\ldots,x_{n}] \) and \( I=(x_{1},\ldots,x_{n}) \). We can further reduce easily to the case where 
\( A=\ZZ[x] \) and \( I=(x) \).

Recall from Theorem \ref{thmFilterdatA1modGm} that as graded \( A[t] \)-algebras  \[ \mrm{Rees}(A,x^{\bullet})\cong A[t]\oplus (x\cdot t^{-1})\oplus (x^{2}\cdot t^{-2})\oplus\cdots\oplus (x^{N}\cdot t^{-N})\oplus\cdots\cong A[t,t']/(tt'-x),\]
and
\begin{align*}
\mrm{Rees}(A/x^{N},x^{\bullet})\cong A/(x)^{N}[t]\oplus ((x)/(x)^{N}\cdot t^{-1})
\oplus\cdots\oplus ((x)^{N-1}/(x)^{N}\cdot t^{-N+1})\\\cong A[t,t']/(tt'-x,t^{\prime,N}).
\end{align*}
Thus by \cite[Proposition 2.5.7]{Juan2024analyticdeRham}, \[
\varinjlim_{N}\AnSpec(\mrm{Rees}(A/x^{N},x^{\bullet}))\to \AnSpec(\mrm{Rees}(A,x^{\bullet}))
\] is a monomorphism,
and thus \[\Spf^{+}(A,x^{\bullet})=
\varinjlim_{N}\AnSpec(\mrm{Rees}(A/x^{N},x^{\bullet}))/\GG_{m}\to \AnSpec(\mrm{Rees}(A,x^{\bullet}))/\GG_{m}
\] is also a monomorphism. 


Unwinding the definition, in terms of functor of points over \( \AnSpec(A)\times [\mA^{1}/\GG_{m}] \), \( [\AnSpec(\mrm{Rees}(A,x^{\bullet}))/\GG_{m}] \) represents \[
(B\supset \mcal{L})\mapsto \Hom_{A}(A/(x), B/^{\mbb{L}}\mcal{L})\cong\Hom((x)\otimes_{A}B,\mcal{L}),
\] where the second \( \Hom \) 
is taken in the category of quasi-ideals of \( B \). 
Such a point lies in the full sub-anima \( \Spf^{+}(A,x^{\bullet})(B\supset \mcal{L}) \) if and only if 
there exists a \( ! \)-surjection \( B\to B' \),  such that
the corresponding \( (x)\otimes \mcal{L}^{-1}\to B' \) is a (uniformly) nilpotent ideal.  

We define  \(i: \Hom_{A}(A/x,B/^{\mbb{L}}(\mcal{N}\otimes \mcal{L}))\to \Hom_{A}(A/x,B/^{\mbb{L}}\mcal{L}) \) (for some uniformly nilpotent ideal \(\mcal{N}\))
by composing with \( B/^{\mbb{L}}(\mcal{N}\otimes \mcal{L})\to B/^{\mbb{L}}\mcal{L} \). Taking colimits along \( \mcal{N} \) and taking sheafification, we obtain a map \[ i:\AnSpec(A/x)^{\pdR/(\AnSpec(A)),+}\to \AnSpec(\mrm{Rees}(A,x^{\bullet}))/\GG_{m}. \]
Moreover, note that \( ! \)-locally,  the corresponding map of quasi-ideals \( (x)\to \mcal{L} \) 
factors through \( \mcal{N}\otimes\mcal{L} \) 
for some uniformly nilpotent ideal \( \mcal{N} \). Thus \( (x)\otimes\mcal{L}^{-1} \) is also uniformly nilpotent, and \( i \)  
has a unique factorization \[ i:\AnSpec(A/x)^{\pdR/(\AnSpec(A)),+}\to \Spf^{+}(A,x^{\bullet}). \]

Let us study the fiber of \( i \). Assume that we have \(\AnSpec(B)\to \Spf^{+}(A,x^{\bullet})\), corresponding to \( f:A/(x)\to B/^{\mbb{L}}\mcal{L} \)  
such that \( (x)\otimes\mcal{L}^{-1}\to B \) 
is uniformly nilpotent. We can further localize and assume that \( \mcal{L}\cong B \), and \(B\cong \mcal{L}\to B \) corresponds to \( y\in \pi_{0}(B) \). 
Then the fiber 
\begin{align*}
\AnSpec(A/x)^{\pdR/(\AnSpec(A)),+,pre}\times_{\Spf^{+}(A,x^{\bullet})}\AnSpec(B)
\\\cong 
\varinjlim_{\mcal{N}\subset B}\Hom((x),\mcal{N}\otimes\mcal{L})\times_{\Hom((x),\mcal{L})}*,
\end{align*}
where the colimit is taken over all the uniformly nilpotent ideals \(\mcal{N}\), and \( * \to \Hom((x),\mcal{L})\) corresponds to \(f:(x)\to \mcal{L}\).
Taking \( \mcal{N}=(x)\otimes\mcal{L}^{-1} \), we know that this anima is non-empty.

On the other hand, we can choose \( \mcal{N} \) such that it is pulled back from uniformly nilpotent ideal of \( \pi_{0}(B) \), as such \( \mcal{N} \) form a cofinal system. Then  \( \Hom((x),\mcal{N}\otimes\mcal{L})\to \Hom((x),\mcal{L}) \) is a monomorphism.
To be more precise, by definition, \[
\Hom((x),\mcal{N}\otimes\mcal{L})\cong \Hom_{A}(A/x,B/(\mcal{N}\otimes\mcal{L})),
\] 
Since \[ B/(\mcal{N}\otimes\mcal{L})\cong B/\mcal{N}\times_{(B/\mcal{N})/^{\mbb{L}}y}B/^{\mbb{L}}y, \]
up to replacing \( B \) by \( B/\mcal{N} \), 
it suffices to show that \[
\Hom_{A}(A/x,B)\to \Hom_{A}(A/x,B/^{\mbb{L}}y)
\] is a monomorphism when \( B\cong\pi_{0}(B) \). But this is clear:
 \( B\times_{B/^{\mbb{L}}y}B\) is \( 0 \)-truncated, 
and
\( \Hom_{A}(A/x,B\times_{B/^{\mbb{L}}y})\cong \Hom_{A}(A/x,B) \), since both are \( 0 \)-truncated, and are singletons if \( x=0\in \pi_{0}(B) \).   
\end{proof}

\subsection{Logarithmic de Rham stack}\label{subsecProperAlgLogDRstack}
In this subsection, we introduce a version of the ``algebraic logarithmic de Rham stack" (Definition \ref{dfnFilterDRAlg}). 

Before we start, let us explain the idea of our construction: for smooth rigid varieties \(\mX=\Spa(A,A^{+})\) over \(L\), by Theorem \ref{thmAnDRStackGeneral} (1)(3),
the algebraic de Rham stack (relative to \(L\)) is isomorphic to
\(\mX/\hat{\Delta}_{X}\),
where \(\hat{\Delta}_{X}\) is the completion along the Zariski closed immersion \(\Delta:X\hookrightarrow X\times_{L} X\) regarded as an ind-rigid varieties. 
Therefore, to define a log version of the algebraic de Rham stack,
we want a log version of \(\hat{\Delta}_{X}\). 

One way of reconstructing \(\hat{\Delta}_{X}\) 
is to consider the fiber product \(X\times_{X^{\pdR}}X\).
On the other hand, one expects (the filtered complete version of) \(X^{\pdR}\) to be affine, and isomorphic to \(\Spa(\dR(A))\),
for example in the sense of \cite{MathewMondal25} or \cite[Definition 7.15]{BhattLurie2022prismatization} (see Lemma \ref{lemLogDRStackIsAffine} below). 
Therefore, one may expect that \(\hat{\Delta}_{X}\cong \Spf(A\otimes_{\dR(A)}A)\) in a suitable sense. This turns out to be true (see Lemma \ref{lemDiagonalLogProperty} (2)). This also enables us to define a log version of \(\hat{\Delta}_{X}\),
by simply replacing \(\dR(A)\) by \(\dR_{\log}(A)\).
\begin{rmk}
When the log structure is induced by a strict normal crossings divisor, the log de Rham stack has been constructed in \cite{Barz2025logarithmicrhamstacksnonabelian}, following an idea of Martin Olsson. In general, one expects to define it as the relative de Rham stack with respect to the stack of log structures, where the latter has been studied in \cite{Zhang2026derivedlogarithmicdeformationtheory}.
\end{rmk}
\subsubsection{Construction}
We work in the one of the following settings:
\begin{enumerate}
\item (Algebraic setting) Let \(L\) be an algebraic field.
Let \(X\) be a fine log smooth (algebraic) variety over \(L\) of dimension \(n\) (\cite[Definition 3.1.1]{OgusArthur2018Lola}) realized as a solid stack by Example \ref{egScheme}.
\item (Analytic setting) Let \(L\) be a \(p\)-adic field over \(\QQ_{p}\).
Let \(X\) be an fs log smooth rigid variety over \(L\) of dimension \(n\) (\cite[Definition 3.1.9]{DLLZ2019logarithmicFoundational}) realized as a solid stack by Example \ref{egAnAdicSpace}.
\end{enumerate}

Then by \cite[Definition IV.1.2.1, Proposition 3.2.1]{OgusArthur2018Lola} and \cite[Definition 3.3.6, Theorem 3.3.17]{DLLZ2019logarithmicFoundational}, we have the locally free sheaf of log differentials \(\Omega^{1}_{X,\log}:=\Omega^{1}_{X/L,\log}\) 
with a connection \[\nabla_{\log}:\mO_{X}\to \Omega^{1}_{X,\log}.\]
\begin{dfn}
We define \[\dR_{\log}(\mO_{X}):=\dR_{\log}(\mO_{X}/L):=\left[\mO_{X}\to \Omega^{1}_{X,\log}\to \Omega^{2}_{X,\log}
\to \cdots\to \Omega^{n}_{X,\log}
\right],\]
which is a sheaf of \(\mbb{E}_{\infty}\)-algebras on the analytic site \(X_{\an}\). We endow \(\dR_{\log}(\mO_{X})\)
with the filtration defined by the stupid truncation,
that is, \[\gr^{i}(\dR_{\log}(\mO_{X}))\cong \begin{cases}
	\Omega^{i}_{X,\log},&0\le i\le n;\\
	0,&\mrm{else}.
\end{cases}\]
\end{dfn}

\begin{lem}\label{lemDiagonalLogProperty}
(1)
Consider the \emph{filtered complete} tensor product \(\mO_{X}\hatotimes_{\dR_{\log}(\mO_{X})}\mO_{X}\). To be more precise,
we first consider the tensor product \(\mO_{X}\otimes_{\dR_{\log}(\mO_{X})}\mO_{X}\), which has an induced filtration, by putting the trivial filtration on both \(\mO_{X}\),
and then we take the filter completion.

Then \(\mO_{X}\hatotimes_{\dR_{\log}(\mO_{X})}\mO_{X}\) is concentrated in degree 0, and the graded pieces are of the form \[\gr^{i}(\mO_{X}\hatotimes_{\dR_{\log}(\mO_{X})}\mO_{X})\cong \Sym^{i}_{\mO_{X}}(\Omega^{1}_{X,\log}),i\in \NN,\]
and \(\Sym^{i}\gr^{1}\to \gr^{i}\) is an isomorphism. 
In particular, analytic locally, we have a non-canonical filtered isomorphism \(\mO_{X}\hatotimes_{\dR_{\log}(\mO_{X})}\mO_{X}\cong \mO_{X}[[\Omega^{1}_{X,\log}]]\).

(2) If \(X\) is endowed with the trivial log structure,
then \[\mO_{X}\hatotimes_{\dR(\mO_{X})}\mO_{X}\cong (\mO_{X}\otimes_{L}\mO_{X})^{\wedge}_{I},\]
where \(I\) is the ideal defining the Zariski closed embedding \(X\hookrightarrow X\times X\).


(3)
We can endow 
\(\mO_{X}\hatotimes_{\dR_{\log}(\mO_{X})}\mO_{X}\)
with a bi-\(\mO_{X}\)-Hopf algebra structure via 
\begin{align*}
\mO_{X}\hatotimes_{\dR_{\log}(\mO_{X})}\mO_{X}&
\xrightarrow{f\otimes f'\mapsto f\otimes 1\otimes f'}
\mO_{X}\hatotimes_{\dR_{\log}(\mO_{X})}\mO_{X}\hatotimes_{\dR_{\log}(\mO_{X})}\mO_{X} \\
&\cong
(\mO_{X}\hatotimes_{\dR_{\log}(\mO_{X})}\mO_{X})\otimes_{\mO_{X}}(\mO_{X}\hatotimes_{\dR_{\log}(\mO_{X})}\mO_{X}),\end{align*}
and we have an isomorphism of sheaves of bi-\(\mO_{X}\)-algebras \[\RHom_{\mO_{X}}(\mO_{X}\hatotimes_{\dR_{\log}(\mO_{X})}\mO_{X},\mO_{X})
\cong D_{X,\log},\]
the algebra of log differential operators. Here the \(\RHom(-)\)
 is taken in the filtered sense,
 that is, \[
	\RHom_{\mO_{X}}(\mO_{X}\hatotimes_{\dR_{\log}(\mO_{X})}\mO_{X},\mO_{X}):=\varinjlim_{n}
	\RHom_{\mO_{X}}(\mO_{X}\hatotimes_{\dR_{\log}(\mO_{X})}\mO_{X}/\Fil^{n},\mO_{X}).
 \]
\end{lem}
\begin{proof}
(1) 
Since we have taken filtered completion, it suffices to show that \[
\gr^{*}(\mO_{X}\times_{\dR_{\log}(\mO_{X})}\mO_{X})\cong \bigoplus_{i}\Sym^{i}(\Omega^{1}_{X,\log}).
\]
Note that by \cite[Proposition 3.2.1]{Lurie2015rotation}, the LHS is isomorphic to \(\mO_{X}\otimes_{\gr^{*}(\dR_{\log}(\mO_{X}))}\mO_{X}\),
where \(\mO_{X}\) is endowed with the trivial grading.
We also have \[
\gr^{*}(\dR_{\log}(\mO_{X}))
\cong \bigoplus_{i}\wedge^{i}(\Omega^{1}_{X,\log})[-i].
\] as a graded commutative algebra.

Thus it suffices to show the following algebraic theorem:
if \(R\) is an \(\mbb{E}_{\infty}\)-ring over \(\QQ\), \(M\) is a vector bundle over \(R\),
then we have a natural isomorphism of graded algebras \begin{align}\label{alignKoszulDuality}
R\otimes_{\bigoplus_{i}\wedge^{i}M[-i]}R\cong \bigoplus_{i}\Sym^{i}M.
\end{align} Now note that by the argument in \cite[Lemma 7.8]{BhattLurie2022prismatization}, \[
\bigoplus_{i}\wedge^{i}M[-i]\cong R\Gamma(BM,\mO_{BM})
\] as graded \(\mbb{E}_{\infty}\)-algebras,
and we have a natural map \(R\otimes_{\bigoplus_{i}\wedge^{i}M[-i]}R\to \bigoplus_{i}\Sym^{i}M\) by the isomorphism \(\Spec(R)\times_{BM}\Spec(R)\cong M\). 
To verify that the natural map is an isomorphism, we can localize to assume that \(M\) is free over \(R\) of rank \( 1 \). Then it follows from an explicit calculation using Barr construction (\cite[Construction 4.4.2.7, Theorem 4.4.2.8]{Lurie2017HA}).

Alternatively, by the Cartier duality (\cite[Proposition 2.4.4]{BhattLurie2022prismatic}), 
\(\QCoh(B\GG_{a})\) is identified with the category of \(k[t]\)-modules where \(t\) acts nilpotently, which implies that \(\mO_{B\GG_{a}}\) is a compact generator of \(\QCoh(B\GG_{a})\).
Thus by Barr-Beck-Lurie theorem (\cite[Theorem 4.7.3.5]{Lurie2017HA}), \[
\QCoh(B\GG_{a})\cong \Mod(R\Gamma(B\GG_{a},\mO_{B\GG_{a}})),
\] and then 
\begin{align*}
\Mod(\Sym^{*}M)
\cong \QCoh(M)
\cong \Mod(R)
\otimes_{\QCoh(BM)}\Mod(R)\\
\cong \Mod(R\otimes_{R\Gamma(BM,\mO_{BM})}R),
\end{align*}
where the second isomorphism holds by applying \cite[Theorem 4.8.5.16 (4)]{Lurie2017HA} to
\( \Mod(R)\cong \QCoh(M/M)\cong\Mod_{\Sym^{*}_{R}M^{\vee}}(\QCoh(BM)) \), 
and \[ \QCoh(M)\cong \QCoh((M\times M)/M)\cong \Mod_{\Sym^{*}_{R}M^{\vee}\otimes \Sym^{*}_{R}M^{\vee}}(\QCoh(BM)), \]
and the last isomorphism also follows from \cite[Theorem 4.8.5.16 (4)]{Lurie2017HA}. 
This implies the desired isomorphism (\ref{alignKoszulDuality}).

(2)
We now consider \(X\) with the trivial log structure,
and \(X=\Spa(A,A^{+})\).
Then there is a natural map \(A\hatotimes_{L}A\to A\hatotimes_{\dR(A)}A\)
mapping \(f\otimes f'\)
to \(f\otimes f'\).
For any element \(a\otimes 1-1\otimes a\in I\),
we see that it is mapped into \(\Fil^{1}(A\hatotimes_{\dR(A)}A)\),
so the map extends uniquely to a morphism between filtered colimit algebras \[(A\hatotimes_{L}A)^{\wedge}_{I}\to A\hatotimes_{\dR(A)}A.\]
Note that on both sides, \(\gr^{i}\)
is the \(i\)-th symmetric power of \(\gr^{1}\).
Now in order to check that it is an isomorphism,
it suffices to show that there is an isomorphism on
\(\gr^{1}(-)\). We can find analytic locally a \'etale map to the torus (\cite[Corollary 1.6.10]{Huber13}), and thus reduce to the dimension 1 case.
We now consider the Barr complex computing \(A\hatotimes_{\dR(A)}A/\Fil^{2}\) (\cite[Construction 4.4.2.7, Theorem 4.4.2.8]{Lurie2017HA}). Explicitly, it is given by the fiber of the morphism from the complex formed by the first row to that of the second row \[\begin{tikzcd}\cdots \arrow[r] &[-0.3cm]
	A\otimes	A\otimes A\otimes A\arrow[r]\arrow[d] &
	A\otimes A\otimes A\arrow[r,"f"]\arrow[d,"1\otimes d\otimes 1"] & [-0.3cm] A\otimes A\arrow[d]
	\\
	\cdots \arrow[r] &
(A\otimes	\Omega^{1}_{A}\otimes A\otimes A)
\oplus (A\otimes	A\otimes \Omega^{1}_{A}\otimes A)
\arrow[r] &
A\otimes \Omega^{1}_{A}\otimes A\arrow[r] &0.
\end{tikzcd}\]
Now for \(a\otimes 1-1\otimes a\in A\otimes A\),
\(1\otimes a\otimes 1\in A\otimes A\otimes A\)
is a lift along \(f\),
and its image in 
\(A\otimes \Omega^{1}_{A}\otimes A\)
is \(1\otimes da\otimes 1\).
This shows that \(a\otimes 1-1\otimes a\in\gr^{1}((A\hatotimes_{L}A)^{\wedge}_{I})\)
is mapped to \(da\in \gr^{1}(A\hatotimes_{\dR(A)}A)\),
which implies that the map induces an isomorphism on \(\gr^{1}(-)\) via the classical isomorphism \(I/I^{2}\cong \Omega^{1}_{X}\).

(3) 
Note that by adjunction, \[
\RHom_{\mO_{X}}(\mO_{X}\otimes_{\dR_{\log}(\mO_{X})}\mO_{X},\mO_{X})
\cong \RHom_{\dR_{\log}(\mO_{X})}(\mO_{X},\mO_{X}),
\] which is an isomorphism of bi-\( \mO_{X} \)-algebras.
Moreover, both sides 
are concentrated in degree 0, as (non-canonically) \((\mO_{X}\hatotimes_{\dR_{\log}(\mO_{X})}\mO_{X})/\Fil^{n}\cong \bigoplus_{i=0}^{n-1}\Sym^{i}(\Omega^{1}_{X,\log})\).
Therefore, it suffices to construct an isomorphism \( D_{X,\log}\to \pi_{0}\RHom_{\dR_{\log}(\mO_{X})}(\mO_{X},\mO_{X}) \).

After localization, we assume that \( X \) is affine or affinoid. 
Then both 
\( H^{0}(X,D_{X,\log}) \)
and \( H^{0}(X,\End_{\dR_{\log}(\mO_{X})}(\mO_{X})) \) can be regarded as subalgebras of \( H^{0}(X,\End_{L}(\mO_{X})) \). Clearly, \( H^{0}(X,\mO_{X}) \) is contained in both subalgebras. For any first order log differential \( D\in H^{0}(X,\Omega^{1,\vee}_{X,\log})\subset H^{0}(X,D_{X,\log}) \),
\( D:\mO_{X}\to \mO_{X} \) is induced by the composition 
\begin{align*}
&\mO_{X}\xrightarrow{a\mapsto 1\otimes a}
\mO_{X}\otimes_{\dR_{\log}(\mO_{X})}\mO_{X}
\xrightarrow{a\otimes b\mapsto (ab,ad(b))}\\&
(\mO_{X}\otimes_{\dR_{\log}(\mO_{X})}\mO_{X})/\Fil^{2}
\cong \mO_{X}\oplus \Omega^{1}_{X,\log}
\xrightarrow{(0,D)}
\mO_{X},
\end{align*}
where the description of the second map is given by the argument in (2),
and the third map is induced by the splitting \[
\mO_{X}\to 
(\mO_{X}\otimes_{\dR_{\log}(\mO_{X})}\mO_{X})/\Fil^{2},a\mapsto a\otimes 1
\]

Therefore, \( H^{0}(X,\Omega^{1,\vee}_{X,\log})\subset \End_{\dR_{\log}(\mO_{X})}(\mO_{X}) \), and induces an isomorphism on \( \gr^{1} \).
Since \( H^{0}(X,\Omega^{1,\vee}_{X,\log}) \)
is generated by \( H^{0}(X,\Omega^{1,\vee}_{X,\log}) \)
over \( H^{0}(X,\mO_{X}) \), 
we know that \[
H^{0}(X,D_{X,\log}) \xhookleftarrow{i_{0}} H^{0}(X,\End_{\dR_{\log}(\mO_{X})}(\mO_{X})) \subset H^{0}(X,\End_{L}(\mO_{X})).
\] 
Since by (1), \( \gr^{\bullet}(\mO_{X}\otimes_{\dR_{\log}(\mO_{X})}\mO_{X})\cong \mO_{X}[\Omega^{1}_{X}] \) as a graded bi-\( \mO_{X} \)-Hopf algebra, 
\[ \gr^{\bullet}(\RHom_{\mO_{X}}(\mO_{X}\otimes_{\dR_{\log}(\mO_{X})}\mO_{X},\mO_{X}))\cong  \mO_{X}[\Omega^{1,\vee}_{X}] \]
as a graded bi-\( \mO_{X} \)-algebra, with \( \gr^{n} \) generated by \( \gr^{1} \). 
Since \( i_{0} \) induces isomorphisms on \( \gr^{0} \)
and \( \gr^{1} \), we know that it also induces isomorphisms on \( \gr^{i} \) for any \( i\in\NN \), as desired.
\end{proof}

\begin{dfn}\label{dfnLogDiagnal}
We define \[\hat{\Delta}_{X,\log}^{+}:=\unl{\Spf}^{+}_{X}((\mO_{X}\hatotimes_{\dR_{\log}
(\mO_{X})}\mO_{X}),\]
where the RHS is as in Definition \ref{dfnSpfFunctor} (2).
Then \(\hat{\Delta}_{X,\log}^{+}\)
has two projections \(p_{1},p_{2}\) to \(X\times[\mA^{1}/\GG_{m}]\),
with a section \(\Delta:X\times [\mA^{1}/\GG_{m}]\to \hat{\Delta}_{X,\log}^{+}\).

We denote by  \(\hat{\Delta}_{X,\log}\) the restriction of \(\hat{\Delta}^{+}_{X,\log}\)
to \([\GG_{m}/\GG_{m}]\subset [\mA^{1}/\GG_{m}]\), which also has \(p_{1},p_{2}:\hat{\Delta}_{X,\log}\to X\)
and \(\Delta:X\to \hat{\Delta}_{X,\log}\).
\end{dfn}
\begin{lem}\label{lemThisIsGroupoid}
	Consider \(U:N(\Delta)^{op}\to \mrm{SolidStk}\),
	given by \(U([0]):=X\),
	\(U([1]):=\hat{\Delta}_{X,\log}^{+}\),
	and \( U([n]) \) be the \(n\)-th fiber product \[\hat{\Delta}_{X,\log}^{+}\times_{p_{1},X\times[\mA^{1}/\GG_{m}],p_{2}}\hat{\Delta}_{X,\log}^{+}\times_{p_{1},X\times[\mA^{1}/\GG_{m}],p_{2}}\cdots \times_{p_{1},X\times[\mA^{1}/\GG_{m}],p_{2}}\hat{\Delta}_{X,\log}^{+}.\] Then \(U\) defines a groupoid object in the sense of \cite[Definition 6.1.2.7]{Lurie2009HTT} in the category of solid stacks over \(\mA^{1}/\GG_{m}\).
\end{lem}
\begin{proof}
We easily reduce to the case where 
\(X=\AnSpec(A)\) is affine.
By \cite[Proposition 6.1.2.6]{Lurie2009HTT},
it suffices to verify that for any \(n\ge 0\)
and \([n]=S\cup S'\)
with \(S\cap S'=\{s\}\),
the natural morphism \(U([n])\to U(S)\times_{U(\{s\})}U(S')\) is an isomorphism.
By Lemma \ref{lemSpfCommuteWithLimit}, \(U([n])\)
is canonically associated to the filtered complete algebra \(A^{\otimes_{\dR_{\log}(A)}n}\) via the functor \(\unl{\Spf}^{+}_{X}(-)\)
in Definition \ref{dfnSpfFunctor},
and the cosimplicial system \(A^{\otimes_{\dR_{\log}(A)}n}\) is precisely the \v Cech nerve associated to the morphism \(\dR_{\log}(A^{\rhd})\to A^{\rhd}\) in the category of filtered complete \(\mbb{E}_{\infty}\)-rings,
which then is necessarily a groupoid object by applying \cite[Proposition 6.1.2.11]{Lurie2009HTT} to the opposite category of filtered colimit \(\mbb{E}_{\infty}\)-rings.
\end{proof}
We will now define the (filtered) algebraic logarithmic de Rham stack.
\begin{dfn}[Algebraic logarithmic de Rham stack]\label{dfnFilterDRAlg}
We define the \emph{filtered algebraic logarithmic de Rham stack} (relative to \(L\)) as the quotient of \(X\times [\mA^{1}/\GG_{m}]\) by the groupoid object in Lemma \ref{lemThisIsGroupoid}: 
\[X^{\pdR/L,+}_{\log}:=(X\times[\mA^{1}/\GG_{m}])/\hat{\Delta}_{X,\log}^{+}.\]
Note that by \cite[Theorem 6.1.0.6]{Lurie2009HTT},
we know that \[(X\times [\mA^{1}/\GG_{m}])\times_{X^{\pdR/L,+}_{\log}}(X\times [\mA^{1}/\GG_{m}])\cong \hat{\Delta}_{X,\log}^{+}.\]

We define the \emph{algebraic logarithmic de Rham stack} \(X^{\pdR/L}_{\log}\) (resp. \emph{algebraic logarithmic Hodge stack} \(X^{\alg-\mrm{Hodge}/L}_{\log}\), resp. \emph{filtered complete algebraic logarithmic de Rham stack} \(X^{\pdR/L,\hat{+}}_{\log}\)) as the restriction of \(X^{\pdR/L,+}_{\log}\)
to \([\GG_{m}/\GG_{m}]\subset [\mA^{1}/\GG_{m}]\) (resp. \(B\GG_{m}\subset [\mA^{1}/\GG_{m}]\), resp. \([\hat{\mA}^{1}/\GG_{m}]\subset [\mA^{1}/\GG_{m}]\)). 
\end{dfn}
\begin{rmk}\label{rmkFilDRtoDR}
By Remark \ref{rmkFromFilterToNoFil},
we have natural morphisms \[
X\times [\mA^{1}/\GG_{m}]\to X^{\pdR/L,+}_{\log}\to X^{\pdR/L}_{\log}\times [\mA^{1}/\GG_{m}]\to [\mA^{1}/\GG_{m}].
\]
\end{rmk}

The following lemma ensures that there is no confliction of notation.
\begin{lem}\label{lemTwoAlgDRSame}
(1) 
There is a natural factorization \( X\to X^{\pdR/L}_{\log}\to X^{\pdR/L} \), where \( X^{\pdR/L} \) is as in Definition \ref{dfnAlgDRJuan}. 

(2)
If \(X\) is smooth over \(L\) endowed with the trivial log structure, then \(X^{\pdR/L,+}_{\log}\) (Definition \ref{dfnFilterDRAlg})
coincides with \(X^{\pdR/L,+}\) (Definition \ref{dfnAlgDRJuan}).
\end{lem}
\begin{proof}
For (1), note that by  Definition \ref{dfnAlgDRJuan}, \[ \Hom(\hat{\Delta}_{X,\log},X^{\pdR/L})\cong \Hom(X,X^{\pdR/L}). \] Thus the simplicial system in Definition \ref{dfnFilterDRAlg} becomes constant after applying \( \Hom(-,X^{\pdR/L}) \), which induces \( X^{\pdR/L}_{\log}\to X^{\pdR/L} \).

For (2),
by \cite[Proposition 5.1.4]{Juan2024analyticdeRham}, we know that \( X\times[\mA^{1}/\GG_{m}]\to X^{\pdR/L,+} \) 
is a \( ! \)-surjection, and \[
(X\times[\mA^{1}/\GG_{m})]\times_{X^{\pdR/L,+}}(X\times[\mA^{1}/\GG_{m})]
\cong X^{\pdR/(X\times_{L}X),+}\cong \hat{\Delta}_{X}^{+},
\] where the last isomorphism is given by Lemma \ref{lemClosedImmeAlgDRFil}.
In particular, after pulling back to \( X\times \mA^{1} \), both sides are \( 0 \)-truncated, 
and is an isomorphism of groupoid objects over \( X\times\mA^{1} \). After taking simplicial colimits, we obtain the desired isomorphism. 
\end{proof}

\subsubsection{Properties}
\begin{lem}\label{lemAlgebHTStack}
We have an isomorphism \[X^{\alg-\mrm{Hodge}/L}_{\log}\cong [(X\times B\GG_{m})/\hat{T}_{X,\log}\{-1\}],\]
where 
\(\hat{T}_{X,\log}\{-1\}\)
denotes the formal completion of 
\(T_{X,\log}\{-1\}=(\Omega^{1}_{X,\log})^{\vee}\{-1\}\),
along the zero section, 
and \(\{-1\}\) is as in Notation \ref{notationTwistByAGm}.
\end{lem}
\begin{proof}
By Lemma \ref{lemDiagonalLogProperty} (1), we know \[\hat{\Delta}^{+}_{X,\log}|_{X\times B\GG_{m}}\cong \varinjlim_{i}\unl{\AnSpec}_{X\times B\GG_{m}} (\Sym^{*}(\Omega^{1}_{X,\log}\{1\})/\Fil^{i})\cong \widehat{T}_{X,\log}\{-1\},
\] and \(p_{1},p_{2}\)
are both given by the natural projection. We thus have the desired isomorphism.
\end{proof}
\begin{lem}\label{lemCalDiagonalSmooth}
Consider \(p_{1},p_{2}:\hat{\Delta}_{X,\log}^{+}\to X\times [\mA^{1}/\GG_{m}]\). Then \(p_{i}\) are suave \(!\)-covers (Definition \ref{dfnVariousPropertyMorphiDescent}) for \(i=1,2\), and \(p_{1}^{!}(\mO_{X\times [\mA^{1}/\GG_{m}]})\) is a line bundle shifted to (cohomological) degree \(-\dim(X)\).
\end{lem}
\begin{proof}
Locally after fixing a lift of \(\gr^{1}\), 
we know by Lemma \ref{lemDiagonalLogProperty} (1) that \[\hat{\Delta}_{X,\log}\cong \varinjlim_{n}\underline{\AnSpec}_{X\times [\mA^{1}/\GG_{m}]}(\mO_{X}[\Omega^{1}_{X,\log}\{1\}]/(\Sym^{n}\Omega^{1}_{X,\log})) \]
is isomorphic to the completion at the zero section of the tangent bundle \(T_{X,\log}\{-1\}=\unl{\AnSpec}_{X\times[\mA^{1}/\GG_{m}]}(\mO_{X}[\Omega^{1}_{X,\log}\{1\}])\), with the natural morphism to \( X\times [\mA^{1}/\GG_{m}] \) being induced by \( p_{1} \). 
We can conclude as \(\widehat{T}_{X,\log}\{-1\}\) is cohomologically smooth over \(X\times [\mA^{1}/\GG_{m}]\) by \cite[Proposition 4.2.4 (2)]{Juan2024analyticdeRham}. It has a section, so is a \( ! \)-surjection. Thus it is a suave \( ! \)-cover.  
\end{proof}

\begin{lem}\label{lemBasicFiltered}
The map \(h:X\times [\mA^{1}/\GG_{m}]\to X_{\log}^{\pdR/L,+}\)
is a suave \( ! \)-cover (Definition \ref{dfnVariousPropertyMorphiDescent}) 
and 
\(h^{!}(\mO_{X_{\log}^{\pdR/L,+}})\cong \omega_{X,\log}\{d\}[d]\),
with \(d:=\dim(X)\), \(\omega_{X,\log}:=\wedge^{d}(\Omega^{1}_{X,\log})\), and \(\{d\}\) is as in Notation \ref{notationTwistByAGm}.
In particular, \( h \) satisfies universal \( ! \)-descent and universal \( * \)-descent.  
\end{lem}
\begin{proof}
\(h\) is a \(!\)-surjection by construction, and thus \( h^{*} \) 
is conservative.
We also have by Definition \ref{dfnFilterDRAlg} that \[(X\times [\mA^{1}/\GG_{m}])\times_{X^{\pdR/L,+}_{\log}}(X\times [\mA^{1}/\GG_{m}])\cong \hat{\Delta}_{X,\log}^{+}.\]
Then the cohomological smoothness of \(h\) follows from Lemma \ref{lemCalDiagonalSmooth} and Theorem \ref{thmVariousDescent}. 
By Lemma \ref{lemBasicLogDRSta}, we know \(h^{!}(\mO_{\mX_{\log}^{\pdR/L,+}})\) is a line bundle over \(X\times [\mA^{1}/\GG_{m}]\) shifted to degree \(-d\). Consider the restriction to \(i:B\GG_{m}\to \mA^{1}/\GG_{m}\). Then \[
\gr^{*}(h^{!}(\mO_{\mX_{\log}^{\pdR/L,+}}))\cong h^{\prime,!}(\mO_{X^{\alg-\mrm{Hodge}/L}_{\log}}),
\] for \(\gr^{*}(-):=i^{*}\), and \[
h':X\times B\GG_{m}\to X^{\alg-\mrm{Hodge}/L}_{\log}\cong [X\times B\GG_{m}/\hat{T}_{X,\log}(-1)],
\] where the last isomorphism is given by Lemma \ref{lemAlgebHTStack}.
By \cite[Proposition 4.2.5 (3)]{Juan2024analyticdeRham}, we know \[
h^{\prime,!}(\mO_{X^{\alg-\mrm{Hodge}/L}_{\log}})\cong (\wedge^{d}(T_{X,\log}(-1))[-d])^{-1}\cong \omega_{X,\log}\{d\}[d].
\] Since \(h^{!}(\mO_{X_{\log}^{\pdR/L,+}})\) is a perfect complex, we conclude that \(h^{!}(\mO_{\mX^{\pdR/L,+}_{\log}})\cong \omega_{X,\log}\{d\}[d]\) by Lemma \ref{lemPerfComplexA1Gm}.

The last statment follows from Theorem \ref{thmVariousDescent}.
\end{proof}

\begin{lem}\label{lemBasicLogDRSta}
(1) The map \(h:X\to X_{\log}^{\pdR/L}\)
is a suave \( ! \)-cover  and \(!\)-surjective,
and 
\(h^{!}(\mO_{X_{\log}^{\pdR/L}})\cong \omega_{X,\log}[d]\),
with \(d:=\dim(X)\), and \(\omega_{X,\log}:=\wedge^{d}(\Omega^{1}_{X,\log})\).

Moreover, \( X^{\pdR/L,+}_{\log}\times_{[\mA^{1}/\GG_{m}]}[\mA^{1,\an}/\GG_{m}^{\an}] \) and \( X^{\pdR/L}_{\log} \) 
are essentially Tate (Definition \ref{dfnEssentialTate}), and we have equivalences \( \Vect(X^{\pdR/L}_{\log})\cong \Vect((X^{\pdR/L}_{\log})^{Tate}) \).

By abuse of notation, we will write \( X^{\pdR/L}_{\log}:=(X^{\pdR/L}_{\log})^{Tate} \). 

(2) Assume in addition that \(X\) is smooth over \(L\).
We have a natural factorization \[X\to X^{\pdR/L}_{\log}\to X^{\dR},\]
and the induced map \(X^{\dR}\to (X^{\pdR/L}_{\log})^{\dR}\) is a monomorphism,
and
admits a natural retract.

(3) 
If \(f:X'\to X\) is a log smooth morphism, then we have a natural morphism \(f^{\pdR/L}:X^{\prime,\pdR/L,+}_{\log}\to X^{\pdR/L,+}_{\log}\) such that there is a commutative diagram of solid stacks \[\begin{tikzcd}
	X'\times {[\mA^{1}/\GG_{m}]}\arrow[r]\arrow[d] & X^{\prime,\pdR/L,+}_{\log}\arrow[d]\\
	X\times {[\mA^{1}/\GG_{m}]}\arrow[r] & X^{\pdR/L,+}_{\log}.
\end{tikzcd}\]
If \(f\) is log \'etale,
then the diagram is Cartesian.
\end{lem}
\begin{proof}
The first part of
(1) follows from Lemma \ref{lemBasicFiltered}. 
Note that by Lemma \ref{lemDiagonalLogProperty} (1), locally after fixing a lift of \( \gr^{1} \) and a trivialization of \( \Omega^{1}_{X,\log} \),  \[ \hat{\Delta}_{X,\log}^{+}\times_{[\mA^{1}/\GG_{m}]}[\mA^{1,\an}/\GG_{m}^{\an}]\cong X\times(\hat{\mA}^{1})^{d}/\GG_{m}^{\an},
 \] which is essentially Tate by Example \ref{egEssentialTate} (3) and Lemma \ref{lemBasicPropertyEsseTate} (6), and thus
\( X^{\pdR/L,+}_{\log} \) 
is also essentially Tate by Lemma \ref{lemBasicPropertyEsseTate} (6).
Moreover, \( X\to X^{\pdR/L}_{\log} \) is a \( ! \)-surjection, and \( X \) can be covered by \( \AnSpec(A_{i}) \)
for \( A_{i} \) bounded and Fredholm by Theorem \ref{thmSolidTateIsFredholm}. Thus by Lemma \ref{lemBasicPropertyEsseTate} (1), we have  \( \Vect(X^{\pdR/L}_{\log})\cong \Vect((X^{\pdR/L}_{\log})^{Tate}) \).

(2) follows from Theorem \ref{thmAnDRStackGeneral} (3),
as we have a factorization \(\hat{\Delta}_{X,\log}\to \hat{\Delta}_{X}\to \Delta(X)^{\dagger}\). Also note that as we have seen, \(
\mO_{X}\hatotimes_{\dR_{\log}(\mO_{X})}\mO_{X}\to \mO_{X}
\) induces an isomorphism after \(\dagger\)-reduction,
and thus \[X^{\dR}\times_{(X^{\pdR/L}_{\log})^{\dR}}X^{\dR}\cong (\hat{\Delta}_{X,\log})^{\dR}\cong X^{\dR},\]
so the map is a monomorphism.
The natural retract is induced by \(X^{\pdR/L}_{\log}\to X^{\dR}\). 

For (3), by localizing, we can assume that \(X=\Spa(A,A^{+})\)
and \(X'=\Spa(A',A^{\prime,+})\), then 
\(f\) induces an exact sequence \[
0\to f^{*}\Omega^{1}_{A,\log}\to \Omega^{1}_{A',\log}\to \Omega^{1}_{A'/A,\log}\to 0,
\] and thus a filtered morphism \(\dR_{\log}(A)\to \dR_{\log}(A')\) whose \(\gr^{0}\)
is given by \(A\to A'\). This gives the desired functoriality. 

If \(f\) is log \'etale, \(f^{*}\Omega^{1}_{X,\log}\cong \Omega^{1}_{X',\log}\). Therefore, we have a natural map \[f^{*}(\mO_{X}\hatotimes_{\dR_{\log}(\mO_{X})}\mO_{X})\to \mO_{X'}\hatotimes_{\dR_{\log}(\mO_{X'})}\mO_{X'}\]
induces the natural map \[f^{*}\Sym^{i}_{\mO_{X}}\Omega^{1}_{X,\log}\to \Sym^{i}_{\mO_{X'}}\Omega^{1}_{X',\log}\] on \(\gr^{i}(-)\)
which is thus also an isomorphism. This implies that \(f^{*}\hat{\Delta}_{X,\log}^{+}\cong \hat{\Delta}_{X',\log}^{+}\), and thus we have the desired Cartesian diagram.
\end{proof}

\subsubsection{Logarithmic D-modules}


\begin{lem}\label{lemDRcomplexVSpushforward}
(1)
For \(h:X\times [\mA^{1}/\GG_{m}]\to X^{\pdR/L,+}_{\log}\), the functor 
\[h^{*}:\QCoh(X^{\pdR/L,+}_{\log})\to \QCoh(X\times [\mA^{1}/\GG_{m}])\]
lifts canonically to a \(\QCoh([\mA^{1}/\GG_{m}])\)-linear equivalence of symmetric monoidal categories \[
\QCoh(X^{\pdR/L,+}_{\log})\cong \Mod_{D_{X,\log}}(\Fil^{\ZZ}(\QCoh(X))),
\] where \(D_{X,\log}\) is endowed with the degree filtration \(\Fil^{-i}D_{X,\log}:=\mrm{Diff}^{\le i}_{\log}(\mO_{X},\mO_{X})\).

This functor identifies the subcategory \(\Vect(X^{\pdR/L,+}_{\log})\) of vector bundles on \(X^{\pdR/L,+}_{\log}\) with the subcategory of filtered vector bundles equipped with integrable log connections satisfying Griffiths transversality.

(2)
Similarly, \(h':X\to X^{\pdR/L}_{\log}\)
induces an equivalence of symmetric monoidal categories \[
\QCoh(X^{\pdR/L}_{\log})\cong \Mod_{D_{X,\log}}(\QCoh(X))=: D_{X,\log}-\Mod.
\] 
This functor identifies the subcategory \(\Vect(X^{\pdR/L}_{\log})\)with the subcategory of vector bundles equipped with integrable log connections.

(3) 
For any 
\(M\in \QCoh(X^{\pdR/L}_{\log})\) (resp. \(M\in\QCoh(X^{\pdR/L,+}_{\log})\)), we have an isomorphism (resp. an isomorphism in \( \QCoh([\mA^{1}/\GG_{m}]) \)) \[
R\Gamma(X^{\pdR/L}_{\log},M)\cong R\Gamma_{\log-\dR}(X,h^{*}M):=\RHom_{D_{X,\log}}(\mO_{X},h^{*}(M)).
\]
\end{lem}
\begin{rmk}\label{rmkResolveObyDX}
Note that \(\mO_{X}\) admits a resolution \[0\to
D_{X,\log}\otimes\wedge^{\dim(X)}(\Omega^{1}_{X,\log})^{\vee}\to \cdots
\to 
D_{X,\log}\otimes (\Omega^{1}_{X,\log})^{\vee}
\to D_{X,\log}\to \mO_{X}\to 0
\] which can be used for calculating \(\RHom_{D_{X,\log}}(\mO_{X},h^{*}(M)).\)
\end{rmk}




\begin{proof}
By Barr-Beck-Lurie theorem (\cite[Theorem 4.7.3.5]{Lurie2017HA}), the argument in \cite[Proposition 3.1.27]{Juan2024analyticdeRham} and Lemma \ref{lemBasicLogDRSta} (1), \(h^{*}\)
lifts canonically to an equivalence \[
\QCoh(X^{\pdR/L,+}_{\log})\cong \Mod_{h^{*}h_{\natural}}(\QCoh(X\times[\mA^{1}/\GG_{m}])).
\] 
We need to understand the monoid \( h^{*}h_{\natural} \).
We have
 \[
h^{*}h_{\natural}(\mF)\cong h^{*}h_{!}(\mF\otimes (h^{!}1)^{-1})
\cong p_{2,!}p_{1}^{*}(\mF\otimes (h^{!}1)^{-1})
\cong p_{2,\natural}p_{1}^{*}(\mF).
\]
Note that \(p_{2,\natural}p_{1}^{*}(\mO_{X})\) is a bi-\(\mO_{X}\)-module,
with an additional right action \(*_{r}\) of \(\mO_{X}\) induced via functoriality by its action on \(\mO_{X}\).

We now separate the proof into several steps:

\emph{Step (A)}: 
We claim that for \(\mF\in\QCoh(X\times[\mA^{1}/\GG_{m}]),
\)
\begin{align}\label{alignFindMonodialByAlg}
p_{2,\natural}p_{1}^{*}(\mF)\cong p_{2,\natural}p_{1}^{*}(\mO_{X})\otimes_{*_{r},\mO_{X}}\mF\cong D_{X,\log}\otimes_{*_{r},\mO_{X}}\mF.
\end{align}
Note that (\ref{alignFindMonodialByAlg}) implies in particular that \( p_{2,\natural}p_{1}^{*}(\mO_{X}) \) is reflexive for the na\"ive dual 
\[
(-)^{\vee}:=\unl{\RHom}_{X\times [\mA^{1}/\GG_{m}]}(-,\mO_{X}).
\]

For the proof of (\ref{alignFindMonodialByAlg}), denote \[
f_{n}:\unl{\AnSpec}_{X\times[\mA^{1}/\GG_{m}]}(\mO_{X}\hatotimes_{\dR_{\log}(\mO_{X})}\mO_{X}/\Fil^{n})
\hookrightarrow \hat{\Delta}_{X,\log}^{+},
\] which is cohomologically smooth, This can be seen by finding a lift of the \(\gr^{1}\) as in Lemma \ref{lemDiagonalLogProperty} (1), fixing a trivialization of \(\Omega_{X,\log}^{1}\), and reducing to the general statement that \(\AnSpec(L)\hookrightarrow \AnSpec(L\langle t\rangle)\)
is cohomologically smooth by \cite[Proposition 3.6.9]{Juan2024analyticdeRham}.

Then by \(!\)-descent, \((p_{2}\circ f_{n})_{\natural}\cong \varinjlim_{n}(p_{2}\circ f_{n})_{\natural}f_{n}^{*}.\) 
Note that by Lemma \ref{lemDiagonalLogProperty} (1), \(p_{2}\circ f_{n}\) is a finite flat map, and thus 
\begin{align*}
\RHom((p_{2}\circ f_{n})_{\natural}(p_{1}\circ f_{n})^{*}(\mF),&\mcal{G})\cong \RHom(\mF,(p_{1}\circ f_{n})_{*}(p_{2}\circ f_{n})^{*}(\mcal{G}))\\
&\cong \RHom(\mF,(\mO_{X}\hatotimes_{\dR_{\log}(\mO_{X})}\mO_{X})/\Fil^{n}\otimes_{\mO_{X}}\mcal{G})
\\
&\cong \RHom(((\mO_{X}\hatotimes_{\dR_{\log}(\mO_{X})}\mO_{X})/\Fil^{n})^{\vee}\otimes\mF,\mcal{G}).
\end{align*}
Hence \( (p_{2}\circ f_{n})_{\natural}(p_{1}\circ f_{n})^{*}(\mF)\cong (\mO_{X}\hatotimes_{\dR_{\log}(\mO_{X})}\mO_{X})/\Fil^{n})^{\vee}\otimes_{\mO_{X}}\mF \cong D_{X,\log}\otimes_{*_{r},\mO_{X}}\mF\).
Thus 
\begin{align*}
p_{2,\natural}p_{1}^{*}(\mF)&\cong \varinjlim_{n}
(p_{2}\circ f_{n})_{\natural}(p_{1}\circ f_{n})^{*}(\mF)\\
&\cong\varinjlim_{n}
((\mO_{X}\hatotimes_{\dR_{\log}(\mO_{X})}\mO_{X})/\Fil^{n})^{\vee }\otimes_{\mO_{X}}\mF\cong D_{X,\log}\otimes_{\mO_{X}}\mF.
\end{align*}

\emph{Step (B)}: 
Via (\ref{alignFindMonodialByAlg}), the monoid structure of \( p_{2,\natural}p_{1}^{*} \) induces a filtered bi-\(\mO_{X}\)-module on \( p_{2,\natural}p_{1}^{*}(\mO_{X}) \).
We claim that as a filtered bi-\(\mO_{X}\)-modules,
\begin{align}\label{alignIdentWithDXlog}
p_{2,\natural}p_{1}^{*}(\mO_{X})\cong D_{X,\log}. 
\end{align} 


Now to identify \( p_{2,\natural}p_{1}^{*}(\mO_{X}) \) canonically with \( D_{X,\log} \),
we consider
\begin{align}\label{alignDualDXtoDiagnal}
(p_{2,\natural}p_{1}^{*}(\mO_{X}))^{\vee}
&\cong p_{2,*}\unl{\RHom}_{\hat{\Delta}_{X,\log}^{+}}(p_{1}^{*}(\mO_{X}),p_{2}^{*}(\mO_{X}))\\
&\cong 
p_{1,*}p_{2}^{*}(\mO_{X})
\cong \mO_{X}\hatotimes_{\dR_{\log}(\mO_{X})}\mO_{X},
\end{align} 
along which the monoid structure on \( p_{2,\natural}p_{1}^{*} \) is translated into the comonoid structure of \( p_{1,*}p_{2}^{*} \).
Therefore, by Lemma \ref{lemDiagonalLogProperty} (3), we have 
\(p_{2,\natural}p_{1}^{*}(\mO_{X})\cong D_{X,\log}\)
as a filtered bi-\(\mO_{X}\)-algebra, as desired.

(\ref{alignFindMonodialByAlg}) and 
(\ref{alignDualDXtoDiagnal}) implies the desired description of the category of quasi-coherent sheaves.


Finally, since \(h\) is a \(!\)-surjection (Lemma \ref{lemBasicLogDRSta} (1)), \(M\in\QCoh(X^{\pdR/L,+}_{\log})\)
is a vector bundle if and only if \(h^{*}(M)\)
is a vector bundle on \(X\times [\mA^{1}/\GG_{m}]\). We thus obtain the identification of vector bundles.
This finishes the proof of (1). 

The statement of (2) follows from (1) by base change along \(\QCoh([\mA^{1}/\GG_{m}])\to\QCoh([\GG_{m}/\GG_{m}])\).


(3) follows immediately from (2).
\end{proof}
\begin{lem}\label{lemPerfFilDRtoFilCompDR}
Consider \( \widehat{j}:X_{\log}^{\pdR/L,\hat{+}}\to X_{\log}^{\pdR/L,+} \). Then \( (\widehat{j}^{*},\widehat{j}_{*}) \)  
induces equivalences \( \Perf(X^{\pdR/L,\hat{+}}_{\log})\cong \Perf(X^{\pdR/L,+}_{\log}) \) 
and \( \Vect(X^{\pdR/L,\hat{+}}_{\log})\cong \Vect(X^{\pdR/L,+}_{\log}) \). 
\end{lem}
\begin{proof}
Consider the Cartesian diagram \[
\begin{tikzcd}
X\times \hat{\mA}^{1}/\GG_{m}\arrow[d]
\arrow[r,"\widehat{j}"]
& X\times \mA^{1}/\GG_{m}\arrow[d]\\
X_{\log}^{\pdR/L,\hat{+}}\arrow[r,"\widehat{j}"] & X_{\log}^{\pdR/L,+}.
\end{tikzcd}
\]
Then the result follows formally from Lemma \ref{lemPerfComplexA1Gm}, by Lemma \ref{lemSmBaseChange} (1)
and Lemma \ref{lemBasicLogDRSta} (1), as in the proof of Lemma \ref{lemPerfComplexA1Gm} (2).
\end{proof}

\begin{cor}\label{corDRStackPrim}
If \( s:X \to\AnSpec(L_{\square}) \) is prim, then 
\( s^{\pdR}:X^{\pdR/L}_{\log}\to \AnSpec{L_{\square}} \) is also prim.
\end{cor}
\begin{proof}
By Lemma \ref{lemBasicLogDRSta}, \( h:X\to X^{\pdR/L} \)
is cohomologically smooth. 
Since \( \mO_{X} \) is \( s \)-prim by assumption, it follows that any \( V\in \Perf(X) \) is \( s \)-prim. 
By \cite[Proposition 3.1.29 (2)]{Juan2024analyticdeRham}, \( h_{\natural}(V) \) is 
\( s^{\pdR} \)-prim for any \( V\in \Perf(X) \). Then by Lemma \ref{lemDRcomplexVSpushforward} and Remark \ref{rmkResolveObyDX}, we can find 
\( \mO_{X^{\pdR/L}_{\log}} \)
in the stable subcategory generated by \( \{h_{\natural}(V):V\in\Perf(X)\} \). Thus \( \mO_{X^{\pdR/L}_{\log}} \) is also \( s^{\pdR} \)-prim, that is, \( s^{\pdR} \)
is prim.
\end{proof}

\begin{lem}[Realizing differential operators via \( \natural \)-push-forward]
	\label{lemDiffOpeToLogStack}
Given two vector bundles \(V_{1}\)
and \(V_{2}\) on \(X\),
let
\(\mrm{Diff}_{\log}(V_{1},V_{2})\)
denote the space of the log differential operators from \(V_{1}\)
to \(V_{2}\),
and let \(\mrm{Diff}_{\log}^{\le n}(V_{1},V_{2})\) denote the subspace of log differential operators of order \(\le n\).

(1) There is an isomorphism \[\mrm{Diff}_{\log}(V_{1},V_{2})\cong \Hom_{X_{\log}^{\pdR/L}}\left(h_{\natural}V_{2}^{\vee},h_{\natural}V_{1}^{\vee}\right)\cong \Hom_{X_{\log}^{\pdR/L}}\left(h_{*}V_{1},h_{*}V_{2}\right),\]
for \(h:X\to X^{\pdR/L}_{\log}.\) Here the second isomorphism is induced by \[ \unl{\Hom}(h_{\natural}V,\mO_{X_{\log}^{\pdR/L}})\cong h_{*}(V^{\vee}). \]

(2) There is an isomorphism \[
\mrm{Diff}_{\log}^{\le n}(V_{1},V_{2})\cong \Hom_{X_{\log}^{\pdR/L,+}}\left(h_{*}(V_{1})\{-n\},h_{*}(V_{2})\right)
\]
for \( h:X\times [\mA^{1}/\GG_{m}]\to X^{\pdR/L,+}_{\log} \), and \( \{-n\} \) is as in Notation \ref{notationTwistByAGm}.
This isomorphism is compatible with that in (1) along the natural map \( \mrm{Diff}^{\le n}_{\log}(V_{1},V_{2})\to \mrm{Diff}^{\le n}(V_{1},V_{2}) \) on the LHS, and the forgetful map on the RHS.
\end{lem}
\begin{proof}
For (1),
by the proof of Lemma \ref{lemDRcomplexVSpushforward}, we have \( h^{*}h_{\natural}(\mF)\cong D_{X,\log}\otimes \mF \). Thus \[
\Hom_{X_{\log}^{\pdR/L}}\left(h_{\natural}V_{2}^{\vee},h_{\natural}V_{1}^{\vee}\right)\cong \Hom_{X}\left(V_{2}^{\vee},h^{*}h_{\natural}V_{1}^{\vee}\right)\cong 
 \Hom_{X}\left(V_{2}^{\vee},D_{X,\log}\otimes V_{1}^{\vee}\right),
\] which is the usual linearization of differential operators. For the second isomorphism, it suffices to show that the natural morphism \[
\Hom_{X}\left(V_{2}^{\vee},D_{X,\log}\otimes V_{1}^{\vee}\right)\to \Hom_{X}\left(D_{X,\log}^{\vee}\otimes V_{1},V_{2}\right)
\] is an isomorphism.
Choosing analytic locally on \( X \)
a trivialization of \( \Omega^{1}_{X,\log} \) and a lift of \( \gr^{1}(\mO_{X}\hatotimes_{\dR_{\log}(\mO_{X})}\mO_{X}) \), it suffices to show that the natural map induces an isomorphism \[
A[x_{1}^{\vee},\ldots,x_{d}^{\vee}]\cong \Hom_{A_{\square}}(A[[x_{1},\ldots,x_{d}]],A)
\] for \( A \) (topologically) of finite type over \( L \). If \( d=1 \), this is because \( A[[x_{1},\ldots,x_{d}]]\cong A_{\square}[\NN\cup\{\infty\}]/A[\{\infty\}] \), 
and for any \( M\in D(A_{\square}) \), \[
\unl{\Hom}_{A_{\square}}(A[[x]],M)\cong \fib(\unl{\Hom}(\NN\cup\{\infty\},M)\xrightarrow{\text{evaluate at \( \infty \)}} M
)\cong M^{\bigoplus \NN}.
\] In general, using \cite[Example 6.4]{CS19}, \( A[[x_{1},\ldots,x_{d}]]\cong A[[x_{1}]]\otimes_{A}\cdots\otimes_{A}A[[x_{d}]] \), so 
\begin{align*}
\unl{\Hom}_{A_{\square}}(A[[x_{1},\ldots,x_{d}]],A)&\cong \unl{\Hom}_{A_{\square}}(A[[x_{1},\ldots,x_{d-1}]],\unl{\Hom}(A[x_{d}],A))\\
&\cong \unl{\Hom}_{A_{\square}}(A[[x_{1},\ldots,x_{d-1}]],A[x_{d}^{\vee}]),
\end{align*}
and one can conclude the proof by induction.

For (2), 
we have 
\begin{align*}
\Hom_{X_{\log}^{\pdR/L,+}}(h_{*}(V_{1})\{-n\},&h_{*}(V_{2}))
\cong \Hom_{X\times \mA^{1}/\GG_{m}}(h^{*}h_{*}(V_{1})\{-n\},V_{2})\\
&\cong \Hom_{\Fil^{\ZZ}(\QCoh(X))}(D_{X,\log}^{\vee}\otimes V_{1}\{-n\},V_{2})\\
&\cong \Hom_{X}((D_{X,\log}^{\le n})^{\vee}\otimes V_{1},V_{2})\\
&\cong \Hom_{X}(V_{2}^{\vee}, D_{X,\log}^{\le n}\otimes V_{1}^{\vee})\cong \mrm{Diff}^{\le n}_{\log}(V_{1},V_{2}).
\end{align*}
The compatibility with (1) is clear from the construction.
\end{proof}


The following shows that filtered complete logarithmic de Rham stack is ``affine'' in a suitable sense.
\begin{lem}\label{lemLogDRStackIsAffine}
(1)
Let \( s \) denote the structure map \(s:X^{\pdR/L,\hat{+}}_{\log}\to [\hat{\mA}^{1}/\GG_{m}]\).
Then if \(X=\Spa(A,A^{+})\), the functor \(s_{*}\)
induces an equivalence of symmetric monoidal categories \[
\QCoh(X^{\pdR/L,\hat{+}}_{\log})\cong \Mod_{\dR_{\log}(A)}(\widehat{\Fil}^{\ZZ}(D(L_{\square}))).
\] 

(2) Consider \(
h:X\times [\hat{\mA}^{1}/\GG_{m}]\to X^{\pdR/L,\hat{+}}_{\log}.
\) Then 
\[
h^{*}:\QCoh(X^{\pdR/L,\hat{+}}_{\log})\rightleftharpoons \QCoh(X\times [\hat{\mA}^{1}/\GG_{m}]):h_{*}
\] coincides with \[
-\hatotimes_{\dR_{\log}(A)}A:\Mod_{\dR_{\log}(A)}(\widehat{\Fil}^{\ZZ}(D(L_{\square})))\rightleftharpoons \Mod_{A}(\widehat{\Fil}^{\ZZ}(D(L_{\square}))):\mrm{Forget}.
\]
\end{lem}
\begin{proof}
By \( * \)-descent \[
\QCoh(X^{\pdR/L,\hat{+}}_{\log})\cong \varprojlim_{[n]\in\Delta}\QCoh(U([n])\times_{[\mA^{1}/\GG_{m}]}[\hat{\mA}^{1}/\GG_{m}])
\] where \(U([n])\) is as in Lemma \ref{lemThisIsGroupoid}. By construction of \(\unl{\Spf}^{+}\) (Definition \ref{dfnSpfFunctor}) and Lemma \ref{lemSpfCommuteWithLimit}, the simplicial diagram above coincides with the simplicial diagram \[
\Mod_{\dR_{\log}(X)}(\widehat{\Fil}^{\ZZ}(D(L_{\square})))\cong \varprojlim_{[n]\in\Delta}\Mod_{A_{n}}(\widehat{\Fil}^{\ZZ}(D(L_{\square}))),
\] where \(A_{n}\) is the \((n+1)\)-th filter completed tensor product of \(A\) over \(\dR_{\log}(A)\). This induces the desired equivalence on the LHS. 

For the last part, by adjunction, it suffices to show that \(h_{*}\) coincides with the forgetful functor. This follows from the fact that \(h_{*}\) is compatible with \( s_{*} \). 
\end{proof}

\subsection{Analytic versus algebraic de Rham stacks}\label{subsecAnalyticVSAlgeDR}
We also need the following relation between analytic and algebraic de Rham stacks. 
The new result is the identification of the categories of vector bundles.
Recall that \( X^{\pdR/L} \) is essentially Tate by Lemma \ref{lemBasicLogDRSta},
and \( X^{\pdR/L}:=(X^{\pdR/L})^{Tate} \).

\begin{prop}\label{propAnToAlgDRJuan}
Let \(X\) be a smooth rigid variety over \(L\). 

(1)
The natural map \(g:X^{\pdR/L}\to X^{\dR/L}\) is cohomologically co-smooth (Definition \ref{dfnVariousPropertyMorphiDescent} (10)). 

(2)
The natural morphism \(\mO_{X^{\dR/L}}\to g_{*}\mO_{X^{\pdR/L}}\)
is an isomorphism, \(g^{*}:\QCoh(X^{\dR/L})\to\QCoh(X^{\pdR/L})\)
is fully faithful, and \(g_{*}\circ g^{*}\cong \mrm{id}\). 

(3) \((g^{*},g_{*})\)
induces an equivalence \( \Vect(X^{\pdR/L})
\cong \Vect(X^{\dR/L}) \). In particular, \( g_{*}|_{\Vect(X^{\pdR/L})} \) is symmetric monoidal.
\end{prop}
\begin{proof}
(1) and (2) are given by \cite[Proposition 5.2.11]{Juan2024analyticdeRham}.
For (3), \( g^{*}:\Vect(X^{\dR/L})\to \Vect(X^{\pdR/L}) \)
is fully faithful by (2). It suffices to show that \( g_{*} \)
preserves vector bundles, and \( g^{*}g_{*}\cong \id \) when restricted to vector bundles.

We can construct the following Cartesian diagram \[\begin{tikzcd}
\Delta(X)^{\dagger}\arrow[d,"h'"]
\arrow[r,"p"] & X\arrow[d,"h"]\\
(\Delta(X)^{\dagger})^{\pdR/X}\arrow[r,"f'"]\arrow[d,"g'"] & X^{\pdR/L}\arrow[d,"g"]\\
X\arrow[r,"f"] & X^{\dR/L}
\end{tikzcd}\] as follows:
by Theorem \ref{thmAnDRStackGeneral} (3),
we have the big Cartesian square. Then we can take relative algebraic de Rham stack along each column. By definition, \( (X^{\dR/L})^{\pdR/L}\cong X^{\dR/L} \), and thus \( X^{\pdR/(X^{\dR/L})}\cong X^{\pdR/L} \).

Since \( g \) is cohomologically co-smooth, by Lemma \ref{lemSmBaseChange} (2),
we know \( f^{*}g_{*}\cong g'_{*}f^{\prime,*} \). So it suffices to show that \( g'_{*} \)
preserves vector bundles, and \( g^{\prime,*}g_{*}'\cong \id \) when restricted to vector bundles.

By \cite[Corollary 1.6.10]{Huber13}, we can assume that \( X \)
admits an \'etale map \( X\to \TT^{n}_{L} \). Then \( \Delta(X)^{\dagger}\cong X\times_{\TT^{n}}\Delta(\TT^{n}_{L})^{\dagger} \), and \( \hat{\Delta}_{X}^{\dagger}\cong X\times_{\TT^{n}}\hat{\Delta}_{\TT^{n}_{L}}^{\dagger} \).
Then we have a Cartesian diagram
\[
\begin{tikzcd}
(\Delta(X)^{\dagger})^{\pdR/X}\arrow[d,"g'"]\arrow[r] & 
(\Delta(\TT^{n}_{L})^{\dagger})^{\pdR/X}\arrow[d,"\ti{g}"] 
\\X\arrow[r] & \TT^{n}_{L}.
\end{tikzcd}
\] Note that
\[ \Delta(\TT^{n}_{L})^{\dagger}\cong \Spa(L\langle T_{1}^{\pm1},\ldots,T_{n}^{\pm1}\rangle\{x_{1},\ldots,x_{n}\}^{\dagger})\cong \TT^{n}_{L}\times \GG_{a}^{\dagger}. \]
Thus \( (\Delta(X)^{\dagger})^{\pdR/X}\cong X\times \GG_{a}^{\dagger}/\GG_{a} \). Then we are done by Lemma \ref{lemDModOnGaDagger}.
\end{proof}
The proof of the following lemma is motivated by that of \cite[Theorem 9.7]{Shimizu2022adic}.
\begin{lem}\label{lemDModOnGaDagger}
Let \( X=\Spa(A,A^{+}) \) be an analytic adic space over \( L \), and \( n\in\NN \). Let
\( p: X\times (\GG_{a}^{\dagger}/\hat{\GG}_{a})^{n}\to X \) denote the projection to the first factor.
Then for any vector bundle \( \mF \)  on  \( X\times (\GG_{a}^{\dagger}/\hat{\GG}_{a})^{n} \), 
\( p_{*}(\mF) \) is a vector bundle on \( X \), and
the natural map \( p^{*}p_{*}(\mF)\to \mF \) is an isomorphism, which induces an equivalence \[
\Vect(X\times (\GG_{a}^{\dagger}/\hat{\GG}_{a})^{n})\cong \Vect(X).
\]
In particular, the natural morphism \(\mO_{X}\to p_{*}\mO_{X\times (\GG_{a}^{\dagger}/\hat{\GG}_{a})^{n}} \) is an isomorphism.
\end{lem}
\begin{proof}
Let \( i:X\cong X\times (\hat{\GG}_{a}/\hat{\GG}_{a})^{n}\to X\times (\GG_{a}^{\dagger}/\hat{\GG}_{a})^{n} \) denote the \( 0 \)-th section.
Let \( \mF \) be any vector bundle on  \( X\times (\GG_{a}^{\dagger}/\hat{\GG}_{a})^{n} \). 
Up to replacing \( X \) by an open covering,
we can assume that \( i^{*}M \) is isomorphic to the trivial vector bundle.

\( \mF \) pulls back to a vector bundle \( M \)  on \( X\times (\GG_{a}^{\dagger})^{n} \).   
Note that \[ X\times (\GG_{a}^{\dagger})^{n}\cong \AnSpec((A\{x_{1},\ldots,x_{n}\}^{\dagger},A^{+})_{\square}), \] and \( (A\{x_{1},\ldots,x_{n}\}^{\dagger},A^{+})_{\square} \) 
is a bounded \( \QQ_{p} \)-algebra whose uniform completion (\cite[Definition 2.3.1]{AnschützBoscoLeBrasCamargoScholze2025analyticrhamstacksfarguesfontaine}) is \( (A,A^{+})_{\square} \).
Thus by \cite[Proposition 3.3.5]{AnschützBoscoLeBrasCamargoScholze2025analyticrhamstacksfarguesfontaine} and Theorem \ref{thmSolidTateIsFredholm}, \( (A\{x_{1},\ldots,x_{n}\}^{\dagger},A^{+})_{\square} \) 
is also Fredholm.
In particular, by Lemma \ref{lemFredholmVectorBundle},
\( M \) 
comes from a finite projective module \( M^{\delta} \) over \( A\{x_{1},\ldots,x_{n}\}^{\dagger}(*) \). 
Note that \( f\in A\{x_{1},\ldots,x_{n}\}^{\dagger}(*)  \) is invertible if and only if its image in \( A \) is invertible.
By the assumption in the first paragraph,
\( M^{\delta}\otimes_{A\{x_{1},\ldots,x_{n}\}^{\dagger}(*)}A(*)\cong A^{\oplus r} \),
so by lifting a basis,  \( M^{\delta}\cong A\{x_{1},\ldots,x_{n}\}^{\dagger}(*)^{\oplus r} \),
and  \( M\cong (A\{x_{1},\ldots,x_{n}\}^{\dagger})^{\oplus r} \).

Let \( U\subset \mA^{1,\an} \)
be an open affinoid neighborhood of \( 0 \),
then we have a monomorphism \( \GG_{a}^{\dagger}/\hat{\GG}_{a}\hookrightarrow U^{\pdR/\QQ_{p}} \). Then by Lemma \ref{lemDRcomplexVSpushforward}, 
\( \mF \) corresponds to an object in \( \Mod_{D_{U^{n}/\QQ_{p}}}(\QCoh(X\times U^{n})) \). 
Shrinking \( U \) and taking colimits along \( (-)^{*} \) in \( \mrm{Pr}^{L} \),
\( \mF \) gives rise to an object in \( \Mod_{D_{(\GG_{a}^{\dagger})^{n}/\QQ_{p}}}(\QCoh(X\times (\GG_{a}^{\dagger})^{n})) \), 
with \[ D_{(\GG_{a}^{\dagger})^{n}/\QQ_{p}}\cong \QQ_{p}\{x_{1},\ldots,x_{n}\}^{\dagger}\left[\Partial{}{x_{1}},\ldots,\Partial{}{x_{n}}\right]. \]
By Lemma \ref{lemDRcomplexVSpushforward} (and Remark \ref{rmkResolveObyDX}), we know that \[
p_{*}(\mF)\cong R\Gamma_{\dR}(M,\nabla).
\]

Since  \( M\cong (A\{x_{1},\ldots,x_{n}\}^{\dagger})^{\oplus r} \), \( \mF \) corresponds to \( (A\{x_{1},\ldots,x_{n}\}^{\dagger})^{\oplus r} \) equipped with an integrable connection \[ \nabla:(A\{x_{1},\ldots,x_{n}\}^{\dagger})^{\oplus r} \to \bigoplus_{i=1}^{r}(A\{x_{1},\ldots,x_{n}\}^{\dagger})^{\oplus r}\cdot dx_{i}. \]
Since \( \QQ_{p}\{x_{1},\ldots,x_{n}\}^{\dagger}\cong \varinjlim_{0\subset U\subset \mA^{1}}H^{0}(U^{n},\mO_{U^{n}}) \), we can find an open affinoid neighborhood \( U\subset \mA^{1,\an} \) of \( 0 \), such that \( M \) spreads out to a vector bundle \( M_{1} \) with an integrable connection on \( X\times U^{n} \), i.e. a vector bundle \( \mF_{1} \) on \( X\times (U^{n})^{\pdR/\QQ_{p}} \).

Without loss of generality, we can assume that \[ U^{n}=\Spa(\QQ_{p}\langle x_{1},\ldots,x_{n}\rangle,\ZZ_{p}\langle x_{1},\ldots,x_{n}\rangle). \]
We further fix a \( A^{+}\langle x_{1},\ldots,x_{n}\rangle \)-lattice \( M^{\circ}_{1}\subset M_{1} \). 
Then there exists \( N\in\NN \), such that \( \nabla:M_{1}\to \bigoplus_{i=1}^{n}M_{1}\cdot dt_{i} \) induces \(\nabla: M^{\circ}_{1}\to \bigoplus_{i=1}^{n} p^{-N}M^{\circ}_{1}\cdot dt_{i} \). 

Consider \[
P:M_{1}\to M_{1}\otimes_{A^{+}\langle x_{1},\ldots,x_{n}\rangle}A[[x_{1},\ldots,x_{n}]],\;
f\mapsto \sum_{\unl{j}\in \NN^{n}}\frac{(-1)^{|\unl{j}|}}{\unl{j}{!}} D_{\unl{j}}(f)\cdot \unl{x}^{\unl{j}},
\] where for \( \unl{j}=(j_{1},\ldots,j_{n}) \), \( |\unl{j}|:=\sum_{i}j_{i} \), \( \unl{j}!:=\prod_{i}(j_{i}!) \), \( D_{\unl{j}}=(\Partial{}{x_{1}})^{j_{1}}\cdots (\Partial{}{x_{n}})^{j_{n}} \). 

Then one can verify easily that \( \nabla(P(f))=0 \)
for any \( f\in M_{1}(*) \).
Moreover, for \( f\in M_{1}^{\circ} \), \( D_{\unl{j}}(f)\in p^{-N|\unl{j}|}M_{1}^{\circ} \). Thus there exists \( N'>N \),
such that for \( f\in M_{1}^{\circ} \), \( \frac{(-1)^{|\unl{j}|}}{\unl{j}{!}} D_{\unl{j}}(f)\in  p^{-N'|\unl{j}|}M_{1}^{\circ}\). Hence, we have a factorization
 \begin{align*}
 P:M_{1}\to  M_{1}\otimes_{A^{+}\langle x_{1},\ldots,x_{n}\rangle}A^{+}\langle p^{-N'}x_{1},\ldots,p^{-N'}x_{n}\rangle\\\subset M_{1}\otimes_{A\langle x_{1},\ldots,x_{n}\rangle}
A\{x_{1},\ldots,x_{n}\}^{\dagger}\cong M.
 \end{align*}
 Further extending coefficients, we obtain an \( A \)-linear map \( P:M\to M \), such that \( P(M)\subset M^{\nabla=0} \). Moreover, by construction, the composition \( M\xrightarrow{P}M\to M/(x_{1},\ldots,x_{d}) \) coincides with the natural projection,
 so \( P(M)\to M/(x_{1},\ldots,x_{d}) \) is an \( A \)-linear surjection.
 Since \( M/(x_{1},\ldots,x_{d}) \) is free over \( A \) of rank \( r \),
 we can find a splitting and realize \( A^{\oplus r} \)
 as a direct summand of \( P(M) \). Such that the composition \( A^{\oplus r}\to P(M)\to M\to M/(x_{1},\ldots,x_{d}) \) is an isomorphism.
 
 Therefore, we have \( A^{\oplus r}\subset M^{\nabla=0} \),
 which we can extend to an \( A\{x_{1},\ldots,x_{n}\}^{\dagger} \)-linear map between vector bundles  on \( X\times \GG_{a}^{\dagger} \) with connections (relative to \( X \)): \[\alpha:(A\{x_{1},\ldots,x_{n}\}^{\dagger})^{\oplus r}\hookrightarrow M.\] Again by the Fredholm property, this corresponds to a map between finite projective modules over \( A\{x_{1},\ldots,x_{n}\}^{\dagger}(*) \). Moreover, it is an isomorphism after \( -\otimes_{A\{x_{1},\ldots,x_{n}\}^{\dagger}(*)}A(*) \). 
Note that \( f\in A\{x_{1},\ldots,x_{n}\}^{\dagger}(*)  \) is invertible if and only if its image in \( A \) is invertible. Thus we know the same for matrices, which
shows that \( \alpha \)
is an isomorphism.
 
Thus we can assume that \( \mF \) is the trivial vector bundle on \( X\times \GG_{a}^{\dagger}/\hat{\GG}_{a} \). Then it suffices to show that \( R\Gamma_{\dR}(\QQ_{p}\{x_{1},\ldots,x_{d}\}^{\dagger})\cong \QQ_{p} \), which is the Poincaré lemma for \( \GG_{a}^{\dagger} \). See the proof of \cite[Lemma 5.2.12]{Juan2024analyticdeRham}.
\end{proof}

\begin{cor}\label{corXdRDescendableCover}
Let \( X \) be a smooth rigid variety over \(K\). Then \( X\to X^{\dR/K} \) is a prim \( ! \)-cover (Definition \ref{dfnVariousPropertyMorphiDescent}).  
\end{cor}
\begin{rmk}
A Gelfand version of this theorem is shown in \cite[Proposition 5.4.4]{AnschützBoscoLeBrasCamargoScholze2025analyticrhamstacksfarguesfontaine}. 
\end{rmk}
\begin{proof}
Consider the composition \[
X\xrightarrow{h}X^{\pdR/K}
\xrightarrow{g}X^{\dR/K}.
\]
We start by showing that \( \mO_{X} \) 
is \( (g\circ h) \)-prim.  
By Theorem \ref{thmAnDRStackGeneral} (3),
\( X\to X^{\dR/K} \) 
is a \( ! \)-surjection, and \( X\times_{X^{\dR/K}}X\cong (X\subset X\times_{K}X)^{\dagger} \). By Theorem \ref{thmVariousDescent} (4), suffices to show that for \( p_{i}: (X\subset X\times_{K}X)^{\dagger}\to X\),   
\( \mO_{(X\subset X\times_{K}X)^{\dagger}} \) 
is \( p_{i} \)-prim. By choosing toric chart, we can assume that \( X\cong \Spa(K\langle T^{\pm1}\rangle) \), and then \[
(X\subset X\times_{K}X)^{\dagger}\cong \AnSpec(\Spa(K\langle T^{\pm1}\rangle\{\epsilon\}^{\dagger}))\cong (X\times \{0\}\subset X\times \mA^{1})^{\dagger},
\] which is affine proper over \( X \), and thus \( \mO_{(X\subset X\times_{K}X)^{\dagger}} \) is \( p_{1} \)-prim, as desired.  

By Lemma \ref{lemDRcomplexVSpushforward} and Remark \ref{rmkResolveObyDX} (to be proven in the next subsection), \( \mO_{X^{\pdR/K}} \) has a resolution
\[ h_{*}(\mO_{X})\to h_{*}(\Omega^{1}_{X/K})\cdots h_{*}(\Omega^{d}_{X/K}). \] 
This implies tha \( (g\circ h)_{*}(\mO_{X}) \) is descendable 
by Proposition \ref{propAnToAlgDRJuan}. 
\end{proof}
\begin{cor}\label{corAnDRstackEssenTate}
Let \( X \) be a smooth rigid variety over \(K\). Then the natural map \( ((X^{\dR/K})_{Solid})^{Tate}\to X^{\dR/K} \) is an isomorphism.

Thus by abuse of notation, we will write \( X^{\dR/K}:=(X^{\dR/K})_{Solid} \).
\end{cor}
\begin{proof}
By Theorem \ref{thmAnDRStackGeneral} (3), \( X^{\dR/K}\cong \varinjlim_{[n]\in\Delta^{op}}(\Delta(X)^{\dagger})^{n/X} \). Note that by the proof of Theorem \ref{thmAnDRStackGeneral} (3), \( (\Delta(X)^{\dagger})^{n/X} \) is analytic locally on \( X \) of the formal \( \AnSpec^{b}(A_{n}) \)
for \( A_{n}\in\mrm{AnRing}^{b}_{\QQ_{p}} \). Thus \( (((\Delta(X)^{\dagger})^{n/X})_{Solid})^{Tate}\cong (\Delta(X)^{\dagger})^{n/X} \). The rest of the proof is essentially the same as that of Lemma \ref{lemBasicPropertyEsseTate} (6). It suffices to show that \( X\to ((X^{\dR/K})_{Solid})^{Tate} \) is a \( ! \)-surjection. 
This is because the properties of being prim and of being descendable are both preserved under \( (-)_{Solid} \)
and \( ((-)_{Solid})^{Tate} \): being prim can be verified using diagonal by \cite[Proposition 6.9]{Scholze2022sixFunctor} as mentioned in Definition \ref{dfnVariousPropertyMorphiDescent}, and these functors are compatible with six functor formalisms by Lemma \ref{lemSolidToTateAnalyticficationFunctor}. Thus \( X\to ((X^{\dR/K})_{Solid})^{Tate} \)
is a prim \( ! \)-cover, and thus a \( ! \)-surjection by Theorem \ref{thmVariousDescent} (7).
\end{proof}

\subsection{Some representation theories}\label{subsecRepTheory}

\subsubsection{Partial flag varieties}\label{subsecPartialFl}
Let \(L\) be a field over \( \QQ \). 
Let \(G\) be a reductive group over \(L\), \(P\subset G\) be a parabolic subgroup, \(N\) be its unipotent radical, and \(M\) be its Levi quotient.
Let \(\fg:=\Lie(G)\), \(\p:=\Lie(P)\), \(\fn:=\Lie(N)\) and \(\mm:=\Lie(M)\). 
\begin{dfn}\label{dfnFlandM}
We
define the \emph{partial flag variety} associated to \(P\)
as the scheme \(\mFl_{P}:=P\backslash G\) over \(L\), which is equipped with a right action of \(G\). 

We define the 
\(\mcal{P}_{P}:=G\xrightarrow{f'}\mFl_{P}\), which we regard as a right \(P\)-torsor over \(\mFl_{P}\) by composing the left action of \(P\) with \(g\mapsto g^{-1}\), and we let
\(\mcal{M}_{P}:=N\backslash G\xrightarrow{f}\mFl_{P}\) be its \(M\)-reduction.

We will write \(\mFl:=\mFl_{P}\) and \(\mcal{M}:=\mcal{M}_{P}\)
when the context is clear.
\end{dfn}
\begin{rmk}
We will always assume that \( L \) is a \( p \)-adic field when we talk about analytic de Rham stacks. In that case, \( \mFl_{P} \) coincides with its analytification \( \Fl_{P} \) by Proposition \ref{propAnGAGA}.
\end{rmk}

\begin{notation}[Equivariant sheaves on \(\mFl_{P}\), \cite{BB83}]\label{notationEquShvOnFlGeneral}
Let
\(\fg^{0}:=\fg\otimes \mO_{\mFl}\), and let \(\p^{0}\) (resp. \(\fn^{0}\)) be the sub-vector bundle of \(\fg^{0}\) whose total space has the description \[\p^{0}=\{(X\in \fg,x\in \mFl):X\in x^{-1}\p x\},\; (\text{resp. } \fn^{0}=\{(X\in \fg,x\in \mFl):X\in x^{-1} \fn{x}\}).
\] Let \(\mm^{0}:=\p^{0}/\fn^{0}\) be the Levi quotient. 

We also define \(G^{0}:=G\times \mFl\), \[P^{0}=\{(g\in G,x\in \mFl):g\in x^{-1}P x\},\; N^{0}=\{(g\in G,x\in \mFl):X\in g\in x^{-1}N x\}.\]
\end{notation}
\begin{dfn}\label{dfnCompleteGrpObj}
Assume that we have a groupoid object \((Y\rightrightarrows X,\Delta:X\hookrightarrow Y)\) in the category of solid stacks (resp. Tate stacks) (\cite[Definition 6.1.2.7]{Lurie2009HTT}). \(\Delta:X\to Y\) is a map between groupoid objects over \(X\).
We define \[\widehat{Y}:=X^{\pdR/Y},\;(\text{resp. }Y^{\dagger}:=X^{\dR/Y}),\]
which is a groupoid object by the finite-limit-preserving functoriality of \((-)^{\pdR/-}\) (resp. \((-)^{\dR/-}\)) (Theorem \ref{thmAnDRStackGeneral} (1)).

The operations \((-)^{\dagger}\)
and \(\widehat{(-)}\)
are clearly functorial in \(Y\), and we also have natural transformations \(\widehat{(-)}\to (-)^{\dagger}\to \id\).
\end{dfn}
\begin{lem}\label{lemFullGrpObj}
(1)
The groupoid object \(G^{0}\rightrightarrows \mFl \) induced by the action of \(G\)
on \(\mFl\) canonically factors through another groupoid object \[
G^{0}/N^{0}\rightrightarrows \mFl,
\]
which further canonically factors through a third groupoid object \[
G^{0}/P^{0}\rightrightarrows \mFl,
\] which is isomorphic as a groupoid object to \( \Fl\times_{L}\Fl\).

(2) The groupoid object \(G^{0}\times_{\mFl}\mcal{M}\rightrightarrows \mcal{M} \) induced by the action of \(G\)
on \(\mcal{M}\) canonically factors through another groupoid object \[
G^{0}/N^{0}\times_{\mFl}\mcal{M}\rightrightarrows \mcal{M}.
\]
\end{lem}
\begin{proof}
The proof of
\cite[Lemma IV.1.1]{Scholze2024RealLanglands} works verbatim here.
\end{proof}
\begin{cor}\label{corDaggerGroupObject}
We have \[\mFl^{\dR/(G^{0}/N^{0})^{\an}}\cong (\fg^{0}/\fn^{0})^{\dagger},\;\mFl^{\pdR/(G^{0}/N^{0})}\cong \widehat{\fg^{0}/\fn^{0}},\]
\[\mFl^{\dR/(G^{0}/P^{0})}\cong (\fg^{0}/\p^{0})^{\dagger},\;\mFl^{\pdR/(G^{0}/P^{0})}\cong \widehat{\fg^{0}/\p^{0}},\]
\[(\mcal{M}^{\an})^{\dR/(G^{0}/N^{0}\times_{\mFl}\mcal{M})^{\an}}\cong (f^{*}(\fg^{0}/\fn^{0}))^{\dagger},\;\mcal{M}^{\pdR/(G^{0}/N^{0}\times_{\mFl}\mcal{M})}\cong \widehat{f^{*}(\fg^{0}/\fn^{0})},\]
for \(f:\mcal{M}\to \mFl\). Here \(\widehat{(-)}\)
and \((-)^{\dagger}\)
are as in Definition \ref{dfnVBvariousNBHD}. 

In particular, the right hand sides can be endowed with the structure of a groupoid 
object by Definition \ref{dfnCompleteGrpObj}.
\end{cor}
\begin{proof}
By \cite[Proposition 5.1.4 \& 5.2.4]{Juan2024analyticdeRham}, \( \mFl^{\dR/(G^{0}/N^{0})}\cong (\mFl\subset G^{0}/N^{0})^{\dagger} \) and 
\( \mFl^{\pdR/(G^{0}/N^{0})}\cong (\mFl\subset G^{0}/N^{0})^{\wedge/\mFl} \). We note that \( \mFl^{\dR/G^{0}}\cong \fg^{\dagger}\times \mFl \) and \( \mFl^{\pdR/G^{0}}\cong \widehat{\fg}\times \mFl \) as in Definition \ref{dfnDaggerGroupObject}. Similarly, \( \mFl^{\dR/N^{0}}\cong \fn^{0,\dagger} \) and \( \mFl^{\pdR/N^{0}}\cong \widehat{\fn^{0}} \), induced by exponential maps.
Thus we have \[\mFl^{\dR/(G^{0}/N^{0})}\cong \mFl^{\dR/G^{0}}/\mFl^{\dR/N^{0}}\cong (\fg^{\dagger}\times \mFl)/\fn^{0,\dagger}\cong (\fg^{0}/\fn^{0})^{\dagger}.\]
The proof of the other cases is similar.
\end{proof}
\begin{lem}\label{lemCompareDRStac}
(1) We have isomorphisms \(\mFl/(\fg^{0}/\p^{0})^{\dagger}\cong \mFl^{\dR/L}\)
and \(\mFl/\widehat{\fg^{0}/\p^{0}}\cong \mFl^{\pdR/L}\), compatible with the natural maps from \(\mFl\). 

Equivalently, we have \((\mFl\subset \mFl\times_{L}\mFl)^{\dagger}\cong (\fg^{0}/\fn^{0})^{\dagger}\)
and \((\mFl\subset \mFl\times_{L}\mFl)^{\wedge}\cong \widehat{\fg^{0}/\fn^{0}}\) as groupoid objects. In particular, we have an isomorphism \((\fg^{0}/\p^{0})^{\vee}\cong \Omega^{1}_{\mFl/L}\).

(2) We have isomorphisms \[
\mcal{M}^{\an}/(f^{*}(\fg^{0}/\fn^{0}))^{\dagger}\cong (\mcal{M}^{\an})^{\dR/L},\;\mcal{M}/\widehat{f^{*}(\fg^{0}/\fn^{0})}\cong \mcal{M}^{\pdR/L},
\]
which are compatible with the natural maps from \(\mFl\). 

(3) We have the following commutative diagram \[
\begin{tikzcd}
\mcal{M}^{\an}\arrow[r,"f"]
\arrow[d] & 
\mFl \arrow[d]\arrow[r,equal]
& \mFl\arrow[d]
\\
(\mcal{M}^{\an})^{\dR/L}\arrow[r,"f'"]
& \mFl/(\fg^{0}/\fn^{0})^{\dagger}\arrow[r,"u"]
& \mFl^{\dR/L},
\end{tikzcd}\]
where the square on the left is Cartesian. In particular, \[
\mFl/(\fg^{0}/\fn^{0})^{\dagger}\cong (\mcal{M}^{\an})^{\dR/L}/M^{\an}.
\]
\end{lem}
\begin{proof}
The groupoid structure induces a natural morphism \(G^{0}/P^{0}\to \mFl\times_{L} \mFl\), which is an isomorphism of groupoid objects. Applying Definition \ref{dfnCompleteGrpObj}, we have \(\mFl^{\dR/G^{0}/P^{0}}\cong (\fg^{0}/\p^{0})^{\dagger}\cong \Delta(\mFl)^{\dagger}  \),
and thus \(\mFl/(\fg^{0}/\p^{0})^{\dagger}\cong (\mFl\subset \mFl\times_{L}\mFl)^{\dagger}\) by Theorem \ref{thmAnDRStackGeneral} (3). The proof of the other isomorphisms is similar. (3) follows from (2) immediately.
\end{proof}

\subsubsection{Representations of Lie algbras}\label{subsubsecRepLieAlg}
Let \( L \) be a field over \( \QQ \), and \( G \) be an algebraic group over \( L \), and let \( \fg \) be its Lie algebra. We will consider \( \hat{\fg}^{+} \)
as in Definition \ref{dfnDaggerGroupObject}.
\begin{lem}\label{lemDRStackOfG}
We have an isomorphism \( G^{\pdR/L,+}\cong (G\times [\mA^{1}/\GG_{m}])/\hat{\fg}^{+} \) over \( [\mA^{1}/\GG_{m}] \). Here the quotient is taken over \( [\mA^{1}/\GG_{m}] \).
\end{lem}
\begin{proof}
By Lemma \ref{lemClosedImmeAlgDRFil}, \( \hat{\fg}^{+}\cong \AnSpec(L)^{\pdR/G,+} \) where \( \AnSpec(L)\hookrightarrow G \) is the unit. 
In particular, we have a natural group morphism \( \hat{\fg}^{+}\to G\times [\mA^{1}/\GG_{m}] \). 


Using the isomorphism \( (G\times G,p_{1},p_{2})\cong (G\times G,p_{1},m),(g_{1},g_{2})\mapsto (g_{1},g_{1}^{-1}g_{2}) \)  of groupoid objects over \( G \), the diagonal map \( \Delta:G\to G\times G \)
is translated to \( i_{1}:G\cong G\times 1\hookrightarrow G\times G \). Therefore, using the functoriality of Definition \ref{dfnCompleteGrpObj}, we have an isomorphism of groupoid objects \( \hat{\Delta}_{G}^{+}\cong G\times \hat{\fg}^{+} \). This implies that \( G^{\pdR/L,+}\cong (G\times [\mA^{1}/\GG_{m}])/\hat{\fg}^{+} \).
\end{proof}
\begin{lem}\label{lemQuotientClosedSubgorup}
Let \( H\subset G \) be a Zariski closed subgroup over \( L \). Then \[
(H\backslash G)^{\pdR/L,+}/G\cong [\mA^{1}/\GG_{m}]_{L}/H^{\pdR/G,+}
\] where the quotient on the RHS is taken over \( [\mA^{1}/\GG_{m}]_{L} \) and that on the LHS is over \( L \).
\end{lem}
\begin{proof}
By Theorem \ref{thmAnDRStackGeneral} (1), \( 
(H\backslash G)^{\pdR/L,+}\cong H^{\pdR/L,+}\backslash G^{\pdR/L,+} \). By definition, \[
H^{\pdR/G,+}=H^{\pdR/L,+}\times_{G^{\pdR/L,+}}(G\times [\mA^{1}/\GG_{m}]),
\] and \( G\times [\mA^{1}/\GG_{m}]\to G^{\pdR/L,+}\)
is a \( ! \)-surjection by Theorem \ref{thmAnDRStackGeneral} (3). Thus \[
H^{\pdR/L,+}\backslash G^{\pdR/L,+}
\cong H^{\pdR/G,+}\backslash (G\times [\mA^{1}/\GG_{m}]).
\] Quotienting out \( G\times [\mA^{1}/\GG_{m}] \)
on both sides, we obtain the desired isomorphism.
\end{proof}
\begin{lem}\label{lemRealizeLieRep}
We write \( B\hat{\fg}:=\AnSpec(L)/\hat{\fg} \), and \( B\hat{\fg}^{+}:= G\backslash G^{\pdR/L,+}\). 

(1) We have \( B\hat{\fg}^{+}\cong [\mA^{1}/\GG_{m}]_{L}/\hat{\fg}^{+} \), where the quotient is taken over \( [\mA^{1}/\GG_{m}]_{L} \). In particular, 
\( B\hat{\fg}^{+}|_{\GG_{m}/\GG_{m}}\cong B\hat{\fg} \), and
\( h:[\mA^{1}/\GG_{m}]_{L}\to B\hat{\fg}^{+} \) is cohomologically smooth.
Moreover, the structure map \(g: B\hat{\fg}^{+}\to [\mA^{1}/\GG_{m}]_{L} \) is cohomologically smooth. 

(2) There are equivalences \( \QCoh(B\hat{\fg}^{+})\cong \Mod_{U(\fg)}(\Fil^{\ZZ}(D(L_{\square}))) \), and \( \QCoh(B\hat{\fg})\cong \Rep(\fg) \), where \( U(\fg) \) is induced with the degree filtration \( \Fil^{-i}U(\fg):=\im(T^{\le i}(\fg)\to U(\fg)) \).

In particular, we have a filtered isomorphism \( Z(\fg)\cong \pi_{0}\End(\id_{\QCoh(B\hat{\fg}^{+})}) \),
where \( Z(\fg) \) is the center of \( U(\fg) \).
\end{lem}
\begin{proof}
For (1), the first isomorphism is given by Lemma \ref{lemQuotientClosedSubgorup} by putting \( H=\{1\} \).
The cohomological smoothness of
\( h \) follows by Theorem \ref{thmVariousDescent} (4) from that of \( G\times [\mA^{1}/\GG_{m}]\to G^{\pdR/L,+} \) by Lemma \ref{lemBasicLogDRSta} (1).
The cohomological smoothness of \( g \) follows from those of \( h \) and \( \id_{[\mA^{1}/\GG_{m}]}\) by Theorem \ref{thmVariousDescent} (5).

(2)
This is essentially \cite[Proposition III.1.4]{Scholze2024RealLanglands}. Let us give a sketch here.  
As in the proof of Lemma \ref{lemDRcomplexVSpushforward}, by Barr-Beck-Lurie theorem (\cite[Theorem 4.7.3.5]{Lurie2017HA}) and \cite[Proposition 3.1.27]{Juan2024analyticdeRham},
the structure of the monoid \( h^{*}h_{\natural} \) equipped \( h^{*}h_{\natural}(L) \) with a filtered associative algebra structure, such that \( \QCoh(B\hat{\fg})\cong \Mod_{h^{*}h_{\natural}(L)}(\Fil^{\ZZ}(D(L_{\square}))) \).

We claim that \( h^{*}h_{\natural}(L)\cong U(\fg) \) as a filtered associative algebra. 
Moreover, we have a Cartesian diagram \[
\begin{tikzcd}
G\times {[\mA^{1}/\GG_{m}]} \arrow[r]\arrow[d,"h'"] & {[\mA^{1}/\GG_{m}]}\arrow[d,"h"]\\
G^{\pdR/L,+} \arrow[r] & B\hat{\fg}^{+}.
\end{tikzcd}
\] 
This implies that there is a natural morphism of associative algebras
\( h^{*}h_{\natural}(L)\to h^{\prime,*}h'_{\natural}(\mO_{G}) \), and the latter is identified with \( D_{G/L} \)
by Lemma \ref{lemDRcomplexVSpushforward}, which is equipped with a \( G \)-action induced by the right multiplication of \( G \) on \( G \), and \( h^{*}h_{\natural}(L)\cong R\Gamma(G,h^{\prime,*}h'_{\natural}(\mO_{G}))\cong R\Gamma(G,D_{G/L}) \).

On the other hand, we have a natural map \( U(\fg)\to   H^{0}(G,D_{G/L}) \) induced by taking derivative of the action induced by the left multiplication, and it is injective, and is compatible with the degree filtration, and induces an isomorphism on \( \gr_{0} \) and \( \gr_{1} \). This implies that it is an isomorphism. We thus have \( D_{G/L}\cong U(\fg)\otimes_{L}\mO_{G} \), and thus \( R\Gamma(G,D_{G/L})\cong U(\fg) \).
\end{proof}
We learn the following idea from a talk of Arthur-C\'esar Le Bras.
\begin{lem}[Infinitesimal action]\label{lemInfiniteAction}
Let \( U \) be a solid stack over \( L \) equipped with a right action of \( G \). Consider \( Y:=G\times [\mA^{1}/\GG_{m}]\backslash U^{\pdR/L,+} \), where the quotient is taken over \( [\mA^{1}/\GG_{m}]_{L} \), and \( G \) acts on \( U^{\pdR/L,+} \) via \( G\times [\mA^{1}/\GG_{m}]\to G^{\pdR/L,+} \). 
Note that \( G^{\pdR/L,+} \) acts on \( U^{\pdR/L,+} \) by Theorem \ref{thmAnDRStackGeneral} (1).

Then we have a filtered morphism \( Z(\fg)\to \End(\id_{\QCoh(Y)}) \), where \( Z(\fg) \) is endowed with the degree filtration inheritted from \( U(\fg) \).

The construction is natural in \( U \), and when \( U=G \), we specialize to the action of \( Z(\fg) \) in Lemma \ref{lemRealizeLieRep}.
\end{lem}
\begin{proof}
We will write \( G^{+}:=G\times [\mA^{1}/\GG_{m}] \).

We start by defining an action of \( \QCoh(G^{+}\backslash G^{\pdR/L,+}/G^{+}) \) on \( \QCoh(Y) \). 
Consider the correspondence \[
G^{+}\backslash U^{\pdR/L,+}\xleftarrow{p_{2}}G^{+}\backslash (G^{\pdR/L,+}\times^{G^{+}}U^{\pdR/L,+})\xrightarrow{m}G^{+}\backslash U^{\pdR/L,+},
\] where \( p_{2} \) (resp. \( m \)) is induced by 
\( p_{2}:G^{\pdR/L,+}\times U^{\pdR/L,+}\to U^{\pdR/L,+} \) (resp. induced by the action of \( G^{\pdR/L,+} \) on \( U^{\pdR/L,+} \)).

Note that \( m \)
fits in the Cartesian diagram \[
\begin{tikzcd}
G^{+}\backslash (G^{\pdR/L,+}\times^{G^{+}}U^{\pdR/L,+})\arrow[d,"m"]\arrow[r]
& BG^{+}\arrow[d,"g"]
\\
G^{+}\backslash U^{\pdR/L,+}\arrow[r] & BG^{\pdR/L,+},
\end{tikzcd}
\] and \( g \) is cohomologically smooth by Lemma \ref{lemRealizeLieRep}. Thus \( m \) is also cohomologically smooth.

Consider \[
p_{1}:G^{+}\backslash (G^{\pdR/L,+}\times^{G^{+}}U^{\pdR/L,+})
\to 
G^{+}\backslash G^{\pdR/L,+}/G^{+}.
\]
Then for any \( F\in \QCoh(G^{+}\backslash G^{\pdR/L,+}/G^{+}) \),
we define 
\begin{align*}
F*(-):\QCoh(G^{+}\backslash U^{\pdR/L,+})
&\to \QCoh(G^{+}\backslash U^{\pdR/L,+}),\\
\mcal{G}&\mapsto m_{\natural}(p_{2}^{*}(\mcal{G})\otimes p_{1}^{*}(F)).
\end{align*}

Let \( \delta:G^{+}\backslash G^{+}/G^{+}\to G^{+}\backslash G^{\pdR/L,+}/G^{+} \). Note that \( \delta \) is also cohomologically smooth by Lemma \ref{lemBasicLogDRSta} (1). We claim that \( \delta_{\natural}(1)*(-)\cong \id_{\QCoh(G^{+}\backslash U^{\pdR/L,+})} \). For this, we consider the following diagram \[
\begin{tikzcd}
G^{+}\backslash U^{\pdR/L,+}\arrow[d,equal]
&[-0.5cm]\arrow[l,"p_{2}'"]
G^{+}\backslash (G^{+}\times^{G^{+}}U^{\pdR/L,+})\arrow[r,"p_{1}'"]\arrow[d,"\delta'"]
&[-0.5cm]
G^{+}\backslash G^{+}/G^{+}
\arrow[d,"\delta"]\\
G^{+}\backslash U^{\pdR/L,+}
&\arrow[l,"p_{2}"]
G^{+}\backslash (G^{\pdR/L,+}\times^{G^{+}}U^{\pdR/L,+})\arrow[r,"p_{1}"]
&
G^{+}\backslash G^{\pdR/L,+}/G^{+}
\end{tikzcd}
\] where the square on the right is Cartesian.
Then \[
\delta_{\natural}(1)*\mcal{G}
\cong m_{\natural}(p_{1}^{*}(\delta_{\natural}(1))\otimes p_{2}^{*}\mcal{G})
\cong m_{\natural}(\delta_{\natural}'\circ p_{1}^{\prime,*}(1)\otimes p_{2}^{*}\mcal{G})
\cong (m\circ \delta')_{\natural}p^{\prime,*}_{2}(\mcal{G}).
\] However, both \( p_{2}' \)
and \( m\circ \delta' \) are isomorphisms, and the composition \( p_{2}'\circ (m\circ \delta')^{-1}\cong \id \). This implies that  \( (m\circ \delta')_{\natural}p^{\prime,*}_{2}\cong \id \), as desired.

Now given \( \delta_{\natural}(1)*(-)\cong \id_{\QCoh(Y)} \),
it suffices to construct \( Z(\fg)\to \End(\delta_{\natural}(1)) \).
This is given by \[
\End_{G^{+}\backslash G^{\pdR/L,+}/G^{+}}(\delta_{\natural}(1))
\cong R\Gamma(G,\End_{G^{+}\backslash G^{\pdR/L,+}}(\delta''_{\natural}(1)))
\cong R\Gamma(G,U(\fg)),
\] where the second isomorphism is given by Lemma \ref{lemRealizeLieRep}. 

The construction is clearly natural in \( U \) equipped with an action of \( G \). 

When \( U=G \), 
note that \( G\backslash G^{\dR}/G\cong \AnSpec(L)/(\Delta(G)\subset G\times G)^{\wedge} \),
and \( \delta_{\natural}(1) \) corresponds to the compact support induction of the trivial representation of \( G \)
to \( (\Delta(G)\subset G\times G)^{\wedge} \).
When restricted to \( \hat{\fg}\times \hat{\fg}\subset (\Delta(G)\subset G\times G)^{\wedge} \),
\( \delta_{\natural}(1) \)
corresponds to \( U(\fg) \) regarded as a \( U(\fg)\otimes U(\fg) \)-representation via
\( (*_{r},*_{l}) \),
where \( *_{r} \) and \( *_{l} \) denote the
left and right actions respectively.
Then the identification \( \delta_{\natural}(1)*(-)\cong \id_{\QCoh(Y)} \) boils down to the isomorphism \[ \left(R\Gamma(\fg,(U(\fg),*_{r})\otimes V),*_{l}\right)\cong V. \]
Then the action of \( Z(\fg) \) on \( U(\fg) \) on the LHS clearly induces the \( Z(\fg) \)-action on the RHS.
\end{proof}


\subsubsection{Parabolic BGG category \texorpdfstring{\( \mscr{O} \)}{O}}
\label{subsubsecParBGGCatO}
We continue with the setting of \S \ref{subsecPartialFl}. We write \[ X^{+}:=X\times [\mA^{1}/\GG_{m}] \] for any solid stack \( X \), unless \( X=\hat{\fg} \) or \( B\hat{\fg}^{+} \), in which case we still use Definition \ref{dfnDaggerGroupObject} and Lemma \ref{lemRealizeLieRep}. Let \[
h^{+}:\mFl_{P}^{+}/G^{+}\to 
\mFl_{P}^{\pdR/L,+}/G^{+},\;
h:\mFl_{P}/G\to \mFl_{P}^{\pdR/L}/G.
\]
\begin{dfn}
We define the \emph{filtered parabolic BGG category \( \mscr{O}_{P}^{+} \)} to be the full \emph{stable} subcategory of \( \QCoh(\mFl_{P}^{\pdR/L,+}/G^{+}) \) generated by \( h_{\natural}^{+}(V)\{n\} \) for \( V\in  \Vect(\mFl_{P}/G) \) and \( n\in\ZZ \).

We define the \emph{parabolic BGG category \( \mscr{O}_{P} \)} to be the full \emph{stable} subcategory of \( \QCoh(\mFl_{P}^{\pdR/L}/G) \) generated by \( h_{\natural}(V) \) for \( V\in \Vect(\mFl_{P}/G) \).
\end{dfn}

Note that \( \mFl_{P}/G\cong BP \), which admits a natural t-structure. Let \( \Rep(P)^{\heartsuit} \) denote the heart of the natural \( t \)-structure, and \(  \Rep(P)^{\heartsuit,\omega} \) denote the subcategory of compact objects. Then \( \Vect(\mFl_{P}/G)\cong \Rep(P)^{\heartsuit,\omega} \), which is generated by \( \Rep(M)^{\heartsuit,\omega} \).
\begin{lem}\label{lemRelateClassicalCateO}
(1) 
For any \( V\in\QCoh(\mFl_{P}/G)\cong\Rep(P) \), we have \[
(h^{+})^{*}h^{+}_{\natural}(V)\cong U(\fg)\otimes_{U(\p)}V,
\] where the RHS is regarded as a filtered \( P \)-representation by the diagonal action, where \( U(\fg) \)  is endowed with the adjoint action of \( P \) and the degree filtration as in Lemma \ref{lemRealizeLieRep}.

Moreover, the monad \( (h^{+})^{*}h^{+}_{\natural} \) is induced by the bi-\( U(\p) \)-algebra 
structure of \( U(\fg) \), and \[
\QCoh(\mFl_{P}^{\pdR/L,+}/G)\cong \Mod_{(h^{+})^{*}h^{+}_{\natural}}(\Rep(P)).
\]

(2)
Consider the map \( f:B\hat{\fg}^{+}\cong G^{\pdR/L,+}/G^{+}\to \mFl_{P}^{\pdR/L,+}/G^{+} \).
Then via the equivalence in Lemma \ref{lemRealizeLieRep} (2), for \( V\in \Rep(M) \), 
\[ f^{*}(h_{\natural}^{+}(V))\cong U(\fg)\otimes_{U(\p)}V. \]
where the RHS is equipped with the filtration induced from the degree filtration on \( U(\fg) \) as in Lemma \ref{lemRealizeLieRep}.
Moreover, for \( V,V'\in\Rep(M) \) and \( n\in\ZZ \), 
\begin{align*}&f^{*}:
\pi_{0}
\Hom_{\mFl_{P}^{\pdR/L,+}/G^{+}}(h^{+}_{\natural}(V)\{n\},h^{+}_{\natural}(V'))\\&
\cong \pi_{0}\Hom_{\Mod_{U(\fg)}(\Fil^{\ZZ}(D(L_{\square})))}(U(\fg)\otimes_{U(\p)}V\{n\},U(\fg)\otimes_{U(\p)}V').
\end{align*}

In particular,
\( f^{*} \) induces a conservative functor from \( \mscr{O}_{P} \) to the stable subcategory of \( \Rep(\fg) \) generated by the usual parabolic BGG category (see for example \cite[\S 9.3]{Humphreys2008RosL}), which is compatible with the action of \( Z(\fg) \) provided by Lemma \ref{lemInfiniteAction}.
\end{lem}
\begin{rmk}
Comparing with the usual parabolic BGG category, we are putting the additional structure of an algebraic action of \( P \).
\end{rmk}
\begin{proof}
We start with (2).
Consider the Cartesian diagram
\begin{equation}\label{equationCartePtoG}
\begin{tikzcd}
G^{\pdR/\mFl_{P},+}/G^{+}\arrow[r,"f'"]\arrow[d,"\alpha"]
& \mFl_{P}^{+}/G^{+}\arrow[d,"h^{+}"]\\
G^{\pdR/L,+}/G^{+}
\arrow[r,"f"]
& \mFl_{P}^{\pdR/L,+}/G^{+}.
\end{tikzcd}
\end{equation}
Note that we have \( P\cong G\times_{\mFl_{P}}\{\infty\} \),
and thus \( P^{+}\cong P^{\pdR/L,+}\times_{G^{\pdR/\mFl_{P},+}} G^{+} \). 
Therefore, \begin{equation}\label{eqRelaDRquatoG}
G^{\pdR/\mFl_{P},+}/G^{+}\cong P^{\pdR/L,+}/P^{+}\cong B\hat{\p}^{+}\xrightarrow{\alpha}B\hat{\fg}^{+}\cong G^{\pdR/L,+}/G^{+}.
\end{equation} Then by Lemma \ref{lemSmBaseChange} (1), it suffices to \( \alpha_{\natural}(V)\cong U(\fg)\otimes_{U(\p)}V \). This is clear since \( \alpha^{*} \) corresponds to the forgetful map by the proof of Lemma \ref{lemRealizeLieRep}.

The identification of \( \pi_{0}\Hom \)
is given by \cite[Theorem 2]{Faltings1983CohoOfLocSymmetric}, for which we use Lemma \ref{lemDiffOpeToLogStack} to relate the LHS to classical differential operators.
Concretely, by adjunction, it boils down to 
\begin{align*}
\pi_{0}
\Hom_{\Fil^{\ZZ}(\Rep(P))}(V\{n\},U(\fg)\otimes_{U(\p)}V')\\
\cong \pi_{0} \Hom_{\Fil^{\ZZ}(\Rep(\p))}(V\{n\},U(\fg)\otimes_{U(\p)}V').
\end{align*}

Finally, \( f^{*} \) is conservative since \( f \) is a \( ! \)-surjection,
and \( f^{*} \) is compatible with the action of \( Z(\fg) \) by the functoriality in Lemma \ref{lemInfiniteAction}. 

For (1), by Lemma \ref{lemQuotientClosedSubgorup}, we have \( \mFl_{P}^{+}/G^{+}\cong [\mA^{1}/\GG_{m}]_{L}/P^{\pdR/G,+} \). Thus \[
\mFl_{P}^{+}/G^{+}\times_{\mFl_{P}^{\pdR/L,+}/G^{+}} \mFl_{P}^{+}/G^{+}
\cong P^{+}\backslash P^{\pdR/G,+}/P^{+}
\cong P^{+}\backslash (P^{\pdR/G,+}/P^{+}),
\] where in the second term, the two \( P^{+} \) act by left and right multiplication, and in the third term, the action of \( P^{+} \) is induced by the adjoint action.

By (\ref{equationCartePtoG}), (\ref{eqRelaDRquatoG}) and Lemma \ref{lemRealizeLieRep}, we have a \( P^{+} \)-equivariant isomorphism
\( \hat{\fg}^{+}/\hat{\p}^{+}\cong P^{\pdR/G,+}/P^{+} \), 
where the actions of \( P^{+} \) on the both sides 
are induced by the adjoint action. Thus we have a Cartesian diagram \[
\begin{tikzcd}
P^{+}\backslash (\hat{\fg}^{+}/\hat{\p}^{+})\arrow[r,"s_{1}"]\arrow[d,"s_{2}"] &
BP^{+}\arrow[d,"h^{+}"]\\
BP^{+}\arrow[r,"h^{+}"]& \mFl_{P}^{\pdR/L,+}/G^{+},
\end{tikzcd}
\] where \( s_{1} \) and \( s_{2} \) are 
 compositions \[ s_{i}:P^{+}\backslash (\hat{\fg}^{+}/\hat{\p}^{+})\xrightarrow{u} [\mA^{1}/\GG_{m}]_{L}/(P^{+}\ltimes \hat{\p}^{+})\xrightarrow{\bar{s}_{i}} BP^{+}, \]
where the first map is induced by the structure map \( \hat{\fg}^{+}\to [\mA^{1}/\GG_{m}]_{L} \), and the second map is induced by \( \bar{s}_{i}:P^{+}\ltimes \hat{\p}^{+}\to P^{+} \) for \( i=1,2 \), with \( \bar{s}_{1} \) being the projection to the first term, and \( \bar{s}_{2} \) being the multiplication.

Thus by Lemma \ref{lemSmBaseChange} (1), \[ (h^{+})^{*}h^{+}_{\natural}(V)\cong s_{1,\natural}s_{2}^{*}(V)\cong \bar{s}_{1,\natural}u_{\natural}u^{*}\bar{s}_{2}^{*}(V)\cong \bar{s}_{1,\natural}(u_{\natural}(\mO_{\hat{\fg}^{+}})\otimes V)\cong \bar{s}_{1,\natural}(U(\fg)\otimes V),\]
where the last isomorphism is given by Lemma \ref{lemRealizeLieRep}, and \( P^{+}\ltimes \hat{\p}^{+} \) acts on \( \hat{\fg}^{+} \) as in the quotient \( P^{+}\backslash (\hat{\fg}^{+}/\hat{\p}^{+}) \) and acts on \( V \) via \( \bar{s}_{2} \).
Consider the Cartesian diagram \[
\begin{tikzcd}
B\hat{\p}^{+}\arrow[d,"\bar{s}_{1}'"]
\arrow[r] & {[\mA^{1}/\GG_{m}]_{L}}/(P^{+}\ltimes \hat{\p}^{+})\arrow[d,"\bar{s}_{1}"]\\
{[\mA^{1}/\GG_{m}]_{L}}\arrow[r] & {[\mA^{1}/\GG_{m}]_{L}}/P^{+}.
\end{tikzcd}\]
Thus by Lemma \ref{lemRealizeLieRep} applied to \( \p \),
\( \bar{s}_{1,\natural}(U(\fg)\otimes V)\cong U(\fg)\otimes_{U(\p)}V \) as filtered \( P \)-representations.

As a result, applying \( (h^{+})^{*}h^{+}_{\natural} \) \( n \)-times to \( V \)
gives \[ \underbrace{U(\fg)\otimes_{U(\p)}U(\fg)\cdots\otimes_{U(\p)}U(\fg)}_{n-\text{times}}\otimes_{U(\p)}V. \] Putting \( V:=U(\p) \) (equipped with the degree filtration and the adjoint action), then the monad structure of \( (h^{+})^{*}h^{+}_{\natural} \) corresponds to a unique bi-\( U(\p) \)-algbra structure on \( U(\fg) \). Note that the objects involved in the latter are concentrated in degree \( 0 \). We can show that it has to be the natural bi-algebra structure of \( U(\fg) \) by comparing with the monad in Lemma \ref{lemRealizeLieRep}.

The last statement follows from Barr-Beck-Lurie theorem (\cite[Theorem 4.7.3.5]{Lurie2017HA}), the argument in \cite[Proposition 3.1.27]{Juan2024analyticdeRham} and Lemma \ref{lemBasicLogDRSta} (1).
\end{proof}



\begin{notation}
Let \( V=V_{\alpha}\in \Rep(M)^{\omega} \) be the irreducible representation of highest weight \( \alpha \).
We denote by \( \lambda_{V}=\lambda_{\alpha}:Z(\fg)\to L \) the character through which \( Z(\fg) \)
acts on \( U(\fg)\otimes_{U(\p)}V_{\alpha} \), which exists by \cite[Proposition 2.1]{Lepowsky1977470}. 
\end{notation}
\begin{lem}\label{lemCompareInfiniteAction}
Let \( V\in \Rep(M)^{\omega} \).
Then the map \( Z(\fg) \to \pi_{0}\End (h^{+}_{\natural}(V)) \) provided by Lemma \ref{lemInfiniteAction} factors through  \( \lambda_{V} \).
\end{lem}
\begin{proof}
This follows from Lemma \ref{lemRelateClassicalCateO}.
\end{proof}
\begin{cor}\label{corNoMapDiffInfinit}
Let \( V,V'\in \Rep(M)^{\omega} \).
If \( \lambda_{V}\ne \lambda_{V'} \), then 
\[ \RHom(h_{\natural}^{+}(V),h_{\natural}^{+}(V'))\cong 0. \]
\end{cor}
\begin{proof}
Given any \( f\in \pi_{m}(\RHom(h_{\natural}^{+}(V),h_{\natural}^{+}(V'))) \), we claim that \( f\cong 0 \).
\( f \) corresponds to 
\( f:h_{\natural}^{+}(V)\to h_{\natural}^{+}V'[-m] \). 
Let \( X\in Z(\fg) \) such that \( \lambda_{V}(X)\ne \lambda_{V'}(X) \). Consider \[Y:= \frac{X-\lambda_{V}(X)}{\lambda_{V'}(X)-\lambda_{V}(X)}\in Z(\fg). \]
By Lemma \ref{lemInfiniteAction}, we have a commutative diagram \[
\begin{tikzcd}
h_{\natural}^{+}(V)\arrow[d,"Y"] \arrow[r,"f"]
& h_{\natural}^{+}(V'){[-m]}\arrow[d,"Y"]
\\
h_{\natural}^{+}(V) \arrow[r,"f"]
& h_{\natural}^{+}(V'){[-m]}.
\end{tikzcd}
\] On the other hand, \( Y|_{h_{\natural}^{+}(V)}=0 \in \pi_{0}\End(h^{+}_{\natural}(V))\) and \( Y|_{h_{\natural}^{+}(V')}=1 \in \pi_{0}\End(h^{+}_{\natural}(V'))\). This implies that \( f\cong 0 \), as desired.
\end{proof}
\begin{dfn}
For \( \lambda:\ZZ(\fg)\to L \), we define
\( \mscr{O}_{P,\lambda}^{+} \) to be the stable subcategory generated by \( h_{\natural}(V_{\alpha}) \) for 
\( \lambda_{\alpha}=\lambda \).
\end{dfn}
\begin{cor}\label{corDecomposeCatOInfiniteAction}
There is a decomposition \( \mscr{O}_{P}^{+} \cong \bigoplus_{\lambda}\mscr{O}_{P,\lambda}^{+} \). 

For \( M\in\mscr{O}_{P}^{+} \),
we denote its factor in \( \mscr{O}_{P,\lambda}^{+} \) by \( M^{\widehat{\lambda}} \), which we refer to as the generalized eigenspace for \( Z(\fg)=\lambda \). Then we have \( M\cong \bigoplus_{\chi}M^{\widehat{\chi}} \).
\end{cor}
\begin{proof}
This follows immediately from Corollary \ref{corNoMapDiffInfinit}.
\end{proof}

\subsubsection{BGG complexes on partial flag varieties}
\label{subsubsecBGGComplex}
We continue with the setting of \S \ref{subsecPartialFl} and \S \ref{subsubsecParBGGCatO}. We extend the field \( L \) such that \( G \) is split over \( L \). Let \( B\subset G \) be a Borel subgroup contained in \( P \), and let \( T\subset B \) be a split maximal torus of \( G \) contained in \( M \). 
Let \( W \) (resp. \( W_{M} \)) denote the absolute Weyl group of \( G \) (resp. \( M \)), and let \( {}^{M}W \) denote the set of Kostant representatives of \( W_{M}\backslash W \). Let \( w \) be the longest element in \( W \). 
We will moreover restrict ourselves to the following setting:
\begin{notation}\label{notationPartialFlagMinuscule}
Let \( \mu\in X_{*}(T) \) be a minuscule cocharacter. Let \( L\subset\CC \) be a number field.
We define \[P_{\mu,\CC}^{\std}:=\{g\in G_{\CC}:\lim_{t\to\infty}\mrm{Ad}(\mu(t))(g)\text{ exists}\},
\]
and 
\[P_{\mu,\CC}:=\{g\in G_{\CC}:\lim_{t\to 0}\mrm{Ad}(\mu(t))(g)\text{ exists}\},
\] both of which descends to \(L\), which we denote as \(P_{\mu}^{\std}\) and \(P_{\mu}\). Note that \(P_{\mu}^{std}= P_{\mu^{-1}}\).
Let \(M_{\mu}\) be their common Levi quotient, which is the centralizer of \(\mu\). Let \(N_{\mu}\) and \(N_{\mu}^{\std}\)
be the unipotent radical of \(P_{\mu}\)
and \(P_{\mu}^{\std}\) respectively.
We denote the corresponding Lie algebras as \( \p_{\mu}=\mm_{\mu}\oplus \fn_{\mu} \)
and \( \bar{\p}_{\mu}=\mm_{\mu}\oplus \bar{\fn}_{\mu} \) respectively.

We further denote \( \mFl_{G,\mu}:=\mFl_{P_{\mu}}:=P_{\mu}\backslash G  \)
and \( \mFl_{G,\mu}^{\std}:=P^{\std}_{\mu}\backslash G  \), which are schemes over \(L\) equipped with a right action of \(G\). 
\end{notation}
\begin{lem}\label{lemNmuSimpleLem}
Let \( \mu\in X_{*}(G) \) be a minuscule cocharacter.
For any root \( \alpha \), \( \alpha(\mu)\in\{-1,0,1\} \), and \( N_{\mu} \) is an abelian group scheme,
with \( \fn_{\mu} \) (resp. \( \hat{\fn}_{\mu} \))
consisting of \( \alpha \)-eigenspaces for \( \alpha(\mu)=1 \) (resp. \( =-1 \)).
\end{lem}
\begin{proof}
Give a root \( \alpha \), one can find \( i_{\alpha}:\mfk{sl}_{2}\hookrightarrow \fg \), such that \( \alpha(i_{\alpha}(\mrm{diag}\{1,-1\}))=2 \). Then \( i \)
lifts to \( \SL_{2}\to G \).
If \( \alpha(\mu)\ge 2 \), then \( \mu-i_{\alpha}(\mrm{diag}\{1,-1\}) \) will be a strictly smaller cocharacter.
The description of \( \fn_{\mu} \) follows from the definition. For any \( X,X'\in \fn_{\mu} \), \( \mrm{Ad}\circ \mu \) acts on \( [X, X'] \) by \( 2 \), which implies that \( [X, X']=0 \).
\end{proof}
\begin{lem}\label{lemFilOnRep}\label{dfnGTorsorFilDRFl}
Let \( \mu\in X_{*}(G) \) be a minuscule cocharacter, and \( (V,\rho)\in\Rep(G) \). Then \( V \)
admits a unique \( P_{\mu} \)-stable filtration \( \Fil_{\mu}^{\bullet} \) such that \( \GG_{m} \) acts on \( (\gr^{i}V,\rho\circ \mu) \) via \( t\mapsto t^{i} \).

Moreover, for any \( X\in \fg \), the action of \( X \) on \( V \) induces \( \Fil^{i}_{\mu}(V)\to \Fil^{i-1}_{\mu}(V) \).

In particular, 
\( (V,\Fil^{\bullet}_{\mu})\in \Fil^{\ZZ}(\Rep(P_{\mu})) \) descends to \( \mFl_{G,\mu}^{\pdR/L,+}/G \), which we denote as \( \mcal{G}^{+}_{\mu}(V)\in \QCoh(\mFl_{G,\mu}^{\pdR/L,+}/G) \).
Let \( \mcal{G}^{+}_{\mu} \) be the associated \( G \)-torsor. Then  \( \mcal{G}_{\mu}:=\mcal{G}^{+}_{\mu}|_{\mFl_{G,\mu}^{\pdR/L}/G} \) is induced by the structure map \( \mFl_{G,\mu}^{\pdR/L}/G\to BG \).
\end{lem}
\begin{proof}
Consider \( (V,\rho\circ\mu)\in \Rep(\GG_{m})\cong \gr^{\ZZ}(D(L_{\square})) \) by Theorem \ref{thmFilterdatA1modGm}, which decomposes as \( \bigoplus_{i\in\ZZ}V_{i}  \), such that \( \GG_{m} \) acts on \( V_{i} \)
by \( t\mapsto t^{i} \).
We then put \( \Fil_{\mu}^{i}(V):=\bigoplus_{i'\ge i}V_{i} \). It suffices to show that \( \Fil_{\mu}^{i}(V) \) is \( P_{\mu} \)-stable. Note that \( V_{i}\subset V \)
is \( M_{\mu} \)-stable since \( M_{\mu} \) commutes with \( \mu \). To show that it is \( N_{\alpha} \)-stable, since \( N_{\alpha} \) is connected, it suffices to consider the \( \fn_{\alpha} \)-action, which takes \( V_{i} \) to \( V_{i+1} \) by Lemma \ref{lemNmuSimpleLem}. 
Now for the action of \( \fg \), we have \( \fg=\p_{\mu}\oplus \bar{\fn}_{\mu} \), \( \p_{\mu} \) preserves \( \Fil^{\bullet}_{\mu} \), and \( \bar{\fn}_{\mu} \)
takes \( \Fil_{\mu}^{i} \)
to \( \Fil_{\mu}^{i-1} \) by Lemma \ref{lemNmuSimpleLem}.
In other words, \( U(\fg)\otimes_{U(\p)}V\to V \) can be upgraded to a filtered morphism for \( \Fil^{\bullet}_{\mu} \)
on \( V \) and the degree filtration on \( U(\fg) \).
The last part thus follows from the description of \( \QCoh(\mF^{\pdR/L,+}_{G,\mu}/G) \)
in Lemma \ref{lemRelateClassicalCateO} (1).
\end{proof}


\begin{notation}
For \( \alpha\in X^{*}(T) \) \( G \)-dominant, let \( V_{\alpha}\in \Rep(G) \) the corresponding representation with highest weight \( \alpha \).
Similarly, for \( \alpha\in X^{*}(T) \) \( M \)-dominant, let \( W_{\alpha}\in \Rep(M) \) the corresponding representation with highest weight \( \alpha \).
Let \( \rho \) denote the half sum of positive roots of \( G \).
We define the dotted action of \( W \) on \( X^{*}(T) \) by \( w\cdot \alpha=w(\alpha+\rho)-\rho \).
\end{notation}

\begin{prop}[Filtered BGG complex]\label{propFilterBGG}
Let
\( \alpha\in X^{*}(T) \) be a \( G \)-dominant character, and let \( \mu\in X_{*}(G) \) be a minuscule cocharacter. Consider \( \mcal{G}^{+}_{\mu}(V_{\alpha})\in \QCoh(\mFl_{G,\mu}^{\pdR/L,+}/G) \) in Lemma \ref{dfnGTorsorFilDRFl}.

Then \( \mcal{G}^{+}_{\mu}(V_{\alpha})\in\mscr{O}_{P}^{+} \),
and it is isomorphic to a so-called \emph{BGG complex} \[
\mrm{BGG}_{\alpha}^{+}:=
\left[L_{\dim(\fn)}\to \cdots \to L_{1}\to L_{0} \right],
\] with \( L_{i}\cong \bigoplus_{w\in {}^{M}W,l(w)=i}h_{\natural}^{+}(W_{w\cdot \alpha})\{(w\cdot\alpha)(\mu)\} \) in cohomological degree \( -i \).
\end{prop}
\begin{rmk}
The result differs from \cite{Faltings1983CohoOfLocSymmetric}, because \cite{Faltings1983CohoOfLocSymmetric} fixes the Borel \( B\supset P_{\mu}^{\std} \), while we are fixing \( B\supset P_{\mu} \).
\end{rmk}
\begin{proof}
The following proof is essentially the same as that in \cite{Faltings1983CohoOfLocSymmetric}.
We will consider the decomposition given by Corollary \ref{corDecomposeCatOInfiniteAction}. Note that the action of \( Z(\fg) \) on \( \mcal{G}^{+}_{\mu}(V_{\alpha}) \) factors through \( \lambda_{\alpha} \) 
by the functoriality in Lemma \ref{lemInfiniteAction}.

We start with the case where \( \alpha=0 \) and \( V_{\alpha}\cong \mO_{BG} \). Then by Lemma \ref{lemDRcomplexVSpushforward} and Remark \ref{rmkResolveObyDX}, we have a resolution 
\begin{align*}
0\to h_{\natural}^{+}(\wedge^{\dim(\fn)}\Omega^{1,\vee}_{\mFl_{G,\mu}/L})\{-\dim(\fn)\}\to \cdots\to h_{\natural}^{+}(\Omega^{1,\vee}_{\mFl_{G,\mu}/L})\{-1\}\\\to h_{\natural}^{+}(\mO_{\mFl_{G,\mu}})\to \mO_{\mFl_{G,\mu}^{\pdR/L,+}}\to0.
\end{align*}
More precisely, inductively, we need to construct \[
h^{+}_{\natural}(\wedge^{i}\Omega^{1,\vee}_{\mFl_{G,\mu}/L})\to \left[
	h_{\natural}^{+}(\wedge^{i-1}\Omega^{1,\vee}_{\mFl_{G,\mu}/L})\to \cdots\to h_{\natural}^{+}(\mO_{\mFl_{G,\mu}})\to \mO_{\mFl_{G,\mu}^{\pdR/L,+}},
\right]
\] which is induced by the following isomorphism given by Lemma \ref{lemDRcomplexVSpushforward} and Remark \ref{rmkResolveObyDX}: \[
\wedge^{i}\Omega^{1,\vee}_{\mFl_{G,\mu}/L}\cong \Fil^{-i}\left((h^{+})^{*}
	\left[
	h_{\natural}^{+}(\wedge^{i}\Omega^{1,\vee}_{\mFl_{G,\mu}/L})\{-i\}\to \cdots
	\to \mO_{\mFl_{G,\mu}^{\pdR/L,+}}\right]
\right).
\]

This implies that \( \mO_{\mFl_{G,\mu}^{\pdR/L,+}}\in \mscr{O}_{P}^{+} \). 
Note that \( \Omega^{1,\vee}_{\mFl_{G,\mu}/L} \) corresponds to \( \fg/\p\in \Rep(M) \), and thus we have the correct twist \( \{-\} \) by Lemma \ref{lemNmuSimpleLem}.
By the argument in \cite[70]{Faltings1983CohoOfLocSymmetric} and Lemma \ref{lemCompareInfiniteAction},
\[ (h_{\natural}^{+}(\wedge^{i}\Omega^{1,\vee}_{\mFl_{G,\mu}/L}) )^{\widehat{\lambda_{0}}}\cong \bigoplus_{w\in {}^{M}W,l(w)=i}h_{\natural}^{+}(W_{w\cdot 0}^{\vee}).
\] Thus taking \( (-)^{\widehat{\lambda_{0}}} \) of the resolution above, we obtain the desired resolution.

For general \( \alpha \), we tensor the BGG resolution for \( \alpha=0 \) with \( \mcal{G}^{+}_{\mu}(V_{\alpha})\) and then apply \( (-)^{\widehat{\lambda_{\alpha}}} \). By Lemma \ref{lemSmBaseChange} (1), \( \mscr{O}_{P} \)
is closed under \( -\otimes \mcal{G}^{+}_{\mu}(V_{\alpha}) \), and for \( w\in {}^{M}W \),
\( \mcal{G}^{+}_{\mu}(V_{\alpha})\otimes h_{\natural}^{+}(W_{w\cdot 0})\cong h_{\natural}^{+}(W_{w\cdot 0}\otimes (V_{\alpha},\Fil^{\bullet}_{\mu})) \). As in \cite[71-72]{Faltings1983CohoOfLocSymmetric},
\( W_{w\cdot 0}\otimes V_{\alpha} \) as a \( P_{\mu} \)-representation admits a subquotient \( W_{w\cdot \alpha} \), corresponding to weight \( w\alpha \)-vector in \( V_{\alpha} \), such that \[
U(\fg)\otimes_{U(\p)}W_{w\cdot \alpha}\{w\alpha(\mu)\}
\cong \left(V_{\alpha}\otimes (U(\fg)\otimes_{U(\p)}W_{w\cdot \alpha})\right)^{\widehat{\lambda_{\alpha}}}.
\] Thus by the conservativity in Lemma \ref{lemRelateClassicalCateO}, we have \[h_{\natural}^{+}(W_{w\cdot \alpha})\{w\alpha(\mu)\}\cong \left(\mcal{G}^{+}_{\mu}(V_{\alpha})\otimes h_{\natural}^{+}(W_{w\cdot \alpha})\right)^{\widehat{\lambda_{\alpha}}}.
\] This gives the desired BGG resolution.
\end{proof}
Proposition \ref{propFilterBGG} can be understood as saying that \[ \mcal{G}_{\mu}^{+}(V_{\alpha})\in \QCoh(\mFl_{G,\mu}^{\pdR/L,+}/G) \] admits a \( \ZZ \)-filtration \( \Fil_{\mrm{BGG}} \), such that \[
\gr^{i}_{\mrm{BGG}}(\mcal{G}_{\mu}^{+}(V_{\alpha}))\cong L_{-i}[i]\cong \bigoplus_{w\in {}^{M}W,l(w)=i}h_{\natural}^{+}(W_{w\cdot \alpha})\{(w\cdot\alpha)(\mu)\}[i].
\]
When \( G \) is isogeny to a product \( \prod_{j=1}^{n}G_{j} \), the construction of Proposition \ref{propFilterBGG} can be refined to construct a \( \ZZ^{n} \)-filtration:
\begin{prop}\label{propPartialBGGDual}
Let \( \alpha \)
and \( \mu \)
be as in Proposition \ref{propFilterBGG}.
Assume that \( \prod_{j=1}^{n}G_{j}\to G \) is a central extension, and \( \mu:\GG_{m}\to G \) lifts to a minuscule cocharacter \( \ti{\mu}=(\mu_{j}):\GG_{m}\to \prod_{j=1}^{n}G_{j} \). 
Let \( M_{j,\mu_{j}}\subset G_{j} \) be the Levi subgroup defined by \( \mu_{j} \) (Notation \ref{notationPartialFlagMinuscule}), and let \( {}^{M_{\mu_{j}}}W_{G_{j}} \) be the set of Kostant representatives of \( W_{M_{\mu_{j}}}\backslash W_{G_{j}} \).

Then \( \prod_{j=1}^{n}M_{j,\mu_{j}}\to M_{\mu} \) is a central extension, and \( {}^{M}W\cong \prod_{j=1}^{n}{}^{M_{\mu_{j}}}W_{G_{j}} \).
Moreover,
\( \mcal{G}^{+}_{\mu}(V_{\alpha})\in \QCoh(\mFl_{G,\mu}^{\pdR/L,+}/G) \)
admits a \( \ZZ^{n} \)-filtration \( \Fil^{\bullet}_{\mrm{BGG}} \), such that 
\begin{align*}
&\gr^{-\unl{i}}_{\mrm{BGG}}(\mcal{G}^{+}_{\mu}(V_{\alpha}))
\cong \\&\bigoplus_{(w_{j})_{j=1}^{n}\in \prod_{j=1}^{n}{}^{M_{\mu_{J}}}W_{G_{j}},l(w_{j})=i_{j}}h^{+}_{\natural}(W_{(\prod_{j}w_{j})\cdot \alpha})\left\{((\prod_{j}w_{j})\cdot \alpha)(\mu)\right\}\left[\sum_{j}i_{j}\right]
\end{align*}
 for \( \unl{i}=(i_{1},\ldots,i_{n}) \).
\end{prop}
\begin{proof}
The proof follows that of Proposition \ref{propFilterBGG}. We have a central extension \( \prod_{j}P_{j,\mu_{j}}\to P_{\mu} \), and thus \(\Fl_{G,\mu}\cong \prod_{j=1}^{n}\Fl_{G_{j},\mu_{j}}\) and \[
\Fl_{G,\mu}^{\pdR/L,+}/G\cong \Fl_{G_{1},\mu_{1}}^{\pdR/L,+}/G\times_{BG}\cdots\times_{BG}\Fl_{G_{n},\mu_{n}}^{\pdR/L,+}/G.
\] 
Thus when \( \alpha=0 \), the resolution 
\begin{align*}
0\to h_{\natural}^{+}(\Omega^{\dim(\fn),\vee}_{\mFl_{G,\mu}/L})\{-\dim(\fn)\}\to \cdots\to h_{\natural}^{+}(\Omega^{1,\vee}_{\mFl_{G,\mu}/L})\{-1\}\to \\h_{\natural}^{+}(\mO_{\mFl_{G,\mu}})\to \mO_{\mFl_{G,\mu}^{\pdR/L,+}}\to 0
\end{align*}
in the proof of Proposition \ref{propFilterBGG}
admits a \( \ZZ^{n} \)-filtration such that \[
\gr^{\unl{i}}(\Omega^{m,\vee}_{\mFl_{G,\mu}/L})
\cong \begin{cases}
0,&m\ne \sum_{j=1}^{n}i_{j}\\
h_{1,\natural}^{+}(\Omega^{i_{1},\vee}_{\mFl_{G_{j},\mu_{j}}/L})\boxtimes\cdots\boxtimes 
h_{n,\natural}^{+}(\Omega^{i_{n},\vee}_{\mFl_{G_{j},\mu_{j}}/L}),&m=\sum_{j=1}^{n}i_{j},
\end{cases}
\] where \( h_{j}^{+}:\Fl_{G_{j},\mu_{j}}/G\times[\mA^{1}/\GG_{m}]\to \Fl_{G_{j},\mu_{j}}^{\pdR/L,+}/G \). Here we have implicitly used the K\"unneth formula (Lemma \ref{lemKunnethProper}).
The rest then follows from the proof of Proposition \ref{propFilterBGG}.
\end{proof}

\begin{cor}[Filtered dual BGG complex]\label{corFilterDualBGG}
Let
\( \alpha\in X^{*}(T) \) be a \( G \)-dominant character, and let \( \mu\in X_{*}(G) \) be a minuscule cocharacter. Consider \( \mcal{G}^{+}_{\mu}(V_{\alpha})\in \QCoh(\mFl_{G,\mu}^{\pdR/L,+}/G) \) in Lemma \ref{dfnGTorsorFilDRFl}.

Then \( \mcal{G}^{+}_{\mu}(V_{\alpha}) \) is isomorphic to a so-called \emph{dual BGG complex} \[
\mrm{dBGG}_{\alpha}^{+}:=
\left[M^{0}\to M^{1}\cdots \to M^{\dim(\fn)} \right],
\] with \( M^{i}\cong \bigoplus_{w\in {}^{M}W,l(w)=i}h_{*}^{+}(W^{\vee}_{w\cdot (-w_{0}\alpha)})\{-(w\cdot(-w_{0}\alpha))(\mu)\} \) in cohomological degree \( i \).
\end{cor}
\begin{proof}
This follows from Proposition \ref{propFilterBGG} by taking \( \unl{\RHom}(-,\mO_{\mFl_{G,\mu}^{\pdR/L,+}/G}) \).
\end{proof}
\begin{cor}\label{corPartialBGGDual}
Let \( \alpha \), \( \mu \), \( (G_{j},\mu_{j}) \), \( M_{j,\mu_{j}} \) and \( {}^{M_{\mu_{j}}}W_{G_{j}} \)
be as in Proposition \ref{propPartialBGGDual}. 
Then
\( \mcal{G}^{+}_{\mu}(V_{\alpha})\in \QCoh(\mFl_{G,\mu}^{\pdR/L,+}/G) \)
admits a \( \ZZ^{n} \)-filtration \( \Fil^{\bullet}_{\mrm{BGG}^{\vee}} \), such that 
\begin{align*}
&\gr^{\unl{i}}_{\mrm{BGG}^{\vee}}(\mcal{G}^{+}_{\mu}(V_{\alpha}))\cong \bigoplus_{(w_{j})_{j=1}^{n}\in \prod_{j=1}^{n}{}^{M_{\mu_{j}}}
W_{G_{j}},l(w_{j})=i_{j}}\\
&h^{+}_{\natural}\left(W_{-(\prod_{j}w_{j})\cdot (-w_{0}\alpha)}\right)
\left\{-((\prod_{j}w_{j})\cdot (-w_{0}\alpha))(\mu)\right\}\left[-\sum_{j=1}^{n}i_{j}\right]
\end{align*}
for \( \unl{i}=(i_{1},\ldots,i_{n}) \).
\end{cor}
\begin{proof}
This follows from Proposition \ref{propPartialBGGDual} by taking \( \unl{\RHom}(-,\mO_{\mFl_{G,\mu}^{\pdR/L,+}/G}) \).
\end{proof}
\begin{eg}[Theta operator]\label{egThetaP1}
Let \( G=\GL_{2} \), \( \mu:\GG_{m}\to G,t\mapsto \mrm{diag}(t,1) \),
and \( P_{\mu}=B \), the subgroup of upper triangular matrices. 
Then the dual BGG complex for \( V_{\alpha}=\Sym^{k-1}\mrm{Std}\otimes \mrm{det}^{l} \) for \( \alpha=(k-1+l,l) \) is a two-term resolution of \( \mcal{G}^{+}_{\mu}(V_{\alpha}) \) \[ \theta^{k}:h^{+}_{*}\left((l,k+l-1)\right)\{l\}\to h^{+}_{*}\left((k+l,l-1)\right)\{k+l\}. \]
Thus \( \theta^{k} \) corresponds to a differential operator of degree \( k \) by Lemma \ref{lemDiffOpeToLogStack}.
Applying \( (h^{+})^{*} \), we obtain a fiber sequence in \( \QCoh(\Fil^{\ZZ}(\Rep(P_{\mu}))) \) \[
(V_{\alpha},\Fil^{\bullet}_{\mu})\to (h^{+})^{*}h^{+}_{*}\left((l,k+l-1)\right)\{l\}\xrightarrow{\theta^{k}} (h^{+})^{*}h^{+}_{*}\left((k+l,l-1)\right)\{k+l\}.
\]  Note that 
\( \gr^{i}_{\mu}(V_{\alpha})\cong 0 \) for \( i\notin [l,k+l] \). Thus \( \theta^{k} \) induces an isomorphism on \( \gr^{i} \) for \( i\notin [l,k+l] \).
\end{eg}
\subsection{de Rham period sheaves}\label{subsecDeRhamPeriodSheaves}
In this subsection, we realize the de Rham period sheaves (\cite{Scholze13}, \cite{DLLZ2022logarithmicJAMS}) as quasi-coherent sheaves on the algebraic de Rham stacks (Definitions \ref{dfnAlgDRJuan} \& \ref{dfnFilterDRAlg}). 

\begin{construction}\label{constructionBdROnDRStack}
Let \(S=\Spa(R,R^{+})\)
be an affinoid perfectoid space over \(L\) for some \(p\)-adic field  \(L\subset\bar{\QQ}_{p}\), and let \(\mX_{\rF}(S)\) denote the relative Fargues-Fontaine curve associated to \(S^{\flat}\) (\cite[Definition II.1.15]{FarguesScholze2021geometrization}). Then there is a closed embedding \(S\hookrightarrow \mX_{\rF}(S)\) (\cite[Proposition 11.3.1]{ScholzeWeinstein2020berkeley}). 
Let \(\BdR^{+}(R)\) denote the formal completion of \(\mX_{\rF}(S)\)
along the divisor \(S\). 

For any \(n\in\NN\),
we can endow \(\BdR^{+}(R)/t^{n}\) with the solid ring structure associated to the standard topology induced from \(([\varpi],p)\)-adic topology on \(W(R^{+,\flat})[1/p]\twoheadrightarrow \BdR^{+}(R)/t^{n}\), where \([\varpi]\) is the Teichmüller lift of a pseudo-uniformizer \(\varpi\in R^{+,\flat}\). Note that \(\BdR^{+}(R)/t^{n}\) is indeed bounded (\cite[Definition 2.6.10]{Juan2024analyticdeRham}) with the subring of power bounded elements (\cite[Definition 2.6.1 (3)]{Juan2024analyticdeRham}) given by the preimage of \(R^{\circ}\subset R\) along \(\BdR^{+}(R)/t^{n}\twoheadrightarrow R\).

We further endow \(\BdR^{+}(R)/t^{n}\)
 with the analytic structure induced by that of \(R\) via \cite[Proposition 12.23]{CS20}, which is again bounded. 
 
Using Definition \ref{dfnSpfFunctor},
 we define \[\BdR^{+}(S)^{+}:=\Spf^{+}(\BdR^{+}(R))\to [\mA^{1}/\GG_{m}].
 \] We denote \( \BdR^{+}(S):=\BdR^{+}(S)^{+}\times_{[\mA^{1}/\GG_{m}]}[\GG_{m}/\GG_{m}] \).

 Then by Definition \ref{dfnAlgDRJuan}, \(S\to S^{\pdR/L}\) extends naturally to
\[f_{n}:\AnSpec(\BdR^{+}(R)/t^{n})\to S^{\pdR/L},\;f_{\infty}:\BdR^{+}(S)\to S^{\pdR/L}
\] 
and we define \[
\BdR^{+}/t^{n}:=f_{n,*}(\mO_{\AnSpec(\BdR^{+}(R)/t^{n})})\in\QCoh(S^{\pdR}).
\]

We further define \[
\BdR^{+}:=\varprojlim_{n}\BdR^{+}/t^{n},
\;\BdR:=\BdR^{+}[1/t]\in \QCoh(S^{\dR}),
\] which are endowed with the \(t\)-adic filtration. 
\end{construction}
\begin{rmk}
Instead of considering the formal completion, we could also consider the dagger completion along \(S\hookrightarrow \mX_{\rF}(S)\),
which also descends to \(S^{\dR}\), defining an alternative period sheaf \(\BdR^{\dagger}\in\QCoh(S^{\dR})\).
\end{rmk}

This construction is closely related to the sheaf \(\OBdRlog^{+}\)
introduced by \cite{Scholze13}, \cite{DLLZ2022logarithmicJAMS}. See \cite{GuoLi2021period} for an alternative interpretation. 

\begin{dfn}[{\cite[Definition 6.8]{Scholze13}, \cite[Definition 2.2.10]{DLLZ2022logarithmicJAMS}}]\label{dfnOBdRDFN}
Let \(X=\Spa(A,A^{+})\) be a log smooth rigid variety over \(L\) for some \(p\)-adic field \(L\subset\bar{\QQ}_{p}\). Then we have the period sheaf \(\OBdRlog^{+}\) on the pro-Kummer \'etale site of \(X\).
Assume that \(\ti{X}=\Spa(B,B^{+})\to X\) is a pro-Kummer \'etale morphism, and \(\ti{X}\) is perfectoid. 

Then \(\OBdRlog^{+}(B)\) is equipped with a (complete) filtration \(\Fil^{\bullet}\), contains \(\BdR^{+}(B)\) and \(A\), and is equipped with an integrable connection \[\nabla:\OBdRlog^{+}(B)\to \OBdRlog^{+}(B)\otimes_{A}\Omega^{1}_{A/L,\log}\]
satisfying the Griffiths transversality,
such that the filtered Poincar\'e sequence below is exact (\cite[Corollary 6.13]{Scholze13}, \cite[Corollary 2.4.2]{DLLZ2022logarithmicJAMS}): 
\begin{align*}
0\to \BdR^{+}(B)\to \OBdRlog^{+}(B)\xrightarrow{\nabla}\OBdRlog^{+}(B)\otimes_{A}\Omega^{1}_{A/L,\log}\\\to \cdots
\to \OBdRlog^{+}(B)\otimes_{A}\Omega^{\dim(X)}_{A/L,\log}\to 0.
\end{align*}
\end{dfn}
\begin{construction}\label{constructionOBdRTate}
Let \(\ti{X}\to X\) as in Definition \ref{dfnOBdRDFN}. \(\OBdRlog^{+}(B)/\Fil^{i}\) is a nilpotent thickening of \(B\), and we endow it with the analytic ring structure induced by \cite[Proposition 12.23]{CS20}, which is again bounded. 
We define \[\OBdRlog^{+}(\ti{X})^{+}:=
\Spf^{+}(\OBdRlog^{+}(B)):=\varinjlim_{n}\AnSpec^{+}(\OBdR^{+}(B)/\Fil^{n})
\] over \( [\mA^{1}/\GG_{m}] \), where the RHS is defined in Definition \ref{dfnSpfFunctor}.
We denote \(\OBdRlog^{+}(\ti{X}):=\OBdRlog(\ti{X})^{+}\times_{[\mA^{1}/\GG_{m}]}[\GG_{m}/\GG_{m}]\).

Then there are natural maps \[
\begin{tikzcd}
\ti{X}\times [\mA^{1}/\GG_{m}]\arrow[r] & 
\OBdRlog^{+}(\ti{X})^{+}\arrow[r]\arrow[d] & \BdR^{+}(\ti{X})^{+}\\
& X\times [\mA^{1}/\GG_{m}]
\end{tikzcd}
\] where \(\BdR^{+}(\ti{X})^{+}\) is as in Construction \ref{constructionBdROnDRStack}. 
The construction globalizes by Lemma \ref{lemBdRTorsor} (2) below. Note that the horizontal maps are \(!\)-surjections, being colimits of descendable covers (\cite[Proposition 3.35]{Mathew2016galois}). 

We will omit the subscript \(\log\) when \(X\) is smooth with the trivial log structure.

The construction has good functorial properties.
\end{construction}
\begin{lem}\label{lemBdRTorsor}
Let \(f:S=\Spa(R,R^{+})\to S'=\Spa(R',R^{\prime,+})\) be a map
between affinoid perfectoid spaces, which induces
\begin{align*}&
f^{[i]}:\AnSpec(\BdR^{+}(R)/t^{i})\to \AnSpec(\BdR^{+}(R')/t^{i}),\\&
\ti{f}^{[i]}:\AnSpec(\OBdR^{+}(R)/\Fil^{i})\to \AnSpec(\OBdR^{+}(R')/\Fil^{i}).
\end{align*}

(1) If \(f\) is light pro\'etale morphism for a locally profinite group \( \Gamma \), then \(f^{[i]}\) and \(\ti{f}^{[i]}\)
are \(\unl{\Gamma}\)-torsors.

(2) If \(f\) is an open immersion, then \(f^{[i]}\) and \(\ti{f}^{[i]}\) are open immersions.

In particular, the constructions globalize and works for perfectoid spaces \(S\) over \(\QQ_{p}\).
\end{lem}
\begin{proof}
The case \(i=1\): 
note that \( S \) 
and \( S' \) 
are essentially Tate (Lemma \ref{lemBasicPropertyEsseTate} (3)).
By Theorem \ref{thmAnDRStackGeneral} (4), \( S^{Tate}\to S^{\prime,Tate} \) 
is a \( \unl{\Gamma} \)-torsor. Thus applying \( (-)_{Solid} \)  \( S\to S' \) 
is a \( \unl{\Gamma} \)-torsor. 

In general, note that we have a Cartesian diagram \[
\begin{tikzcd}
S\times[\mA^{1}/\GG_{m}]\arrow[r]\arrow[d] &\AnSpec^{+}(\OBdR^{+}(R)/\Fil^{i})\arrow[d]\arrow[r]
&
\AnSpec^{+}(\BdR^{+}(R)/t^{i})\arrow[d]\\
S'\times[\mA^{1}/\GG_{m}]\arrow[r]&\AnSpec^{+}(\OBdR^{+}(R')/\Fil^{i})\arrow[r] &
\AnSpec^{+}(\BdR^{+}(R')/t^{i}),
\end{tikzcd}
\] 
where the large square is clearly Cartesian, and the square on the right is Cartesian by Lemma \ref{lemOBdRIsPushoutOfBdR}.
Now we conclude by noting that the horizontal maps are \(!\)-surjections by \cite[Proposition 3.35]{Mathew2016galois}. 
\end{proof}

The following result is a logarithmic version of \cite[Theorem 1.1, Remark 4.14]{GuoLi2021period}.
\begin{lem}\label{lemOBdRIsPushoutOfBdR}
We endow \(\dR_{\log}(A)\) with the Hodge filtration (i.e. induced by stupid truncation).
Then there is a structure of a filtered complete \(\dR_{\log}(A)\)-algebra on \((\BdR^{+}(B),t^{\bullet})\)
such that we have a \emph{push-out square} of filtered complete algebras \[
\begin{tikzcd}
\dR_{\log}(A)\arrow[r]\arrow[d] &  A\arrow[d]\\
(\BdR^{+}(B),t^{\bullet})\arrow[r] & (\OBdRlog^{+}(B),\Fil^{\bullet}).
\end{tikzcd}
\]
\end{lem}
\begin{proof}
Let \[
\dR_{\log}(\OBdRlog^{+}(B)):=[\OBdRlog^{+}(B)\xrightarrow{\nabla}\cdots
\to \OBdRlog^{+}(B)\otimes_{A}\Omega^{\dim(X)}_{A/L,\log}],\]
which is a filtered commutative differential graded algebra (\cite[Definition 7.1.4.8]{Lurie2017HA}) via the filtration \(\Fil^{\bullet}\) on \(\OBdRlog^{+}(B)\).
We have a morphism of filtered commutative differential graded algebras
\(\dR_{\log}(A)\to \dR_{\log}(\OBdRlog^{+}(B)),\)
but using the Poincar\'e squence, the RHS is also identified with \(\BdR^{+}(B)\). 
We thus obtain the commutative diagram. 

To verify that the diagram is a push-out, it is possible to prove by a tedious calculation with Barr resolution. We prove it using the filtered complete log de Rham stack.
By Lemma \ref{lemDRcomplexVSpushforward} (1), \((\OBdRlog^{+}(B),\Fil^{\bullet},\nabla)\) defines a unique element in \(\mF\in\QCoh(X_{\log}^{\pdR/L,\hat{+}})\),
and by Lemma \ref{lemDRcomplexVSpushforward} (3) and Remark \ref{rmkResolveObyDX}, 
we have a filtered isomorphism
\[R\Gamma(X_{\log}^{\pdR/L,\hat{+}},\mF)\cong R\Gamma_{\log-\dR}(\OBdRlog^{+}(B))\cong \BdR^{+}(B).\]
Then by Lemma \ref{lemLogDRStackIsAffine}, we know that \[
\BdR^{+}(B)\otimes_{\dR_{\log}(A)}A\cong \OBdRlog^{+}(B),
\] as desired.
%
\end{proof}
\begin{dfn}\label{dfnOBdRDescend}
By Lemma \ref{lemDRcomplexVSpushforward}, \((\OBdRlog^{+}(B),\Fil^{\bullet},\nabla)\) defines a unique element in \(\QCoh(X_{\log}^{\pdR/L,\hat{+}})\), which we denote (by abuse of notation) as \(\BdR^{+}(B)^{+}\), which is further equipped with a complete \(t\)-adic filtration \(t^{\bullet}\). 

By Lemma \ref{lemOBdRIsPushoutOfBdR}, 
\(\BdR^{+}(B)^{+}\in\QCoh(X_{\log}^{\pdR/L,\hat{+}})\)
corresponds to \[(\BdR^{+}(B),t^{\bullet})\in \Mod_{\dR_{\log}(A)}(\widehat{\Fil}^{\ZZ}(D(L_{\square})))\]
via the equivalence in Lemma \ref{lemLogDRStackIsAffine}. 

We define \(\BdR(B)^{+}:=(\BdR^{+}(B)^{+}[1/t],t^{\bullet})\in \widehat{\Fil}^{\ZZ}(\QCoh(X_{\log}^{\pdR/L,\hat{+}}))\). 

We will regard \( \BdR(B)^{+} \) and  \( \BdR^{+}(B)^{+} \) 
as objects in \( \Fil^{\ZZ}(\QCoh(X_{\log}^{\pdR/L,+})) \) via \(\widehat{j}_{*}\) (Notation \ref{rmkEmbedFilterComplete}), and denote by \( \BdR(B) \) 
and \( \BdR^{+}(B) \) 
by their restrictions to \( X_{\log}^{\pdR/L} \). 
\end{dfn}
\begin{lem}\label{lemGrBdRonDRStack}
We have 
\begin{align*}
\gr^{i}_{t^{\bullet}}(\BdR(B)^{+})\cong f_{*}(\mO_{\ti{X}\times[\mA^{1}/\GG_{m}]})
\{i\}(i),\;i\in\ZZ
\end{align*}
and 
\[\gr^{i}_{t^{\bullet}}(\BdR^{+}(B)^{+})\cong
\begin{cases}
f_{*}(\mO_{\ti{X}\times[\mA^{1}/\GG_{m}]})
\{i\}(i),&i\in\NN
    \\
0,&i\in\ZZ_{<0}.
\end{cases}
\]
Here \(f\) denotes the composition \(\ti{X}\times \mA^{1}/\GG_{m}\to X\times \mA^{1}/\GG_{m}\to X^{\pdR/L,+}_{\log}\), \( \{i\} \) 
denotes the twist of the filtration (Notation \ref{notationTwistByAGm}), and \( (i) \) 
denotes the Tate twist.
\end{lem}
\begin{proof}
Via the equivalence in Lemma \ref{lemLogDRStackIsAffine}, 
for \(i\in \NN\), \(\gr^{i}_{t^{\bullet}}(\BdR^{+}(B)^{+})\) is isomorphic to \(B\{i\}(i)\) regarded as a \(\dR_{\log}(A)\)-module via \(\dR_{\log}(A)\to\gr^{0}(\dR_{\log}(A))\cong A\to B\). 
This corresponds to \(f_{*}(\mO_{\ti{X}\times[\hat{\mA}^{1}/\GG_{m}]})\{i\}(i)\in \QCoh(X_{\log}^{\pdR/L,\hat{+}})\)
by Lemma \ref{lemLogDRStackIsAffine}. Then by Lemma \ref{lemPerfComplexA1Gm}, it is the same as \(f_{*}(\mO_{\ti{X}\times[\mA^{1}/\GG_{m}]})\{i\}(i)\in\QCoh(X_{\log}^{\pdR/L,+})\).
The proof for \(\BdR(B)\) is similar.
\end{proof}

\begin{lem}\label{lemBdRMapToLogDRStack}
Let \(X\) be a log smooth rigid variety over \(L\) for some \(p\)-adic field \(L\subset\bar{\QQ}_{p}\), and let \(f:\ti{X}\to X\) be a pro-Kummer \'etale morphism, where \(\ti{X}\) is perfectoid. 

Then there is a natural map \(f^{\pdR}:\BdR^{+}(\ti{X})\to X^{\pdR/L}_{\log}\), 
such that there is a commutative diagram \[
\begin{tikzcd}
\ti{X}\times {[\mA^{1}/\GG_{m}]}\arrow[r]\arrow[rr,bend left,"f"] &
\OBdRlog^{+}(\ti{X})^{+}\arrow[r]\arrow[d] & X\times {[\mA^{1}/\GG_{m}]}\arrow[d]\\ &
\BdR^{+}(\ti{X})^{+}\arrow[r,"f^{\pdR}"] & X^{\pdR/L,+}_{\log}\arrow[r] & {[\mA^{1}/\GG_{m}]},
\end{tikzcd}
\] 
where the square is Cartesian \emph{after} restricting to \( [\hat{\mA}^{1}/\GG_{m}] \). 
Here  \(\OBdRlog^{+}(\ti{X})^{+}\) and \(\BdR^{+}(\ti{X})^{+}\) are as in Construction \ref{constructionOBdRTate} and Construction \ref{constructionBdROnDRStack}.

Moreover, if \(\ti{X}=\Spa(B,B^{+})\to X=\Spa(A,A^{+})\),
then \(f^{\pdR}_{*}(\mO_{\BdR^{+}(\ti{X})^{+}})\cong \BdR^{+}(B)^{+}\), 
\end{lem}
\begin{proof}
Localizing, 
we can assume that \(X=\Spa(A,A^{+})\)
and \(\ti{X}=\Spa(B,B^{+})\).
By Lemma \ref{lemOBdRIsPushoutOfBdR}, we have a push-out diagram of filtered complete algebras \[
\begin{tikzcd}
\BdR^{+}(B)\arrow[r]
& \OBdRlog^{+}(B)\\
\dR(A/L)\arrow[u]\arrow[r]
& A\arrow[u].
\end{tikzcd}
\] 
In particular, we have a morphism between cosimplicial system of filtered complete algebras \[
(A^{\otimes_{\dR(A/L)}n})_{[n]\in\Delta}\to 
(\OBdRlog^{+}(B)^{\otimes_{\BdR^{+}(B)}n})_{[n]\in\Delta}.
\]
Applying the functor \(\Spf^{+}(-)\) (Definition \ref{dfnSpfFunctor}), we
 have a morphism between simplicial system of solid stacks \[
\left(\Spf^{+}\left(\OBdRlog^{+}(B)^{\otimes_{\BdR^{+}(B)}n}\right)\right)_{[n]\in\Delta^{op}}\to 
\left(\Spf^{+}\left(A^{\otimes_{\dR(A/L)}n}\right)\right)_{[n]\in\Delta^{op}}. 
\] Note that the RHS is precisely the groupoid object in Lemma \ref{lemThisIsGroupoid}, and the LHS is the \v Cech nerve associated to the \(!\)-cover \(\OBdRlog^{+}(\ti{X})\to \BdR^{+}(\ti{X})\) by Construction \ref{constructionOBdRTate} and Lemma \ref{lemSpfCommuteWithLimit}. 

So taking simplicial colimits on both sides, by Definition \ref{dfnFilterDRAlg},
we obtain a morphism \[\Spf^{+}(\BdR^{+}(B))\to X^{\pdR/L,+}.
\] such that we have a commutative diagram \[
\begin{tikzcd}
\Spf^{+}(\OBdRlog^{+}(B))\arrow[r]\arrow[d] \mcorner & X\times {[\mA^{1}/\GG_{m}]}\arrow[d] \\
\Spf^{+}(\BdR^{+}(B))\arrow[r]& X^{\pdR/L,+}_{\log},
\end{tikzcd}
\] that becomes Cartesian after restricting of \( [\hat{\mA}^{1}/\GG_{m}] \) by Lemma \ref{lemOBdRIsPushoutOfBdR}. 
\end{proof}

In the smooth case,
we have two constructions \(\BdR^{+}(\ti{X})\to X^{\pdR/L}\) (Construction \ref{constructionBdROnDRStack} and Lemma \ref{lemBdRMapToLogDRStack}). Actually they are the same. 
\begin{lem}\label{lemTwoBdRCoincide}
Let \(\ti{X}\to X\) be as in Definition \ref{dfnOBdRDescend}.
Assume in addition that \(X\) is smooth and is equipped with the trivial log structure. 

Then 
we have the following commutative diagram \[
\begin{tikzcd}
\BdR^{+}(\ti{X})\arrow[r,"f_{\infty}"]\arrow[rd,"f^{\pdR}"] & \ti{X}^{\pdR/L}\arrow[d,"\pi"] \\
& X^{\pdR/L}
\end{tikzcd}
\] where \(f_{\infty}\) is as in Construction \ref{constructionBdROnDRStack}, and \(f^{\pdR}\) is as in Lemma \ref{lemBdRMapToLogDRStack}.

In particular,
\((\BdR^{+}(B),t^{\bullet})\in\Fil^{\ZZ}(\QCoh(X^{\pdR/L}))\) in Definition \ref{dfnOBdRDescend}
is isomorphic to \((\pi_{*}(\BdR^{+}),t^{\bullet})\),  
where \(\BdR^{+}\in \QCoh(\ti{X}^{\pdR/L})\) is given by  in Construction
\ref{constructionBdROnDRStack}.
\end{lem}
\begin{proof}
By definition of \(X^{\pdR/L}\) (Definition \ref{dfnAlgDRJuan}), \[
\mrm{Map}_{L}(\BdR^{+}(\ti{X}),X^{\pdR/L})
\cong 
\mrm{Map}_{L}(\ti{X},X^{\pdR/L}).
\] Thus there are forced to be the same, as both maps are the unique extension of \(f:\ti{X}\to X\).
\end{proof}






\section{Shimura varieties as analytic stacks}\label{sectionLAVectorSV}

In this section, we work with general Shimura varieties. 
We want to study the locally analytic vectors in the structure sheaf of infinite-level Shimura varieties,
following \cite{Pan22} and \cite{Juan2022.09locallyShi}.
In the framework of analytic stacks, we will define so-called locally analytic Shimura varieties.
The two main results are Theorem \ref{thmLAstrucGMHTper} and Theorem \ref{thmReformulateRHCorr}. 

Theorem \ref{thmLAstrucGMHTper} says roughly that the geometry of locally analytic Shimura varieties is determined by its de Rham stack plus a so-called ``Grothendieck-Messing-Hodge-Tate period map''. 
We learn about this period map first in the local setting in \cite{DospinescuJuan2024jacquet}, but the version for global Shimura varieties is more subtle, for which we first need to develop a theory of ``Grothendieck-Messing period map'' in Subsection \ref{subsecGrothMessTheory}.

Theorem \ref{thmReformulateRHCorr} is about the analytic de Rham stack of perfectoid spaces. We explain how one can fit the (logarithmic) Riemann-Hilbert correspondence into our framework. 

Here is a more detailed plan:

In Subsection \ref{subsecGrothMessTheory}, we give a reformulation of the classical Kodaira-Spencer isomorphism (Propositions \ref{propAlgGMtheory} \& \ref{propAnalyticGMTheory}), which could be understood as a characteristic 0 version of Grothendieck-Messing theory, but note that our results hold for \emph{any} Shimura varieties. As a corollary, we give a reformulation of (dual) BGG complexes on Shimura varieties.

In Subsection \ref{subsectionInfiniteLevelShimuraVar}, we define two versions---a completed version (Lemma \ref{lemTorCompacAsTateStack}) and a smooth / uncompleted version (Definition \ref{dfnNonCompletedInfiniteLevel})---of the original construction of the ``perfectoid Shimura variety" of \cite{Scholze15}, in the frame work of analytic stacks. The slight trickiness comes from the toroidal boundaries. The completed infinite-level Shimura varieties also admit Hodge-Tate period maps to \( \Fl_{G,\mu} \) in the category of Tate stacks.

In Subsection \ref{subsecDeriveLocallyAnalyticVectors},
we introduce the derived locally analytic vectors as defined in \cite{JRC2021solid}, \cite{JacintoJoaquínJuan2023SolidLARepII}, and apply the construction to define the main object of our study---locally analytic Shimura varieties. We show that its cohomology calculates the locally analytic part of the completed cohomology of \cite{Emerton06}. 

In Subsection \ref{subsecGMHTperiodMap}, we define a so-called ``Grothendieck-Messing-Hodge-Tate period map'', which combines the Hodge-Tate period map and the Grothendieck-Messing period map. We state a theorem that could be understood as a Grothendieck-Messing theory of locally analytic Shimura varieties (Theorem \ref{thmLAstrucGMHTper}). This is one of the deepest results in the section. 

In Subsection \ref{subsecGeometricSen}, 
we recall the geometric Sen theory from \cite{Pan22}, \cite{Pilloni22}, \cite{Juan2022.09locallyShi}, which is a key tool for studying locally analytic Shimura varieties. This will allow us to relate locally analytic Shimura varieties to smooth infinite-level Shimura varieties in Subsection \ref{subsecLocallyAnalyticShimuraVar} (Proposition \ref{propLAvsSmoothSV}). 

In Subsection \ref{subsecSMdRstack}, we further relate smooth infinite-level Shimura varieties to its de Rham stack. Since we don't know whether taking algebraic de Rham stacks commutes with limits in general, the latter is defined as the inverse limit of de Rham stacks of finite level Shimura varieties.
Combining Subsections \ref{subsecLocallyAnalyticShimuraVar} \& \ref{subsecSMdRstack}, we complete the proof of Theorem \ref{thmLAstrucGMHTper}.

In Subsection \ref{subsecHorLocaAnaShimuraVarie},
we study the 
``horizontal action'' of \(\mfk{m}^{0}\), and define 
the horizontal locally analytic Shimura variety, which is the quotient of the locally analytic Shimura variety by this action.
The key result is 
Corollary \ref{corCartesianFlandSMGeneral}, which follows immediately from Theorem \ref{thmLAstrucGMHTper}.
shows in particular that the map from the horizontal locally analytic Shimura variety \(\mX^{\la,0}_{K^{p}}\) to \(\mX_{\Kpp}\times \Fl_{G,\mu}\) is ``\'etale'' in a suitable sense. 
In the next section, we will use this to pull back the classical theta operators on \(\mX_{\Kpp}\)
and \(\Fl_{G,\mu}\) to define the operators \(d^{k_{\tau}}_{\tau}\)
and \(\bar{d}^{k_{\tau}}_{\tau}\).

Subsections \ref{subsecBdRSV} and \ref{subsecRiemannHilbert} focus on another topic: a reformulation of the logarithmic Riemann-Hilbert correspondence on the Shimura varieties. 
In Subsection \ref{subsecBdRSV},
we start with defining a \( \BdR^{+} \)-version of the infinite-level Shimura varieties. This is essentially in Subsection \ref{subsecDeRhamPeriodSheaves}, except some additional treatment is needed when the perfectoidness of the infinite-level Shimura varieties is unknown. 
In Subsection \ref{subsecRiemannHilbert}, we reformulate the logarithmic Riemann-Hilbert correspondence (Theorem \ref{thmReformulateRHCorr}, Corollary \ref{corReformulateRiemmanHilbPerfdCase}). The key point is that the Riemann-Hilbert correspondence can be formulated in a way that is intrinsic to a perfectoid space equipped with an action of a \( p \)-adic Lie group. 

In Subsection \ref{subsecLocalAnalogueLASv}, we give an account of the previous results in the setting of basic local Shimura varieties (\cite{ScholzeWeinstein2020berkeley}). Part of these results are essentially in \cite{DospinescuJuan2024jacquet}.

\subsection{Grothendieck-Messing theory}\label{subsecGrothMessTheory}
\subsubsection{Set-up}\label{subsecSetUpSV}
As before, we fix an identification \(\iota:\CC\cong\bar{\QQ}_{p}\).
\begin{notation}\label{notationShimuraSetUp}
Let \((G,X)\) be any Shimura datum, \(E\subset \CC\) be the reflex field, and \(E_{\p}
\) be the \(p\)-adic completion along the composition \(E\subset \CC\cong \bar{\QQ}_{p}\). In particular, we have natural embeddings \(E_{\p}\subset\bar{\QQ}_{p}\subset \CC_{p}\).

(1)
For any neat open compact subgroup \(K\subset G(\mA_{f})\), let \(X_{K}\) 
be 
the associated Shimura variety defined over \(E\) (\cite{Milne1990canonical}), and let \( X_{K}^{\tor} \) be a toroidal compactification of \( X_{K} \) (\cite{Pink1989arithmetical}). Here we fix a suitable choice of the cone decomposition such that for a fixed level \(K^{p}K_{p,0}\), \(X_{K^{p}K_{p,0}}^{\tor}\) is smooth, and the boundary \(\partial X_{K}^{\tor}\) is a normal crossings divisor (\cite[Example 2.3.17]{DLLZ2019logarithmicFoundational}). 
Later, we will only consider \( K\subset K^{p}K_{p,0} \).

For any finite extension \( L\subset \bar{\QQ}_{p} \) of \( E_{\p} \), 
let \(\mX_{K,L}\) (resp. \(\mX_{K,L}^{\tor}\))
be the analytification of \(X_{K}\times_{E}L\) (resp. \(X_{K}^{\tor}\times_{E}L\)), 
and denote \(\mX_{K}:=\mX_{K,E_{\p}}\times_{\Spa(E_{\p})}\Spa(\CC_{p})\) (resp. \(\mX_{K}^{\tor}:=\mX^{\tor}_{K,E_{p}}\times_{\Spa(E_{\p})}\Spa(\CC_{p})\)). 

(2)
Let \(\mu:\GG_{m,\bar{\QQ}}\to G_{\bar{\QQ}}\)
be the Hodge cocharacter determined by the Shimura datum \((G,X)\). Let \( P^{\std}_{\mu},\;P_{\mu},\;M_{\mu},\;N_{\mu}^{\std},\;N_{\mu} \) be as in Notation \ref{notationPartialFlagMinuscule}.

Let
\[\mrm{Fl}_{G,\mu,\bar{\QQ}}:=P_{\mu}\backslash G_{\bar{\QQ}},\;\mrm{Fl}_{G,\mu,\bar{\QQ}}^{\std}:=P_{\mu}^{\std}\backslash G_{\bar{\QQ}},
\] which are schemes over \(\bar{\QQ}\) equipped with a right action of \(G_{\bar{\QQ}}\). 
They both descend to schemes over \(E\), which we denote as \( \mrm{Fl}_{G,\mu} \) and \( \mrm{Fl}_{G,\mu}^{\std} \).

For any finite extension \( L\subset \bar{\QQ}_{p} \) of \( E_{\p} \), 
we denote by \(\Fl_{G,\mu,L}\) and \(\Fl^{\std}_{G,\mu,L}\) the associated rigid varieties over \(\Spa(L)\), and denote by \(\Fl_{G,\mu}\)
and \(\Fl^{\std}_{G,\mu}\) their base change to \(\Spa(\CC_{p})\).
\end{notation}
\begin{lem}\label{lemFiniteLevelProper}
(1)
The solid stacks \( \mX^{\tor}_{K} \), \( \Fl_{G,\mu}^{\std} \)
and \( \Fl_{G,\mu} \) are cohomologically co-smooth over \( \AnSpec(\CC_{p,\square}) \). \( (\mX^{\tor}_{K})^{\pdR/\CC_{p}}\to \AnSpec(\CC_{p,\square}) \)
is prim.

(2) \( \Fl_{G,\mu}^{\dR}\to \AnSpec(\bar{\QQ}_{p}) \) and
\( \Delta_{\Fl}^{\dR}:\Fl_{G,\mu}^{\dR}\to \Fl_{G,\mu}^{\dR}\times_{\bar{\QQ}_{p}}\Fl_{G,\mu}^{\dR} \)
are cohomologically co-smooth.
\end{lem}
\begin{proof}
The first part follows from Proposition \ref{propAnGAGA} and Example \ref{egScheme}. The part for algebraic de Rham stacks is given by Corollary \ref{corDRStackPrim}. For the second part, the cohomological co-smoothness of the first map follows from Theorem \ref{thmVariousDescent} (5),
and Theorem \ref{thmAnDRStackGeneral} (3) and (6).
For the second map \( \Delta^{\dR}_{\Fl} \), we note that \( \Delta_{Fl}:\Fl_{G,\mu}\to \Fl_{G,\mu}\times_{\CC_{p}}\Fl_{G,\mu} \)
is cohomogically co-smooth  by Proposition \ref{propAffineProperCover}, and
\( \Fl_{G,\mu}\times_{\CC_{p}}\Fl_{G,\mu}\to \Fl_{G,\mu}^{\dR}\times_{\bar{\QQ}_{p}}\Fl_{G,\mu}^{\dR}
\) is cohomogically co-smooth again by Theorem \ref{thmVariousDescent} (5)
and Theorem \ref{thmAnDRStackGeneral} (3) and (6). Thus we can again conclude by Theorem \ref{thmVariousDescent} (5).
\end{proof}
\begin{notation}\label{notationWidetildeGroup}
We define \(Z_{c}\subset Z=Z(G)\) the maximal \(\QQ\)-anisotropic \(\RR\)-split torus, and  \(G^{c}:=G/Z_{c}\). For any subgroup \(K_{p}\subset G(\QQ_{p})\),
define \(\widetilde{K_{p}}:=K_{p}/\overline{(K^{p}K_{p}\cap Z(\QQ))}\). For any subgroup \(H\subset G_{\bar{\QQ}}\) containing \(Z_{c}\), we denote \(H^{c}:=H/Z_{c}\). Note that \(\widetilde{K_{p}}\) depends implicitly on the neat open compact subgroup \(K^{p}\).
\end{notation}

\begin{notation}
	\label{notationEquShvOnFl}
Let \(\fg:=\Lie(G)\otimes\CC_{p}\), \(\fg^{0}:=\fg\otimes \mO_{\Fl_{G,\mu}}\), and let \(\p_{\mu}^{0}:=\p^{0}\), \(\fn^{0}_{\mu}:=\fn^{0}\) and \(\mm_{\mu}^{0}:=\mm^{0}\), where the RHS is as in Notation \ref{notationEquShvOnFlGeneral} for \(L=\CC_{p}\) and \(P=P_{\mu}\). We will write \(\p^{0}\), \(\fn^{0}\) and \(\mm^{0}\) if it causes no confusion. We write \(\fg^{c,0},\p^{c,0},\mm^{0}\)
for the quotient of \(\p^{0}\), \(\fn^{0}\) and \(\mm^{0}\) by \(\mfk{z}_{c}:=\Lie(Z_{c})\otimes\CC_{p}\).

We write \(\mcal{M}_{\mu}^{c}:=\mcal{M}_{P_{\mu}^{c}}\), \(\mcal{P}_{\mu}^{c}:=\mcal{P}_{P_{\mu}^{c}}\) over \( \mFl_{G,\mu} \) and \(\mcal{M}_{\mu}^{\std,c}:=\mcal{M}_{P^{\std,c}_{\mu}}\), \(\mcal{P}_{\mu}^{\std,c}:=\mcal{P}_{P_{\mu}^{\std,c}}\) over \( \mFl_{G,\mu}^{\std} \), where the RHS is as in Definition \ref{dfnFlandM} for \(P=P_{\mu}^{c}\) and \(P=P_{\mu}^{\std,c}\) respectively.

After base change to \( \CC_{p} \), by Proposition \ref{propAnGAGA}, \( (\mFl_{G,\mu})_{\CC_{p}} \cong \Fl_{G,\mu}\). We will write \( \mcal{M}_{\mu}^{c,\an}:= ((\mcal{M}_{\mu}^{c})^{Tate})_{Solid}\), which by Lemma \ref{lemSolidToTateAnalyticficationFunctor} is a reduction of structure of \( \mcal{M}_{\mu}^{c} \) to a \( M_{\mu}^{c,\an} \)-torsor over \( \Fl_{G,\mu} \). We define \(\mcal{P}_{\mu}^{c,\an}\), \(\mcal{M}_{\mu}^{\std,c,\an}\), and \(\mcal{P}_{\mu}^{\std,c,\an}\) similarly.
\end{notation}

\begin{notation}[de Rham torsor]\label{notationDRtorsor}
By \cite{Milne1990canonical}, \cite{Harris1989functorial}, for any \(V\in \Rep(G^{c})\), there is a \(\mu^{-1}\)-filtered vector bundle \(D_{\dR,K}(V)\) (``the canonical extension'')
over \(X_{K}^{\tor}\) equipped with an integrable log connection \(\nabla_{\dR,K}\). 
The construction \( D_{\dR,K}(-) \) 
is moreover a \( \otimes \)-functor. 

By Tannakian formalism,
this induces a (right) \(P^{\std,c}_{\mu}\)-torsor \(\mcal{P}_{\dR,K}^{\std,c}\) over \(X^{\tor}_{K,\bar{\QQ}}\).
We denote \[
\mcal{G}_{\dR,K,\bar{\QQ}}^{c}:=(\mcal{P}_{\dR,K}^{\std,c}\times G^{c}_{\bar{\QQ}})/P^{\std,c}_{\mu},\;
\mcal{M}_{\dR,K}^{c}:=(\mcal{P}_{\dR,K}^{\std,c}\times M^{c}_{\mu})/P^{\std,c}_{\mu}
\] over \(X^{\tor}_{K,\bar{\QQ}}\). Moreover, \( \mcal{G}_{\dR,K,\bar{\QQ}}^{c} \) 
descends to a \( G^{c} \)-torsor \( \mcal{G}_{\dR,K}^{c} \)  
over \( X^{\tor}_{K} \). 
Then we have a Cartesian diagram \[
\begin{tikzcd}
\mcal{G}_{\dR,K,\bar{\QQ}}^{c}\arrow[r]\arrow[d]& \mrm{Fl}^{\std}_{G,\mu,\bar{\QQ}}\arrow[d]\\
X_{K,\bar{\QQ}}^{\tor}\arrow[r,"\pi_{\GM,K}"] & {[\mrm{Fl}^{\std}_{G,\mu,\bar{\QQ}}/G^{c}_{\bar{\QQ}}]}
\end{tikzcd}
\] which also descends to \(E\).

After base change to \( \CC_{p} \), by Proposition \ref{propAnGAGA}, \( (X^{\tor}_{K})_{\CC_{p}} \cong \mX^{\tor}_{K}\), and we will write \( \mcal{P}^{\std,c,\an}_{\dR,K} \), \( \mcal{G}^{c,\an}_{\dR,K} \) and \( \mcal{M}^{c,\an}_{\dR,K} \) for the corresponding torsors obtained by applying \( ((-)^{Tate})_{Solid} \).
\end{notation}

\subsubsection{Algebraic Grothendieck-Messing theory}
\begin{lem}	\label{lemDRtorsorDescends}
There is a \(G^{c}\)-torsor \(\bar{\mcal{G}}^{c,+}_{\dR,K}\)
over \((X^{\tor}_{K})^{\pdR/E,+}_{\log}\) such that for any \(V\in \Rep(G^{c})\), the associated filtered vector bundle \(\bar{\mcal{G}}^{c,+}_{\dR,K}(V)\) with an integrable log connection over \((X^{\tor}_{K})^{\pdR/E,+}_{\log}\)
coincides via Lemma \ref{lemDRcomplexVSpushforward} with 
the canonical extension \((D_{\dR,K}(V),\nabla_{\dR,K})\) in Notation \ref{notationDRtorsor}. 

Consider \(h_{K}:X^{\tor}_{K}\times [\mA^{1}/\GG_{m}]\to (X^{\tor}_{K})^{\pdR/E,+}_{\log}\). Then we have a commutative diagram of solid stacks \[
\begin{tikzcd}
X^{\tor}_{K}\times [\mA^{1}/\GG_{m}]\arrow[r,"\pi_{\GM,K}"] \arrow[d,"h_{K}"]& 
{\left[\mrm{Fl}^{\std}_{G,\mu}/G^{c}\right]\times [\mA^{1}/\GG_{m}]}\arrow[d,"h_{\mrm{Fl}^{\std}}"]
\\
(X^{\tor}_{K})^{\pdR/E,+}_{\log}\arrow[r,"\pi_{\GM,K}^{\pdR,+}"] &
{\left[(\mrm{Fl}^{\std}_{G,\mu})^{\pdR/E,+}/G^{c}\right]}.
\end{tikzcd}
\]
The morphism \( \pi_{\GM,K^{p}K_{p,0}} \) is cohomologically smooth, and we have \[
\mcal{P}_{\dR,K}^{\std,c}\cong \pi_{\GM,K}^{*}(\mcal{P}_{\mu}^{\std,c}),
\mcal{M}_{\dR,K}^{c}\cong \pi_{\GM,K}^{*}(\mcal{M}_{\mu}^{\std,c}),
\] 
Moreover, \[ \bar{\mcal{G}}_{\dR,K}^{c,+}\cong (\pi_{\GM,K}^{\pdR,+})^{*}(\mcal{G}^{\std,c,+}_{\mu}), \]
where \( \mcal{G}^{\std,c,+}_{\mu}:=\mcal{G}_{\mu^{-1}}^{c,+}\to(\mFl_{G^{c},\mu^{-1}})^{\pdR/E,+}/G^{c} \) as in Lemma \ref{lemFilOnRep}.

We will refer to \(\pi_{\GM,K}\) as the \emph{Grothendieck-Messing period map}.
\end{lem}
\begin{proof}
The first paragraph follows from Lemma \ref{lemDRcomplexVSpushforward} and the Tannakian formalism. We will write \( \mcal{\bar{G}}^{c}_{\dR,K}:=\mcal{\bar{G}}^{c,+}_{\dR,K}|_{X_{K}^{\pdR/E}} \).

Note that \(\mcal{G}^{c}_{\dR,K}\cong h_{K}^{*}(\mcal{\bar{G}}^{c}_{\dR,K})|_{\mX^{\tor}_{K}}\to X^{\tor}_{K} \) admits a reduction of structure to a \( P_{\mu}^{\std,c} \)-torsor as in Notation \ref{notationDRtorsor}.
This induces \(\pi_{\GM,K}:X^{\tor}_{K}\to {[\mrm{Fl}^{\std}_{G,\mu}/G^{c}]}\cong BP_{\mu}^{\std,c}\).
Note that by trivializing the \( P^{c}_{\mu} \)-torsor \( \mcal{P}_{\dR,K}^{c} \) locally on \( X^{\tor}_{K} \), the fiber of \( \mcal{G}^{c}_{\dR,K}\to \mrm{Fl}^{\std}_{G,\mu} \) 
is locally isomorphic to \( X^{\tor}_{K}\times P^{c}_{\mu} \). 
In particular, as \( X^{\tor}_{K^{p}K_{p,0}} \) is assumed to be smooth, \( \pi_{\GM,K^{p}K_{p,0}} \) is cohomologically smooth by \cite[Proposition 3.6.13]{Juan2024analyticdeRham}. 

We continue to construct \( \pi_{\GM,K}^{\pdR,+} \). 
We start by defining \[
\pi_{\GM,K}^{\pdR}:(X^{\tor}_{K})^{\pdR/E}_{\log}\to (\mFl^{\std}_{G,\mu})^{\pdR/E}/G^{c}.
\]
By Lemma \ref{lemTwoAlgDRSame}, 
there is a natural factorization \( X_{K}\to (X_{K}^{\tor})_{\log}^{\dR/E}\to (X_{K}^{\tor})^{\dR/E} \). 
Consider the Cartesian diagram \[
\begin{tikzcd}
\mcal{G}^{c}_{\dR,K}
\arrow[r]\arrow[d] & \mcal{\bar{G}}^{c}_{\dR,K}\arrow[d]\\
X^{\tor}_{K}\arrow[r] & (X^{\tor}_{K})^{\pdR/E}_{\log}.
\end{tikzcd}
\] Then by Theorem \ref{thmAnDRStackGeneral} (1), there is an isomorphism \[
(\mcal{G}^{c}_{\dR,K})^{\pdR}\times_{(X^{\tor}_{K})^{\pdR}}(X^{\tor}_{K})^{\pdR/E}_{\log}\cong \mcal{\bar{G}}^{c}_{\dR,K}
\times^{G^{c}}(G^{c})^{\pdR},
\] where we note that \( E \) is \'etale over \( \QQ \), so \( (-)^{\pdR,+}=(-)^{\pdR/E,+} \) by Lemma \ref{lemBasicLogDRSta} (3).
Now \( \mcal{G}^{c}_{\dR,K}\to \mFl^{\std}_{G,\mu} \)
induces \( (\mcal{G}^{c}_{\dR,K})^{\pdR}\to (\mFl^{\std}_{G,\mu})^{\pdR} \),
and thus, using the isomorphism, induces a \( G^{c} \)-equivariant map 
\( \bar{G}^{c}_{\dR,K}\to (\mFl^{\std}_{G,\mu})^{\pdR} \). Thus quotienting out \( G^{c} \) on both sides gives the desired \( \pi^{\pdR}_{\GM,K} \).

Now to extend \( \pi^{\pdR}_{\GM,K} \) to \( \pi^{\pdR,+}_{\GM,K} \),
we observe that \( \pi^{\pdR}_{\GM,K} \)
induces a morphism of groupoid objects \[
\begin{tikzcd}
\hat{\Delta}_{X^{\tor}_{K},\log}
\arrow[d,bend left,"p_{1}'"]
\arrow[d,bend right,"p_{2}'"']
\arrow[r] & \hat{\Delta}_{\mFl^{\std}_{G,\mu}}/G^{c}\arrow[d,bend left,"p_{1}"]
\arrow[d,bend right,"p_{2}"']\\
X^{\tor}_{K}
\arrow[u,"\Delta'"]
\arrow[r] & \mFl_{G,\mu}^{\std}/G^{c}.
\arrow[u,"\Delta"]
\end{tikzcd}
\]
As in Definition \ref{dfnCompleteGrpObj}, we can take \( (-)^{\pdR,+} \) along \( \Delta \),
which gives a morphism of groupoid objects \[
\begin{tikzcd}
(X^{\tor}_{K})^{\pdR/(\hat{\Delta}_{X^{\tor}_{K},\log}),+}
\arrow[d,bend left,"p_{1}'"]
\arrow[d,bend right,"p_{2}'"']
\arrow[r] & 
(\mFl_{G,\mu}^{\std}/G^{c})^{\pdR/(\hat{\Delta}_{\mFl^{\std}_{G,\mu}}/G^{c}),+}
\arrow[d,bend left,"p_{1}"]
\arrow[d,bend right,"p_{2}"']\\
X^{\tor}_{K}
\arrow[u,"\Delta'"]
\arrow[r] & \mFl_{G,\mu}^{\std}/G^{c}.
\arrow[u,"\Delta"]
\end{tikzcd}
\]
However, by Lemma \ref{lemClosedImmeAlgDRFil} and Lemma \ref{lemDiagonalLogProperty} (1),
we have isomorphisms of groupoid objects
\[
(X^{\tor}_{K})^{\pdR/(\hat{\Delta}_{X^{\tor}_{K},\log}),+}\cong \hat{\Delta}_{X^{\tor}_{K},\log}^{+},\;
(\mFl^{\std}_{G,\mu})^{\pdR/(\hat{\Delta}_{\mFl^{\std}_{G,\mu}}),+}\cong \hat{\Delta}_{\mFl^{\std}_{G,\mu}}^{+}.
\]
By Theorem \ref{thmAnDRStackGeneral} (1), we also have \[
(\mFl_{G,\mu}^{\std}/G^{c})^{\pdR/(\hat{\Delta}_{\mFl^{\std}_{G,\mu}}/G^{c}),+}
\cong (\mFl^{\std}_{G,\mu})^{\pdR/(\hat{\Delta}_{\mFl^{\std}_{G,\mu}}),+}/G^{c}\cong \hat{\Delta}_{\mFl^{\std}_{G,\mu}}^{+}/G^{c}.
\] Quotienting out these groupoid objects as in Definition \ref{dfnFilterDRAlg},
we obtain the desired \( \pi_{\GM,K}^{\pdR,+} \).

We are now left to comparing the various torsors. 
The isomorphism \( \mcal{P}_{\dR,K}^{\std,c}\cong \pi_{\GM,K}^{*}(\mcal{P}_{\mu}^{\std,c}) \) is clear from the construction of \( \pi_{\GM,K} \), and the comparison of the \( M_{\mu}^{c} \)-torsor follows immediately.
The isomorphism \( \bar{\mcal{G}}_{\dR,K}^{c,+}\cong (\pi_{\GM,K}^{\pdR,+})^{*}(\mcal{G}^{\std,c,+}_{\mu}) \) requires some explanation. From the construction of \( \pi_{\GM,K}^{\pdR} \), the \( G^{c} \)-torsor \( (\mFl^{\std}_{G,\mu})^{\pdR/E}\to (\mFl^{\std}_{G,\mu})^{\pdR/E}/G^{c} \) pulls back to \( \bar{\mcal{G}}^{c}_{\dR,K}\to (X^{\tor}_{K})^{\pdR/E}_{\log} \),
but the former is isomorphic to \( \mcal{G}^{\std,c,+}_{\mu}|_{\left[(\mrm{Fl}^{\std}_{G,\mu})^{\pdR/E}/G^{c}\right]} \). This induces an isomorphism \[ \bar{\mcal{G}}_{\dR,K}^{c,+}|_{(X^{\tor}_{K})^{\pdR/E}_{\log}}\cong (\pi_{\GM,K}^{\pdR,+})^{*}(\mcal{G}^{\std,c,+}_{\mu})|_{(X^{\tor}_{K})^{\pdR/E}_{\log}}. \]

Now by Tannakian formalism and Lemma \ref{lemDRcomplexVSpushforward}, we are left to verify that for any \( V\in\Rep(G^{c}) \), the induced isomorphism of vector bundles is compatible with filtrations. This is true since on both sides, the filtrations are induced by \( \Fil_{\mu_{-1}}^{\bullet} \) in Lemma \ref{lemFilOnRep} and the reduction of structure of \( \mcal{G}^{c}_{\dR,K} \) to \( P^{c}_{\mu} \), and the latter is shown to be compatible above.
\end{proof}
\begin{rmk}[Subcanonical extension]\label{rmkSubCanonicalExtension}
We will also consider the \emph{subcanonical extension} considered in \cite{Harris1989functorial}. For this, assume that \( X_{K^{p}K_{p,0}}^{\tor} \) is smooth with normal crossings divisor, and let \( \mscr{I}_{X_{K^{p}K_{p,0}}^{\tor}} \) be the ideal sheaf defining the boundary. Then the usual differential 
\[ \mscr{I}_{X_{K^{p}K_{p,0}}^{\tor}}\hookrightarrow \mO_{X_{K^{p}K_{p,0}}^{\tor}}\to \Omega^{1}_{X_{K^{p}K_{p,0}}/E}\cong \mscr{I}_{X_{K^{p}K_{p,0}}^{\tor}}\otimes \Omega^{1}_{X_{K^{p}K_{p,0}}/E,\log} \]
defines a log differential on \( \mscr{I}_{X_{K^{p}K_{p,0}}^{\tor}} \). 
Therefore, by Lemma \ref{lemDRcomplexVSpushforward}, it (equipped with the trivial filtration) descends to a line bundle on \( (X_{K^{p}K_{p,0}}^{\tor})_{\log}^{\pdR/E,+} \).
 We denote this line bundle by  \( \mscr{I}_{K^{p}K_{p,0}}\in \QCoh((X_{K^{p}K_{p,0}}^{\tor})_{\log}^{\pdR/E,+}) \). 
 For \( K_{p}\subset K_{p,0} \), we denote by \( \mscr{I}_{K^{p}K_{p,0}} \in \QCoh((X_{\Kpp}^{\tor})_{\log}^{\pdR/E,+}) \) its pull-back to \( (X_{\Kpp}^{\tor})_{\log}^{\pdR/E,+} \). 
Then for \( V\in\Rep(G^{c}) \),  the subcanonical extension 
is given by \( \bar{\mcal{G}}^{c}_{\dR,K}(V)\otimes \mscr{I}_{K}\in \mrm{Vect}((X^{\tor}_{K})^{\pdR/E,+}_{\log}) \). 
\end{rmk}
\begin{prop}[Algebraic Grothendieck-Messing theory]\label{propAlgGMtheory}
The commutative diagram in Lemma \ref{lemDRtorsorDescends} is Cartesian.
\end{prop}
\begin{lem}[Kodaira-Spencer isomorphism]\label{lemIsoOmega1KS}
There are a natural \(\mO_{X^{\tor}_{K}}\)-linear isomorphism \[\pi_{\GM}^{*}(\Omega^{1}_{\mrm{Fl}_{G,\mu}})\cong \Omega^{1}_{X^{\tor}_{K},\log}\]
and a natural \(\mO_{X^{\tor}_{K}}\)-linear filtered isomorphism \[(\pi_{\GM}^{*}(p_{1,*}(\mO_{\hat{\Delta}^{+}_{\mrm{Fl}_{G,\mu}}})))^{\wedge}_{\Fil}
\cong p_{1,*}(\mO_{\hat{\Delta}^{+}_{X^{\tor}_{K},\log}}).
\]
Here we use the notation of Definition \ref{dfnLogDiagnal},
\((-)^{\wedge}_{\Fil}\) denotes the filtered completion, and
 we regard \(\Omega^{1}_{\mrm{Fl}_{G,\mu}}\)
and \(p_{1,*}(\mO_{\hat{\Delta}^{+}_{\mrm{Fl}_{G,\mu}}})\)
as objects in \(\QCoh(\mrm{Fl}_{G,\mu}/G^{c})\) by descent, where \(G^{c}\) acts on \(\hat{\Delta}^{+}_{\mrm{Fl}_{G,\mu}}\)
diagonally.
\end{lem}
\begin{proof}
The commutative diagram in Lemma \ref{lemDRtorsorDescends} induces a map between groupoid objects \[
\begin{tikzcd}
\hat{\Delta}^{+}_{X^{\tor}_{K},\log} \arrow[r]
\arrow[d,bend right,"p_{1}"]\arrow[d,bend left,"p_{2}"] & {[\hat{\Delta}^{+}_{\mrm{Fl}^{\std}_{G,\mu}}/G^{c}]}
\arrow[d,bend right,"p_{1}"]\arrow[d,bend left,"p_{2}"] \\
{X^{\tor}_{K}\times [\mA^{1}/\GG_{m}]} \arrow[r,"\pi_{\GM}"] & {[\mrm{Fl}^{\std}_{G,\mu}/G^{c}]\times [\mA^{1}/\GG_{m}]},
\end{tikzcd}
\]
which induces a natural filtered morphism 
\[
i':
(\pi_{\GM}^{*}(p_{1,*}(\mO_{\hat{\Delta}^{+}_{\mrm{Fl}^{\std}_{G,\mu}}})))^{\wedge}_{\Fil}
\to p_{1,*}(\mO_{\hat{\Delta}^{+}_{X_{K}^{\tor},\log}}).
\]
Taking \(\gr^{1}\), we obtain (by Lemma \ref{lemDiagonalLogProperty} (1)) a natural morphism 
\[i:\pi_{\GM}^{*}(\Omega^{1}_{\mrm{Fl}^{\std}_{G,\mu}})\to \Omega^{1}_{X_{K}^{\tor},\log}.
\] 

This is an isomorphism by the classical Kodaira-Spencer isomorphism. See for example \cite[Lemma A.10]{GrahamPilloniJacinto2023pNearlyOveconvergent}. 
Let us give a sketch here: 

Over \(\mrm{Fl}^{\std}_{G,\mu}\), we have a canonical identification \(\Omega^{1}_{\mrm{Fl}^{\std}_{G,\mu}}\cong (\fg^{0}/\p_{\mu}^{std,0})^{\vee}\), and we have a filtered vector bundle with connection \(\fg^{0}\to \fg^{0}\otimes \Omega^{1}_{\mrm{Fl}^{\std}_{G,\mu}}\). Taking the graded pieces, we obtain a \(\mO_{\mrm{Fl}^{\std}_{G,\mu}}\)-linear morphism \(\mm^{0}_{\mu}\to (\fg^{0}/\p^{std,0}_{\mu})\otimes \Omega^{1}_{\mrm{Fl}^{\std}_{G,\mu}}\), which gives a \(\mO_{X_{K}^{\tor}}\)-linear map \[f_{1}:\pi_{\GM}(\mm^{0}_{\mu})\otimes\pi_{\GM}(\fg^{0}/\p^{std,0}_{\mu})^{\vee} \to \pi_{\GM}^{*}(\Omega^{1}_{\mrm{Fl}^{\std}_{G,\mu}}).
\] 

On the other hand, 
\(\pi_{\GM}(\fg^{0})\)
is a filtered vector bundle over \(X_{K}^{\tor}\) with connection,
which by the same argument gives another \(\mO_{X_{K}^{\tor}}\)-linear map \[f_{2}:\pi_{\GM}(\mm^{0}_{\mu})\otimes\pi_{\GM}(\fg^{0}/\p^{std,0}_{\mu})^{\vee}\to \Omega^{1}_{X_{K}^{\tor},\log}.
\] By tracing the diagram, we can see the \(f_{1}\)
and \(f_{2}\)
are compatible along the natural morphism \(i\) above. 
However, both \(f_{1}\)
and \(f_{2}\) are surjective. This can by checked on the scheme level by working on the complex manifolds. Therefore, \(i\) is surjective, and thus it is an isomorphism since both sides are vector bundles of the same rank. 

We are left to show that \(i'\) is also an isomorphism. Note that by Lemma \ref{lemDiagonalLogProperty} (1), \(\gr^{*}(\mO_{\hat{\Delta}_{\mrm{Fl}^{\std}_{G,\mu}}})\cong \Sym^{*}\Omega^{1}_{\mrm{Fl}^{\std}_{G,\mu}}\), \(\gr^{*}(\mO_{X_{K}^{\tor}})\cong \Sym^{*}\Omega^{1}_{X_{K}^{\tor},\log}\), and \(i=\gr^{1}(i')\) is an isomorphism. This implies that \(\gr^{*}(i')\) is an isomorphism. Since both sides are filtered complete, this shows that \( i \) 
is also an isomorphism. 
\end{proof}
\begin{proof}[Proof of Proposition \ref{propAlgGMtheory}]
For a solid stack \( X \), we will write \( X^{+}:=X\times [\mA^{1}/\GG_{m}] \).
By Lemma \ref{lemIsoOmega1KS} (2),
the natural morphism 
\begin{align*}
X_{K}^{\tor,+}\times_{(X^{\tor}_{K})^{\pdR/E,+}_{\log}}X_{K}^{\tor,+}\to  
X_{K}^{\tor,+}\times_{\mrm{Fl}^{\std,\pdR/E,+}_{G,\mu}/G^{c}}\mrm{Fl}^{\std,+}_{G,\mu}/G^{c}\\
\cong X_{K}^{\tor,+}\times_{\mrm{Fl}^{\std,+}_{G,\mu}/G^{c},\pr_{1}}\hat{\Delta}_{\mrm{Fl}^{\std}_{G,\mu}}^{+}/G^{c}
\end{align*}
is an isomorphism. 
Note that the RHS is also isomorphic to \[X_{K}^{\tor,+}\times_{(X^{\tor}_{K})^{\pdR/E,+}_{\log}}((X^{\tor}_{K})^{\pdR/E,+}_{\log}\times_{\mrm{Fl}^{\std,\pdR/E,+}_{G,\mu}/G^{c}}\mrm{Fl}^{\std,+}_{G,\mu}/G^{c}).\] Since \(X_{K}^{\tor,+}\to (X_{K}^{\tor})^{\pdR/E,+}_{\log}\)
is a \(!\)-cover by construction, we know that the natural morphism 
\[X_{K}^{\tor,+}\to (X^{\tor}_{K})^{\pdR/E,+}_{\log}\times_{\mrm{Fl}^{\std,\pdR/E,+}_{G,\mu}/G^{c}}\mrm{Fl}^{\std,+}_{G,\mu}/G^{c}
\]
is an isomorphism.
\end{proof}
\begin{rmk}
Proposition \ref{propAlgGMtheory} actually gives another proof of \cite[Theorem 1.5]{DLLZ2022logarithmicJAMS}, as the fibers of \(h_{K}\) at a \(\CC_{p}\)-point \(x\) is precisely the formal neighborhood of \(x\). 
\end{rmk}
\subsubsection{BGG complex on Shimura varieties}
\label{subsubsecBGGSV}
The following BGG complexes are
studied in \cite{Faltings1983CohoOfLocSymmetric}, \cite[VI.\S 5]{FC13}, \cite[\S 5]{LanPolo2018dualBGG}, \cite[Theorem 6.1.7]{LanLiuZhu2023rham}. We give a quick proof here using Proposition \ref{propAlgGMtheory}. This reconstruction will be important for our future construction of a locally analytic Jacquet-Langlands transfer (\cite[Theorem \ref{2-thmJacquetLanglandsLA}]{Jiang2025HMF}), since we are going to define the transfer first for the period domains.
\begin{prop}[BGG complex on Shimura varieties]\label{propBGGFilSV}
Fix a Borel \( B\subset G_{\bar{\QQ}} \) such that \( B\subset P_{\mu} \).
Let
\( \alpha\in X^{*}(T) \) be a \( G \)-dominant character.



(1)  \( \bar{\mcal{G}}^{c,+}_{\dR,K}(V_{\alpha})^{\vee} \) is isomorphic to a so-called \emph{BGG complex} \[
\mrm{BGG}_{\alpha}^{+}:=
\left[L_{\dim(\fn)}\to \cdots \to L_{1}\to L_{0} \right],
\] with \( L_{i}\cong \bigoplus_{w\in {}^{M}W,l(w)=i}h_{K,\natural}(\mcal{M}_{\dR,K}^{c}(W_{w\cdot \alpha}^{\vee}))\{(w\cdot\alpha)(\mu)\} \) in cohomological degree \( -i \).

(2) \( \bar{\mcal{G}}^{c,+}_{\dR,K}(V_{\alpha}) \) is isomorphic to a so-called \emph{dual BGG complex} \[
\mrm{dBGG}_{\alpha}^{+}:=
\left[M^{0}\to M^{1}\cdots \to M^{\dim(\fn)} \right],
\] with \( M^{i}\cong \bigoplus_{w\in {}^{M}W,l(w)=i}h_{K,*}(\mcal{M}_{\dR,K}^{c}(W_{w\cdot \alpha}))\{-(w\cdot\alpha)(\mu)\} \) in cohomological degree \( i \).
\end{prop}
\begin{proof}
Using Lemma \ref{lemDRtorsorDescends} and Proposition \ref{propAlgGMtheory},
the first isomorphisms follows from
Proposition \ref{propFilterBGG} via proper base change (Lemma \ref{lemSmBaseChange} (2)). The second isomorphism follows from the first by taking dual.
 Note that the difference on the formulation comes from the different choices of the Borel subgroup as \( B\subset P_{\mu} \) rather than \( B\subset P_{\mu}^{\std} \).
\end{proof}
\subsubsection{Analytic Grothendieck-Messing theory}
We now continue to the analytic version. Due to the lack of a good definition of the ``analytic logarithmic de Rham stack'', we restrict to the open part \(\mX_{K}\subset\mX_{K}^{\tor}\).
\begin{prop}[Analytic Grothendieck-Messing theory]\label{propAnalyticGMTheory}
There is a \(G^{c,\an}_{E_{\p}}\)-torsor \(\bar{\mcal{G}}^{c,\an}_{\dR,K}\)
over \(\mX_{K,E_{\p}}^{\dR/E_{\p}}\) such that for any \(V\in \Rep(G^{c})\), the associated vector bundle \(\bar{\mcal{G}}^{c,\an}_{\dR,K}(V)\) over \(\mX_{K,E_{\p}}^{\dR/E_{\p}}\)
coincides via Lemma \ref{lemDRcomplexVSpushforward} with the analytification of \((D_{\dR,K}(V),\nabla_{\dR,K})\) in Notation \ref{notationDRtorsor}. 

Consider \(h_{K}^{\an}:\mX_{K,E_{\p}}\to \mX_{K,E_{\p}}^{\dR}\), then we have a \emph{Cartesian} diagram of solid stacks over \( E_{\p} \) \[
\begin{tikzcd}
\mX_{K,E_{\p}}\arrow[r,"\pi_{\GM,K}"] \arrow[d,"h_{K}^{\an}"]& 
{[\Fl^{\std}_{G,\mu,E_{\p}}/G^{c,\an}_{E_{\p}}]}\arrow[d,"h_{\Fl^{\std}}^{\an}"]
\\
\mX_{K,E_{\p}}^{\dR}\arrow[r,"\pi_{\GM,K}^{\dR}"] &
{[\Fl^{\std,\dR}_{G,\mu,E_{\p}}/G^{c,\an}_{E_{\p}}]}.
\end{tikzcd}
\]
Here we use implicitly Corollary \ref{corAnDRstackEssenTate} to regard analytic de Rham stacks as solid stacks.
\end{prop}
Recall that by Theorem \ref{thmAnDRStackGeneral} (3), 
for any smooth rigid varieties \(X/E_{\p}\), \[
X\times_{X^{\dR}}X\cong {\Delta}(X)^{\dagger}:= (X\subset X\times_{E_{\p}} X)^{\dagger},
\] which is equipped with two projections \(\pr_{1},\pr_{2}:{\Delta}(X)^{\dagger}\to X\), and there is a map of groupoid objects \((\hat{\Delta}_{X},\pr_{i})\to (\Delta(X)^{\dagger},p_{i})\), corresponding to \(X^{\pdR/E_{\p}}\to X^{\dR}\) in Proposition \ref{propAnToAlgDRJuan}, which gives rise to an injection \(\pr_{1,*}\mO_{\Delta(X)^{\dagger}}\hookrightarrow p_{1,*}\mO_{\hat{\Delta}_{X}}\). 
\begin{lem}\label{lemDiagonalLocally}
Let \(U=\Spa(A,A^{+})\) be an affinoid rigid variety that admits an \'etale chart \(f:U\to \TT^{n}\)
corresponding to \(f_{1},\ldots,f_{n}\in \mO_{U}(U)^{\times}\). 
Then \begin{align*}
\Delta(U)^{\dagger}&\cong \AnSpec((A,A^{+})_{\square}\{Y_{1},\ldots, Y_{n}\}^{\dagger}),\\
\hat{\Delta}_{U}&\cong \varinjlim_{n} \AnSpec((A,A^{+})_{\square}[Y_{1},\ldots, Y_{n}]/I^{n}),
\end{align*} where \(Y_{i}:=f_{i}\otimes 1-1\otimes f_{i}\), and \(I:=(Y_{1},\ldots,Y_{n})\). The natural morphism \(\pr_{1,*}\mO_{\Delta(X)^{\dagger}}\hookrightarrow p_{1,*}\mO_{\hat{\Delta}_{X}}\)
is given by the natural injection \[
A\{Y_{1},\ldots, Y_{n}\}^{\dagger}\hookrightarrow A[[Y_{1},\ldots, Y_{n}]]. 
\]
\end{lem}
\begin{proof}
This follows from the Cartesian diagram 
\[\begin{tikzcd}
\Delta(U)^{\dagger}\arrow[d,"\pr_{1}"]\arrow[r,"f'"]
& \Delta(\TT^{n})^{\dagger}\arrow[d,"\pr_{1}"]\\
U \arrow[r,"f"]
& \TT^{n},
\end{tikzcd}
\] and the concrete description of \(\Delta(\TT^{n})^{\dagger}\). 
\end{proof}

 
\begin{lem}\label{lemDiagonalFlIsEasy}
We have the following isomorphism of filtered algebras in \(\QCoh(\Fl_{G,\mu}^{\std})\) \[p_{1,*}(\mO_{\hat{\Delta}_{\Fl_{G,\mu}^{\std}}})
\cong \mO_{\Fl_{G,\mu}^{\std}}[[(\fg^{0}/\p^{std,0}_{\mu})^{\vee}]]\cong \mO_{\Fl_{G,\mu}^{\std}}[[\Omega^{1}_{\Fl_{G,\mu}^{\std}}]],
\] and the following isomorphism of algebras in \(\QCoh(\Fl_{G,\mu}^{\std})\)
\[\pr_{1,*}(\mO_{{\Delta}(\Fl_{G,\mu}^{\std})^{\dagger}})
\cong \mO_{\Fl_{G,\mu}^{\std}}\{(\fg^{0}/\p^{std,0}_{\mu})^{\vee}\}^{\dagger}\cong \mO_{\Fl_{G,\mu}^{\std}}\{\Omega^{1}_{\Fl_{G,\mu}^{\std}}\}^{\dagger}.
\] Moreover, the natural morphism \(\pr_{1,*}(\mO_{{\Delta}(\Fl_{G,\mu}^{\std})^{\dagger}})\to p_{1,*}(\mO_{\hat{\Delta}_{\Fl_{G,\mu}^{\std}}})\)
corresponds to the natural inclusion \( \mO_{\Fl_{G,\mu}^{\std}}\{\Omega^{1}_{\Fl_{G,\mu}^{\std}}\}^{\dagger}\hookrightarrow \mO_{\Fl_{G,\mu}^{\std}}[[\Omega^{1}_{\Fl_{G,\mu}^{\std}}]] \). 
\end{lem}
\begin{proof}
This is a reformulation of Lemma \ref{lemCompareDRStac}.
\end{proof}
\begin{cor}\label{corOverConvergentNbhdNotTooDifficult}
We have 
the following commutative diagram 
\[\begin{tikzcd}[column sep=0.32cm]
\pr_{1,*}(\mO_{{\Delta}(\mX_{K})^{\dagger}})\arrow[d,hook]& \arrow[l,"\alpha^{\dagger}"']\mO_{\mX_{K}}\{\Omega^{1}_{\mX_{K}}\}^{\dagger}\arrow[d,hook]
&\arrow[l] \pi_{\GM}^{*}(\mO_{\Fl_{G,\mu}^{\std}}\{\Omega^{1}_{\Fl_{G,\mu}^{\std}}\}^{\dagger})\arrow[d,hook]
\\
p_{1,*}(\mO_{\hat{\Delta}_{\mX_{K}}})& \mO_{\mX_{K}}[[\Omega^{1}_{{\mX_{K}}}]]\arrow[l,"\alpha"']
& \arrow[l]\pi_{\GM}^{*}(\mO_{\Fl_{G,\mu}^{\std}}[[\Omega^{1}_{\Fl_{G,\mu}^{\std}}]])^{\wedge}_{\Fil}
\end{tikzcd}
\]
where the horizontal maps in the middle are induced by \(\pi_{\GM}^{*}(\Omega^{1}_{\Fl_{G,\mu}^{\std}})\cong \Omega^{1}_{\mX_{K}}\) in Lemma \ref{lemIsoOmega1KS}. 
Moreover, all the horizontal maps are isomorphisms.
\end{cor}
\begin{proof}
The square on the right is given by Lemma \ref{lemIsoOmega1KS}.
Thus
we obtain 
the commutative diagram 
above, except that we need to show that the natural map (induced by the commutative diagram in the proof of Proposition \ref{propAnalyticGMTheory} below) \[
i:\mO_{\mX_{K}}\{\Omega^{1}_{\mX_{K}}\}^{\dagger}
\cong 
\pi_{\GM}^{*}(\mO_{\Fl_{G,\mu}^{\std}}\{\Omega^{1}_{\Fl_{G,\mu}^{\std}}\}^{\dagger})\hookrightarrow \pr_{1,*}(\mO_{\Delta(\mX_{K})^{\dagger}})
\] is an isomorphism. For this, we can check analytic locally on \(\mX_{K}\), so let \(U\subset \mX_{K}\)
be an affinoid open subspace admitting an \'etale chart corresponding to \( f_{1},\ldots,f_{n}\in \mO_{U}(U)^{\times} \).  
So
we reduce to the setting of Lemma \ref{lemDiagonalLocally}. We have \[
\mO_{\mX_{K}}\{\Omega^{1}_{\mX_{K}}\}^{\dagger}\hookrightarrow \pr_{1,*}(\mO_{\Delta(\mX_{K})^{\dagger}})
\hookrightarrow p_{1,*}\mO_{\hat{\Delta}_{X}}
\cong \mO_{\mX_{K}}[[\Omega^{1}_{\mX_{K}}]].
\] 
Comparing with Lemma \ref{lemDiagonalLocally}, the composition \[\alpha:
 \mO_{U}[[\Omega^{1}_{U}]]\cong \mO_{U}[[df_{1},\ldots,df_{n}]]\to 
\mO_{U}[[Y_{1},\ldots,Y_{n}]]\cong p_{1,*}\mO_{\hat{\Delta}_{U}}
\] is a filtered isomorphism, 
and 
\[
\gr^{*}(\alpha):\mO_{U}[df_{1},\ldots,df_{n}]\cong\mO_{U}[Y_{1},\ldots,Y_{n}],\;df_{i}\mapsto Y_{i}.
\]

Along the isomorphism \(\alpha\),
we have an induced \[
\alpha^{\dagger}:\mO_{\mX_{K}}\{\Omega^{1}_{\mX_{K}}\}^{\dagger}\cong 
\mO_{\mX_{K}}\{df_{1},\ldots,df_{n}\}^{\dagger}
\hookrightarrow 
\mO_{\mX_{K}}\{Y_{1},\ldots,Y_{n}\}^{\dagger}\cong \pr_{1,*}\mO_{\Delta(X)^{\dagger}}.
\]

Therefore \[\alpha^{\dagger}(df_{i})=Y_{i}+\epsilon_{i}
\] where \(\epsilon_{i}\in (Y_{1},\ldots,Y_{n})^{2}
\mO_{\mX_{K}}\{Y_{1},\ldots,Y_{n}\}^{\dagger}\subset 
\mO_{\mX_{K}}\{Y_{1},\ldots,Y_{n}\}^{\dagger}\).
Up to replacing \(Y_{i}\) by \(Y_{i}/p^{N}\),
we can assume that \[\epsilon_{i}\in p(Y_{1},\ldots,Y_{n})^{2}\mO^{+}_{\mX_{K}}\langle Y_{1},\ldots,Y_{n} \rangle.
\] Then we claim that \(\alpha^{\dagger}\) already induces an isomorphism \[
\mO_{\mX_{K}}^{+}\langle Y_{1}/p^{t},\ldots,Y_{n}/p^{t}\rangle 
\cong \mO_{\mX_{K}}^{+}\langle df_{1}/p^{t},\ldots,df_{n}/p^{t}\rangle.  
\] for any \(t\in \NN\).
For this, without losing generality, let us assume \(t=1\). Since both sides are \(p\)-torsion free and \(p\)-complete,
it suffices to verify modulo \(p\), but then \(df_{i}\)
is mapped to \(Y_{i}\). Now taking colimit along \(t\)
gives the desired isomorphism.

Thus taking colimit along \(t\), we know that \(\alpha^{\dagger}\) is an isomorphism.
\end{proof}

\begin{proof}[Proof of Proposition \ref{propAnalyticGMTheory}]
Let \( g_{K}:\mX^{\pdR/E_{\p}}_{K,E_{\p}}\to \mX^{\dR}_{K,E_{\p}} \). 
We define \[ \mcal{G}^{c,\an}_{\dR,K}(V):=g_{K,*}(\mcal{G}^{c}_{\dR,K}(V)), \] where the RHS is as in Lemma \ref{lemDRtorsorDescends}. By Proposition \ref{propAnToAlgDRJuan}, \( \mcal{G}^{c,\an}_{\dR,K}(V) \) is a vector bundle, and \( \mcal{G}^{c,\an}_{\dR,K} \) is an exact symmetric monoidal functor, and thus define a \( G^{c} \)-torsor \( \bar{\mcal{G}}^{c,\an'}_{\dR,K} \) over \( \mX^{\dR}_{K,E_{\p}} \). 
Moreover, the solid stack \( \mX^{\dR}_{K,E_{\p}} \) is essentially Tate by Corollary \ref{corAnDRstackEssenTate}. Since \( (-)_{Solid} \) and \( (-)^{Tate} \) in Lemma \ref{lemSolidToTateAnalyticficationFunctor} commutes with finite limits, and \( (G^{c})^{Tate}\cong (G^{c}_{\QQ_{p}})^{\an} \) by Proposition \ref{propAnGAGA}, we know that \[ \bar{\mcal{G}}^{c,\an'}_{\dR,K}:=((\bar{\mcal{G}}^{c,\an'}_{\dR,K})^{Tate})_{Solid}\to \mX^{\dR}_{K,E_{\p}} \]
is a reduction of structure of \( \bar{\mcal{G}}^{c,\an'}_{\dR,K} \) to a \( G^{c,\an}_{E_{\p}} \)-torsor.
By the same argument in the proof of Proposition \ref{lemDRtorsorDescends}, we obtain the commutative diagram. 

This diagram is Cartesian by the natural isomorphism \[\pi_{\GM}^{*}(\pr_{1,*}(\mO_{\Delta(\Fl_{G,\mu}^{\std})^{\dagger}}))\cong \pr_{1,*}(\mO_{\Delta(\mX_{K})^{\dagger}})\] in Corollary \ref{corOverConvergentNbhdNotTooDifficult} through the same argument in the proof of Proposition \ref{propAlgGMtheory}. 
\end{proof}



\subsection{Infinite-level Shimura varieties}\label{subsectionInfiniteLevelShimuraVar}
The goal of this subsection is to introduce two versions of the infinite-level Shimura varieties---the completed version \(\mX_{K^{p}}^{\tor}\) (Lemma \ref{lemTorCompacAsTateStack}) and the uncompleted / smooth version \(\mX_{K^{p}}^{\tor,\sm}\) (Definition \ref{dfnNonCompletedInfiniteLevel}), and we will show that the natural map from the completed version to the smooth version is a \(!\)-covering.

The construction of the infinite-level Shimura varieties was introduced by \cite{Scholze15} and further studied in \cite{CS17}, \cite{Shen2017perfectoid}, \cite{HansenJohansson2020perfectoid}. Away from the boundary, one defines \(\mX_{K^{p}}:=\varprojlim_{\Kpp}\mX_{\Kpp},\)
as a diamond, which is represented by a perfectoid space by \cite{HansenJohansson2020perfectoid} when the Shimura datum is of pre-abelian type. However, we do not know a general functor from the category of diamonds to the category of Tate stacks, sending perfectoid spaces to the associated Tate stack, as pro\'etale covers are not necessarily \(!\)-covers. Therefore, some non-formal argument is needed for the realization of \(\mX_{K^{p}}\) as a Tate stack.

We continue with the setting of Notation \ref{notationShimuraSetUp}. In particular,
we fix a base level \(K_{p,0}\), and a cone decomposition, such that the associated toroidal compactification \(\mX^{\tor}_{K^{p}K_{p,0}}\)
for \(\mX_{K^{p}K_{p,0}}\) is smooth
whose boundary is a normal crossings divisor.
For each \(K_{p}\), by rigid Abhyankar's lemma (\cite[Proposition 4.2.1]{DLLZ2022logarithmicJAMS}), 
the transition maps \(\mX^{\tor}_{\Kpp}\to \mX^{\tor}_{K^{p}K_{p,0}}\) are finite Kummer \'etale. 
\begin{rmk}
\(\mX^{\tor}_{\Kpp}\) are normal, but we do not guarantee that they are smooth.
\end{rmk}

We define \[\mX_{K^{p},\proket}^{\tor}:=\varprojlim_{K_{p}}\mX_{\Kpp}^{\tor},\]
where the limit is taken in the pro-Kummer \'etale site of \(\mX^{\tor}_{K^{p}K_{p,0}}\). Then \(\mX^{\tor}_{K^{p},\proket}\) is a pro-Kummer 
\'etale \(\widetilde{K_{p}}\)-torsor over \(\mX^{\tor}_{\Kpp}\), where \(\widetilde{K_{p}}\)
is defined in Notation \ref{notationWidetildeGroup}.


The lemma below gives a realization of \(\mX^{\tor}_{K^{p},\proket}\) as a Tate stack. See Lemma \ref{lemAnalyticVSproket} for the relation.
\begin{notation}\label{notationToricChartTower}
We define \(\mbb{B}^{n}:=\Spa(\CC_{p}\langle T_{1},\ldots,T_{n}\rangle,\mO_{\CC_{p}}\langle T_{1},\ldots,T_{n}\rangle)\), equipped with the log structure induced by the divisor \(\{T_{1}\cdots T_{n}=0\}\) (\cite[Example 2.2.21]{DLLZ2019logarithmicFoundational}).

Consider the finite Kummer \'etale cover
\[\mbb{B}^{n}_{m}:=\Spa\left(\CC_{p}\langle T^{1/m}_{1},\ldots,T^{1/m}_{n}\rangle,\mO_{\CC_{p}}\langle T^{1/m}_{1},\ldots,T^{1/m}_{n}\rangle
\right)\to \mbb{B}^{n},
\] where the LHS is equipped with the log structure induced by the divisor \(\{T_{1}^{1/m}\cdots T_{n}^{1/m}=0\}\). Denote by \( \mbb{B}^{n}_{\infty} \) the following pro-finite Kummer \'etale \(\Gamma_{n}\)-torsor over \( \mbb{B}^{n} \)
\[\Spa\left(\left(\varinjlim_{m\in\ZZ_{\ge1}}\mO_{\CC_{p}}\langle T^{1/m}_{1},\ldots,T^{1/m}_{n}\rangle\right)^{\wedge}_{p}[\frac{1}{p}],\left(\varinjlim_{m\in\ZZ_{\ge1}}\mO_{\CC_{p}}\langle T^{1/m}_{1},\ldots,T^{1/m}_{n}\rangle\right)^{\wedge}_{p}\right),\]
where \(\Gamma_{n}:= \left(\underline{\hat{\ZZ}^{\times}}\right)^{n}\).
\end{notation}
\begin{construction}[Local chart of completed infinite-level Shimura varieties]
	\label{constructionUChartTower}
Let \(U_{K_{p,0}}\subset \mX_{K^{p}K_{p,0}}^{\tor}\) be an open affinoid subspace 
such that there exists a strictly \'etale morphism \(U_{K_{p,0}}\to \mbb{B}^{n}\), which descends to a finite extension \(L/E_{\p}\).
Note that such open subspaces always exist analytically locally by the assumption on \(\mX^{\tor}_{K^{p}K_{p,0}}\) (see \cite[Example 2.3.17 \& 3.1.13]{DLLZ2019logarithmicFoundational}).

For any \(K_{p}\subset K_{p,0}\),
let \(U_{K_{p}}\subset \mX^{\tor}_{\Kpp}\) denote
the preimage of \(U_{K_{p,0}}\),
and let \(U_{K_{p},m}\) be
the normalization of \(U_{K_{p}}\times_{\mbb{B}^{n}}\mbb{B}^{n}_{m}\) (equivalently, the fiber product in the category of fs log adic spaces by rigid Abhyankar's lemma (\cite[Proposition 4.2.1]{DLLZ2019logarithmicFoundational})).

We define 
\(U_{K_{p},\infty}:=\varprojlim_{m}U_{K_{p},m},\)
where the limit is taken in the pro-Kummer \'etale site of \(\mX^{\tor}_{\Kpp}\), which is represented by a perfectoid space by Lemma \ref{lemTorCompacAsTateStack} (2) below. We define a Tate stack via \[\ti{U}:=\left[\ti{U}_{\infty}/\unl{\Gamma_{n}}\right].\]
\end{construction}
\begin{lem}[Completed infinite-level Shimura varieties]\label{lemTorCompacAsTateStack}
We continue the setting of Construction \ref{constructionUChartTower}.

(1) For a fixed \(K_{p}\), \(U_{K_{p},m}\to \mbb{B}^{n}_{m}\) is finite Kummer \'etale, and when \(m\) is sufficiently divisible, it is finite \'etale. Taking limit,
\(U_{K_{p},\infty}\to \mbb{B}^{n}_{\infty}\)
is finite \'etale. In particular,
\(U_{K_{p},\infty}\)
is represented by a perfectoid space. 

(2)
For any \(K_{p}'\subset K_{p}\), \(U_{K_{p}',\infty}\to U_{K_{p},\infty}\)
is finite \'etale.
We define \(\ti{U}_{\infty}:=\varprojlim_{K_{p}}U_{K_{p},\infty}\),
then \(\ti{U}_{\infty}\)
is also perfectoid.

(3) \(\ti{U}\to U_{K_{p,0}}\)
is independent of the choice of the toric chart \(U_{K_{p,0}}\to \mbb{B}^{n}\). In particular, for any \(U'\subset U\subset \mX^{\tor}_{K^{p}K_{p,0}}\), we have an natural open immersion
\(\ti{U}'\subset \ti{U}\). 

Therefore, by gluing along open immersions, we obtain a Tate 
stack, which we still denote as \(\mX^{\tor}_{K^{p}}\), and refer to as the \emph{completed infinite-level Shimura variety} of tame level \(K^{p}\).

(4)
There are natural maps \(\pi_{K_{p}}:\mX^{\tor}_{K^{p}}\to \mX^{\tor}_{\Kpp}\) as morphisms between Tate stacks, which
are essentially affine proper and descendable of index \(\le 8\).

(5) We have the Hodge-Tate period map \(\pi_{\HT,K^{p}}:\mX^{\tor}_{K^{p}}\to \Fl_{G,\mu}\)
as a morphism between Tate stacks. We will write 
\(\pi_{\HT}:=\pi_{\HT,K^{p}}\) when it causes no confusion.

(6) 
Let \(\mX_{K^{p}}\subset \mX^{\tor}_{K^{p}}\)
be the preimage of \(\mX_{\Kpp}\subset \mX_{\Kpp}^{\tor}\). If the diamond \(\varprojlim_{K_{p}}\mX_{\Kpp}\) is representable by a perfectoid space \(\mX_{K^{p}}^{\diamond}\) over \(\CC_{p}\), then we have \(\mX_{K^{p}}\cong \mX_{K^{p}}^{\diamond}\).
\end{lem}
\begin{rmk}
In the Hodge type cases, \(\mX^{\tor}_{K^{p}}\)
is representable by a perfectoid space by \cite{PilloniStroh2016cohomologie} and \cite{Lan2022closed} for good choices of cone decompositions. Then the associated Tate stack \(\mX^{\tor}_{K^{p}}\) that we define here is expected to be the Tate stack associated to that perfectoid space, if the tower of toroidal compactifications is arbitrarily ramified at boundaries; for example, this is known in the PEL type case.
\end{rmk}
\begin{proof}
	By rigid Abhyankar's lemma (\cite[Proposition 4.2.1]{DLLZ2019logarithmicFoundational}), \(U_{K_{p},m}\to \mbb{B}^{n}_{m}\) is finite Kummer \'etale.
	By construction, the saturated monoid defining the log structure of \(\mbb{B}_{\infty}^{n}\) is \(m\)-divisible for any \(m\in\ZZ_{\ge 1}\),
	and thus any Kummer \'etale morphism over \(\mbb{B}_{\infty}^{n}\)
	is automatically \'etale. Since the tower \(\{\mbb{B}_{m}^{n}\}\)	
	is arbitrarily ramified, the map becomes finite \'etale for \(m\) sufficiently large.
	Then for sufficiently divisible \(m|m'\), \(U_{K_{p},m'}\cong U_{K_{p},m}\times_{\mbb{B}^{n}_{m}}\mbb{B}^{n}_{m'}\), and thus \(U_{K_{p},\infty}\cong U_{K_{p},m}\times_{\mbb{B}^{n}_{m}}\mbb{B}^{n}_{\infty}\).
	This finishes the proof of (1) and (2).
	
	For (3),
	if we have two such choices \(f_{i}:U_{K_{p,0}}\to \mbb{B}^{n}\)
	for \(i=1,2\),
	and denote by \(\ti{U}^{(1)}_{K_{p},\infty}\)
	and \(\ti{U}^{(2)}_{K_{p},\infty}\)
	the corresponding perfectoid spaces. Then
	consider \(\ti{U}^{(12)}_{K_{p},\infty}:=\ti{U}^{(1)}_{K_{p},\infty}\times_{U_{K_{p}}}\ti{U}^{(2)}_{K_{p},\infty}\),
	where the fiber product is taken in the category of diamonds. Then \(\ti{U}^{(12)}_{K_{p},\infty}\to \ti{U}^{(i)}_{K_{p},\infty}\)
	is pro-finite-Kummer-\'etale, for \(i=1,2\).
	Again since the relevant monoids are divisible,
	we know that these morphisms are pro-finite-\'etale,
	with Galois groups \(\Gamma_{n}\).
	Taking limits along \(K_{p}\),
	and writing \(\ti{U}^{(12)}_{\infty}:=\varprojlim_{K_{p}}\ti{U}^{(12)}_{K_{p},\infty}\),
	we have a diagram \[\ti{U}^{(2)}_{\infty}\leftarrow\ti{U}^{(12)}_{\infty}\to \ti{U}^{(1)}_{\infty},\]
	where both morphisms are pro\'etale \(\Gamma_{n}\)-torsors.

	By Theorem \ref{thmAnDRStackGeneral} (4),
	we know that \[[\ti{U}^{(2)}_{\infty}/\Gamma_{n}]\cong [\ti{U}^{(12)}_{\infty}/\Gamma_{n}\times \Gamma_{n}]\cong [\ti{U}^{(1)}_{\infty}/\Gamma_{n}],\]
	which shows that the construction is independent of the choice of the charts.


For (4),
we again work locally on \(U_{K_{p}}\).
It suffices to show the same statement for \(\ti{U}_{\infty}\to U_{K_{p}}\).
Note that \(U_{K_{p},\infty}\to U_{K_{p}}\)
is pulled-back from \(U_{\infty}\to U\),
which is affine proper and admits a splitting,
and thus descendable of index \(\le 1\). 
Now for any \(K_{p}'\subset K_{p}\),
\(U_{K_{p}',\infty}\to U_{K_{p},\infty}\) is finite \'etale,
and thus \(\ti{U}_{\infty}\to U_{K_{p}}\)
is affine proper and descendable of index \(\le 4\) by Theorem \ref{thmAnDRStackGeneral} (4). 
Therefore, \(\ti{U}_{\infty}\to U_{K_{p}}\)
is affine proper and descendable of index \(\le 8\).

For (5), by Riemann-Hilbert correspondence in \cite{DLLZ2022logarithmicJAMS}, for any algebraic representation \(V\) of \(G^{c}\),
we have a pro-Kummer-\'etale local system \(\mcal{M}_{B,\Kpp}(V)\) over \(\mX_{\Kpp}^{\tor}\),
with associated vector bundles \((\mcal{B}_{\dR,\Kpp}(V),\nabla)\) with log connection. We have the Hodge-Tate filtration on \(\mcal{M}_{B,\Kpp}(V)|_{\ti{U}_{\infty}}\otimes \mO_{\ti{U}}\)
for \(\ti{U}\)
as in (3),
which is \(\Gamma_{n}\)-equivariant.
Since \(\mcal{M}_{B,\Kpp}(V)|_{\ti{U}_{\infty}}\)
is canonically trivialized by the definition of \(\mX_{K^{p}}^{\tor}\),
the Hodge-Tate filtration
defines a \(\Gamma_{n}\)-equivariant morphism \(\ti{U}_{\infty}\to \Fl_{G,\mu}\),
and thus \([\ti{U}_{\infty}/\Gamma_{n}]\to \Fl_{G,\mu}\).
This glues together to give the desired Hodge-Tate period map 
\(\pi_{\HT,K^{p}}:\mX^{\tor}_{K^{p}}\to \Fl_{G,\mu}\).

For (6), by construction, \(\ti{U}_{\infty}\to \mX^{\diamond}_{K^{p}}|_{U_{K_{p}}}\) 
is pro-finite \'etale \(\Gamma_{n}\)-torsor between perfectoid spaces. Then by Theorem \ref{thmAnDRStackGeneral} (4), we know that \([\ti{U}_{\infty}/\unl{\Gamma_{n}}]\cong \mX^{\diamond}_{K^{p}}|_{U_{K_{p}}}\).
\end{proof}
The construction in Lemma \ref{lemTorCompacAsTateStack} is designed to ensure the following:
\begin{lem}\label{lemAnalyticVSproket}
Let \(U_{K_{p}}\subset \mX^{\tor}_{\Kpp}\)
be an open subspace, and let \(\ti{U}:=\pi_{K_{p}}^{-1}(U_{K_{p}})\) which is an open subspace of the Tate stack \(\mX^{\tor}_{K^{p}}\) defined in Lemma \ref{lemTorCompacAsTateStack} (3). Then we have a natural isomorphism \[
R\Gamma(\ti{U},\mO_{\mX^{\tor}_{K^{p}}})\cong R\Gamma(\ti{U}_{\proket},\hat{\mO}_{\mX^{\tor}_{\Kpp,\proket}}),
\] where the LHS is the cohomology of Tate stacks, and the RHS is the pro-Kummer \'etale cohomology of the open subspace \(\ti{U}_{\proket}\subset \mX^{\tor}_{K^{p},\proket}\).
\end{lem}
\begin{proof}
We start with the case where \(U_{K_{p}}\) is affinoid, and
admits a log \'etale morphism to \(\mbb{B}^{n}\)
as in Lemma \ref{lemTorCompacAsTateStack}. Then \(\ti{U}\cong [\ti{U}_{\infty}/\Gamma_{n}]\),
and thus by \(*\)-descent, \[
R\Gamma(\ti{U},\mO_{\mX^{\tor}_{K^{p}}})
\cong R\Gamma(\Gamma_{n},H^{0}(\ti{U}_{\infty},\mO_{\ti{U}})).
\]
On the other hand, \(\ti{U}_{\infty,\proket}\to \ti{U}_{\proket}\)
is a pro-Kummer \'etale \(\Gamma_{n}\)-torsor, and by \cite[Theorem 5.4.3]{DLLZ2019logarithmicFoundational}, 
\begin{align*}
R\Gamma(\ti{U}_{\proket},\hat{\mO}_{\mX^{\tor}_{\Kpp,\proket}})
\cong R\Gamma(\Gamma_{n},R\Gamma(\ti{U}_{\infty},\hat{\mO}_{\mX^{\tor}_{\Kpp,\proket}}))\\
\cong R\Gamma(\Gamma_{n},H^{0}(\ti{U}_{\infty},\mO_{\ti{U}})).
\end{align*}
This induces an isomorphism \[
R\Gamma(\ti{U},\mO_{\mX^{\tor}_{K^{p}}})\cong
R\Gamma(\ti{U}_{\proket},\hat{\mO}_{\mX^{\tor}_{\Kpp,\proket}}).
\]
By the same argument as in the proof of Lemma \ref{lemTorCompacAsTateStack} (3), 
we can see that this isomorphism is independent of the choice of the toric chart, in particular, it is compatible with restriction along open inclusions \(U_{K_{p}}'\subset U_{K_{p}}\). Thus we can treat the general cases using the \v Cech resolution.
\end{proof}

\begin{dfn}[Smooth infinite-level Shimura varieties]\label{dfnNonCompletedInfiniteLevel}
We define the \emph{smooth infinite-level Shimura varieties} \(\mX^{\tor,\sm}_{K^{p}}\) of tame level \(K^{p}\) as
 \[\mX^{\tor,\sm}_{K^{p}}:=\varprojlim_{K_{p}\subset G(\QQ_{p})}\mX^{\tor}_{\Kpp}\supset \mX^{\sm}_{K^{p}}:=\varprojlim_{K_{p}\subset G(\QQ_{p})}\mX_{\Kpp},\]
where the limits are taken in the category of solid stacks over \( \CC_{p} \). 

For any \( E_{\p}\subset L\subset \CC_{p} \), we denote by \( \mX^{\tor,\sm}_{K^{p},L} \) and \( \mX^{\sm}_{K^{p},L} \)  
the descent of \( \mX^{\tor,\sm}_{K^{p}} \) and \( \mX^{\sm}_{K^{p}} \) 
from \( \CC_{p} \) 
to \( L \). 


We will write \(\mO^{\sm}_{\mX^{\tor}_{K^{p}}}\)
for the structure sheaf of \(\mX^{\tor,\sm}_{K^{p}}\).
We write \(\mscr{I}^{\sm}_{\mX^{\tor}_{K^{p}}}\)
for the pull-back of \(\mscr{I}_{\mX^{\tor}_{\Kpp}}:=\mO_{\mX^{\tor}_{\Kpp}}(-D)\),
the ideal sheaf that defines the boundary.
We will also write \(\mO^{\sm}_{K^{p}}\)
and \(\mscr{I}^{\sm}_{K^{p}}\) if there is no chance of confusion.
We will write \( \mO^{\sm}_{K^{p},L} \) 
and \( \mscr{I}^{\sm}_{K^{p},L} \) 
for their descent to \( L \).  
\end{dfn}
\begin{dfn}[Chartable affinoid open subspace]\label{dfnChartableAffOp}
Let \( L\subset \bar{\QQ}_{p} \) be any finite extension of \( E_{\p} \) such that the irreducible components of the boundaries of \( (\mX^{\tor,rat}_{K^{p}K_{p,0}})_{L} \)  are geometrically irreducible.

Let \(K_{p}\) be an open compact subgroup of \(G(\QQ_{p})\), and \(U_{K_{p}}\) be an
open affinoid subspace of \(\mX^{\tor}_{\Kpp}\) that descends to \( L \), and that admits a toric chart (i.e. a strictly \'etale morphism \(U_{K_{p}}\to \mbb{B}^{n}\) that is a composition of rational localizations and finite étale morphisms).

For any \(K'_{p}\subset K_{p}\), let 
\(U_{K'_{p}}\cong \Spa(A_{K_{p}'},A_{K_{p}'}^{+})\subset \mX^{\tor}_{K^{p}K_{p}'}\) be the preimage of \(U_{K_{p}}\) in \(\mX^{\tor}_{K^{p}K_{p}'}\),
and let \(U^{\sm}:=\varprojlim_{K_{p}'}U_{K'_{p}}=\AnSpec((A^{\sm},A^{\sm,+})_{\square})\), with \(A^{\sm,+}:=\varinjlim_{K_{p}'}A_{K_{p}'}^{+}\) and \(A^{\sm}:=A^{\sm,+}[1/p]\).
Then \(U^{\sm}\) is the preimage of \(U_{K_{p}}\) in \(\mX^{\tor,\sm}\). 

We will refer to such \(U^{\sm}\) as a \emph{chartable affinoid open subspace} of \(\mX^{\tor,\sm}_{K^{p}}\).
In this case, we denote the descent of \( U_{K_{p}'} \) and \( U^{\sm} \)   to \( L \) 
as \( U_{K_{p}',L}\cong \Spa(A_{K_{p}',L},A^{+}_{K_{p}',L}) \) 
and \( U^{\sm}_{L}\cong \Spa(A^{\sm}_{L},A^{+,\sm}_{L}) \). 

A chartable affinoid open subspace \(U^{\sm}\) is \emph{good} 
if, in addition, there exists an open affinoid subspace \(V\subset\Fl_{G,\mu}\) such that \(\fn^{0}|_{V}\) and \(\fg^{c,0}/\fn^{0}|_{V}\) are finite free, and \((\pi^{\sm})^{-1}(U^{\sm})\subset \pi_{\HT}^{-1}(V)\)
for \[
\mX^{\tor,\sm}_{K^{p}}\xleftarrow{\pi^{\sm}}
\mX^{\tor}_{K^{p}}\xrightarrow{\pi_{\HT}}
\Fl_{G,\mu}.
\]
\end{dfn}
\begin{lem}\label{lemOpenSetsFromFulltoSM}
(1) The set of chartable affinoid open subspace is closed under intersections. Moreover, we can (and we will) fix a suitable cone decomposition \(\Sigma\) and \(K_{p,0}\) (in Notation \ref{notationShimuraSetUp}) such that the set of good chartable affinoid open subspace forms a cover of \(\mX_{K^{p}}^{\tor,\sm}\). 

(2) The solid stack \( \mX^{\tor,\sm}_{K^{p}} \) 
is essentially Tate (Definition \ref{dfnEssentialTate}). 

Thus by abuse of notation, we will write \( \mX^{\tor,\sm}_{K^{p}}:=(\mX^{\tor,\sm}_{K^{p}})^{Tate} \).

(3)
The natural morphism \[\pi^{\sm}:\mX^{\tor}_{K^{p}}\to \mX^{\tor,\sm}_{K^{p}}\]
is essentially affine proper and descendable of index \(\le 16\).

(4) The solid stack \( \mX^{\tor,\sm}_{K^{p}} \)
is cohomologically co-smooth \( \mX^{\tor}_{\Kpp} \) and over \( \CC_{p} \).
\end{lem}
\begin{proof} 
For (1),
the set is closed under intersection,
since \(\mX^{\tor}_{\Kpp}\)
is separated. 
We can choose a finite open affinoid cover \(  \{V_{i}\} \)  of \(\Fl_{G,\mu}\)
such that \(\fn^{0}|_{V_{i}}\) and \(\fg^{c,0}/\fn^{0}|_{V_{i}}\) are finite free for any \(i\). Then we can choose a sufficiently small open subgroup \(K_{p,0}\subset G(\QQ_{p})\), such that \(V_{i}\) is \(K_{p,0}\)-stable for any \(i\). Then \(\{\pi_{\HT}^{-1}(V_{i})\}\)
descends to an open cover \(  \{U_{K_{p,0},i}\} \)  of \(\mX^{\tor}_{K^{p}K_{p,0}}\)
(for any choice of \(\Sigma\)). Finally, we fix \(\Sigma\) such that \(\mX^{\tor}_{K^{p}K_{p,0}}\) is smooth with the log structure
induced by a normal crossings divisor (\cite[Example 2.3.17]{DLLZ2019logarithmicFoundational}). 
Then for each \(i\), \(U_{K_{p,0},i}\) analytically locally admits a toric chart by \cite[Theorem 1.18]{Kiehl1967rham}, and thus
\( \pi_{\HT}^{-1}(V_{i}) \) can be covered by chartable affinoid open subspaces.  

(2) follows from descent from the open covering given by \( U^{\sm} \) using Lemma \ref{lemBasicPropertyEsseTate} (6).
Note that \( U^{\sm} \) 
is essentially Tate Lemma \ref{lemBasicPropertyEsseTate} (3), since \( A^{\sm} \) 
is bounded, being colimits of bounded rings by \cite[Proposition 2.6.14 (1)]{Juan2024analyticdeRham}. 

For (3),
the natural map is induced by \(\pi_{K_{p}}:\mX^{\tor}_{K^{p}}\to \mX^{\tor}_{\Kpp}\) in Lemma \ref{lemTorCompacAsTateStack} (4), 
which
is essentially affine proper and descendable of index \(\le 8\). More precisely,
locally there is a cover \(\ti{U}_{\infty}\to \ti{U}\subset\mX^{\tor}_{K^{p}}\),
such that \(\ti{U}_{\infty}\to U_{K_{p}}\)
is affine proper and descendable of index \(\le 8\),
and thus \(\pi^{\sm}\)
is essentially affine proper and descendable of index \(\le 16\) by \cite[Lemma 11.22]{BhattScholze2017projectivity}.

The part (4) follows from Proposition \ref{propAffineProperCover} and Lemma \ref{lemFiniteLevelProper} (1).
\end{proof}

\subsection{Locally analytic Shimura varieties}
\label{subsecDeriveLocallyAnalyticVectors}
\subsubsection{Derived locally analytic vectors}
\begin{notation}\label{notationStalkAt1DaggerGrp}
Let \(L/\QQ_{p}\) be \(p\)-adic field, and \( G \) be a \( p \)-adic Lie group.  
Recall \( \mO_{G,1} \) 
defined in Definition \ref{dfnIncarnationGroup} (4).
The space $\mO_{G,1}$ carries two actions of $\fg$, induced respectively by the left multiplication and the right multiplication of \(G\). We denote them as $*_{1}$ and $*_{2}$ respectively.

Given any solid $\fg$-representation $(V,*_{3})$ over \(L\), we have various actions of \(\fg\) on \(V\hatotimes \mO_{G,1}\). 
For any \(I\subset \{1,2,3\}\), we denote by \(*_{I}\) the diagonal action obtained by combining \(*_{i}\) for \(i\in I\). For example,
\(*_{1,3}\) is the diagonal action induced by \(*_{1}\) and \(*_{3}\). 
\end{notation}
\begin{dfn}[derived smooth / locally analytic vector, {\cite[Definition 5.1.4 \& Definition 3.2.3]{JacintoJoaquínJuan2023SolidLARepII}}]\label{dfnLAvectors}
Let \( G \) 
be a \( p \)-adic Lie group. 
For any solid $G$-representation $V$ over \(L\), we define the \emph{derived smooth vector functor} as  $V^{R-G-\sm}:=\varinjlim_{K\subset G}R\Gamma(K,V)$, and we define
the \emph{derived locally analytic vector functor} as \[
V^{R-G-\la}:=(V\hatotimes\mO_{G,1})^{R-(G,*_{1,3})-\sm}:=\varinjlim_{K\subset G}(V\hatotimes C^{\an}(K,L))^{R-(G,*_{1,3})-\sm}.
\] Here the colimits are taken over open compact subgroups \( K \) 
of \( G \). 

Then \( V^{R-G-\sm} \) (resp. \( V^{R-G-\la} \)) is naturally a representation of \( G^{\sm} \)  (resp. \( G^{\la} \)) by \cite[Proposition 5.4.2 \& Theorem 4.3.3]{JacintoJoaquínJuan2023SolidLARepII}.  
\end{dfn}

We will also discuss the descent datum from \( \CC_{p} \) to the cyclotomic extension \( E_{\p,\infty} \)  of \( E_{\p} \).  
\begin{notation}\label{notationEpInftyHandGamma}
Let \( L\subset \bar{\QQ}_{p} \) 
be a finite extension of \( \QQ_{p} \). 
Let \(L_{\infty}\subset \bar{\QQ}_{p}\) be the algebraic extension of \( L \) 
obtained by adding all the \(p^{n}\)-th roots of unity \( \zeta_{p^{n}} \)  for all \( n\in\NN \). 
We denote by \( \hat{L}_{\infty}\subset \CC_{p} \) the \( p \)-adic completion of \( L_{\infty} \). 
Let \( H_{L}:=\Gal(\bar{\QQ}_{p}/L_{\infty}) \), 
and \( \Gamma_{L}:=\Gal(L_{\infty}/L) \). 
\end{notation}
\begin{dfn}\label{dfnDsenFunctor}
For any \( V \in \Rep_{\QQ_{p,\square}}(\unl{\Gal_{L}}) \), we define \[
D_{\sen,L}(V):=(R\Gamma(H_{L},V))^{R-\Gamma_{L}-\la}\in \Rep_{\QQ_{p,\square}}(\Gamma_{L}^{\la}),
\]
and \[ D_{\sen}(V):=\varinjlim_{L'\subset\bar{\QQ}_{p}}D_{\sen,L'}(V), \] 
where the colimit is taken over all the finite extensions \( L'/L \). 

Note that \(\Gamma_{\QQ_{p}}\twoheadrightarrow \Gamma_{L}\) with finite kernel. 
We fix a generator \( \Theta\in \Lie(\Gamma_{\QQ_{p}})\cong\QQ_{p} \).
By abuse of notation, we denote by \( \Theta \) the image of \( \Theta \)  in \( \Lie(\Gamma_{L}) \).
We have the \emph{arithmetic Sen operator} \( \Theta\in\End_{L_{\infty}}(D_{\sen,L}(V)) \) induced by the action of \( \Theta\in \Lie(\Gamma_{L}) \).
We normalize the choice of \( \Theta \) such that \( \Theta \) 
acts on \( D_{\sen,L}(\chi_{\mrm{cycl}}\otimes\CC_{p}) \) by \( 1 \). 
\end{dfn}
\begin{dfn}\label{dfnAdmitsSen}
For any \( V \in \Mod_{\CC_{p}}(\Rep_{\QQ_{p,\square}}(\unl{\Gal_{L}})) \) (i.e. solid semilinear representation over \( \CC_{p} \)), then \( D_{\sen,L}(V)\in \Mod_{L_{\infty}}(\Rep_{\QQ_{p,\square}}(\Gamma_{L}^{\la})) \),  
and there is a natural morphism \( D_{\sen,L}(V)\hatotimes_{L_{\infty}}\CC_{p}\to V \).
We say that \(V\) \emph{admits an arithmetic Sen operator} 
if this natural morphism is an isomorphism,
in which case we define \emph{the arithmetic Sen operator} \( \Theta\in\End_{\CC_{p}}(V) \)
to be the \( \CC_{p} \)-linear extension of the action of \( \Theta \)  on \( D_{\sen,L}(V) \). Note that by Galois descent, the property of ``admitting an arithmetic Sen operator''
is insensitive to restricting to open subgroups of \( \Gal_{L} \). 
\end{dfn}
\begin{dfn}\label{dfnFontaineOperator}
For any \( V \in \Mod_{B_{\dR}}(\widehat{\Fil}^{\ZZ}(\Rep_{\QQ_{p,\square}}(\unl{\Gal_{L}}))) \) (i.e. filtered complete solid semilinear representation over \( B_{\dR} \)), then 
\[ D_{\sen,L}(V)\in \Mod_{L_{\infty}((t))}(\widehat{\Fil}^{\ZZ}(\Rep_{\QQ_{p,\square}}(\Gamma_{L}^{\la}))), \]
and there is a natural morphism \( D_{\sen,L}(V)\hatotimes_{L_{\infty}((t))}B_{\dR}\to V \).
We say that \(V\) \emph{admits a Fontaine operator} 
if this natural morphism is an isomorphism,
in which case we define \emph{the Fontaine operator} \( \Theta\in\End_{\CC_{p}}(V) \).
\end{dfn}
The following functor will be needed for the future study of Fontaine operators
as well as for the formulation of \cref{thmReformulateRiemmanHilbPerfdCaseIntro}.
\begin{dfn}\label{dfnE0functor}
In the algebraic setting, we define a functor \( E_{0}:D(\GG_{a})\to D(\hat{\GG}_{a}) \) as follows for \( \GG_{a}=\Spec(\QQ[Y]) \) and \( \hat{\GG}_{a}:=\varinjlim \Spec(\QQ[Y]/Y^{n}) \): by Cartier duality (\cite[Proposition 2.2.13 \& 2.4.4]{BhattLurie2022prismatic}), we have equivalences of symmetric monoidal categories
\( \alpha:\QCoh(\GG_{a})\cong D(\QQ[Y])\cong \QCoh(\Spec(\QQ)/\hat{\GG}_{a}) \), and  \(\beta:\QCoh(\hat{\GG}_{a})\cong D_{Y^{\infty}-\mrm{torsion}}(\QQ[Y])\cong \QCoh(\Spec(\QQ)/\GG_{a}) \), with the RHS equipped with the convolution symmetric monoidal structure. 

Let \( f:\Spec(\QQ)/\hat{\GG}_{a}\to \Spec(\QQ)/\GG_{a} \), and we define the functor taking \emph{the part where \( Y \) is nilpotent} as
\[ E_{0}:D(\GG_{a})\to D_{Y}(\hat{\GG}_{a}),\;E_{0}:=\beta^{-1}\circ f_{*}\circ \alpha. \] 
Then
\( E_{0} \) is lax symmetric monoidal for the convolution symmetric monoidal structures.

In the setting of solid stacks, by Cartier duality (\cite[Proposition 4.2.5]{Juan2024analyticdeRham}), \( \QCoh(B\hat{\GG}_{a})\cong D(\QQ_{p,\square}[Y])\cong D(\QQ_{p,\square})\otimes_{D(\QQ)}D(\QQ[Y]) \), and \( E_{0} \) defines a lax symmetric monoidal functor of \[E_{0}: \QCoh(B\hat{\GG}_{a})\to \QCoh(\hat{\GG}_{a}), \]
for the convolution symmetric monoidal structure on the RHS.
\end{dfn}
\begin{lem}
The functor 
\( E_{0}:D(\QQ[Y])\to D_{Y^{\infty}-\mrm{torsion}}(\QQ[Y]) \) 
is isomorphic to \( M\mapsto \varinjlim_{n}M/^{\mbb{L}}Y^{n}[-1] \), with transition maps given by multiplication by \( Y \).

In particular, if \( M\in D(\QQ[Y]) \) 
lies in the essential image of \( D(\QQ[Y]/(Y-i)) \) 
for \( i\in\QQ \), then \( E_{0}(M)\cong M \)  
if \( i=0 \) 
and \( E_{0}(M)\cong 0 \) 
if \( i\ne 0 \). 
\end{lem}
\begin{proof}
This follows from adjuction, and the identification of \( f^{*} \) 
with the natural inclusion \( D_{Y^{\infty}-\mrm{torsion}}(\QQ[Y])\to D(\QQ[Y]) \). See \cite[Proposition 2.42]{Jiang2023thetaModCur} for details. 
\end{proof}
\begin{lem}\label{lemDsenCommuteWithFlatBC}
Let \( M\in D(L_{\infty,\square})^{\heartsuit} \) be a Banach space.
 We regard \( M\otimes_{L_{\infty}}\CC_{p}\in \Mod_{\CC_{p}}(\Rep(\unl{\Gal_{L}})) \) by letting \( \Gal_{L} \) 
act on \( \CC_{p} \). Then \( D_{\sen,L}(M\otimes_{L_{\infty}}\CC_{p})\cong M \).  
\end{lem}
\begin{proof}
We first show that \( R\Gamma(H_{L},M\otimes_{L_{\infty}}\CC_{p})\cong M\otimes_{L_{\infty}}\hat{L}_{\infty} \). The case where \( M\cong L_{\infty} \) is the classical theorem of Ax-Sen-Tate. The general case follows from the observation that the calculation of the continuous cochain complex commutes with \( -\otimes_{L_{\infty}}M \): 
\[
\Hom(\unl{H_{L}^{i}},M\otimes \CC_{p})
\cong \Hom(\unl{H_{L}^{i}},L_{\infty})\otimes_{L_{\infty,\square}}(M\otimes_{L_{\infty,\square}} \CC_{p})
\cong
\Hom(\unl{H_{L}^{i}},\CC_{p})\otimes_{L_{\infty},\square}M
\] by \cite[Lemma 3.13]{JRC2021solid} and Lemma \ref{lemHomStoBanach} below.

We then take \( (-)^{R-\Gamma_{L}-\la} \), which can be calculated by \( -\otimes \mO_{\Gamma_{L},1} \) and the Lazard-Serre's resolution (see \cite[Theorem 5.7]{JRC2021solid}), and thus also commutes with  \( -\otimes_{L_{\infty,\square}}M \).
\end{proof}
\begin{lem}\label{lemHomStoBanach}
Let \( S \) be a profinite set. Then there exists a set \( I \), such that for any Banach space \( M \) over \( \QQ_{p} \),  \[
\Hom(S,M)\cong \widehat{\bigoplus}_{I}M.
\]
\end{lem}
\begin{proof}
By compactness, any map \( S\to M \) factors through a lattice \[ M^{\circ}\cong \widehat{\bigoplus}_{J}(\ZZ_{p})\cong \varprojlim_{n}\bigoplus_{J}(\ZZ/p^{n}) \]
for some set \( J \).
Let \( \ZZ^{\bigoplus I}:=\Hom(S,\ZZ) \), and then \[
\Hom(S,M^{\circ})\cong \varprojlim_{n}\bigoplus_{J}\Hom(S,\ZZ/p^{n}) 
\cong \varprojlim_{n}\bigoplus_{J\times I}(\ZZ/p^{n})
\cong \widehat{\bigoplus}_{I\times J}\ZZ_{p}\cong \widehat{\bigoplus}_{I}M^{\circ}.
\] Inverting \( p \), we obtain the desired isomorphism.
\end{proof}
\begin{eg}\label{egApplyToLinfty}
We have 
\begin{align*}
\Mod_{L_{\infty}}(\widehat{\Fil}^{\ZZ}(\Rep_{\QQ_{p,\square}}(\Gamma^{\la}_{L})))\cong 
\widehat{\Fil}^{\ZZ}(\Mod_{L_{\infty}}(\Rep_{\QQ_{p,\square}}(\Gamma^{\la}_{L})))\\
\cong \widehat{\Fil}^{\ZZ}\QCoh(\AnSpec(L_{\infty})/\Gamma^{\la}_{L})
\end{align*}
by \cite[Theorem 4.3.3]{JacintoJoaquínJuan2023SolidLARepII}. 
Here \( \Gamma^{\la}_{L} \) 
acts on \( \AnSpec(L_{\infty})=\varprojlim_{n}\AnSpec(L(\zeta_{p^{n}})) \) 
via \( \Gamma^{\la}_{L}\to \Gamma^{\sm}_{L}=\varprojlim_{n}\unl{\Gal(L(\zeta_{p^{\infty}})/L)} \). Thus we have natural morphisms \[h:\AnSpec(L_{\infty})\times B\hat{\GG}_{a}\to (\AnSpec(L_{\infty})\times B(1\subset\Gamma^{\la}_{L})^{\dagger})/\Gamma^{\sm}_{L}\cong \AnSpec(L_{\infty})/\Gamma^{\la}_{L}, \] where the first morphism is induced by \( \hat{\fg}\to \fg^{\dagger} \) in Definition \ref{dfnDaggerGroupObject}.

Unwinding the definition,
for any \( V\in \Mod_{L_{\infty}}(\widehat{\Fil}^{\ZZ}(\Rep_{\QQ_{p,\square}}(\Gamma^{\la}_{L}))) \), \[ h^{*}(V)\in \widehat{\Fil}^{\ZZ}\QCoh(\AnSpec(L_{\infty})\times B\hat{\GG}_{a})\cong \widehat{\Fil}^{\ZZ}D(L_{\infty,\square}[\Theta]) \] corresponds precisely to the action of \( \Theta\in\Lie(\Gamma_{L}) \). 
We will denote \[
E_{0}(V):=E_{0}(h^{*}(V))\in \widehat{\Fil}^{\ZZ}\QCoh(\hat{\GG}_{a})\cong \widehat{\Fil}^{\ZZ}(D_{\Theta^{\infty}-\mrm{torsion}}(L_{\infty,\square}[\Theta])),
\] for \( \hat{\GG}_{a}=\Spf(L_{\infty,\square}[[\Theta]]) \).  
\end{eg}

\subsubsection{Construction}

\begin{notation}\label{notationRlaConvention}
For \( V\in \Rep_{\QQ_{p,\square}}(\unl{G(\QQ_{p})}) \), by abuse of notation,
we write \[V^{R-\la}:=V^{R-G^{c}(\QQ_{p})-\la}:=V^{R-\widetilde{Z_{c}(\ZZ_{p})}-\sm,R-G^{c}(\ZZ_{p})-\la}.\]
Here \(\widetilde{G(\ZZ_{p})}\) is as in Notation \ref{notationWidetildeGroup}.
Note that \(G^{c}(\ZZ_{p})\)
doesn't really act on \(V^{R-\widetilde{Z_{c}(\ZZ_{p})}-\sm}\),
but \(G^{c}(\ZZ_{p})\)-locally analytic vector is still defined.

More precisely, 
we define \[V^{R-\la}:=
	\varinjlim_{K_{p}}R\Gamma(K_{p}\cap\widetilde{Z_{c}(\ZZ_{p})},V)^{R-K_{p}/(K_{p}\cap \widetilde{Z_{c}(\ZZ_{p})})-\la}
	\] where the colimit is taken over all the open compact subgroups \( K_{p}\subset \widetilde{G(\ZZ_{p})} \). 

As a result, \( V^{R-\la} \) 
is equipped with an action of \( G(\QQ_{p})^{\la}/\mfk{z}_{c}^{\dagger} \) (Notation \ref{notationEquShvOnFl}) as \( \mfk{z}_{c}^{\dagger}\cong\varprojlim_{Z_{p}\subset\widetilde{Z_{c}(\ZZ_{p})}}Z_{p}^{\la} \) (Definition \ref{dfnDaggerGroupObject}). 
\end{notation}
\begin{notation}\label{notationOlaIla}
	Define \[\hat{\mO}_{K^{p}}=\hat{\mO}_{\mX^{\tor}_{K^{p}}}:=R\pi_{*}^{\sm}(\mO_{\mX^{\tor}_{K^{p}}})\in\mrm{CAlg}(\QCoh(\mX^{\tor,\sm}_{K^{p}})),\]
	\[\mO_{K^{p}}^{\la}:=\mO_{\mX^{\tor}_{K^{p}}}^{\la}:=(\hat{\mO}_{K^{p}})^{R-\la}\in \mrm{CAlg}(\QCoh(\mX^{\tor,\sm}_{K^{p}})),\]
	and \[\mscr{I}_{K^{p}}^{\la}:=\mscr{I}_{\mX^{\tor}_{K^{p}}}^{\la}:=\mO^{\la}_{K^{p}}\otimes_{\mO^{\sm}_{K^{p}}}\mscr{I}^{\sm}_{K^{p}},\]
	where \(\mscr{I}^{\sm}_{K^{p}}\) is as in Definition \ref{dfnNonCompletedInfiniteLevel}. We will omit the subscript \( K^{p} \) when it causes no confusions.
\end{notation}
\begin{rmk}
	Note that assuming Leopoldt's conjecture, \(\widetilde{Z_{c}(\ZZ_{p})}\) is a finite discrete group (as \(K^{p}\) is neat), and \(\hat{\mO}_{K^{p}}\) is automatically \(Z_{c}(\ZZ_{p})\)-smooth. Therefore, we will always take \((-)^{R-\widetilde{Z_{c}(\ZZ_{p})}-\sm}\) (which is conjecturally a trivial operation), and work exclusively with \(Z_{c}(\ZZ_{p})\)-smooth vectors. 
\end{rmk}

\begin{thm}\label{thmOlaConcentrated0}
	Let \(U^{\sm}=\AnSpec(A^{\sm},A^{\sm,+})\subset \pi_{\HT}^{-1}(V)\) be a good chartable affinoid open subspace that descends to \( L/E_{\p} \)  (Definition \ref{dfnChartableAffOp}).
	Then \(\mO^{\la}_{K^{p}}(U^{\sm})\in D^{[0,0]}(A^{\sm})\). 
	Moreover, \( \mO^{\la}_{K^{p}}(U^{\sm}) \) admits an arithmetic Sen operator, and \[
	D_{\sen,L}(\mO^{\la}_{K^{p}}(U^{\sm}))\in D^{[0,0]}(A^{\sm}_{L}).
	\]
	The same holds for \(\mscr{I}^{\la}_{K^{p}}\).
\end{thm}
\begin{proof}
This is essentially
\cite[Proposition 6.2.8 \& Theorem 6.3.6]{Juan2022.09locallyShi}, except for our convention at Notation \ref{notationRlaConvention}. More precisely, our \( \mO^{\la} \) can be obtained by applying \( R\Gamma(\widetilde{\mfk{z}}_{c},-) \) to \( \mO^{\la}_{Sh} \) in \cite{Juan2022.09locallyShi}, with \( \widetilde{\mfk{z}}_{c}:=\Lie(\widetilde{Z_{c}(\ZZ_{p})}) \). In particular, the existence of arithmetic Sen operators follows readily from \cite[Theorem 6.3.6]{Juan2022.09locallyShi}. 

For the concentration at degree \( 0 \), by \cite[Proposition 3.2.7]{Juan2022.05GeoSen}, it suffices to show that the composition of \[
\theta_{\mX^{\tor}_{K^{p}}}:\hat{\mO}_{K^{p}}\otimes (\Lie(\widetilde{K_{p}}))^{\vee}_{\ket}\to \hat{\mO}_{K^{p}}\otimes_{\mO_{\mX_{\Kpp}^{\tor}}}\Omega^{1}_{\mX^{\tor}_{\Kpp},\log}(-1)
\] and \( (\Lie(K_{p}/K_{p}\cap Z_{c}(\QQ_{p})))^{\vee}\to (\Lie(\widetilde{K_{p}}))^{\vee} \) remains surjective. 
This is known by \cite[Lemma 5.2.3 \& Theorem 5.2.5]{Juan2022.09locallyShi}.
\end{proof}
\begin{dfn}\label{dfnLocallyAnSV}
We define the \emph{locally analytic Shimura variety}
of tame level \(K^{p}\)
as the solid stack \[\mX^{\tor,\la}_{K^{p}}:=\underline{\AnSpec}_{\mX^{\tor,\sm}_{K^{p}}}(\mO^{\la}_{K^{p}}),\]
and for \( L/E_{\p} \) as in Definition \ref{dfnChartableAffOp}, 
we define \begin{align*}
\mX^{\tor,\la}_{K^{p},L_{\infty}}:=\underline{\AnSpec}_{\mX^{\tor,\sm}_{K^{p},L_{\infty}}}(D_{\sen,L}(\mO^{\la}_{K^{p}})),\end{align*}
where the right hand sides are given by the construction of Lemma \ref{lemAffProperProp} (3), which is applicable by Theorem \ref{thmOlaConcentrated0} and Lemma \ref{lemOpenSetsFromFulltoSM}. 

There is an action of the finite group \( H_{E_{\p}}/H_{L}=\Gal(L_{\infty}/E_{\p,\infty}) \) 
on \( \mX_{K^{p},L_{\infty}}^{\tor,\la}\to \mX_{K^{p},L_{\infty}}^{\tor,\sm} \), and we denote \[
\mX^{\tor,\la}_{K^{p},E_{\p,\infty}}:=\mX_{K^{p},L_{\infty}}^{\tor,\la}/(H_{E_{\p}}/H_{L})\to \mX_{K^{p},E_{\p,\infty}}^{\tor,\sm},
\] which is a descent of \( \mX^{\tor,\la}_{K^{p}}\to \mX^{\tor,\sm}_{K^{p}} \) 
to \( E_{\p,\infty} \), and is equipped with an action of \( \Gamma_{E_{\p}}^{\la} \).  

Then there is a natural factorization of \(\pi^{\sm}\)
as \[\mX^{\tor}_{K^{p}}\xrightarrow{\pi^{\la}} \mX^{\tor,\la}_{K^{p}}\xrightarrow{\pi^{\sm}_{\la}} \mX^{\tor,\sm}_{K^{p}}.\]
\end{dfn}
\begin{rmk}\label{rmkLAsvIsEssTate}
If \( A \) is a Banach algebra with a continuous action of a \( p \)-adic Lie group \( G \),  
then \( A^{\la} \) admits a lattice \( \varinjlim_{K\subset G} (A^{\circ}\otimes C^{\an}(K,\ZZ_{p}))^{K} \), and thus it is bounded. In particular, 
 \( \mO^{\la}_{K^{p}}(U^{\sm}) \) is bounded, and thus \( \mX^{\tor,\la}_{K^{p}} \) 
 is essentially Tate by Lemma \ref{lemBasicPropertyEsseTate} (3).
\end{rmk}
\begin{lem}\label{lemHTPeriodFactor}
The Hodge-Tate period map \[\pi_{\HT,K^{p}}:\mX^{\tor}_{K^{p}}\to \Fl_{G,\mu}\]
has a natural factorization \[\pi_{\HT,K^{p}}^{\la}:\mX^{\tor,\la}_{K^{p}}\to \Fl_{G,\mu},\]
which descends to a \( \Gamma^{\la}_{E_{\p}} \)-equivariant map  \begin{align*}
\pi_{\HT,K^{p}}^{\la}:\mX^{\tor,\la}_{K^{p},E_{\p,\infty}}\to \Fl_{G,\mu,E_{\p,\infty}}.
\end{align*}
\end{lem}
\begin{proof}
It suffices to work locally, so we fix \(U^{\sm}=\AnSpec(A^{\sm},A^{\sm,+})\subset \mX^{\tor}_{K^{p}}\)
a good chartable affinoid open subspace that descends to \( L \) (Definition \ref{dfnChartableAffOp}), and let \(U_{K_{p}}=\Spa(A_{K_{p}},A_{K_{p}}^{+})\subset \mX^{\tor}_{\Kpp}\)
and \(V=\Spa(B,B^{+})\subset \Fl_{G,\mu}\) be as in Definition \ref{dfnChartableAffOp}.
Note that by Lemma \ref{lemOpenSetsFromFulltoSM}, such \(V\) forms a cover of \(\Fl_{G,\mu}\), and such \(U^{\sm}\) forms a cover \(\mX^{\tor,\sm}_{K^{p}}\). Let \(\hat{A}:=H^{0}(U^{\sm},\hat{\mO}_{K^{p}})\)
and \(A^{\la}:=H^{0}(U^{\sm},\mO^{\la}_{K^{p}})\) endowed with the induced analytic ring structure from \((A_{K_{p}},A^{+}_{K_{p}})_{\square}\). Then \(\mX^{\tor,\la}_{K^{p}}|_{U^{\sm}}\cong \AnSpec(A^{\la})\).

The Hodge-Tate period map \(\pi_{\HT}\) (Lemma \ref{lemTorCompacAsTateStack} (5)) \(\pi_{\HT}:\mX^{\tor}_{K^{p}}|_{U^{\sm}}\to V\cong  \Spa(B,B^{+})\)
induces a map of solid \(\CC_{p}\)-algebras \(B\to \hat{A}\).
Since the action of \(G(\QQ_{p})\)
on \(B\) is locally analytic, we have a canonical factorization \(B\to A^{\la}\hookrightarrow \hat{A}\), and thus a morphism of analytic rings \((B,\mO_{\CC_{p}})_{\square}\to A^{\la}\), where
\(f:(B,\mO_{\CC_{p}})_{\square}\) is the ring \(B\) equipped with the induced analytic ring structure from \((\CC_{p},\mO_{\CC_{p}})_{\square}\). Since \(\Fl_{G,\mu}\) is proper, \(V=\Spa(B,B^{+})\to \Fl_{G,\mu}\)
extends canonically to \(\Spa(B,\mO_{\CC_{p}})\to \Fl_{G,\mu}\). Composing with \(f\), we obtain the desired morphism \(\pi_{\HT,K^{p}}^{\la}:\mX^{\tor,\la}_{K^{p}}\cong\AnSpec(A^{\la})\to \Fl_{G,\mu}\).

By functoriality of the construction, the map \( \pi^{\la}_{\HT,K^{p}} \) 
is equivariant for the \( \Gal_{E_{\p}} \)-action. Applying the functor \( D_{\sen,L} \) to the factorization \( B\to A^{\la} \) above, we have \( D_{\sen,L}(B)\to D_{\sen,L}(A) \), which by the same argument induces the Hodge-Tate map \( \mX^{\tor,\la}_{K^{p},L_{\infty}}\to \Fl_{G,\mu,L_{\infty}} \) over \( L_{\infty} \). Quotienting out \( H_{E_{\p}}/H_{L_{\infty}} \) gives the desired map.
\end{proof}

\subsubsection{Relation with completed cohomology}
Now we explain the relation with the completed cohomology of Emerton.
\begin{dfn}[\cite{Emerton06}]\label{dfnCompleteCoh}
Define the \emph{completed cohomology} with compact support with tame level $K^{p}$ as  \[R\Gamma_{c}(K^{p},\ZZ/p^{n}):=\varinjlim_{K_{p}} R\Gamma_{\et,c}(X_{\Kpp,\bar{E}},\ZZ/p^{n}),
\] which carries the action of $G(\QQ_{p})\times\TT(K^{p})\times \Gal_{E}$. We also define \[R\Gamma_{c}(K^{p},\ZZ_{p}):=\Rlim_{n}R\Gamma_{c}(K^{p},\ZZ/{p^{n}}),\;R\Gamma_{c}(K^{p},\QQ_{p}):=R\Gamma_{c}(K^{p},\ZZ_{p})[\frac{1}{p}].\]
In general, for any Banach ring $C$ over $\ZZ_{p}$, $R\Gamma_{c}(K^{p},C):=R\Gamma_{c}(K^{p},\ZZ_{p})\otimes^{L}_{\ZZ_{p}}C$. We will write \(\ti{H}^{i}_{c}(K^{p},C):=H^{i}(R\Gamma_{c}(K^{p},C))\).

Define the \emph{completed cohomology} with tame level $K^{p}$ as  \[R\Gamma(K^{p},\ZZ/p^{n}):=\varinjlim_{K_{p}} R\Gamma_{\et}(X_{\Kpp,\bar{E}},\ZZ/p^{n}),
\]
and \(R\Gamma(K^{p},\ZZ_{p})\), \(R\Gamma(K^{p},\QQ_{p})\), \(R\Gamma(K^{p},C)\) and \(\ti{H}^{i}(K^{p},C)\) are defined similarly as in the compact support case.
	
	\end{dfn}
	\begin{thm}[{\cite[Theorem 6.2.6]{Juan2022.09locallyShi}}]\label{thmLAsvCalculateCompleCoh}
	We have \( \TT^{S}\times G(\QQ_{p})^{\la}\times \Gal_{E} \)-equivariant  isomorphisms \[
	R\Gamma(K^{p},\CC_{p})^{R-\la}\cong 
	R\Gamma
	(\mX^{\tor,\sm}_{K^{p}},\mO^{\la}_{\mX^{\tor}_{K^{p}}})\cong R\Gamma(\mX^{\tor,\la}_{K^{p}},
	\mO_{\mX^{\tor,\la}_{K^{p}}}),\]
	and 
	\[R\Gamma_{c}(K^{p},\CC_{p})^{R-\la}\cong 
	R\Gamma
	(\mX^{\tor,\sm}_{K^{p}},\mscr{I}^{\la}_{\mX^{\tor}_{K^{\p}}})\cong R\Gamma(\mX^{\tor,\la}_{K^{p}},\pi^{\sm,*}_{\la}(\mscr{I}^{\sm}_{K^{p}})).\]
	\end{thm}
\begin{proof}
We focus on the completed cohomology without support. The proof works verbatim for the other case.
For the first isomorphism,
\cite[Theorem 6.2.6]{Juan2024analyticdeRham}	
proves that \(R\Gamma(K^{p},\CC_{p})^{R-\la}\cong R\Gamma_{\an}(\mX^{\tor}_{K^{p}},\mO^{\la})\). 
We claim that \(R\Gamma_{\an}(\mX^{\tor}_{K^{p}},\mO^{\la}) \cong R\Gamma(\mX_{K^{p}}^{\tor,\sm},\mO^{\la}_{\mX^{\tor}_{K^{p}}})\).
For this, let \(K_{p}\) be a small enough open subgroup, and let \(U_{i,K_{p}}\) be an open cover of \(\mX^{\tor}_{\Kpp}\) such that \(U_{i,K_{p}}\)
satisfies the condition of \(U_{K_{p}}\) as in Theorem \ref{thmOlaConcentrated0}.
Then \(\{U_{i}^{\sm}\}\)
forms an open cover of \(\mX^{\tor,\sm}_{K^{p}}\). Then \(U_{i}^{\sm}\cong\AnSpec(A_{i}^{\sm},A_{i}^{\sm,+})\), and \[R\Gamma(U_{i}^{\sm},\mO^{\la}_{\mX^{\tor}_{K^{p}}})\cong R\Gamma_{\proket}(U_{i},\mO_{U_{i}})^{R-\la}\cong R\Gamma_{\an}(U_{i},\mO^{\la}_{U_{i}}), \] where \(U_{i}:=\varprojlim_{K_{p}}U_{i,K_{p}}\)
with the limit
taken in the category of diamonds, the first isomorphism is given by Lemma \ref{lemAnalyticVSproket},
and the second isomorphism is given by \cite[Theorem 6.2.8]{Juan2024analyticdeRham}. Therefore, \(R\Gamma(U_{i}^{\sm},\mO^{\la}_{\mX^{\tor}_{K^{p}}})\cong R\Gamma_{\an}(U_{i},\mO_{U_{i}}^{\la})\). Using the \v Cech complex, we have an isomorphism \(R\Gamma_{\an}(\mX^{\tor}_{K^{p}},\mO^{\la})\cong R\Gamma(\mX^{\tor,\sm}_{K^{p}},\mO^{\la}_{\mX^{\tor}_{K^{p}}})\). 
The second isomorphism follows from the construction of \(\mX^{\tor,\la}_{K^{p}}\). 
\end{proof}

\subsection{Grothendieck-Messing-Hodge-Tate period maps}    \label{subsecGMHTperiodMap}
Recall that we have the Hodge-Tate period map (Lemma \ref{lemHTPeriodFactor}) \[
\pi_{\HT,K^{p}}:\mX_{K^{p}}^{\tor,\la}\to \Fl_{G,\mu},
\] and the Grothendieck-Messing period map (Lemma \ref{lemDRtorsorDescends} \& Proposition \ref{propAnalyticGMTheory})
\[
\pi_{\GM,K}:\mX^{\tor}_{K}\to [\Fl^{\std}_{G,\mu}/G^{c,\an}_{\CC_{p}}].
\]
In this subsection, we will show that the two combines into one \emph{Grothendieck-Messing-Hodge-Tate period map} \(\pi_{\GM\HT}\). 
This period map is studied in the setting of local Shimura varieties in \cite[Proof of Theorem 5.1.10]{DospinescuJuan2024jacquet}.

Recall the following result, which is a corollary of the deep theorem in logarithmic Riemann-Hilbert correspondence (\cite[Theorem 4.6.1]{DLLZ2022logarithmicJAMS}):
\begin{thm}[Hodge-Tate comparison]
	\label{thmIdenTwoTorsorsOverLa}
We have \(\unl{\Gal_{E_{p}}}\times \unl{G(\QQ_{p})}\)-equivariant isomorphisms of \(M^{c,\an}_{\mu}\)-torsors over \(\mX^{\tor,\la}_{K^{p}}\) \[i_{\mrm{RH}}:(\pi_{\GM}^{\la})^{*}(\mcal{M}_{\mu}^{\std,c,\an})\cong 
(\pi_{K}^{\la})^{*}(\mcal{M}_{\dR,K}^{c,\an})\cong (\pi_{\HT,K^{p}}^{\la})^{*}(\mcal{M}_{\mu}^{c,\an})(-1),
\] 
where \(\pi_{\GM}^{\la}:=\pi_{\GM,K}\circ\pi^{\la}_{K}\), \((-1)\) refers to twisting the \(\Gal_{E_{\p}}\)-action by \(\mu\circ \chi_{\mrm{cycl}}^{-1}\), \(\mcal{M}_{\dR,K}^{c,\an}\), \(\mcal{M}_{\mu}^{\std,c,\an}\)
and \(\mcal{M}_{\mu}^{c,\an}\) are as in Notation \ref{notationEquShvOnFl} and Notation \ref{notationDRtorsor},  and
the morphisms are 
as in \[
[\Fl^{\std}_{G,\mu}/G^{c}_{\CC_{p}}]\xleftarrow{\pi_{\GM,K}}
\mX_{K}^{\tor}\xleftarrow{\pi_{K}^{\la}}
\mX_{K^{p}}^{\tor,\la}
\xrightarrow{\pi_{\HT,K^{p}}^{\la}}\Fl_{G,\mu}.
\]
\end{thm}
\begin{proof}
The first isomorphism is given by Lemma \ref{lemDRtorsorDescends}. 

For the second isomorphism,
by \cite[Theorem 4.2.1]{Juan2022.09locallyShi}, we have \[
\pi_{K}^{*}(\mcal{M}_{\dR,K}^{c})\cong \pi_{\HT,K^{p}}^{*}(\mcal{M}_{\mu}^{c})(-1),\]
for \[\mX_{K}^{\tor,\diamond}\xleftarrow{\pi_{K}}
\mX_{K^{p}}^{\tor,\diamond}
\xrightarrow{\pi_{\HT,K^{p}}}\Fl_{G,\mu}^{\diamond},\]
as \(v\)-torsors on diamonds, where \(\mX_{K^{p}}^{\tor,\diamond}:=\varprojlim_{K_{p}}\mX_{\Kpp}^{\tor,\diamond}\) in the category of diamonds over \(\Spd(\CC_{p})\). 
In particular, for any \((U_{K_{p}},U_{K_{p},\infty},\ti{U}_{\infty},\Gamma_{n},\ti{U})\)
as in the construction of \(\mX^{\tor}_{K^{p}}\) in Construction \ref{constructionUChartTower} and Lemma \ref{lemTorCompacAsTateStack},
we have a map of diamonds \(\ti{U}^{\diamond}_{\infty}\to \mX^{\tor,\diamond}_{K^{p}}\),
and thus we have \[
(\pi_{K}|_{\ti{U}_{\infty}})^{*}(\mcal{M}_{\dR,K}^{c})\cong (\pi_{\HT,K^{p}}|_{\ti{U}_{\infty}})^{*}(\mcal{M}_{\mu}^{c})(-1)\]
over \(\ti{U}_{\infty}\). Here \(\ti{U}_{\infty}\) is perfectoid (Lemma \ref{lemTorCompacAsTateStack}) we use \cite[Theorem 3.5.8]{KedlayaLiu2016relative}
to identify the category of \(v\)-vector bundles over \(\ti{U}_{\infty}\) with that of analytic vector bundles. Using descent data, we see that the isomorphism is \(\Gamma_{n}\)-equivariant,
and thus descends to an isomorphism \[
(\pi_{K}|_{\ti{U}})^{*}(\mcal{M}_{\dR,K}^{c})\cong (\pi_{\HT,K^{p}}|_{\ti{U}})^{*}(\mcal{M}_{\mu}^{c})(-1),\]
which further glues to a \(G(\QQ_{p})\)-equivariant isomorphism 
\begin{align}\label{alignIdentTorsorsCompleted}
\pi_{K}^{*}(\mcal{M}_{\dR,K}^{c})\cong \pi_{\HT,K^{p}}^{*}(\mcal{M}_{\mu}^{c})(-1)
\end{align}
over \(\mX^{\tor}_{K^{p}}\). 

We want to descend this isomorphism along \(\pi_{\la}:\mX^{\tor}_{K^{p}}\to \mX^{\tor,\la}_{K^{p}}\). Note that we can cover \(\mX^{\tor,\sm}_{K^{p}}\) by good chartable affinoid open subspaces  \(\{U_{i}^{\sm}\}\) (Definition \ref{dfnChartableAffOp}),
such that \((\pi^{\la}_{K})^{*}(\mcal{M}_{\dR,K}^{c})|_{U_{i}^{\la}}\)
and \((\pi^{*}_{\HT,K^{p}})|_{U_{i}^{\la}}\)
are both trivial \(M^{c}_{\mu}\)-torsors, 
where \(U_{i}^{\la}\) denotes the preimage of \(U_{i}^{\sm}\)
in \(\mX^{\tor,\la}_{K^{p}}\). 

Then by verifying after localizing to each \(U_{i}^{\la}\), we see that for any \((\rho,V)\in\Rep(M^{c}_{\mu})\), 
\begin{align}\label{alignIdenLHSTro}
(\pi^{\la}_{K})^{*}(\mcal{M}_{\dR,K}^{c}(V))\cong \pi_{\la,*}\left(\pi_{K}^{*}(\mcal{M}_{\dR,K}^{c}(V))\right)^{R-\la},
\end{align}
and 
\begin{align}\label{alignIdenRHSTro}
(\pi_{\HT,K^{p}}^{\la})^{*}(\mcal{M}_{\mu}^{c}(V))(\rho\circ \mu\circ \chi_{\mrm{cycl}})\cong \pi_{\la,*}\left(\pi_{\HT,K^{p}}^{*}(\mcal{M}_{\mu}^{c}(V))\right)^{R-\la}(\rho\circ \mu\circ \chi_{\mrm{cycl}}).
\end{align}
Comparing the two isomorphisms, the right hand sides are isomorphic by (\ref{alignIdentTorsorsCompleted}), and thus we have an isomorphism of the left hand sides.
Finally, applying \( ((-)^{Tate})_{Solid} \), we obtain the 
desired isomorphism of \( M^{c,\an}_{\mu} \)-torsors. 
\end{proof}

\begin{dfn}\label{dfnGMHTperDomain}
We define the \emph{Grothendieck-Messing-Hodge-Tate period domain} as \[
\Per_{G,\mu}:=\Per(G,X):=M^{c,\an}_{\mu}\backslash
\left((N_{\mu}^{\an}\backslash G^{c,\an}_{\CC_{p}})(-1)\times (N_{\mu}^{\std,\an}\backslash G^{c,\an}_{\CC_{p}})\right),
\] where \((-1)\) is as in Theorem \ref{thmIdenTwoTorsorsOverLa}.

Then
\(\Per_{G,\mu}\) is a smooth rigid variety over \(\CC_{p}\) equipped with a \(\Gal_{E_{\p}}\)-descent datum. Applying \( D_{\sen,E_{\p}} \) (Notation \ref{notationEpInftyHandGamma}), \( \Per_{G,\mu} \) descends to a smooth rigid variety \( \Per_{G,\mu,E_{\p,\infty}} \)  over \( E_{\p,\infty} \), equipped with a \( \Gamma_{E_{\p}}^{\la} \)-action.  

We denote the right action of \(G^{c,\an}_{\CC_{p}}\) on the first (resp. second) factor of \(\Per_{G,\mu}\) as \(*_{\HT}\) (resp. \(*_{\GM}\)).

We have morphisms \[
\begin{tikzcd}
& \Per_{G,\mu}\arrow[rd,"\pi_{\GM}^{per}"]\arrow[ld,"\pi_{\HT}^{per}"] & \\
\Fl_{G,\mu} & & \Fl_{G,\mu}^{\std},
\end{tikzcd}
\] where for \(\mrm{XX}\in\{\GM,\HT\}\),
 \(\pi_{\mrm{XX}}^{per}\)
is \(G^{c,\an}_{\CC_{p}}\)-equivariant for the \(*_{\mrm{XX}}\)-action on \(\Per_{G,\mu}\).

By construction, for any Tate stack \(X\) over \(\CC_{p}\), \(\mrm{Map}(X,\Per_{G,\mu})\)
is the groupoid of tuples \((\mscr{P}_{\mu}^{c},\mscr{P}_{\mu}^{\std,c},i)\)
where \(\mscr{P}_{\mu}^{c}\) (resp. \(\mscr{P}_{\mu}^{\std,c}\))
is a \(P_{\mu}^{c,\an}\)-structure (resp. \(P_{\mu}^{\std,c,\an}\)-structure) on the trivial \(G^{c,\an}_{\CC_{p}}\)-torsor, and \(i:\mscr{P}_{\mu}^{c}\times^{P_{\mu}^{c,\an}}M_{\mu}^{c,\an}\cong \mscr{P}_{\mu}^{\std,c}\times^{P_{\mu}^{\std,c,\an}}M_{\mu}^{c,\an}(-1)\) is an isomorphism of the associated \(M_{\mu}^{c,\an}\)-torsors.

Then \(\pi^{per}_{\GM}:(\mscr{P}_{\mu}^{c},\mscr{P}_{\mu}^{\std,c},i)\mapsto \mscr{P}_{\mu}^{\std,c}\),
and \(\pi^{per}_{\HT}:(\mscr{P}_{\mu}^{c},\mscr{P}_{\mu}^{\std,c},i)\mapsto \mscr{P}_{\mu}^{c}\).
\end{dfn}
\begin{rmk}
Forgetting the \(\Gal_{E_{\p}}\)-descent datum,
we have \[\Per_{G,\mu}\cong (G^{c,\an}_{\CC_{p}}\times G^{c,\an}_{\CC_{p}})/(P_{\mu}^{c,\an}\times_{M_{\mu}^{c,\an}}P_{\mu}^{\std,c,\an}).\]
\end{rmk}
\begin{dfn}\label{dfnGMHTperMap}
We define the \emph{Grothendieck-Messing-Hodge-Tate period map} as \[
\pi_{\GM\HT,K^{p}}^{\la}:\mX^{\tor,\la}_{K^{p}}\to [\Per_{G,\mu}/(G^{c,\an}_{\CC_{p}},*_{\GM})],
\] as the descent of the map \[
\ti{\pi}_{\GM\HT,K^{p}}^{\la}:(\pi^{\la}_{K})^{*}(\mcal{G}_{\dR,K}^{c,\an})\to \Per_{G,\mu}
\] corresponding to the tuple \[\left((f\circ\pi^{\la}_{\HT,K^{p}})^{*}(\mcal{P}_{\mu}^{c,\an}),
	(f\circ \pi^{\la}_{\GM})^{*}(\mcal{P}_{\mu}^{\std,c,\an}),i_{\mrm{RH}}\right),\]
	where \(f:(\pi^{\la}_{K})^{*}(\mcal{G}_{\dR,K}^{c,\an})\to \mX^{\tor,\la}_{K^{p}}\), and \(i_{\mrm{RH}}\)
	is as in Theorem \ref{thmIdenTwoTorsorsOverLa}.

By Theorem \ref{thmIdenTwoTorsorsOverLa}, \( \pi^{\la}_{\GM\HT,K^{p}} \) is \( \unl{\Gal_{E_{\p}}}\times \unl{G(\QQ_{p})} \)-equivariant, where \( \unl{G(\QQ_{p})} \) 
acts on \( \Per_{G,\mu} \) 
by \( *_{\HT} \).
\end{dfn}
\begin{lem}\label{lemCompatibleWithHTGM}
We have the following commutative diagram over \( \CC_{p} \) \[
\begin{tikzcd}
\mX_{K}^{\tor}\arrow[d,"\pi_{\GM,K}"] &\arrow[l,"\pi^{\la}_{K}"] \mX_{K^{p}}^{\tor,\la}\arrow[d,"\pi_{\GM\HT,K^{p}}"] \arrow[r,equal] & \mX_{K^{p}}^{\tor,\la}
\arrow[d,"\pi_{\HT,K^{p}}"]
\\
\Fl_{G,\mu}^{\std}/G^{c,\an}_{\CC_{p}}
& \Per_{G,\mu}/(G^{c,\an}_{\CC_{p}},*_{\GM})\arrow[r,"\pi^{per}_{\HT}"]
\arrow[l,"\pi^{per}_{\GM}"]
& \Fl_{G,\mu}.
\end{tikzcd}
\]
\end{lem}
\begin{proof}
This follows from the construction of Definition \ref{dfnGMHTperDomain} and Definition \ref{dfnGMHTperMap}. 
\end{proof}

The main result of this section is the following result, which should be understood as a version of Grothendieck-Messing theory on the locally analytic Shimura varieties. 
\begin{thm}\label{thmLAstrucGMHTper}
(1)
We have the following \( \unl{\Gal_{E_{\p}}} \)-equivariant commutative diagram of Tate stacks \[
\begin{tikzcd}
\mX^{\tor,\la}_{K^{p}}
\arrow[d,"\pi^{\la}_{\sm}"]
\arrow[r,"\pi^{\la}_{\GM\HT}"] & 
\Per_{G,\mu}/G^{c,\an}_{\CC_{p}} 
\arrow[r,"\pi_{\HT}^{per}"] \arrow[d,"h_{\Per/\Fl^{\std}}"] &[-0.8cm] 
\Fl_{G,\mu}\arrow[d,"\ti{h}_{\Fl_{G,\mu}}"]\\
\mX^{\tor,\sm}_{K^{p}}
\arrow[d,"h^{\sm}"] \arrow[r,"\pi^{\sm}_{\GM\HT}"] & \Per_{G,\mu}^{\dR/\Fl^{\std}_{G,\mu}}/G^{c,\an}_{\CC_{p}} \arrow[r,"\pi_{\HT}^{per,\dR}"]\arrow[d,"\ti{h}_{\Fl^{\std}}"]\arrow[dr,"\pi_{\GM}^{per}"] & \Fl_{G,\mu}/(\fg^{0}/\fn^{0})^{\dagger}\\
(\mX^{\tor}_{K^{p}})_{\log}^{\pdR/\CC_{p},\sm}\arrow[rrd,bend right=10,"\pi_{\GM}^{\pdR}"]
\arrow[r,"\pi^{\pdR}_{\GM\HT}"] & (\Per_{G,\mu}^{\dR/(\Fl^{\std,\pdR/\CC_{p}}_{G,\mu})})/G^{c,\an}_{\CC_{p}}\arrow[dr,"\pi_{\GM}^{per,\pdR}"] & \Fl_{G,\mu}^{\std}/G^{c,\an}_{\CC_{p}}\arrow[d,"h_{\Fl^{\std}}"]\\
& & \Fl_{G,\mu}^{\std,\pdR/\CC_{p}}/G^{c,\an}_{\CC_{p}},
\end{tikzcd}
\] where \((\mX^{\tor}_{K^{p}})_{\log}^{\pdR/\CC_{p},\sm}:=\varprojlim_{K_{p}}(\mX^{\tor}_{\Kpp})_{\log}^{\pdR/\CC_{p}}\), the actions of \( G^{c,\an}_{\CC_{p}} \) on the middle column are induced by \( *_{\GM} \), and all the squares are \emph{Cartesian}. 

(2) Let \(\mX^{\la}_{K^{p}}\subset \mX^{\tor,\la}_{K^{p}}\) be the preimage of \(\mX^{\sm}_{K^{p}}\subset \mX^{\tor,\sm}_{K^{p}}\). Then we have the following \( \unl{\Gal_{E_{\p}}} \)-equivariant  Cartesian diagram of Tate stacks \[
\begin{tikzcd}
\mX^{\la}_{K^{p}}\arrow[r,"\pi^{\la}_{\GM\HT}"]\arrow[d,"\beta^{\la}"]
& \Per_{G,\mu}/(G^{c,\an}_{\CC_{p}},*_{\GM})\arrow[d]
\\
\mX^{\dR,\sm}_{K^{p}}\arrow[r,"\pi^{\dR}_{\GM\HT}"]
& \Per_{G,\mu}^{\dR}/(G^{c,\an}_{\bar{\QQ}_{p}},*_{\GM}),
\end{tikzcd}
\] for \(\mX^{\dR,\sm}_{K^{p}}:=\varprojlim_{K_{p}}\mX^{\dR}_{\Kpp}\).

If there exists a perfectoid space \( \mX_{K^{p}}^{\diamond} \) over \( \CC_{p} \)
such that \( \mX_{K^{p}}^{\diamond}\sim \varprojlim_{K_{p}}\mX_{\Kpp} \), then \( \mX^{\dR,\sm}_{K^{p}}\cong (\mX_{K^{p}}^{\diamond})^{\dR} \). 
\end{thm}
The proof of this theorem will be finished at the end of Subsection \ref{subsecSMdRstack}. The restriction to \(\mX^{\la}_{K^{p}}\) in (2) is again due to the lack of the ``analytic logarithmic de Rham stacks''.

\subsection{Geometric Sen theory}
\label{subsecGeometricSen}
We now recall the main result of the geometric Sen theory from \cite{Pan22}, \cite{Pilloni22}, \cite{Juan2022.09locallyShi}. This allows us to study locally analytic vectors on the flag variety.
 This will also be used later in the proof of \cite[Theorem \ref{2-thmFontainEqDDbar}]{Jiang2025HMF}. 
	

\begin{notation}
	For any open subspace \(U\subset\Fl_{G,\mu}\),
	denote by \(\mX^{\tor,\sm}_{K^{p}}|_{U}\) be
	the open subspace of \(\mX^{\tor,\sm}_{K^{p}}\) such that \(\pi^{\sm,-1}(\mX^{\tor,\sm}_{K^{p}}|_{U})=\pi_{\HT}^{-1}(U)\subset \mX^{\tor}_{K^{p}}\).
\end{notation}
	\begin{notation}
	Let $U$ be any affinoid open of $\Fl$.
	For any open compact subgroup $K_{p}\subset G^{c}(\QQ_{p})$ such that $K_{p}\cdot U=U$, we denote $\QCoh_{K_{p}}(U):=\QCoh([U/\unl{K_{p}}])$.

	Let $\QCoh_{\fg^{c}}(U)$ be the 2-colimit of $\QCoh_{K_{p}}(U)$, taken along restriction maps over all small enough open subgroups $K_{p}\subset G^{c}(\QQ_{p})$. 

	Then we define the functor $\tVBfun$ as  \begin{align*}
		\tVBfun:\mrm{QCoh}_{\fg^{c}}(U)&\to \QCoh(\mX^{\tor,\sm}_{K^{p}}|_{U}),\\\mF
		&\mapsto (\pi^{\sm}_{*}\circ \pi_{\HT,K^{p}}^{*}(\mF))^{R-\widetilde{G(\ZZ_{p})}-\sm},
	\end{align*}
	where the group \( \widetilde{G(\ZZ_{p})} \) is defined as in Notation \ref{notationWidetildeGroup}, and  the morphisms are as in the diagram \[\begin{tikzcd}
		& \mX^{\tor}_{K^{p}}\arrow[rd,"\pi_{\HT,K^{p}}"]
		\arrow[ld,"\pi^{\sm}"] &\\
		\mX^{\tor,\sm}_{K^{p}}
		&&
		\Fl_{G,\mu}.
	\end{tikzcd}\]

	We define the version with support as \[\tVBn(\mF):=\tVBfun(\mF)\otimes_{\mO^{\sm}_{K^{p}}}\mscr{I}^{\sm}_{K^{p}},\]
	where \(\mscr{I}^{\sm}_{K^{p}}\) is as in Definition \ref{dfnNonCompletedInfiniteLevel}.
	\end{notation}
	
	\begin{rmk}[Prime-to-\(p\) Hecke action]
	The functor
	\(\tVBfun\) lands in the module of \(\TT(K^{p})\)-modules in a suitable sense. One way to formulate is to consider the functor \(\VBn:=R\pi_{\HT,*}\circ \tVBfun\), which clearly enhances to a functor \(\VBn:\QCoh_{\fg^{c}}(U)\to 
	\Mod_{\TT(K^{p})* R\pi_{\HT,*}\mO^{\sm}_{K^{p}}}(\QCoh(U)),\)
	where \(\TT(K^{p})* R\pi_{\HT,*}\mO^{\sm}_{K^{p}}\)
	is taking the push-out in the category of associative algebras.

	Alternatively, \(\varinjlim_{K^{p}}\tVBfun\)
	enhances to a functor \[\ti{\VB}^{\text{naïve}}:\QCoh_{\fg^{c}}(U)\to \QCoh\left(\left(\varprojlim_{K^{p}}\mX^{\tor,\sm}_{K^{p}}|_{U}\right)/G(\mA_{f}^{p})\right).\]
	\end{rmk}

	We will use the equivariant vector bundles over \(\Fl_{G,\mu}\) from Notation \ref{notationEquShvOnFl}. 
	Note that for any $\mF\in \QCoh_{\fg^{c}}(U)$, $\fg^{0}$ acts naturally
	on \(\mF\),
	in the sense that there is a natural map \(\fg^{0}\otimes_{\CC_{p}}\mF\to \mF\), which satisfies the usual axioms of Lie algebra representations. In particular, we have a map \(\fn^{0}\otimes_{\CC_{p}}\mF\to \mF\). Note that the latter is actually a morphism in \(\QCoh_{\fg^{c}}(U)\),
	because the action of $\fn^{0}$ commutes with those of $\fg$ and of $\mO_{U}$.
	
	\begin{dfn}
	We define $\QCoh_{\fg^{c}}(U)^{\fn^{0}}$ to be the \(\infty\)-category of pairs $(\mF,i)$ with $\mF\in\QCoh_{\fg^{c}}(U)$ and $i$ a homotopy equivalence between the morphisms $\mF\to\mF\otimes(\fn^{0})^{\vee}$ and $0$. Note that $\mF\to \mF\otimes(\fn^{0})^{\vee}$ is naturally defined in $\QCoh_{\fg^{c}}(U)$.
	\end{dfn}
	\begin{lem}\label{lemVBnToVB} The functor
	$\mF\mapsto \Fib(\mF\to\mF\otimes(\fn^{0})^{\vee})$ 
	defines a natural functor $R\Gamma(\fn^{0},-):\QCoh_{\fg^{c}}(U)\to\QCoh_{\fg^{c}}(U)^{\fn^{0}}$, and there exists a functor \[\tVBfu:\QCoh_{\fg^{c}}(U)^{\fn^{0}}\to \QCoh(\mX^{\tor,\sm}_{K^{p}}|_{U}),\] such that         \[
		\tVBfun\cong \tVBfu\circ R\Gamma(\fn^{0},-) .
		\]
	We similarly have a decomposition \(\tVBn\cong\tVB\circ R\Gamma(\fn^{0},-)\) for a canonical functor $\tVB$.
	
	Explicitly,
	$\tVBfu((\mF,i)):=R\Gamma(\tilde{\fg}^{0}/\fn^{0},\mO^{\la}_{K^{p}}\otimes^{L}_{\mO_{\Fl},\square}\mF)$, where we have a canonical action of $\tilde{\fg}^{0}/\fn^{0}$ on $\mF$ thanks to the null homotopy $i$ of $\mF\to\mF\otimes(\fn^{0})^{\vee}$, which makes $\mF$ a direct summand of $R\Gamma(\fn^{0},\mF)$.
	\end{lem}
	\begin{proof}
	For $\mcal{G}=\Fib(\mF\to\mF\otimes(\fn^{0})^{\vee})$, there is a canonical homotopy equivalence between $\mcal{G}\to \mcal{G}\otimes(\fn^{0})^{\vee}$ and $0$, and thus we obtain a functor     \( R\Gamma(\fn^{0},-):\QCoh_{\fg^{c}}(U)\to\QCoh_{\fg^{c}}(U)^{\fn^{0}}. \)
Now we want to factorize $\VBn$. By Theorem 1.5 and Theorem 1.7 of \cite{JRC2021solid},      \begin{align*}
\tVBfun(\mF)\cong 
(\pi^{\sm}_{*}\circ \pi_{\HT}^{*}\mF)^{R-\sm}
&\cong R\Gamma(\tilde{\fg},(\hat{\mO}_{K^{p}}\otimes^{L}_{\mO_{\Fl},\square}\mF)^{R-\la})\\
&\cong R\Gamma(\tilde{\fg},{\mO}^{\la}_{K^{p}}
\otimes^{L}_{\mO_{\Fl},\square}\mF)
\end{align*}

Now we can extend the action of $\tilde{\fg}$ to $\tilde{\fg}^{0}=\tilde{\fg}\otimes\mO_{\Fl}$, and 
\begin{align*}
R\Gamma(\tilde{\fg},\mO^{\la}_{K^{p}}\otimes^{L}_{\mO_{\Fl},\square}\mF)&\cong R\Gamma(\tilde{\fg}^{0},\mO^{\la}_{K^{p}}\otimes^{L}_{\mO_{\Fl},\square}\mF)\\
&\cong R\Gamma(\tilde{\fg}^{0}/\fn^{0},R\Gamma(\fn^{0},\mO^{\la}_{K^{p}}\otimes^{L}_{\mO_{\Fl},\square}\mF))
\\&\cong R\Gamma(\tilde{\fg}^{0}/\fn^{0},\mO^{\la}_{K^{p}}\otimes^{L}_{\mO_{\Fl},\square}R\Gamma(\fn^{0},\mF))         ,
\end{align*} where in the last step we use the fact that $\fn^{0}$ acts on $\mO^{\la}_{\mX_{K^{p}}}\cong \mO^{\la}_{K^{p}}$ by zero (\cite[Corollary 6.2.12, Proposition 6.2.8]{Juan2022.09locallyShi}). Therefore, we are done by putting $\tVBfu:=R\Gamma(\tilde{\fg}^{0}/\fn^{0},\mO^{\la}_{K^{p}}\otimes^{L}_{\pi_{\HT}^{-1}\mO_{\Fl},\square}\pi_{\HT}^{-1}(-))$,
and \(\tVB(-):=\tVBfun(-)(-D)\).
\end{proof}

\begin{dfn}[{\cite[Definition 2.3.4]{Juan2022.05GeoSen}}]
Let \(U\) be a quasi-compact open subspace of \(\Fl\), and let \(K_{p}\) be an open compact subgroup of \(G^{c}(\ZZ_{p}) \) that acts on \(U\). We say that
$\mF\in\QCoh_{\fg^{c}}(U)$ is a \emph{relative \(K_{p}\)-analytic} modules, if \(\mF\) is associated to a direct summand of an ON Banach \(\mO(U)\)-module \(V\),
	that admits an \(\mO^{+}(U)\)-lattice \(V^{0}\) with a basis \(\{v_{i}\}_{i\in I}\),
	such that there exist \(\epsilon>0\),
		such that \(K_{p}\)
	acts on \(v_{i}\!\!\mod{p^{\epsilon}}\) trivially.

We say that	\(\mF\) is a \emph{relatively locally analytic module} 
	if there exists a filtered colimit \(\mF\cong \varinjlim_{i\in I} \mF_{i}\) in \(\QCoh_{\fg^{c}}(U)\), such that \(\forall i\in I\), there exists an open
		covering \(U=\cup_{j\in J_{i}}U_{i,j}\), such that for \(\forall j\in J_{i}\),  \(\mF_{i}|_{U_{i,j}}\in\mrm{QCoh}_{\fg^{c}}(U_{i,j})\) 
		is relatively \(K_{i,j}\)-analytic for some open compact subgroup \(K_{i,j}\).

\end{dfn}
If \(\mF\) is relatively locally analytic, the action of \(K_{p}\) on \(R\Gamma(U,\mF)\) is locally analytic, and we have an action of \(\fg^{c}\) (and thus of \(\fg^{c,0}\)) on \(\mF\).
If $\fn^{0}$ acts on $\mF$ by zero, then $\mF$ gives rise to an object in $\QCoh_{\fg^{c}}(U)^{\fn^{0}}$. 
\begin{notation}\label{notationRelLA}

We denote as $\QCoh^{\rla}_{\fg^{c}}(U)^{\fn^{0}}$ the subcategory of $\QCoh_{\fg^{c}}(U)^{\fn^{0}}$ generated (under filtered colimits, extensions, and taking idempotents) by relative locally analytic modules that are killed by $\fn^{0}$. 
\end{notation} 
Note that the underlying sheaves of
relative locally analytic modules are colimits of ON Banach modules by definition, and $\QCoh_{\fg^{c}}^{\rla}(U)^{\fn^{0}}$ is not a stable \(\infty\)-category.

The following is a reformulation of the main theorem of geometric Sen theory:
\begin{thm}[{\cite[Theorem 5.2.1, Theorem 5.2.5]{Juan2022.09locallyShi}}]\label{thmMainThmGeomSenVB}
Let \(U\subset\Fl\) be an open affinoid subspace. Let \(K_{p,0}\) be an open compact subgroup of \(G(\QQ_{p})\) such that \(U\) is \(K_{p,0}\)-stable, and let \(U_{K_{p,0}}\) be the open subspace of \(\mX^{\tor}_{K^{p}K_{p,0}}\)
such that \(\pi_{K_{p,0}}^{-1}(U_{K_{p,0}})=\pi_{\HT}^{-1}(U)\).

(1)
For $\mF\in\QCoh^{\rla}_{\fg^{c}}(U)^{\fn^{0}}$, $\tVB(\mF)$ and \(\tVBfu(\mF)\) are concentrated in degree $0$ in the following sense:

There exists a system of open compact subgroups \(\{K_{p,r}\}_{r\in\NN}\), an open cover \(\{U_{i,K_{p,0}}\}\) of \(U_{K_{p,0}}\)
and a quasi-coherent sheaf \(\VB_{U_{i,K_{p,r}}}(\mF)\) over \(U_{i,K_{p,r}}\) that is a direct summand of ON Banach modules,
such that  \[\tVBfu(\mF)\cong \varinjlim_{r}\pi^{\sm,*}_{K_{p,r}}(\VB_{U_{i},K_{p,r}}(\mF)).
\]
Here \(U_{i,K_{p,r}}\) denotes 
	the preimage of \(U_{i,K_{p,0}}\)
	in \(\mX^{\tor}_{K^{p}K_{p,r}}\), and \(\pi^{\sm}_{K_{p,r}}:\mX^{\tor,\sm}_{K^{p}}\to \mX^{\tor}_{K^{p}K_{p,r}}\).


(2) We have \[(\pi^{\la}_{\sm})^{*}(\tVBfu(\mF))\cong (\pi_{\HT,K^{p}}^{\la})^{*}(\mF),\]
for \[
\mX^{\tor,\sm}_{K^{p}}
\xleftarrow{\pi_{\sm}^{\la}}\mX^{\tor,\la}_{K^{p}}
\xrightarrow{\pi_{\HT,K^{p}}^{\la}}\Fl_{G,\mu},
\]

(3) \(\tVBfu\)
is a symmetric monoidal functor when restricted to \(\mF\in\mrm{QCoh}^{\rla}_{\fg^{c}}(U)^{\fn^{0}}\).
\end{thm}
\begin{proof}
(1) follows essentially from \cite[Theorem 5.2.1, Theorem 5.2.5]{Juan2022.09locallyShi}. A good summary of the argument is summed up in \cite[Theorem 4.5.4]{BoxerCalegariGeePilloni2025modularity}. Note that our definition of the functor \(\tVBfun\) is precisely the functor \(\VB\) in \cite{BoxerCalegariGeePilloni2025modularity}. 

For $\mF\in\QCoh^{\rla}(U)^{\fn^{0}}$, \( R\Gamma(\fn^{0},\mF)\cong\bigoplus_{i\in \NN}\mF\otimes (\wedge^{i}\fn^{0})^{\vee}[-i] \). Thus by Lemma \ref{lemVBnToVB}, we know that
\begin{align}\label{alignCompareSumRGammeN0}
\tVBfun(\mF)\cong\tVBfu(R\Gamma(\fn^{0},\mF))\cong \bigoplus_{i=0}^{d}\tVBfu(\mF\otimes(\wedge^{i}\fn^{0})^{\vee})[-i].
\end{align}
 Compare with \cite[Theorem 4.5.4 (2)]{BoxerCalegariGeePilloni2025modularity}, we see that \(\tVBfu(\mF\otimes(\wedge^{d}\fn^{0})^{\vee})\)
is a direct summand of \(\VB^{0}_{\Sigma}(\mF\otimes(\wedge^{d}\fn^{0})^{\vee})\) loc. cit. 
In particular, \(\tVBfu(\mF\otimes(\wedge^{d}\fn^{0})^{\vee})\) has the desired property of concentration at degree \(0\). 
Replacing $\mF$ with $\mF\otimes\wedge^{d}\fn^{0}$ tells us the concentration of $\tVBfu(\mF)$. Comparing (\ref{alignCompareSumRGammeN0}) again with \cite[Theorem 4.5.4 (2)]{BoxerCalegariGeePilloni2025modularity}, we see that \(\tVBfu(\mF)\cong \VB^{0}_{\Sigma}(\mF)\). 

For (2), one has an isomorphism after pulling back to the diamond. Then
one follows the same proof of Theorem \ref{thmIdenTwoTorsorsOverLa} to descend the isomorphism to \(\mX^{\tor,\la}_{K^{p}}\).

For (3),
it is easy to see that \(\tVBfu(-)\)
is lax symmetric monoidal,
that is, there is a natural morphism \(\tVBfu(\mF)\otimes\tVBfu(\mF')\to \tVBfu(\mF\otimes\mF')\)
is \(\QCoh(\mX^{\tor,\sm}_{K^{p}}|_{U})\). Now since \(\pi^{\sm}\)
is a \(!\)-surjection,
and the morphism above becomes an isomorphism after 
\(\pi^{\sm,*}\),
we know that this natural morphism is an isomorphism.
\end{proof}
\begin{dfn}[Horizontal action]\label{dfnHorActionPilloni}

	Note that for any $\mF\in\QCoh_{\fg^{c}}(U)^{\fn^{0}}$, $\mF$ is carried with a canonical action $\fg^{c,0}/\fn^{0}\otimes\mF\to \mF$. We consider the sub-bundle \(\mfk{m}^{c,0}:=\mfk{p}^{c,0}/\fn^{0}\subset \fg^{c,0}/\fn^0
    \), and we have the action of $\mfk{m}^{c,0}$ on $\mF$,
	which we denote by $(\mfk{m}^{c,0},\theta)$ and refer to it as the \emph{horizontal action}. 
	

	Assuming \(\mF\in\QCoh_{\fg^{c}}^{\rla}(U)^{\fn^{0}}\). Then applying \(\tVBfu\) to the Koszul complex associated to the \(\mm^{c,0}\)-action on \(\mF\) \[
	R\Gamma(\mm^{c,0},\mF):=
	[\mF\to \mF\otimes_{\mO_{\Fl}}(\mm^{c,0})^{\vee}
	\to 
	\cdots \to \mF\otimes_{\mO_{\Fl}}\wedge^{\dim(\mm^{c})}(\mm^{c,0})^{\vee}],
	\]
	we obtain a complex in \(\QCoh(\mO^{\sm})\)
	\begin{align*}
	\tVBfu(R\Gamma(\mm^{c,0},\mF))=[\tVBfu(\mF)\to \tVBfu(\mF)\otimes_{\mO_{\Fl}}\tVBfu((\mm^{c,0})^{\vee})\\
	\to\cdots \to \tVBfu(\mF)\otimes_{\mO_{\Fl}}\wedge^{\dim(\mm^{c})}\tVBfu((\mm^{c,0})^{\vee})],
	\end{align*}
	where we have used the symmetric monoidal structure of \(\tVBfu\) in Theorem \ref{thmMainThmGeomSenVB} (2).

	In particular, we have an action of \(Z(\mfk{m})\) (the center of \(U(\mfk{m})\)),
	By functoriality, $(Z(\mfk{m}),\theta)$ induces an action of \(Z(\mfk{m})
    \) on $\tVBfu(\mF)$, which we also denote at $\theta$.
\end{dfn}

\subsection{Geometry of locally analytic Shimura varieties}
\label{subsecLocallyAnalyticShimuraVar}

\begin{prop}\label{propLAvsSmoothSV}
We have the following commutative diagram \[
\begin{tikzcd}
\mX^{\tor,\la}_{K^{p}}
\arrow[d,"\pi^{\la}_{\sm}"]
\arrow[r,"\pi^{\la}_{\GM\HT}"] & 
\Per_{G,\mu}/(G^{c}_{\CC_{p}},*_{\GM}) 
\arrow[r,"\pi_{\HT}^{per}"] \arrow[d] & 
\Fl_{G,\mu}\arrow[d]\\
\mX^{\tor,\sm}_{K^{p}}\arrow[r,"\pi^{\sm}_{\GM\HT}"] & \Per_{G,\mu}^{\dR/\Fl^{\std}_{G,\mu}}/(G^{c}_{\CC_{p}},*_{\GM}) \arrow[r] & \Fl_{G,\mu}/(\fg^{0}/\fn^{0})^{\dagger}
\end{tikzcd}
\] where both squares are \emph{Cartesian}.
\end{prop}
The proof will be finished at the end of this subsubsection.
\begin{notation}
Recall from Notation \ref{notationStalkAt1DaggerGrp} that
the space
$\mO_{G^{c},1}$ carries two actions $*_{1}$ and $*_{2}$ of $\fg^{c}$. On $\mO_{G^{c},1}\hatotimes\mO_{\Fl}$, there is a third action coming from the action of $G^{c}$ on $\mO_{\Fl}$, which we denote as $*_{3}$. 
Then we have an action \(*_{I}\) of \(\fg^{c}\) on \(\mO_{\Fl}\otimes \mO_{G^{c},1}\) for any \(I\subset\{1,2,3\}\). See Notation \ref{notationStalkAt1DaggerGrp} for details.
We have $(\mO_{G^{c},1}\hatotimes\mO_{\Fl},*_{1,3})\in\QCoh_{\fg^{c}}(\Fl)$.
We define \[C^{\la}:=R\Gamma((\fn^{0},*_{1,3}),\mO_{\Fl}\hatotimes_{\QQ_{p}}\mO_{G^{c},1}),\]
and \[
C^{\la,0}:=R\Gamma((\mfk{p}^{c,0},*_{1,3}),\mO_{\Fl}\hatotimes\mO_{G^{c},1}).
\]
\end{notation}

\begin{lem}\label{lemLocalStalkAcyclicUgRep}
Let \(G\) be a \(p\)-adic Lie group \(G\) and \(\fg:=\Lie(G)\). Then \(R\Gamma(\fg,(\mO_{G,1},*_{1}))\cong \CC_{p}[0]\). 
\end{lem}
\begin{proof}
\(\mO_{G,1}\) is a LB space of compact type. 
Using the duality between LB spaces of compact type, and Fréchet spaces of compact type (\cite[Theorem 3.40]{JRC2021solid}), it suffices to show the dual statement that \(\mO_{G,1}^{\vee}\otimes^{L}_{U(\fg)}\CC_{p}\cong 0\). 
Moreover, \(\mO_{G,1}^{\vee}\cong \varprojlim_{n} D(G_{h^{+}},\CC_{p}) \), where the latter
is defined in \cite[\S 5.3.2]{JRC2021solid}. Then we conclude by \cite[Proposition 5.12]{JRC2021solid}.
\end{proof}
\begin{lem}\label{lemClaConcentraDeg0}
(1) The sheaf
$C^{\la}$ is concentrated degree $0$, $C^{\la}\in\QCoh^{\rla}_{\fg^{c}}(\Fl)^{\fn^{0}}$, and \(C^{\la}\cong\mO_{\Fl}\{(\fg^{c,0}/\fn^{0})^{\vee}\}^{\dagger}\).

(2) The sheaf
\(C^{\la,0}\) is concentrated in degree \(0\),
and \(C^{\la,0}\in \mrm{QCoh}_{\fg^{c}}^{\rla}(\Fl)^{\fn^{0}}\), \(C^{\la,0}\cong \mO_{\Fl}\{(\fg^{c,0}/\mfk{p}^{c,0})^{\vee}\}^{\dagger}\).

\end{lem}
\begin{proof}
We identify \(\fg^{0}\) with its total space \(\fg \times \Fl\).
Locally over \(U\subset \Fl\),
we can fix a section \(s:U\to G\) to the natural map \(G\to \Fl=\Fl_{G,\mu}=P_{\mu}\backslash G\), and then
\(\mrm{Ad}_{s(x)}\) induces an automorphism \(\fg\otimes\mO_{U}\cong \fg\otimes\mO_{U}\), such that
\(\fn^{0}|_{U}=\{(x,\mrm{Ad}_{s(x)}(X)):x\in U,X\in \fn \}\subset \fg\otimes\mO_{U}\) is sent to
\(\fn\otimes\mO_{U}\subset \fg\otimes\mO_{U}\),
Similarly, \(\mfk{p}^{c,0}|_{U}=\{(x,\mrm{Ad}_{s(x)}(X)):x\in U,X\in \mfk{p}^{c} \}\) is sent to \(\p^{c}\otimes\mO_{U}\).

Now \(\mO_{G^{c},1}\cong \CC_{p}\{\fg^{c,\vee}\}^{\dagger} \), so \(\mO_{\Fl}\hatotimes\mO_{G^{c},1}\cong \mO_{\Fl}\{\fg^{c,\vee}\}^{\dagger},\)
and thus \[C^{\la}|_{U}\cong R\Gamma(\fn^{0},\mO_{U}\hatotimes\mO_{G^{c},1})\cong R\Gamma(\fn\otimes\mO_{U},\mO_{U}\{\fg^{c,\vee}\}^{\dagger})\cong \mO_{U}\hatotimes R\Gamma(\fn,\CC_{p}\{\fg^{c,\vee}\}^{\dagger}),\]
where we have implicitly used the fact that the action of \(\fn^{0}\)
is trivial on \(\mO_{\Fl}\).
Fixing a decomposition \(\fg^{c}=\fn\oplus \bar{\p}^{c}\), 
we have \(\CC_{p}\{\fg^{c,\vee}\}^{\dagger}\cong \CC_{p}\{\fn^{\vee}\}^{\dagger}\hatotimes \CC_{p}\{\bar{\p}^{c}\}^{\dagger}\), and the action of \(\fn\) is induced by its action on \(\CC_{p}\{\fn^{\vee}\}\). We know that \(R\Gamma(\fn,\CC_{p}\{\fn^{\vee}\}^{\dagger})\cong \CC_{p}
\) (Lemma \ref{lemLocalStalkAcyclicUgRep}), so
\(R\Gamma(\fn,\CC_{p}\{\fg^{c,\vee}\}^{\dagger})\cong  \CC_{p}\{(\fg^{c}/\fn)^{\vee}\}^{\dagger}\). Twisting back by \(s^{-1}\), we have an isomorphism \(C^{\la}|_{U}\cong \mO_{U}\{(\fg^{c,0}/\fn^{0})^{\vee}\}^{\dagger}\)
that is independent of the choice of \(s\).

The proof of (2) is similar.
\end{proof}

\begin{cor}\label{corVBCla=Ola} 
We have \(\tVBfu(C^{\la},*_{1,3})\cong \mO^{\la}_{K^{p}}\), and 
$\tVB(C^{\la},*_{1,3})\cong \mscr{I}^{\la}_{K^{p}}$. Moreover,
the constant action of \(\fg\)
on \(\mO^{\la}_{K^{p}}\) and \(\mscr{I}^{\la}_{K^{p}}\)
coincides with the action induced by \((\fg,*_{2})\) on \(C^{\la}\).
More generally,
for any \(\mF\in\mrm{QCoh}_{\fg^{c}}^{\rla}(\Fl)^{\fn^{0}}\),
we have \(\tVBfu(\mF\otimes C^{\la})\cong \pi^{\la,*}_{\HT}(\mF)\).
\end{cor}
\begin{proof}
We have 
\begin{align*}
& \tVBfu(\mF\hatotimes C^{\la})
\cong
\tVBfu(R\Gamma(\fn^{0},\mF\hatotimes\mO_{G^{c},1}))
 \cong \tVBfun(\mF\hatotimes\mO_{G^{c},1}) \\
&
\cong (\mF\hatotimes_{\mO_{\Fl}}\hat{\mO}_{\mX^{\tor}_{K^{p}}}\hatotimes\mO_{G^{c},1})^{R-\sm} \cong (\mF\hatotimes_{\mX^{\tor}_{K^{p}}}\hat{\mO}_{\mX^{\tor}_{K^{p}}})^{R-\la}
\cong \mF\hatotimes_{\mO_{\Fl}}\mO_{\mX^{\tor}_{K^{p}}}^{\la} ,
\end{align*} where the last isomorphism holds by \cite[Corollary 2.12]{Jiang2023thetaModCur}. 
The same argument also works for \(\tVB\).
\end{proof}

\begin{prop}\label{propdRstackOfShLA}
    The natural morphism \(\pi^{\la}_{\sm}:\mX_{K^{p}}^{\tor,\la}\to \mX_{K^{p}}^{\tor,\sm}
    \) induces a monomorphism after taking \(
    (-)^{\dR}
    \).
	
	The natural morphism \(\mX^{\tor,\sm}_{K^{p}}\to (\mX^{\tor,\sm}_{K^{p}})^{\dR}\) canonically 
	factors through \((\mX_{K^{p}}^{\tor,\la})^{\dR}\), and \((\mX^{\tor,\la}_{K^{p}})^{\dR/(\mX^{\tor,\sm}_{K^{p}})}\cong \mX^{\tor,\sm}_{K^{p}}\).
\end{prop}
\begin{proof}
    We claim that the projection to the first factor \(p_1:
    \mX_{K^{p}}^{\tor,\la}\times_{\mX_{K^{p}}^{\sm}}\mX_{K^{p}}^{\tor,\la}
    \to \mX_{K^{p}}^{\tor,\la}
    \) becomes an isomorphism after taking \(\dagger\)-reduction. 

    By Theorem \ref{thmMainThmGeomSenVB}, Corollary \ref{corVBCla=Ola} and Lemma \ref{lemClaConcentraDeg0}, we have \[\tVBfu(C^{\la},*_{1,3})\cong\mO^{\la}_{K^{p}},\] 
	and
	\[ \mO^{\la}_{K^{p}}\otimes_{\mO^{\sm}_{K^{p}}}\mO^{\la}_{K^{p}}\cong 
	\pi^{\la,*}_{\sm}\mO^{\la}_{K^{p}}\cong \pi_{\HT}^{\la,*}(C^{\la})\cong \mO^{\la}_{K^{p}}\{(\fg^{c,0}/\fn^{0})^{\vee}\}^{\dagger}\in\QCoh(\mX^{\la}_{K^{p}}),
    \] 
	along which \(p_{1}^{*}: \mO^{\la}_{K^{p}}\to \mO^{\la}_{K^{p}}\otimes_{\mO^{\sm}_{K^{p}}}\mO^{\la}_{K^{p}}
    \) coincides with the natural map \(\mO^{\la}_{K^{p}}\to
    \mO^{\la}_{K^{p}}\{(\fg^{c,0}/\fn^{0})^{\vee}\}^{\dagger}
    \). Then we know that \(p_{1}^{*}\) induces an isomorphism after taking \(\dagger\)-reduction. 

    As a corollary, the morphism \(
    p_1^{\dR}:
    (\mX_{K^{p}}^{\tor,\la})^{\dR}\times_{(\mX^{\tor,\sm}_{K^{p}})^{\dR}}(\mX_{K^{p}}^{\tor,\la})^{\dR}
    \to (\mX_{K^{p}}^{\tor,\la})^{\dR}
    \) is an isomorphism. Equivalently, 
    the natural morphism \((\mX_{K^{p}}^{\tor,\la})^{\dR}
    \to (\mX^{\tor,\sm}_{K^{p}})^{\dR}\) is a monomorphism. 
    
	For the last part, the factorization exists, because \(\mX^{\tor,\la}_{K^{p}}\to \mX^{\tor,\sm}_{K^{p}}\) is a \(!\)-surjection (Lemma \ref{lemOpenSetsFromFulltoSM} (2)); it is unique, because \((\mX_{K^{p}}^{\tor,\la})^{\dR}
    \to (\mX^{\tor,\sm}_{K^{p}})^{\dR}\) is a monomorphism.
	In particular, \((\mX^{\tor,\la}_{K^{p}})^{\dR/(\mX^{\tor,\sm}_{K^{p}})}\cong \mX^{\tor,\sm}_{K^{p}}\).
\end{proof}

\begin{proof}[Proof of Proposition \ref{propLAvsSmoothSV}]
We clearly have a \(M^{c}_{\mu}\)-equivariant Cartesian diagram \[
\begin{tikzcd}
\mcal{M}_{\mu}^{c}(-1)\times \mcal{M}_{\mu}^{\std,c}
\arrow[d]\arrow[r,"p_{1}"]
& \mcal{M}_{\mu}^{c}(-1)\arrow[d]\\
(\mcal{M}_{\mu}^{c})^{\dR}(-1)\times \mcal{M}_{\mu}^{\std,c}\arrow[r]
& (\mcal{M}_{\mu}^{c})^{\dR}(-1).
\end{tikzcd}
\] Quotienting out \(M^{c}_{\mu}\), we obtain a Cartesian square \[
\begin{tikzcd}
\Per_{G,\mu}\arrow[r,"\pi^{per}_{\HT}"]\arrow[d] & \Fl_{G,\mu}\arrow[d]\\
\Per_{G,\mu}^{\dR/\Fl_{G,\mu}^{\std}} \arrow[r] & \Fl_{G,\mu}/(\fg^{c,0}/\fn^{0})^{\dagger}
\end{tikzcd}
\] by Lemma \ref{lemCompareDRStac} (3). Quotienting out \((G^{c}_{\CC_{p}},*_{\GM})\)
on the left column, we obtain the square on the right of Proposition \ref{propLAvsSmoothSV}.

For the square on the left, we have the following commutative diagram \[
\begin{tikzcd}
\mX^{\tor,\la}_{K^{p}}\arrow[r,"\pi^{\la}_{\GM\HT}"] 
\arrow[d] & \Per_{G,\mu}/(G^{c}_{\CC_{p}},*_{\GM})\arrow[d]
\\
\mX^{\tor,\sm}_{K^{p}}\arrow[r] & \Fl^{\std}_{G,\mu}/G^{c}_{\CC_{p}}.
\end{tikzcd}
\] Taking the relative de Rham stacks along the vertical arrows, we have a commutative diagram \[
\begin{tikzcd}
\mX^{\tor,\la}_{K^{p}}\arrow[r,"\pi^{\la}_{\GM\HT}"] 
\arrow[d] & \Per_{G,\mu}/(G^{c}_{\CC_{p}},*_{\GM})\arrow[d]
\\
\mX^{\tor,\sm}_{K^{p}}\arrow[r] & \Per_{G,\mu}^{\dR/\Fl^{\std}_{G,\mu}}/(G^{c}_{\CC_{p}},*_{\GM}).
\end{tikzcd}
\] by Proposition \ref{propdRstackOfShLA}. 

To show that it is Cartesian, since the square on the right has been shown to be Cartesian, it suffices to verify it for the bigger square \[
\begin{tikzcd}
\mX^{\tor,\la}_{K^{p}}\arrow[r,"\pi^{\la}_{\HT,K^{p}}"] 
\arrow[d,"\pi_{\sm}^{\la}"] & \Fl_{G,\mu}\arrow[d]
\\
\mX^{\tor,\sm}_{K^{p}}\arrow[r] & \Fl/(\fg^{c,0}/\fn^{0})^{\dagger}. 
\end{tikzcd}
\] By tracing the construction, the induced map \[
\mX^{\tor,\la}_{K^{p}}\times_{\mX^{\tor,\sm}_{K^{p}}}\mX^{\tor,\la}_{K^{p}}\to (\fg^{c,0}/\fn^{0})^{\dagger}
\] is induced by \[
\mO_{\Fl}\{(\fg^{c,0}/\fn^{0})^{\vee}\}^{\dagger}\hookrightarrow 
\mO^{\la}_{K^{p}}\{(\fg^{c,0}/\fn^{0})^{\vee}\}^{\dagger}
\xrightarrow{\sim}
\mO^{\la}_{K^{p}}\otimes_{\mO^{\sm}_{K^{p}}}\mO^{\la}_{K^{p}},
\] where the second map is as in the proof of Proposition \ref{propdRstackOfShLA}.
The isomorphism shows that 
\begin{align*}
\mX^{\tor,\la}_{K^{p}}\times_{\mX^{\tor,\sm}_{K^{p}}}\mX^{\tor,\la}_{K^{p}}\to (\fg^{c,0}/\fn^{0})^{\dagger}\cong \mX^{\tor,\la}_{K^{p}}\times_{\Fl_{G,\mu}}(\fg^{c,0}/\fn^{0})^{\dagger}\\
\cong \mX^{\tor,\la}_{K^{p}}\times_{\Fl_{G,\mu}/(\fg^{c,0}/\fn^{0})^{\dagger}}\Fl_{G,\mu}.
\end{align*}
Since \(\mX^{\tor,\la}_{K^{p}}\to \mX^{\tor,\sm}_{K^{p}}\) is a \(!\)-surjection by Lemma \ref{lemOpenSetsFromFulltoSM} (2), we know that the natural map \[\mX^{\tor,\la}_{K^{p}}
\to \mX^{\tor,\sm}_{K^{p}}\times_{\Fl_{G,\mu}/(\fg^{c,0}/\fn^{0})^{\dagger}}\Fl_{G,\mu}.
\] is an isomorphism.
\end{proof}
\subsection{Smooth log de Rham stack of Shimura varieties}
\label{subsecSMdRstack}
In this subsection, we will prove Theorem \ref{thmLAstrucGMHTper}.
\begin{dfn}\label{dfnSMratConstruction}
For any \( E_{\p}\subset L\subset \CC_{p} \), and  \( ?\in\{(nothing),+,\hat{+}\} \)
we define the \emph{smooth log de Rham stack of Shimura varieties} as the solid stack \[(\mX^{\tor}_{K^{p},E_{\p}})^{\pdR/E_{\p},?,\sm}_{\log}:=\varprojlim_{K_{p}}(\mX^{\tor}_{\Kpp,E_{\p}})_{\log}^{\pdR/E_{\p},?},\]
and \((\mX^{\tor}_{K^{p},L})^{\pdR/L,?,\sm}_{\log}:= (\mX^{\tor}_{K^{p},E_{\p}})^{\pdR/E_{\p},?,\sm}_{\log}\times_{\AnSpec(E_{\p})}\AnSpec(L) \). 

We will write the open subspace corresponding to the complement of the toroidal boundaries as \( \mX_{K^{p},L}^{\pdR/L,?,\sm} \).

More generally, for \((U_{K_{p}},U^{\sm}:=\varprojlim_{K_{p}}U_{K_{p}},\ti{U})\) as in Construction \ref{constructionUChartTower} that descends to \( U_{K_{p},L}\subset\mX^{\tor}_{\Kpp,L} \) for some \( L/E_{\p} \), 
we denote \[
(\ti{U}_{L})_{\log}^{\pdR/L,?,\sm}:=\varprojlim_{K_{p}}(U_{K_{p},L})^{\pdR/L,?}_{\log}.
\]

We define \[
\Omega^{1,\sm}_{\log,K^{p},L}:=(\pi^{\sm}_{K_{p}})^{*}(\Omega^{1}_{\mX^{\tor}_{\Kpp,L},\log}),\;\omega^{\sm}_{\log,K^{p},L}:=(\pi^{\sm}_{K_{p}})^{*}(\omega_{\mX^{\tor}_{\Kpp,L},\log}).
\] We will omit the subscript \( L \) 
if \( L=\CC_{p} \).  
\end{dfn}
\begin{lem}\label{lemSmDRStackProperty}
(1)
We have a Cartesian diagram of solid stacks over \( L \) \[
\begin{tikzcd}
\mX^{\tor,\sm}_{K^{p},L}\times{[\mA^{1}/\GG_{m}]}\arrow[r,"\pi^{\sm}_{K_{p}}"]\arrow[d,"h^{\sm}"]
& \mX^{\tor}_{\Kpp,L}\times{[\mA^{1}/\GG_{m}]}\arrow[d,"h_{K_{p}}"]\arrow[r,"\pi_{\GM,\Kpp}"]
& (\Fl^{\std}_{G,\mu} \times {[\mA^{1}/\GG_{m}]})/G^{c}_{L}\arrow[d]
\\
(\mX^{\tor}_{K^{p},L})^{\pdR/L,+,\sm}_{\log}\arrow[r,"\pi^{\pdR,\sm}_{K_{p}}"]
& (\mX^{\tor}_{\Kpp,L})^{\pdR/L,+}_{\log}\arrow[r,"\pi_{\GM,\Kpp}^{\pdR}"] & \Fl^{\std,\pdR/L,+}_{G,\mu}/G^{c}_{L},
\end{tikzcd}
\] 

Moreover, \(h^{\sm}\) is a suave \( ! \)-cover, 
and \[h^{\sm,!}(\mO_{(\mX_{K^{p},L}^{\tor})^{\pdR/L,+,\sm}_{\log}})\cong \omega_{\log,K^{p},L}^{\sm}\{d\}[d],\]
with \(d:=\dim(\mX_{\Kpp}^{\tor})\), and \( \{d\} \) 
is as in Notation \ref{notationTwistByAGm}.

In particular, \( (\mX^{\tor}_{K^{p},L})^{\pdR/L,+,\sm}_{\log} \) 
is essentially Tate (Definition \ref{dfnEssentialTate}). By abuse of notation, we will write \[
(\mX^{\tor}_{K^{p},L})^{\pdR/L,+,\sm}_{\log}:=((\mX^{\tor}_{K^{p},L})^{\pdR/L,+,\sm}_{\log})^{Tate}.
\]

(2) The natural map \((\mX^{\tor,\sm}_{K^{p},L})^{\dR/((\mX^{\tor}_{K^{p},L})^{\pdR/L,\sm}_{\log})}\to (\mX^{\tor}_{K^{p},L})^{\pdR/L,\sm}_{\log}\)
is an isormorphism.

(3) The maps \[ s^{\pdR,\sm}_{K_{p}}:(\mX^{\tor}_{K^{p},L})^{\pdR/L,\sm}_{\log}\to (\mX^{\tor}_{\Kpp,L})^{\pdR/L,\sm}_{\log} \] and \[ s^{\pdR,\sm}:(\mX^{\tor}_{K^{p},L})^{\pdR/L,\sm}_{\log}\to \AnSpec(L_{\square}) \] are prim.
\end{lem}
\begin{proof}
For (1), the diagram on the left is Cartesian by Lemma \ref{lemBasicLogDRSta} (3),
and the Cartesian diagram on the right is given by Proposition \ref{propAlgGMtheory}.
The rest follows from
Lemma \ref{lemBasicFiltered} and Lemma \ref{lemBasicPropertyEsseTate} (6).

For (2), by Theorem \ref{thmAnDRStackGeneral} (1), we conclude as \[(\mX^{\tor}_{\Kpp,L})^{\dR}\to((\mX^{\tor}_{\Kpp,L})^{\pdR/L}_{\log})^{\dR}\] is a monomorphism by Lemma \ref{lemBasicLogDRSta} (2).

For (3), the first map is prim by the Cartesian diagram in (1) and Lemma \ref{lemOpenSetsFromFulltoSM} (4).
This fact combined with Lemma \ref{lemFiniteLevelProper} (1) implies that
the second map is also prim.
\end{proof}

\begin{lem}\label{lemQCohSMsv}
(1)
We have isomoprhisms in \(\mrm{Pr}^{L}\) \begin{align*}
\QCoh(\mX^{\tor,\sm}_{K^{p},L})
&\cong \varinjlim_{K_{p}}\QCoh(\mX^{\tor}_{\Kpp,L}),\\
\QCoh((\mX^{\tor}_{K^{p},L})^{\pdR/L,?,\sm}_{\log})
&\cong \varinjlim_{K_{p}}\QCoh((\mX^{\tor}_{\Kpp,L})^{\pdR/L,?,\sm}_{\log})
\end{align*}
where the transition maps are given by \((-)^{*}\), 

(2) We have a \( \QCoh([\mA^{1}/\GG_{m}]) \)-linear isomorphism in \(\mrm{Pr}^{L}\)  \[\QCoh((\mX^{\tor}_{K^{p},L})^{\pdR/L,+,\sm}_{\log})
\cong \Mod_{D_{\log,K^{p},L}^{\sm}}(\Fil^{\ZZ}\QCoh(\mX^{\tor,\sm}_{K^{p},L})),\]
where \(D^{\sm}_{\log,K^{p},L}:=(\pi^{\sm}_{K_{p}})^{*}(D_{\mX^{\tor}_{\Kpp,L},\log})\), equipped with the bi-\(\mO_{\mX^{\tor,\sm}_{K^{p},L}}\)-algebra structure, such that for any \(K_{p}\), \[
(\pi^{\sm}_{K_{p}})^{*}(D^{\sm}_{\log,K^{p},L})\cong \varinjlim_{K_{p}'\subset K_{p}}(\pi^{K_{p}'}_{K_{p}})_{*}(D_{\mX^{\tor}_{K^{p}K_{p}',L},\log})
\] as bi-\(\mO_{\mX^{\tor}_{K^{p}K_{p},L}}\)-algebras, where \(\pi^{K_{p}'}_{K_{p}}:\mX^{\tor}_{K^{p}K_{p}',L}\to\mX^{\tor}_{K^{p}K_{p},L} \).

(3) Let \((U_{K_{p}},U^{\sm}:=\varprojlim_{K_{p}}U_{K_{p}},\ti{U})\) be as in Construction \ref{constructionUChartTower} that descends to \( U_{K_{p},L}\cong \Spa(A_{K_{p},L},A^{+}_{K_{p},L})\subset\mX^{\tor}_{\Kpp,L} \) for some \( L/E_{\p} \). 
Let \( A^{\sm}_{L}:=\varinjlim_{K_{p}}A_{K_{p},L} \) 
and \( \dR_{\log}(A^{\sm}_{L}):=\varinjlim_{K_{p}}\dR_{\log}(A_{K_{p},L}) \). 

Then 
 we have a \( \widehat{\Fil}^{\ZZ}(D(L_{\square})) \)-linear isomorphism in \(\mrm{Pr}^{L}\) \[\QCoh((\ti{U}_{L})^{\pdR/L,\hat{+},\sm}_{\log})
\cong \Mod_{\dR_{\log}(A^{\sm}_{L})}(\widehat{\Fil}^{\ZZ}(D(L_{\square}))),\]
\end{lem}
\begin{proof}
(1) follows from Lemma \ref{lemUniformAffineProper}. 
As in the proof of Lemma \ref{lemUniformAffineProper},
using \cite[Corollary 5.5.3.4]{Lurie2009HTT} and \cite[Proposition 7.1.2.7]{Lurie2017HA},
(2) and (3) follow from Lemma \ref{lemDRcomplexVSpushforward} and Lemma \ref{lemLogDRStackIsAffine} respectively.
\end{proof}
We can define the (dual) BGG complexes on \( (\mX^{\tor}_{K^{p}})^{\pdR/\CC_{p},\sm,+}_{\log} \) using Proposition \ref{propBGGFilSV}: 
\begin{notation}
Denote by \( \mcal{M}^{c}_{\dR} \) the pull-back of \( \mcal{M}^{c}_{\dR,\Kpp} \) along \( \mX^{\tor,\sm}_{K^{p}}\to \mX^{\tor}_{\Kpp} \). We define similarly \( \mcal{P}^{\std,c}_{\dR} \) and \( \mcal{G}^{c}_{\dR} \), and the latter descends to \( \bar{\mcal{G}}^{c}_{\dR} \)
over \( (\mX^{\tor}_{K^{p}})^{\pdR/\CC_{p},+,\sm}_{\log} \), which is defined as the pull-back of \( \bar{\mcal{G}}^{c}_{\dR,K} \) along \( (\mX^{\tor}_{K^{p}})^{\pdR/\CC_{p},+,\sm}_{\log}\to (\mX^{\tor}_{\Kpp})^{\pdR/\CC_{p},+}_{\log} \).
\end{notation}
\begin{prop}\label{propBGGFilSVsm}
Fix a Borel \( B\subset G_{\bar{\QQ}} \) such that \( B\subset P_{\mu} \).
Let
\( \alpha\in X^{*}(T) \) be a \( G \)-dominant character. Consider \( h^{\sm}:\mX^{\tor,\sm}_{K^{p}}\times [\mA^{1}/\GG_{m}]\to (\mX^{\tor}_{K^{p}})^{\pdR/\CC_{p},+,\sm}_{\log} \).

(1)  \( \bar{\mcal{G}}^{c,+}_{\dR}(V_{\alpha})^{\vee} \) is isomorphic to a so-called \emph{BGG complex} \[
\mrm{BGG}_{\alpha}^{+}:=
\left[L_{\dim(\fn)}\to \cdots \to L_{1}\to L_{0} \right],
\] with \( L_{i}\cong \bigoplus_{w\in {}^{M}W,l(w)=i}h^{\sm}_{\natural}(\mcal{M}_{\dR}^{c}(W_{w\cdot \alpha}^{\vee}))\{(w\cdot\alpha)(\mu)\} \) in cohomological degree \( -i \).

(2) \( \bar{\mcal{G}}^{c,+}_{\dR,K}(V_{\alpha}) \) is isomorphic to a so-called \emph{dual BGG complex} \[
\mrm{dBGG}_{\alpha}^{+}:=
\left[M^{0}\to M^{1}\cdots \to M^{\dim(\fn)} \right],
\] with \( M^{i}\cong \bigoplus_{w\in {}^{M}W,l(w)=i}h^{\sm}_{*}(\mcal{M}_{\dR}^{c}(W_{w\cdot \alpha}))\{-(w\cdot\alpha)(\mu)\} \) in cohomological degree \( i \).
\end{prop}
\begin{proof}
(1) follows from Proposition \ref{propBGGFilSV} (1)
and proper base change (Lemma \ref{lemSmBaseChange} (1)).
Then (2) follows from (1) by taking dual \( \unl{\RHom}(-,\mO_{(\mX^{\tor}_{K^{p}})^{\pdR/\CC_{p},\sm,+}_{\log}}) \).
\end{proof}

\begin{prop}\label{propSMdRTorLog}
(1)
We have the following commutative diagram of Tate stacks
\[
\begin{tikzcd}
\mX^{\tor,\sm}_{K^{p}}
\arrow[d,"h^{\sm}"] \arrow[r,"\pi^{\sm}_{\GM\HT}"] &[-0.3cm] \Per_{G,\mu}^{\dR/\Fl^{\std}_{G,\mu}}/G^{c,\an}_{\CC_{p}} \arrow[d]\arrow[r,"\pi_{\GM}^{per}"] & \Fl_{G,\mu}^{\std}/G^{c,\an}_{\CC_{p}}\arrow[d,"h_{\Fl^{\std}}"] \\
(\mX^{\tor}_{K^{p}})_{\log}^{\pdR/\CC_{p},\sm}
\arrow[r,"\pi^{\pdR}_{\GM\HT}"] & \Per_{G,\mu}^{\dR/(\Fl^{\std,\pdR/\CC_{p}}_{G,\mu})}/G^{c,\an}_{\CC_{p}}\arrow[r,"\pi_{\GM}^{per,\pdR}"] & \Fl_{G,\mu}^{\std,\pdR/\CC_{p}}/G^{c,\an}_{\CC_{p}},
\end{tikzcd}
\]
where the actions of \( G^{c,\an}_{\CC_{p}} \) on the middle column are induced by \( *_{\GM} \), and both squares are \emph{Cartesian}.

(2)
We have the following commutative diagram of Tate stacks
\[
\begin{tikzcd}
\mX^{\sm}_{K^{p}}
\arrow[d,"h^{\sm}"] \arrow[r,"\pi^{\sm}_{\GM\HT}"] & \Per_{G,\mu}^{\dR/\Fl^{\std}_{G,\mu}}/(G^{c,\an}_{\CC_{p}},*_{\GM}) \arrow[d]\arrow[r,"\pi_{\GM}^{per}"] & \Fl_{G,\mu}^{\std}/G^{c,\an}_{\CC_{p}}\arrow[d,"h_{\Fl^{\std}}"] \\
\mX_{K^{p}}^{\dR/\CC_{p},\sm}
\arrow[r,"\pi^{\dR}_{\GM\HT}"] & (\Per_{G,\mu}^{\dR/\CC_{p}})/(G^{c,\an}_{\CC_{p}},*_{\GM})\arrow[r,"\pi_{\GM}^{per,\dR}"] & \Fl_{G,\mu}^{\std,\dR/\CC_{p}}/G^{c,\an}_{\CC_{p}},
\end{tikzcd}
\]
where \( \mX_{K^{p}}^{\dR/\CC_{p},\sm}:=\varprojlim_{K_{p}}\mX_{\Kpp}^{\dR/\CC_{p}} \), and  both squares are \emph{Cartesian}.
\end{prop}
\begin{proof}
For (1), by Proposition \ref{propAlgGMtheory} and Proposition \ref{propLAvsSmoothSV}, we have a commutative diagram \[
\begin{tikzcd}
\mX^{\tor,\sm}_{K^{p}}
\arrow[d,"h^{\sm}"] 
\arrow[r,"\pi^{\sm}_{\GM\HT}"] 
&[-0.1cm] \Per_{G,\mu}^{\dR/\Fl^{\std}_{G,\mu}}/(G^{c,\an}_{\CC_{p}},*_{\GM}) 
\arrow[r,"\pi_{\GM}^{per}"] &[-0.1cm] 
\Fl_{G,\mu}^{\std}/G^{c,\an}_{\CC_{p}}
\arrow[d,"h_{\Fl^{\std}}"] 
\\
(\mX^{\tor}_{K^{p}})_{\log}^{\pdR/\CC_{p},\sm}
\arrow[rr] & & \Fl_{G,\mu}^{\std,\pdR/\CC_{p}}/G^{c,\an}_{\CC_{p}},
\end{tikzcd}
\] and thus by taking relative de Rham stack and Lemma \ref{lemSmDRStackProperty}, the lower horizontal arrow factors through \((\Per_{G,\mu}^{\dR/(\Fl^{\std,\pdR/\CC_{p}}_{G,\mu})})/(G^{c,\an}_{\CC_{p}},*_{\GM})\). We thus have the commutative diagram in (1). The square on the right is clearly Cartesian. The large square is Cartesian by Proposition \ref{propAlgGMtheory}. Thus the square on the left is also Cartesian.

The proof of (2) is similar, using Proposition \ref{propAnalyticGMTheory} and Theorem \ref{thmAnDRStackGeneral} (3) instead of Proposition \ref{propAlgGMtheory} and Lemma \ref{lemSmDRStackProperty}.
\end{proof}

\begin{proof}[Proof of Theorem \ref{thmLAstrucGMHTper}]
This follows immediately from Proposition \ref{propLAvsSmoothSV}, Proposition \ref{propSMdRTorLog}, and Theorem \ref{thmAnDRStackGeneral} (6).
The isomorphism \( \mX^{\dR,\sm}_{K^{p}}\cong (\mX_{K^{p}}^{\diamond})^{\dR} \) is provided by Theorem \ref{thmAnDRStackGeneral} (5).
\end{proof}



\subsection{Horizontal locally analytic Shimura varieties}\label{subsecHorLocaAnaShimuraVarie}
In this subsection,
we define a quotient of \(\mX^{\la}_{K^{p}}\) by the following horizontal action.

\begin{dfn}[Horizontal action of {\cite{Pan22}}]\label{dfnHorActionPan}
There is an action of \(\fg^{0}\)
on \(\mO^{\la}\),
induced by the constant action of \(\fg\)
on \(\mO^{\la}\), where the sub-bundle \(\fn^{0}\) acts trivially by \cite[Corollary 6.2.12]{Juan2022.09locallyShi}, and thus induces an action of \(\mm^{0}\), which we denote as \(\theta^{Pan}\). 

More generally, for any \(\mF\in\QCoh_{\fg^{c}}(U)^{\fn^{0}}\), we have an action \((\mm^{c,0},\theta^{Pan})\)
on \[(\pi^{\la}_{\sm})_{*}((\pi^{\la}_{\HT,K^{p}})^{*}(\mF))\cong \mO^{\la}_{K^{p}}\otimes_{\mO_{\Fl}}\mF\] by taking the diagonal action.
\end{dfn}
\begin{lem}\label{lemCompareHorPilloniAndPan}
For any open subspace \(U\subset \Fl_{G,\mu}\),
along the isomorphisms 
\begin{align*}
\mO^{\la}_{K^{p}}\otimes_{\mO^{\sm}_{K^{p}}}\mO^{\la}_{K^{p}}
\cong C^{\la}\otimes_{\mO_{\Fl_{G,\mu}}}\mO^{\la}_{K^{p}},\\
\tVBfu((\mm^{c,0})^{\vee})\otimes_{\mO^{\sm}_{K^{p}}}\mO^{\la}_{K^{p}}
\cong (\mm^{c,0})^{\vee}\otimes_{\mO_{\Fl_{G,\mu}}}\mO^{\la}_{K^{p}}
\end{align*}
in Theorem \ref{thmMainThmGeomSenVB} and Corollary \ref{corVBCla=Ola},
we have an identification of the Koszul complex associated to the action \((\mm^{c,0},\theta_{Pan})\) (Definition \ref{dfnHorActionPan})
\[R\Gamma(\mm^{c,0},(\mO^{\la}_{K^{p}},\theta_{Pan})):=\left[\mO^{\la}_{K^{p}}\otimes_{\mO_{\Fl}}\wedge^{\bullet}((\mm^{c,0})^{\vee})\right]
\]
with the pull-back along \(\pi^{\la}_{\sm}\) of the Koszul complex associated to \((\mm^{c,0},\theta)\) in Definition \ref{dfnHorActionPilloni} \[
\tVBfu(R\Gamma(\mm^{c,0},(C^{\la},\theta))):=[\tVBfu(C^{\la})\otimes_{\mO^{\sm}_{K^{p}}}\wedge^{\bullet}(\tVBfu((\mm^{c,0})^{\vee}))],
\] after modifying the \(i\)-th differential by the signature \((-1)^{n}\).
\end{lem}

\begin{proof}
This is essentially the proof of \cite[Lemme 4.6]{Pilloni22}.
We have a factorization \[\mO^{\la}_{K^{p}}\xhookrightarrow{\Psi}\mO_{G^{c},1}\hatotimes\mO^{\la}_{K^{p}}\xrightarrow{ev_{1}}\mO_{K^{p}}^{\la},\]
where \(\Psi\)
is an injection,
which realizes \(\mO^{\la}_{K^{p}}\)
as the subspace of smooth vectors for the action \(*_{1,3}\), where \(*_{3}\) denotes the constant action on \(\mO^{\la}_{K^{p}}\), and \(*_{1,3}\) is as in Notation \ref{notationStalkAt1DaggerGrp}, and \(ev_{1}\)
is induced by evaluation at \(1\) \(ev_{1}:\mO_{G^{c},1}\to \CC_{p}\). In particular, the composition \(ev_{1}\circ \Psi\)
is the natural identity.
By the property of \(ev_{1}\),
for any \(X\in \fg^{c}\),
and any \(f\in \mO_{G^{c},1}\hatotimes \mO_{K^{p}}^{\la}\),
we know that \(ev_{1}(X*_{1}f)=-ev_{1}(X*_{2}f)\).

We consider the action of \(X\in \mfk{p}^{c,0}\).
Since such an \(X\) acts on \(\mO_{\Fl}\) trivially,
we know \[ev_{1}(X*_{1,3}f)=ev_{1}(X*_{1}f)=-ev_{1}(X*_{2}f).\] In particular,
for any \(f\in \mO^{\la}_{K^{p}}\), and \(X\in \mfk{p}^{c,0}\),
\[ev_{1}(X*_{1,3}\Psi(f))=ev_{1}(X*_{1}\Psi(f))=-ev_{1}(X*_{2}\Psi(f))=-\theta^{Pan}(X)(f),\]
where the last identity is given by Corollary \ref{corVBCla=Ola}.
From this, we see that \(C^{\la}\xrightarrow{*_{1,3}}C^{\la}\otimes \mfk{m}^{c,0,\vee}\)
after applying \(\tVBfu(-)\)
induces \(\mO^{\la}_{K^{p}}\xrightarrow{-\theta^{Pan}}\mO^{\la}_{K^{p}}\otimes \mfk{m}^{c,0,\vee},\) as desired.
\end{proof}

\begin{dfn}
We define \[\mO^{\la,0}_{\mX^{\tor}_{K^{p}}}:=\mO^{\la,0}_{K^{p}}:=\tVBfu(C^{\la,0}),\]
and \[\mscr{I}^{\la,0}_{\mX^{\tor}_{K^{p}}}:=\mscr{I}^{\la,0}_{K^{p}}:=\tVB(C^{\la,0}).\]
\end{dfn}

\begin{lem}\label{lemHorizalActNotSame}
The sheaf \(\mO^{\la,0}\in\mrm{CAlg}(\mrm{QCoh}(\mX^{\tor,\sm}_{K^{p}}))\) in concentrated in degree \(0\),
and is the subsheaf of \(\mO^{\la}\)
where the action of \((\mfk{m}^{c,0},\theta^{Pan})\)
is zero. In fact, \(\mO^{\la,0}\cong R\Gamma((\mfk{m}^{c,0},\theta^{Pan}),\mO^{\la})\). Similar results hold for \(\mscr{I}^{\la,0}_{K^{p}}\).
\end{lem}
\begin{proof}
Concentration follows from Theorem \ref{thmMainThmGeomSenVB}
and Lemma \ref{lemClaConcentraDeg0}.
The rest follows from Lemma \ref{lemCompareHorPilloniAndPan}. 
\end{proof}

By Lemma \ref{lemHorizalActNotSame}, we can apply Lemma \ref{lemAffProperProp} (3) to \(\mO^{\la,0}_{K^{p}}\):
\begin{dfn}\label{dfnHorizontalLAsv}
We define \(\mX^{\tor,\la,0}_{K^{p}}:=\underline{\AnSpec}_{\mX^{\tor,\sm}_{K^{p}}}(\mO^{\la,0}_{K^{p}})\).
\end{dfn}
\begin{rmk}
\( \mO^{\la,0}_{K^{p}}\cong R\Gamma((\mfk{m}^{c,0},\theta^{Pan}),\mO^{\la})\subset \mO^{\la} \),
and thus it follows from Remark \ref{rmkLAsvIsEssTate} that \( \mO^{\la,0}_{K^{p}} \) is also bounded, and \( \mX^{\tor,\la,0}_{K^{p}} \) is also essentially Tate.
\end{rmk}
Then we have a factorization \[\mX^{\tor}_{K^{p}}\to \mX^{\tor,\la}_{K^{p}}\to \mX^{\tor,\la,0}_{K^{p}}\to \mX^{\tor,\sm}_{K^{p}}.\]
\begin{lem}\label{lemHTPeriodFactor0}
The Hodge-Tate period map \(\pi_{\HT}^{\la}\) as in Lemma \ref{lemHTPeriodFactor}
factors through \(\pi_{\HT}^{\la,0}:\mX^{\tor,\la,0}_{K^{p}}\to \Fl_{G,\mu}\).
\end{lem}
\begin{proof}
Since the horizontal 
action of \(\mfk{m}^{c,0}\)on \(\mO_{\Fl_{G,\mu}}\) is zero,
we know that the morphism \(\pi_{\HT}^{\la}:\mX_{K^p}^{\tor,\la}\to \Fl\) factors on the level of structure sheaves through \(\mO^{\la,0}_{K^{p}}\). The rest follows the same proof as in Lemma \ref{lemHTPeriodFactor}.
\end{proof}

\begin{prop}\label{propCartesianLA0} We have a Cartesian diagram \[\begin{tikzcd}
		\mX^{\tor,\la,0}_{K^{p}}\arrow[d,"\pi^{\sm}_{\la,0}"]
		\arrow[r,"\pi_{\HT}"]
		& \Fl_{G,\mu}\arrow[d,"h"]\\
		\mX^{\tor,\sm}_{K^{p}}
		\arrow[r,"\pi_{\HT}^{\sm}"] & \Fl_{G,\mu}^{\dR/\CC_{p}}.
	\end{tikzcd}
	\] 

	Moreover, \( h \)
	and \( \pi^{\sm}_{\HT} \) are both cohomologically co-smooth.
\end{prop}
\begin{proof}
We have a morphism \(\pi^{\sm}_{\HT}\) by composing \(\mX^{\tor,\sm}_{K^{p}}\to\Fl_{G,\mu}/(\fg^{0}/\fn^{0})^{\dagger}\) as in Proposition \ref{propLAvsSmoothSV} with \(\Fl_{G,\mu}/(\fg^{0}/\fn^{0})^{\dagger}\to \Fl_{G,\mu}^{\dR/\CC_{p}}\) in Lemma \ref{lemCompareDRStac} (1).

We thus have the commutative diagram. 
To verify that it is Cartesian, we follows the same proof of Proposition \ref{propLAvsSmoothSV}, and eventually reduces to \(\tVBfu(C^{\la,0})\cong\mO^{\la,0}_{K^{p}}\), and \[
\mO^{\la,0}_{K^{p}}\otimes_{\mO^{\sm}_{K^{p}}}\mO^{\la,0}_{K^{p}}
\cong \mO^{\la,0}_{K^{p}}\otimes_{\mO_{\Fl}}C^{\la,0}.
\] The details are left to readers.

Finally, \( h \) is cohomogically co-smooth by Theorem \ref{thmAnDRStackGeneral} (3) and Proposition \ref{propAffineProperCover},
and \( \pi^{\sm}_{\HT} \)
is cohomogically co-smooth by Lemma \ref{lemFiniteLevelProper} and Lemma \ref{lemOpenSetsFromFulltoSM} (4).
\end{proof}

\begin{cor}\label{corCartesianFlandSMGeneral}\label{corCartesianHilbertFlandSm}
(1)
We have the following commutative diagram of solid stacks over \( \CC_{p} \) \[
\begin{tikzcd}[column sep=-0.1cm]
\mX^{\tor,\la}_{K^{p}}\arrow[r]\arrow[d] & 
\Per_{G,\mu}/G^{c,\an}_{\CC_{p}}\arrow[d]
\\
\mX^{\tor,\la,0}_{K^{p}}\arrow[r]\arrow[d] & 
\Per_{G,\mu}^{\dR/(\Fl_{G,\mu}\times \Fl_{G,\mu}^{\std})}/G^{c,\an}_{\CC_{p}}\arrow[r]\arrow[d]
& \Fl_{G,\mu}\times \Fl_{G,\mu}^{\std}/G^{c,\an}_{\CC_{p}}\arrow[d]
\\
\mX^{\tor,\sm}_{K^{p}}\arrow[r]\arrow[d] & 
\Per_{G,\mu}^{\dR/\Fl_{G,\mu}^{\std}}/G^{c,\an}_{\CC_{p}}\arrow[r]
& \Fl_{G,\mu}^{\dR/\CC_{p}}\times \Fl_{G,\mu}^{\std}/G^{c,\an}_{\CC_{p}}\arrow[d]
\\
(\mX^{\tor}_{K^{p}})^{\pdR/\CC_{p},\sm}_{\log}\arrow[rr] & 
& \Fl_{G,\mu}^{\dR/\CC_{p}}\times \Fl_{G,\mu}^{\std,\pdR/\CC_{p}}/G^{c,\an}_{\CC_{p}}
\end{tikzcd}
\]
where the actions of \( G^{c,\an}_{\CC_{p}} \) on the middle column are induced by \( *_{\GM} \), and all the squares are Cartesian.


(2) We have the following commutative diagram of solid stacks over \( \QQ_{p} \)
\[
\begin{tikzcd}[column sep=0.5cm]
\mX_{K^{p}}^{\la}\arrow[r]\arrow[d] & 
\Per_{G,\mu}/(G^{c,\an},*_{\GM})\arrow[d]\\
\mX_{K^{p}}^{\la,0}\arrow[r]\arrow[d] & 
\Per_{G,\mu}^{\dR/(\Fl_{G,\mu}\times_{\CC_{p}}\Fl^{\std}_{G,\mu})}/(G^{c,\an},*_{\GM})\arrow[r]\arrow[d]
& \Fl_{G,\mu}\times_{\CC_{p}} \Fl_{G,\mu}^{\std}
/G^{c,\an}\arrow[d]\\
\mX_{K^{p}}^{\dR}\arrow[r] & 
\Per_{G,\mu}^{\dR}/(G^{c,\an},*_{\GM})\arrow[r]
& \Fl_{G,\mu}^{\dR}\times_{\bar{\QQ}_{p}} \Fl_{G,\mu}^{\std,\dR}/G^{c,\an},
\end{tikzcd}\]
where all the squares are Cartesian, and \(\mX^{\la,0}_{K^{p}}\subset\mX^{\tor,\la,0}_{K^{p}}\)
is the preimage of \(\mX^{\sm}_{K^{p}}\subset\mX^{\tor,\sm}_{K^{p}}\).
\end{cor}
\begin{proof}
This follows from Proposition \ref{propCartesianLA0}
and Theorem \ref{thmLAstrucGMHTper}.
\end{proof}
We will also need a filtered version:
\begin{cor}\label{corXLa0WithTwoFlag}
We have the following Cartesian diagram of solid stacks over \( \CC_{p} \) \[
\begin{tikzcd}
{\mX^{\tor,\la,0}_{K^{p}}\times [{\mA}^{1}/\GG_{m}]}\arrow[r]\arrow[d,"h^{\la,0,+}"] &\arrow[d,"\beta_{\Fl}\times h^{\sm,+}"]\Fl_{G,\mu}\times 
(\mX^{\tor,\sm}_{K^{p}}\times [{\mA}^{1}/\GG_{m}])
\\
(\mX^{\tor}_{K^{p}})^{\pdR/\CC_{p},{+},\sm}_{\log}\arrow[r] &\Fl_{G,\mu}^{\dR}\times  (\mX^{\tor}_{K^{p}})^{\pdR/\CC_{p},{+},\sm}_{\log}.
\end{tikzcd}
\] 
\end{cor}
\begin{proof}
This
follows from Corollary \ref{corCartesianFlandSMGeneral} (1), since \( \mX^{\tor,\sm}_{K^{p}}\to (\mX^{\tor}_{K^{p}})^{\pdR/\CC_{p},\sm}_{\log} \) factors throug the natural map in Remark \ref{rmkFilDRtoDR}.
\end{proof}

\begin{dfn}\label{dfnBiFilAgLa0}
We define the \emph{smooth bi-filtered algebraic de Rham stack} of \( \mX^{\tor,\la,0}_{K^{p}} \) as the solid stack over \( \CC_{p} \) defined by the following Cartesian diagram
\[
\begin{tikzcd}
(\mX^{\tor,\la,0}_{K^{p}})^{\pdR/\CC_{p},++}_{\log}\arrow[r,"p_{1}"]\arrow[d,"p_{2}"]
& (\mX^{\tor}_{K^{p}})_{\log}^{\pdR/\CC_{p},+,\sm}\arrow[d,"\pi_{\HT}^{\pdR}"]\\
\Fl^{\pdR/\CC_{p},+}_{G,\mu}\arrow[r,"g_{\Fl}^{+}"]
&\Fl^{\dR/\CC_{p}}_{G,\mu}, 
\end{tikzcd}
\] which lies naturally over \( [\mA^{1}/\GG_{m}]^{2} \).
Here \( g_{\Fl}^{+} \) is 
the composition \( \Fl^{\pdR/\CC_{p},+}_{G,\mu}\to \Fl^{\pdR/\CC_{p}}_{G,\mu}\times \mA^{1}/\GG_{m} \) (Remark \ref{rmkFilDRtoDR})
and \( \Fl^{\pdR/\CC_{p}}_{G,\mu}\to \Fl^{\dR/\CC_{p}}_{G,\mu} \).
\end{dfn}
This construction is seemingly ad hoc, but we will see in \cite{Jiang2025HMF} that the Fontaine operator is naturally defined over this space (\cite[Proposition \ref{2-propFontaineIsdRlinear}]{Jiang2025HMF}).
\begin{cor}\label{CordfnBiFilAgLa0}
We have the following Cartesian diagram of solid stacks over \( \CC_{p} \) \[
\begin{tikzcd}
{\mX^{\tor,\la,0}_{K^{p}}\times [{\mA}^{1}/\GG_{m}]^{2}}
\arrow[r,"\pi^{\la,0}_{\HT}\times \pi^{\la,0}_{\sm}"]
\arrow[d,"h^{\la,0,++}"] &
{(\Fl_{G,\mu}\times [\mA^{1}/\GG_{m}])}
\times  
{(\mX^{\tor,\sm}_{K^{p}}\times [\mA^{1}/\GG_{m}])}
\arrow[d,"h_{\Fl}^{+}\times h^{\sm,+}"]
\\ 
(\mX^{\tor,\la,0}_{K^{p}})^{\pdR/\CC_{p},++}_{\log}
\arrow[r,"\pi^{\pdR,++}_{\GM\HT}"] 
& \Fl_{G,\mu}^{\pdR/\CC_{p},+}\times  (\mX^{\tor}_{K^{p}})^{\pdR/\CC_{p},{+},\sm}_{\log}.
\end{tikzcd}
\]  
\end{cor}
\begin{proof}
This follows immediately from Corollary \ref{corXLa0WithTwoFlag} and Lemma \ref{lemSmDRStackProperty} (1).
\end{proof}
\begin{notation}\label{notationFilterCompleteBiFilter}
We will write \[
(\mX^{\tor,\la,0}_{K^{p}})^{\pdR/\CC_{p},\hpp}_{\log}:=
(\mX^{\tor,\la,0}_{K^{p}})^{\pdR/\CC_{p},++}_{\log}\times_{[\mA^{1}/\GG_{m}]^{2}}[\hat{\mA}^{1}/\GG_{m}]^{2},
\] and we will write \( h^{\la,0,\hpp}:=h^{\la,0,++}|_{[\hat{\mA}^{1}/\GG_{m}]^{2}} \). Similarly, \[ \pi^{\pdR,\hpp}_{\GM\HT}:=\pi^{\pdR,++}_{\GM\HT}|_{[\hat{\mA}^{1}/\GG_{m}]^{2}}, h^{\hat{+}}_{\Fl}:=h^{+}_{\Fl}|_{[\hat{\mA}^{1}/\GG_{m}]},h^{\sm,\hat{+}}:=h^{\sm,+}|_{[\hat{\mA}^{1}/\GG_{m}]}. \]
We write \[
(\mX^{\tor,\la,0}_{K^{p}})^{\pdR/\CC_{p}}_{\log}:=(\mX^{\tor,\la,0}_{K^{p}})^{\pdR/\CC_{p},++}_{\log}\times_{[\mA^{1}/\GG_{m}]^{2}}[\GG_{m}/\GG_{m}]^{2}.
\]
\end{notation}

\subsubsection{Differential operators}\label{subsec:Diff-ope}
As a consequence of Corollary \ref{CordfnBiFilAgLa0}, we can construct various (algebraic log) differential operators from the dual BGG complexes on \( \Fl_{G,\mu} \)
and \( \mX^{\tor}_{\Kpp} \) (Proposition \ref{propFilterBGG} and \ref{propBGGFilSV}). 
For this purpose, we will need the following lemma which can be understood as a certain K\"unneth formula.
\begin{lem}\label{lemKunnFlSm}\label{propNablaKWasPullBackDR}
Let \( V \in \QCoh(
\Fl_{G,\mu})\) and let \( W \in \QCoh(\mX^{\tor,\sm}_{K^{p}}) \). 

(1) The natural morphism in \( \QCoh(\Fl_{G,\mu}^{\pdR/\CC_{p},\hat{+}}\times  (\mX^{\tor}_{K^{p}})^{\pdR/\CC_{p},{\hat{+}},\sm}_{\log}) \)
\[
(h^{\hat{+}}_{\Fl})_{*}(V)\boxtimes (h^{\sm,\hat{+}})_{*}(W)\to (h^{\hat{+}}_{\Fl}\times h^{\sm,\hat{+}})_{*}(V\boxtimes W)
\] is an isomorphism.

(2) The natural morphism in \( \QCoh((\mX^{\tor,\la,0}_{K^{p}})^{\pdR/\CC_{p},\hpp}_{\log}) \) 
\begin{align*}
(\pi^{\pdR,\hpp}_{\GM\HT})^{*}\left(
(h^{\hat{+}}_{\Fl})_{*}(V)\boxtimes (h^{\sm,\hat{+}})_{*}(W)\right)\to\\ (h^{\la,0,\hpp})_{*}\left(
	(\pi^{\la,0}_{\HT})^{*}(V)\otimes (\pi^{\la,0}_{\sm})^{*}(W)
\right)
\end{align*}
is an isomorphism.

(3) Let \( M\in\QCoh(\FLT{T}^{\dR/\CC_{p}}) \).
The natural morphism over \( (\mX^{\tor}_{K^{p}})^{\pdR/\CC_{p},\hp,\sm}_{\log} \) \[
h^{\sm,+}_{*}\left(
	V\otimes (\pi^{\sm}_{\HT})^{*}(M)
\right)
\to 
h^{\sm,+}_{*}(V)\otimes (\pi^{\pdR}_{\HT})^{*}(M)
\] is an isomorphism, where the morphisms are as in \[
\pi^{\sm}_{\HT}:\mX^{\tor,\sm}_{K^{p}}\times [\mA^{1}/\GG_{m}]
\xrightarrow{h^{\sm,+}}(\mX^{\tor}_{K^{p}})^{\pdR/\CC_{p},\hp,\sm}_{\log}
\xrightarrow{\pi^{\pdR}_{\HT}}\Fl^{dR/\CC_{p}}_{G,\mu}.
\]
\end{lem} 
\begin{rmk}
As a corollary of Lemma \ref{lemKunnFlSm} (2), differential operators on \( \Fl_{G,\mu} \)
and on \( \mX^{\tor,\sm}_{K^{p}} \) will induces differential operators on \( \mX^{\tor,\la,0}_{K^{p}} \).
\end{rmk}
\begin{proof}
(1)
By Lemma \ref{lemBasicLogDRSta} (1), it suffices to verify that the natural morphism becomes an isomorphism after pulling back along \( h^{\hat{+}}_{\Fl}\times h^{\sm,\hat{+}} \).
By Lemma \ref{lemPerfComplexA1Gm} (1), it further suffices to verify it after pulling back along \(
(B\GG_{m})^{2}\to [\hat{\mA}^{1}/\GG_{m}]^{2}.
\) 
Then by Lemma \ref{lemAlgebHTStack}, we have \[
\Fl_{G,\mu}^{\alg-\mrm{Hodge}/\CC_{p}}=
\Fl_{G,\mu}^{\pdR/\CC_{p},\hat{+}}|_{B\GG_{m}}\cong (\Fl_{G,\mu}\times B\GG_{m})/\hat{T}_{\Fl_{G,\mu}}\{-1\},
\] where \( \hat{T}_{\Fl_{G,\mu}}\{-1\} \) denotes the formal completion of 
\( T_{\Fl_{G,\mu}}\{-1\} \) at the zero section.

Similarly, let \( (\mX^{\tor}_{K^{p}})^{\alg-\mrm{Hodge}/\CC_{p},\sm}_{\log}:=
(\mX^{\tor}_{K^{p}})^{\pdR/\CC_{p},\hat{+},\sm}_{\log}
|_{B\GG_{m}} \). By applying Lemma \ref{lemAlgebHTStack} to \( \mX^{\tor}_{\Kpp} \) and taking limit along \( K_{p} \), by Lemma \ref{lemSmDRStackProperty}, we have \[(\mX^{\tor}_{K^{p}})^{\alg-\mrm{Hodge}/\CC_{p},\sm}_{\log}
\cong (\mX^{\tor,\sm}_{K^{p}}\times B\GG_{m})/\hat{T}^{\sm}_{\mX^{\tor}_{K^{p}},\log}\{-1\},
\] with \( \hat{T}^{\sm}_{\mX^{\tor}_{K^{p}},\log}\{-1\}:=\varprojlim_{K_{p}}\hat{T}_{\mX^{\tor}_{\Kpp},\log}\{-1\} \).

Therefore, 
\begin{align}\label{alignRHSQuotientVB}
	&\Fl_{G,\mu}^{\alg-\mrm{Hodge}/\CC_{p}}\times 
(\mX^{\tor}_{K^{p}})^{\alg-\mrm{Hodge}/\CC_{p},\sm}_{\log}\cong \\
&
\left((\Fl_{G,\mu}\times B\GG_{m})\times
(\mX^{\tor,\sm}_{K^{p}}\times B\GG_{m})\right)/\left(p_{1}^{*}(T_{\Fl_{G,\mu}}\{-1\})\oplus p_{2}^{*}(T^{\sm}_{\mX^{\tor}_{K^{p}},\log}\{-1\})\right)^{\wedge}
\end{align} where \( (-)^{\wedge} \) denotes the formal completion along the zero section.

Thus using Lemma \ref{lemSmBaseChange} (1),
\begin{align*}
(h^{\hat{+}}_{\Fl})^{*}(h^{\hat{+}}_{\Fl})_{*}(V)|_{\Fl_{G,\mu}\times B\GG_{m}}\cong 
(h^{\hat{+}}_{\Fl}|_{B\GG_{m}})^{*}(h^{\hat{+}}_{\Fl}|_{B\GG_{m}})_{*}(V)\\
\in \QCoh(\Fl_{G,\mu}\times B\GG_{m})\cong \gr^{\ZZ}(\QCoh(\Fl_{G,\mu}))
\end{align*}
for \( h^{\hat{+}}_{\Fl}|_{B\GG_{m}}:\Fl_{G,\mu}\times B\GG_{m}\to \Fl_{G,\mu}^{\alg-\mrm{Hodge}/\CC_{p}} \),
where the isomorphism is given by Lemma \ref{lemSmBaseChange} as \( B\GG_{m}\to \mA^{1}/\GG_{m} \) is cohomologically smooth. As a consequence, 
\[
\gr^{i}((h^{\hat{+}}_{\Fl})^{*}(h^{\hat{+}}_{\Fl})_{*}(V))\cong \Sym^{i}(\Omega^{1}_{\Fl_{G,\mu}})\otimes V\in \QCoh(\Fl_{G,\mu}).
\] 
Similarly, \[
\gr^{j}((h^{\hat{+}}_{\Fl})^{*}(h^{\hat{+}}_{\Fl})_{*}(W))\cong \Sym^{j}(\Omega^{1,\sm}_{\mX^{\tor}_{K^{p}},\log})\otimes W\in \QCoh(\mX^{\tor,\sm}_{K^{p}}).
\] with \( \Omega^{1,\sm}_{\mX^{\tor}_{K^{p}},\log}:=\pi^{\sm}_{K_{p}}\Omega^{1}_{\mX^{\tor}_{\Kpp},\log} \).

Similarly, by Corollary \ref{corFilterdatA1modGm}, we can identify \[\QCoh(
(\Fl_{G,\mu}\times B\GG_{m})\times
(\mX^{\tor,\sm}_{K^{p}}\times B\GG_{m}))\cong \gr^{\ZZ\times\ZZ} \QCoh(\Fl_{G,\mu}\times \mX^{\tor,\sm}_{K^{p}})
\] where \( \gr^{\ZZ\times\ZZ}(-):=\mrm{Fun}((\ZZ\times\ZZ)^{ds},-) \) with \((\ZZ\times\ZZ)^{ds}\) is the discrete category corresponds to the set \(\ZZ\times\ZZ\). Then again by Lemma \ref{lemSmBaseChange} (1) and (\ref{alignRHSQuotientVB}), \[
\gr^{(i,j)}((h^{\hat{+}}_{\Fl}\times h^{\sm,\hat{+}})_{*}(V\boxtimes W))\cong (\Sym^{i}(\Omega^{1}_{\Fl_{G,\mu}})\otimes W)\boxtimes (\Sym^{j}(\Omega^{1,\sm}_{\mX^{\tor}_{K^{p}},\log})\otimes W).
\] This shows that the natural morphism is an isomorphism on \( \gr^{(i,J)} \) for any \( i,j \), and thus is an isomorphism.

(2) As in the proof of (1), it suffices to verify that the natural morphism becomes an isomorphism after pulling back to \( \mX^{\tor,\la,0}_{K^{p}}\times (B\GG_{m})^{2} \). Using (\ref{alignRHSQuotientVB}) and the Cartesian diagram in Corollary \ref{CordfnBiFilAgLa0}, we know that 
the right vertical arrow admits a splitting, and thus so does the left vertical arrow, which implies that
\begin{align*}
&(\mX^{\tor,\la,0}_{K^{p}})^{\pdR/\CC_{p},++}_{\log}|_{(B\GG_{m})^{2}}
\cong \\&\left(\mX^{\tor,\la,0}_{K^{p}}\times (B\GG_{m})^{2}\right)/\left(
(\pi^{\la,0}_{\HT})^{*}(\Omega^{1}_{\Fl_{G,\mu}})
\oplus (\pi^{\la,0}_{\sm})^{*}(\Omega^{1,\sm}_{\mX^{\tor}_{K^{p}},\log})
\right)^{\wedge}.
\end{align*}
By Corollary \ref{corFilterdatA1modGm}, we can identify \[\QCoh
(\mX^{\tor,\la,0}_{K^{p}}\times (B\GG_{m})^{2})\cong \gr^{\ZZ\times\ZZ} \QCoh(\mX^{\tor,\la,0}_{K^{p}}),
\] and then by Lemma \ref{lemSmBaseChange} and the argument in (1), 

\begin{align*}&
\gr^{(i,j)}\left((h^{\la,0,\hpp})^{*}(h^{\la,0,\hpp})_{*}
	\left((\pi^{\la,0}_{\HT})^{*}(V)\otimes (\pi^{\la,0}_{\sm})^{*}(W)\right)
\right)\\
&\cong (\pi^{\la,0}_{\HT})^{*}(V\otimes \Sym^{i}\Omega^{1}_{\Fl_{G,\mu}})\otimes
(\pi^{\la,0}_{\sm})^{*}(W\otimes \Sym^{j}\Omega^{1,\sm}_{\mX^{\tor}_{K^{p}},\log})\\
& \cong \gr^{(i,j)}\left((h^{\la,0,\hpp})^{*}
(\pi^{\pdR,\hpp}_{\GM\HT})^{*}\left(
(h^{\hat{+}}_{\Fl})_{*}(V)\boxtimes (h^{\sm,\hat{+}})_{*}(W)\right)\right)
\end{align*} where the last isomorphism is given by the proof of (1). This shows that the natural morphism is indeed an isomorphism.

(3) Again it suffices to verify the isomorphism after applying \( \gr^{\bullet} \). The proof is straight-forward using the description of \( (\mX^{\tor}_{K^{p}})^{\alg-\mrm{Hodge}/\CC_{p}}_{\log} \) above.
\end{proof}
We will also need an analytic version. For this, we work away from the toroidal boundaries:
\begin{lem}\label{lemDiffOpAlgvsAn}
(1)
We have a commutative diagram of solid stacks: 
\[
\begin{tikzcd}
	{\mX^{\la,0}_{K^{p}}\times_{\CC_{p}} [\hat{\mA}^{1}/\GG_{m}]^{2}}
\arrow[r,"\pi^{\la,0}_{\HT}\times\pi^{\la,0}_{\sm}"]
\arrow[d,"h^{\la,0,\hpp}"]
&
{(\Fl_{G,\mu}\times_{\CC_{p}} [\hat{\mA}^{1}/\GG_{m}])}
\times_{\CC_{p}}  
{(\mX^{\sm}_{K^{p}}\times_{\CC_{p}} [\hat{\mA}^{1}/\GG_{m}])}
\arrow[d,"h_{\Fl}^{\hat{+}}\times h^{\sm,\hat{+}}"]
\arrow[dd,"\beta_{\Fl}^{\hat{+}}\times \beta^{\sm,\hat{+}}",bend left=83]
\\ 
(\mX^{\la,0}_{K^{p}})^{\pdR/\CC_{p},\hpp}
\arrow[r,"\pi^{\pdR,\hpp}_{\GM\HT}"] \arrow[d,"g^{\la,0,\hpp}"]
& \Fl_{G,\mu}^{\pdR/\CC_{p},\hat{+}}\times_{\CC_{p}}  \mX_{K^{p}}^{\pdR/\CC_{p},{\hat{+}},\sm}\arrow[d,"g_{\Fl}^{\hat{+}}\times g^{\sm,\hat{+}}"]
\\
\mX^{\dR,\sm}_{K^{p}}\arrow[r,"\pi_{\HT}^{\dR}\times 1"]
&
\Fl_{G,\mu}^{\dR}\times_{\bar{\QQ}_{p}} \mX^{\dR,\sm}_{K^{p}},
\end{tikzcd}
\] in which both squares are Cartesian. 
Here 
\( \mX^{\dR,\sm}_{K^{p}}:=\varprojlim_{K_{p}}\mX_{\Kpp}^{\dR} \), and
\( \beta_{\Fl}^{\hat{+}} \) (resp. \( \beta^{\sm,\hat{+}} \)) is induced by \( \beta_{\Fl}:\Fl_{G,\mu}\to \Fl_{G,\mu}^{\dR} \) (resp. \( \beta^{\sm}:\mX^{\sm}_{K^{p}}\to \mX^{\dR,\sm}_{K^{p}} \)). 
Let \( \beta^{\la,0} \) be the composition \[ \beta^{\la,0}:\mX^{\la,0}\xrightarrow{\pi^{\la,0}_{\sm}}\mX^{\sm}_{K^{p}}\xrightarrow{\beta^{\sm}}\mX^{\dR,\sm}_{K^{p}}, \] and then \( g^{\la,0,\hpp}\circ h^{\la,0,\hpp}\cong \beta^{\la,0}\circ p_{1} \).


(2)
For any \( V\in \QCoh(\Fl_{G,\mu}) \)
and \( W\in \QCoh(\mX^{\sm}_{K^{p}}) \), then we have the following natural commutative diagram
\[
\begin{tikzcd}[column sep=-0.1cm]
g^{\la,0,\hpp}_{*} 
(\pi^{\pdR,\hp\hp}_{\GM\HT})^{*}\left(
h^{\hp}_{\Fl,*}(V)\boxtimes h^{\sm,\hp}_{*}(W)\right)\arrow[r]\arrow[d]
& (\pi^{\dR}_{\HT})^{*}
\beta_{\Fl,*}(V)\otimes \beta^{\sm,\hat{+}}_{*}(W)\arrow[d]
\\
(\beta^{\la,0}\circ p_{1})_{*}\left(
	(\pi^{\la,0}_{\HT})^{*}(V)\otimes (\pi^{\la,0}_{\sm})^{*}(W)
\right)\arrow[r]
&
\beta^{\la,0}_{*}\left(
	(\pi^{\la,0}_{\HT})^{*}(V)\otimes (\pi^{\la,0}_{\sm})^{*}(W)
\right)
\end{tikzcd}
\] where the left vertical arrow is as in Lemma \ref{lemKunnFlSm} (2). Moreover, all the arrows are isomorphisms.
\end{lem}
\begin{rmk}
For comparing with the Fontaine operator (\cite[Theorem \ref{2-thmFontainEqDDbar}]{Jiang2025HMF}), it is better to use the algebraic version, while for proving locally analytic Jacquet-Langlands correspondence (\cite[Theorem \ref{2-thmJacquetLanglandsLA}]{Jiang2025HMF}),
one has to work with the analytic version. Lemma \ref{lemDiffOpeToLogStack} (2) shows that the two versions will give the same constructions.
\end{rmk}
\begin{proof}
(1)
The morphism \( g^{\hat{+}}_{\Fl} \)
is defined to be the composition \[
\Fl_{G,\mu}^{\pdR/\CC_{p},\hat{+}}
\to \Fl_{G,\mu}^{\pdR/\CC_{p}}\times [\hat{\mA}^{1}/\GG_{m}]
\to \Fl_{G,\mu}^{\pdR/\CC_{p}}
\to \Fl_{G,\mu}^{\dR},
\] where the first morphism 
is as in Remark \ref{rmkFilDRtoDR}, and the last is as in Proposition \ref{propAnToAlgDRJuan} (1).
\( g^{\sm,\hat{+}} \) is defined similarly, and \( g^{\la,0,\hpp} \)
is induced by \( g^{\sm,\hat{+}} \).
Then the lower diagram commutes by Definition \ref{dfnBiFilAgLa0}.

By Corollary \ref{corCartesianFlandSMGeneral} (2), the larger square is Cartesian. By Corollary \ref{CordfnBiFilAgLa0}, the upper square is Cartesian, whose vertical arrows are \( ! \)-surjections by Lemma \ref{lemBasicLogDRSta}. These imply that the lower square is also Cartesian.

(2) The left vertical arrow is an isomorphism by Lemma \ref{lemKunnFlSm} (2). The right vertical arrow is an isomorphism by Lemma \ref{lemKunnethProper} as \( X\to X^{\dR} \) is cohomologically co-smooth by Theorem \ref{thmAnDRStackGeneral} (3) and Proposition \ref{propAffineProperCover}.
We are left to shown that the lower horizontal arrow is an isomorphism. This follows from  Corollary \ref{corFilterdatA1modGm}, which implies that for any solid stack \( X \) and \( \mF\in\QCoh(X) \) with \( p_{1}:X\times [\hat{\mA}^{1}/\GG_{m}]\to X \), the natural morphism \(\mF\to  p_{1,*}p_{1}^{*}(\mF) \) is an isomorphism.
\end{proof}

\subsection{\texorpdfstring{\(\BdR^{+}\)}{BdR+}-infinite-level Shimura varieties}\label{subsecBdRSV}
We will explain how to interpretate the logarithmic Riemann-Hilbert correspondence of \cite{DLLZ2022logarithmicJAMS} in terms of de Rham stacks. We start by constructing a \( \BdR^{+} \)-thickening of \( \mX^{\tor}_{K^{p}} \) of Lemma \ref{lemTorCompacAsTateStack}. 

As in the construction of \(\mX^{\tor}_{K^{p}}\) in Lemma \ref{lemTorCompacAsTateStack},
the following argument can be greatly simplified if we assume that the \(\mX^{\tor,\diamond}_{K^{p}}\) is represented by a perfectoid space with infinitely divisible ramifications at the toroidal boundary. 

\begin{lem}\label{lemBdRXtorKp}
Let \((U_{K_{p,0}},i:U_{K_{p,0}}\to\mbb{B}^{n},U_{K_{p}},U_{K_{p},\infty},\ti{U}_{\infty},\Gamma_{n})\) as in Construction \ref{constructionUChartTower} and Lemma \ref{lemTorCompacAsTateStack}. 
We can take \(U_{K_{p,0}}\) to be affinoid, and \(\ti{U}_{\infty}\cong \Spa(B_{\infty},B_{\infty}^{+})\).
Then we define the solid stacks \begin{align*}
\BdR^{+}/t^{i}(\ti{U})^{+}&:=\left[
	\AnSpec^{+}(\BdR^{+}(B_{\infty})/t^{i})/\underline{\Gamma_{n}}
\right],\\
\OBdRlog^{+}/\Fil^{i}(\ti{U})^{+}&:=\left[
	\AnSpec^{+}(\OBdRlog^{+}(B_{\infty})/\Fil^{i})/\underline{\Gamma_{n}}
\right],
\end{align*}
where the RHS are as in Definition \ref{dfnSpfFunctor}.

Then 

(1) \(\BdR^{+}/t^{i}(\ti{U})^{+}\) and \(\OBdRlog^{+}/\Fil^{i}(\ti{U})^{+}\) are independent of choice of the toric chart \(i\).

(2) For \(U_{K_{p,0}}^{(1)}\subset U_{K_{p,0}}^{(2)}\), the corresponding construction induces open immersions \(\BdR^{+}/t^{i}(\ti{U}^{(1)})^{+}\hookrightarrow \BdR^{+}/t^{i}(\ti{U}^{(2)})^{+}\) and \(\OBdRlog^{+}/\Fil^{i}(\ti{U}^{(1)})^{+}\hookrightarrow \OBdRlog^{+}/\Fil^{i}(\ti{U}^{(2)})^{+}\).

In particular, we can glue \(\BdR^{+}/t^{i}(\ti{U})^{+}\)
to obtain a Tate stack \(\BdR^{+}/t^{i}(\mX^{\tor}_{K^{p}})^{+}\), which we refer to as the \emph{\(i\)-truncated \(\BdR^{+}\)-infinite-level Shimura varieties} of tame level \(K^{p}\). 

(3) 
We define \[
\BdR^{+}(\mX^{\tor}_{K^{p}})^{+}:=\varinjlim_{i}\BdR^{+}/t^{i}(\mX^{\tor}_{K^{p}})^{+},\;\OBdR^{+}(\mX^{\tor}_{K^{p}})^{+}:=\varinjlim_{i}\OBdR^{+}/\Fil^{i}(\mX^{\tor}_{K})^{+}.
\] We write \( \BdR^{+}(\mX^{\tor}_{K^{p}}) \)
and \( \OBdR^{+}(\mX^{\tor}_{K^{p}}) \)
for their restrictions to \( [\GG_{m}/\GG_{m}]\subset [\mA^{1}/\GG_{m}] \).
Then there is a commutative diagram \[
\begin{tikzcd}[column sep=0.3cm]
\mX_{K^{p}}^{\tor}\times{ [\mA^{1}/\GG_{m}]}\arrow[r]\arrow[rr,bend left,"\pi_{\sm}"] &
\OBdRlog^{+}(\mX^{\tor}_{K^{p}})^{+}
\arrow[r,"\ti{\pi}_{\sm}"]\arrow[d] & \mX^{\tor,\sm}_{K^{p},\bar{\QQ}_{p}}\times { [\mA^{1}/\GG_{m}]}\arrow[d,"h^{\sm}"]\\
&
\BdR^{+}(\mX^{\tor}_{K^{p}})^{+}
\arrow[r,"\pi_{\sm}^{\pdR}"] & (\mX^{\tor}_{K^{p},\bar{\QQ}_{p}})^{\pdR/\bar{\QQ}_{p},+,\sm}_{\log}\arrow[r,"s"]&{\mA^{1}/\GG_{m}},
\end{tikzcd}
\] 
where the square becomes Cartesian after restricting to \( [\hat{\mA}^{1}/\GG_{m}] \).
\end{lem}
\begin{proof}
For (1) and (2),
following the proof of Lemma \ref{lemTorCompacAsTateStack}, it reduces to Lemma \ref{lemBdRTorsor}.
For (3), we use Lemma \ref{lemBdRMapToLogDRStack}. 
\end{proof}
\begin{dfn}\label{dfnPeriodShvSV}
(1)
We define 
\begin{align*}
\BdRKp^{+}&:=(\pi_{\sm}^{\pdR})_{*}(\mO_{\BdR^{+}(\mX^{\tor}_{K^{p}})^{+}})\in \QCoh((\mX^{\tor}_{K^{p},\bar{\QQ}_{p}})^{\pdR/\bar{\QQ}_{p},\hat{+},\sm}_{\log}),\\
\OBdRlogKp^{+}&:=(\ti{\pi}_{\sm})_{*}(\mO_{\OBdRlog^{+}(\mX^{\tor}_{K^{p}})^{+}})
\in \QCoh(\mX^{\tor,\sm}_{K^{p},\bar{\QQ}_{p}}\times [\hat{\mA}^{1}/\GG_{m}]),
\end{align*} and for \(a\le b\), put \( \BdRKp^{[a,b]}:=t^{a}\BdRKp^{+}/t^{b}\BdRKp^{+} \).
We further define \[\BdRKp:=\BdRKp^{+}[1/t],\;
\OBdRlogKp:=\OBdRlogKp^{+}[1/t]
\]
and \(\OBdRlogKp^{[a,b]}:=\Fil^{a}(\OBdRlogKp)/\Fil^{b}(\OBdRlogKp)\).

(2) We define 
\begin{align*}
\BdR^{\la}&:=\BdR^{R-\la}\in\QCoh((\mX^{\tor}_{K^{p},\bar{\QQ}_{p}})^{\pdR/\bar{\QQ}_{p},\hat{+},\sm}_{\log}),\\
\OBdRlogKp^{\la}&:=\OBdRlogKp^{R-\la}\in\QCoh(\mX^{\tor,\sm}_{K^{p},\bar{\QQ}_{p}}\times[\hat{\mA}^{1}/\GG_{m}]),
\end{align*}
where \((-)^{R-\la}\) is as in Notation \ref{notationRlaConvention}. 

For any \( E_{\p}\subset L\subset \bar{\QQ}_{p}\), by abuse of notation, we will denote by  
the same notation the push-forward of the corresponding sheaves to \( \mX^{\tor,\sm}_{K^{p},L} \times [\hat{\mA}^{1}/\GG_{m}] \) 
or \( (\mX^{\tor}_{K^{p},L})_{\log}^{\pdR/L,\hat{+},\sm} \).  
%
\end{dfn}

\begin{lem}\label{lemBdRLaProperty}
Let \( L\subset\bar{\QQ}_{p} \) 
be a finite extension of \( E_{\p} \). 
We denote by \( \hat{h}^{\sm} \) the following composition
\[ \hat{h}^{\sm}:\mX^{\tor,\sm}_{K^{p}}\times [\hat{\mA}^{1}/\GG_{m}]\to \mX^{\tor,\sm}_{K^{p},L}\times [\hat{\mA}^{1}/\GG_{m}]\xrightarrow{\hat{h}^{\sm}_{L}} (\mX_{K^{p},L}^{\tor})^{\pdR/L,\hat{+},\sm}_{\log}. \]

Then \( \BdRKp \), \( \BdRKp^{\la} \), \( \BdRKp^{R-\widetilde{K_{p}}-\sm} \), \(D_{\sen,L}(\BdRKp^{\la})\)
are complete with respect to \( t^{\bullet} \), and the graded pieces are as follows: for \( i\in\ZZ \),  \begin{align*}
\gr^{i}_{t^{\bullet}}\BdRKp&\cong \hat{h}^{\sm}_{*}\hat{\mO}_{K^{p}}\{i\}(i),
&\gr^{i}_{t^{\bullet}}\BdRKp^{R-\widetilde{K_{p}}-\sm}\cong \hat{h}^{\sm}_{*}\mO^{R-\widetilde{K_{p}}-\sm}_{K^{p}}\{i\}(i),
\\
\gr^{i}_{t^{\bullet}}\BdRKp^{\la}&\cong \hat{h}^{\sm}_{*}\mO^{\la}_{K^{p}}\{i\}(i),
&\gr^{i}_{t^{\bullet}}D_{\sen,L}(\BdRKp^{\la})\cong \hat{h}^{\sm}_{*}D_{\sen,L}(\mO^{\la}_{K^{p}})\{i\}(i),
\end{align*}
Here \( \{i\} \) 
denotes the twist of the filtration (Notation \ref{notationTwistByAGm}), \( (i) \) denotes the Tate twist, and \( D_{\sen,L} \) 
is as in Definition \ref{dfnDsenFunctor}.

Moreover, applying \( (\hat{h}^{\sm}_{L})^{*} \), we have 
\begin{align*}
(\hat{h}^{\sm}_{L})^{*}\BdRKp\cong \OBdRlogKp,
(\hat{h}^{\sm}_{L})^{*}\BdRKp^{R-\widetilde{K_{p}}-\sm}\cong \OBdRlogKp^{R-\widetilde{K_{p}}-\sm},\\
(\hat{h}^{\sm}_{L})^{*}\BdRKp^{\la}\cong \OBdRlogKp^{\la}.
\end{align*}
\end{lem}
\begin{proof}
We can localize to the open subspace \( (\ti{U}_{L})^{\pdR/L,\hat{+},\sm}_{\log}\subset (\mX_{K^{p},L}^{\tor})^{\pdR/L,\sm}_{\log} \), and reduce to the setting of Lemma \ref{lemQCohSMsv} (3). Then \[\QCoh((\ti{U}_{L})^{\pdR/L,\hat{+},\sm}_{\log})
\cong \Mod_{\dR_{\log}(A^{\sm}_{L})}(\widehat{\Fil}^{\ZZ}(D(L_{\square}))),\]
and by construction in Lemma \ref{lemBdRXtorKp}, \( \BdRKp|_{(\ti{U}_{L})^{\pdR/L,\hat{+},\sm}_{\log}} \) corresponds to \( (R\Gamma(\unl{\Gamma}_{n},\BdR(U_{\infty})),t^{\bullet}) \) 
regarded as a filtered complete \( \dR_{\log}(A^{\sm}_{L}) \)-module. 

\( \BdRKp \) 
is complete with respect to \( t^{\bullet} \) by definition.  
The description of \( \gr^{i}_{t^{\bullet}}\BdRKp \) 
follows from Lemma \ref{lemGrBdRonDRStack}. \( (\hat{h}^{\sm}_{L})^{*}\BdRKp\cong \OBdRKp \)
follows from applying smooth base change (Lemma \ref{lemSmBaseChange}) to 
the Cartesian square in Lemma \ref{lemBdRXtorKp} (3).

For the other statement, 
note that the forgetful functor \[ \Mod_{\dR_{\log}(A^{\sm}_{L})}(\widehat{\Fil}^{\ZZ}(D(L_{\square})))\to \widehat{\Fil}^{\ZZ}(D(L_{\square})) \]
is conservative, and compatible with the functors \( (-)^{R-\la},\;(-)^{R-\widetilde{K_{p}}-\sm},\;D_{\sen,L}(-) \). Thus the \( t^{\bullet} \)-completeness and the description of various \( \gr^{i}_{t^{\bullet}} \) follows from the corresponding statements for \( \BdRKp \).

Finally, 
\( (\hat{h}_{L}^{\sm})^{*} \) corresponds to the base change \( -\hatotimes_{\dR_{\log}(A^{\sm}_{L})}A^{\sm}_{L} \). This functor commutes with \( (-)^{R-\widetilde{K_{p}}-\sm} \) since the latter can be calculated by Lazard-Serre's resolution (see \cite[Theorem 5.7]{JRC2021solid}).
It also commutes with \( (-)^{R-\la} \), since \( (-)^{R-\la}=(-)^{R-\widetilde{K_{p}}-\sm}\circ (-\otimes_{L}\mO_{G^{c},1}) \). 
\end{proof}
\begin{cor}\label{corAdmitFontaineBdRLa}
The natural morphism \[ D_{\sen,L}(\BdRKp^{\la})\otimes_{s^{*}(L_{\infty}((t)))}s^{*}(B_{\dR})\to \BdRKp^{\la} \] 
is an isomorphism, where \( s \) 
denotes the structure map \[
s:(\mX^{\tor}_{K^{p},L})^{\pdR/L,\hat{+},\sm}_{\log}\to [\hat{\mA}^{1}/\GG_{m}].
\]
\end{cor}
\begin{proof}
We can again localize as in the proof of Lemma \ref{lemBdRLaProperty}, and apply Lemma \ref{lemLogDRStackIsAffine}. Then \(( s|_{\ti{U}_{L}})^{*} \) can be described as  
\[( s|_{\ti{U}_{L}})^{*}:\widehat{\Fil}^{\ZZ}(D(L_{\square}))\to \Mod_{\dR_{\log}(A^{\sm}_{L})}(\widehat{\Fil}^{\ZZ}(D(L_{\square}))),\;M\mapsto \dR_{\log}(A^{\sm}_{L})\otimes M. \] 
Thus it suffices to show that 
\[ ( s|_{\ti{U}_{L}})_{*}(D_{\sen,L}(\BdRKp^{\la}))\otimes_{L_{\infty}((t))}B_{\dR}\to  ( s|_{\ti{U}_{L}})_{*}(\BdRKp^{\la})  \] 
is an isomorphism. By filter completeness, it suffices to compare both sides after taking \( \gr^{*}(-) \). 
By Lemma \ref{lemBdRLaProperty}, \[ i^{*}(( s|_{\ti{U}_{L}})_{*}(D_{\sen,L}(\BdRKp^{\la})))\cong D_{\sen,L}(\mO^{\la}_{K^{p}})((t)),i^{*}(( s|_{\ti{U}_{L}})_{*}(\BdRKp^{\la}))\cong \mO^{\la}_{K^{p}}((t)), \] 
where the right hand sides are graded with the degree. 
We also have \( \;i^{*}(B_{\dR})\cong \CC_{p}((t)) \). Thus we conclude by Theorem \ref{thmOlaConcentrated0}.
\end{proof}

\subsection{Logarithmic Riemann-Hilbert correspondence}\label{subsecRiemannHilbert}
The following is a reformulation of logarithmic Riemann-Hilbert correspondence of \cite{DLLZ2022logarithmicJAMS} over Shimura varieties:

\begin{thm}\label{thmReformulateRHCorr}
For any \(V\in \Rep(G^{c})\), then we have a \( G(\QQ_{p})^{\sm}\times \unl{\Gal_{E_{\p}}} \)-equivariant isomorphism in \( \QCoh((\mX^{\tor}_{K^{p},E_{\p}})^{\pdR/E_{\p},\hat{+},\sm}_{\log} ) \)  \begin{align*}
(\pi^{\sm,\pdR}_{K_{p}})^{*}\bar{\mcal{G}}^{c}_{\dR,\Kpp}(V)\otimes s^{*}(B_{\dR})&\cong \RHom_{\fg^{c}}(V^{\vee},\BdRKp^{\la}),\\
(\pi^{\sm,\pdR}_{K_{p}})^{*}\bar{\mcal{G}}^{c}_{\dR,\Kpp}(V)\otimes_{E_{\p}} E_{\p,\infty}&\cong \RHom_{\fg^{c}}(V^{\vee},E_{0}(D_{\sen,E_{\p}}(\BdRKp^{\la}))),
\end{align*}
where the morphisms are as in
\[ [\hat{\mA}^{1}/\GG_{m}]_{E_{\p}}\xleftarrow{s} (\mX^{\tor}_{K^{p},E_{\p}})^{\pdR/E_{\p},\hat{+},\sm}_{\log}
\xrightarrow{\pi^{\sm,\pdR}_{K_{p}}} (\mX^{\tor}_{K^{p}K_{p},E_{\p}})^{\pdR/E_{\p},\hat{+}}_{\log}, \]
\(\bar{\mcal{G}}^{c}_{\dR,\Kpp}\) is 
as in Lemma \ref{lemDRtorsorDescends}, and \( E_{0}\circ D_{\sen,E_{\p}} \) 
is as in Definition \ref{dfnDsenFunctor} and Definition \ref{dfnE0functor}.
\end{thm}

Recall the following fact about \(\OBdRlog\):
\begin{prop}[{\cite[Proposition 6.16]{Scholze13}, \cite[Lemma 3.3.2]{DLLZ2022logarithmicJAMS}}]
	\label{propPushForwardOBdR}
Let \(X\) be a log smooth rigid variey over \(L\subset\bar{\QQ}_{p}\).
Consider \( \nu:X_{\CC_{p},\proket}\to X_{\CC_{p},\an} \). Then for any \(a<b\), we have an isomorphism compatible with connections \[R\nu_{*}(\Fil^{a}\OBdRlog/\Fil^{b}\OBdRlog)
\cong \mO_{X}\hatotimes_{L}(t^{a}B_{\dR}^{+}/t^{b}B_{\dR}^{+}).
\] 
\end{prop}
\begin{proof}
The isomorphism in \cite[Lemma 3.3.2]{DLLZ2022logarithmicJAMS} is compatible with connections, 
because there is a natural map from the RHS to the LHS that is compatible with connections.
\end{proof}
\begin{cor}\label{corCompareStructureSheaf}
Consider \( \nu_{K_{p}}:\mX^{\tor}_{\Kpp,\proket}\to \mX^{\tor}_{\Kpp,\an} \). 
We will work with sheaves valued in \( \widehat{\Fil}^{\ZZ}(D(E_{\p,\square})) \). 

Then for any \(V\in\Rep(G^{c})\), we have a filtered \( \Gal_{E_{\p}} \)-equivariant  isomorphism that is compatible with connections  \[
\mcal{G}_{\dR,K_{p}}^{c}(V)\otimes_{E_{\p}}B_{\dR}\cong R\nu_{K_{p},*}(\unl{V}\otimes\OBdRlog),
\] where \( \unl{V} \) comes from descent along the pro-Kummer \'etale tower \(\{\mX^{\tor}_{K^{p}K_{p}',E_{\p}}\}\), and 
\( \mcal{G}_{\dR,K_{p}}^{c} \) is as in Notation \ref{notationDRtorsor}. 
\end{cor}
\begin{proof}
    By \cite[Theorem 5.3.1]{DLLZ2022logarithmicJAMS}, we have an isomorphism as pro-Kummer \'etale sheaves (valued in \( \widehat{\Fil}^{\ZZ}(D(E_{\p,\square})) \)) on \( \mX^{\tor}_{\Kpp,E_{\p}} \)  \[
\unl{V}\otimes \OBdRlog \cong \mcal{G}_{\dR,K_{p}}^{c}(V)\otimes_{\mO_{\mX^{\tor}_{\Kpp}}} \OBdRlog
\] that is compatible with log connections. 
Then applying \( R\nu_{K_{p},*} \) 
on both sides, we obtain 
\begin{align*}
R\nu_{K_{p},*}(\unl{V}\otimes \OBdRlog)\cong \mcal{G}_{\dR,K_{p}}^{c}(V)\otimes_{\mO_{\mX^{\tor}_{\Kpp}}}R\nu_{K_{p},*}(\OBdRlog)\\\cong 
\mcal{G}_{\dR,K_{p}}^{c}(V)\otimes_{E_{\p}}B_{\dR},
\end{align*}
where the last isomorphism is given by Proposition \ref{propPushForwardOBdR}.
\end{proof}

\begin{proof}[Proof of Theorem \ref{thmReformulateRHCorr}]

Let \((U_{K_{p,0}},i:U_{K_{p,0}}\to\mbb{B}^{n},U_{K_{p}},U_{K_{p},\infty},\ti{U}_{\infty},\Gamma_{n})\) as in Construction \ref{constructionUChartTower} and Lemma \ref{lemTorCompacAsTateStack}, where \(U_{K_{p,0}}\subset\mX^{\tor}_{K_{p,0}}\) descends to \( U_{K_{p,0},L} \) over \(L\subset \CC_{p}\) for a finite extension \(L/E_{\p}\).
Assume that \(U_{K_{p,0},L}\cong \Spa(A_{K_{p},L},A_{K_{p},L}^{+})\), and let \(U_{K_{p},L}\subset \mX^{\tor}_{\Kpp,L}\)
be the preimage of \( U_{K_{p,0},L}\subset \mX^{\tor}_{K^{p}K_{p,0},L} \).

Then by descent, for \(\nu_{K_{p}}:\mX^{\tor}_{\Kpp,\proket}\to \mX^{\tor}_{\Kpp,\an}\), \( R\nu_{K_{p},*}(V\otimes\OBdRlog)|_{U_{K_{p}}} \) 
coincides with 
the analytic sheafification of the presheaf valued in \( \widehat{\Fil}^{\bullet}(D(L_{\square})) \) 
\begin{align*}
&(U'_{K_{p},L}\subset U_{K_{p},L})\mapsto R\Gamma(U_{K_{p},\proket}',\unl{V}\otimes\OBdRlog)\\&\cong R\Gamma(\Gamma_{n}\times \widetilde{K_{p}},V\otimes\OBdRlog(\ti{U}'_{\infty}))\cong R\Gamma(\widetilde{K_{p}},V\otimes R\Gamma(\OBdRlog(\ti{U}')^{+})[1/t]),
\end{align*}
where \(U'_{K_{p}}\subset U_{K_{p}}\), \(\ti{U}'_{\infty}\subset \ti{U}_{\infty}\) and \(\OBdRlog(\ti{U}')^{+}\subset \OBdRlog(\ti{U})^{+}\) are the preimages of \(U'_{K_{p},L}\subset U_{K_{p},L}\).

By Proposition \ref{propPushForwardOBdR}, \( R\nu_{K_{p},*}(\unl{V}\otimes\OBdRlog)\cong\mcal{G}^{c}_{\dR,K_{p}}(V)\otimes_{L}B_{\dR}\). 
Taking colimit along \(K_{p}\) and gluing along \( U_{K_{p,0}} \), 
we have \[
(\pi^{\sm}_{K_{p}})_{*}(V\otimes\OBdRlogKp)^{R-\widetilde{K_{p}}-\sm}\cong (\pi^{\sm}_{K_{p}})_{*}(\mcal{G}^{c}_{\dR,\Kpp}(V))\otimes p_{2}^{*}(B_{\dR} )
\] as quasicoherent sheaves over \( \mX^{\tor}_{\Kpp,E_{\p}}\times [\hat{\mA}^{1}/\GG_{m}] \), which is compatible with connections.
Note that by \cite[Theorem 5.5]{JRC2021solid}, the LHS is isomorphic to \( (\pi^{\sm}_{K_{p}})_{*}\RHom_{\fg^{c}}(V^{\vee},\OBdRlogKp^{\la}) \). 

By Lemma \ref{lemBdRLaProperty} and Lemma \ref{lemQCohSMsv}, we know that this isomorphism descends to a \( G(\QQ_{p})^{\sm}\times \unl{\Gal_{E_{\p}}} \)-equivariant isomorphism on \( (\mX^{\tor}_{K^{p},E_{\p}})^{\pdR/E_{\p},\hat{+},\sm}_{\log} \)  
\[
\RHom_{\fg^{c}}(V^{\vee},\BdRKp^{\la})\cong (\pi^{\sm,\pdR}_{K_{p}})^{*}\bar{\mcal{G}}^{c}_{\dR,\Kpp}(V)\otimes s^{*}(B_{\dR}),
\] as desired. 

Since the isomorphism is \( \unl{\Gal_{E_{\p}}} \)-equivariant,  we can apply the functor \( D_{\sen,E_{\p}} \) on both sides.  We have \( E_{0}\circ D_{\sen,E_{\p}} \) 
commutes with \( \RHom_{\fg^{c}}(V^{\vee},-) \) 
since the latter is calculated by a finite resolution. 
On the other hand, we have a natural map 
\begin{align*}
&(\pi^{\sm,\pdR}_{K_{p}})^{*}\bar{\mcal{G}}^{c}_{\dR,\Kpp}(V)\otimes_{E_{\p}} E_{\p,\infty}\to\\
&E_{0}\circ D_{\sen,E_{\p}}\left((\pi^{\sm,\pdR}_{K_{p}})^{*}\bar{\mcal{G}}^{c}_{\dR,\Kpp}(V)\otimes s^{*}(B_{\dR})\right).
\end{align*}
To verify that it is an isomorphism, it suffices to verify after pulling back to \( \mX^{\tor,\sm}_{K^{p},E_{\p}}\times B\GG_{m} \).
Then the RHS is isomorphism to \[
\bigoplus_{i}\bigoplus_{j}E_{0}\circ D_{\sen,E_{\p}}(\gr^{j}(\mcal{G}^{d}_{\dR,\Kpp}(V))\otimes_{E_{\p}} \gr^{i-j}B_{\dR})\{i\}.
\] 
Note that the second direct sum is finite. We then conclude by Lemma \ref{lemDsenCommuteWithFlatBC}. Note that the sections of \( \gr^{j}(\mcal{G}^{d}_{\dR,\Kpp}(V)) \)
form Banach spaces, and are thus flat and nuclear by \cite[Corollary 3.16, Lemma 3.21]{JRC2021solid}.
\end{proof}


\subsubsection{Perfectoid case}

Assume that there exists a perfectoid space \( \mX_{K^{p}} \) such that 
\( \mX_{K^{p}}\sim \varprojlim_{K_{p}}\mX_{\Kpp} \) (\cite[Definition 2.4.1]{SW13}, \cite[Definition 4.4.2]{BoxerPilloni2021higherColeman}). This is known for example if \( (G,X) \) 
is of abelian type (\cite{HansenJohansson2020perfectoid}, \cite[Proposition 4.4.53]{BoxerPilloni2021higherColeman}).

By Theorem \ref{thmAnDRStackGeneral} (4), and the construction of \( \mX_{K^{p}} \) in Lemma \ref{lemTorCompacAsTateStack}, the Tate stack \( \mX_{K^{p}} \) 
is precisely the perfectoid space \( \mX_{K^{p}} \) regarded as a Tate stack.
By Theorem \ref{thmAnDRStackGeneral} (5), we have \( \mX_{K^{p}}^{\dR}\cong (\mX_{K^{p}}^{\sm})^{\dR}\cong \varprojlim_{K_{p}}\mX_{\Kpp}^{\dR} \). So
\begin{lem}\label{lemBasicDRAlgDRandSM}
(1) We have the following diagram \[
\begin{tikzcd}
\mX_{K^{p}}^{\sm}\arrow[r]\arrow[d]
& \mX_{K^{p},\bar{\QQ}_{p}}^{\pdR/\bar{\QQ}_{p},\sm}\arrow[r,"g_{K^{p}}"]\arrow[d]
& \mX_{K^{p}}^{\dR}\arrow[d]\\
\mX_{\Kpp}\arrow[r]
& \mX_{\Kpp,\bar{\QQ}_{p}}^{\pdR/\bar{\QQ}_{p}}\arrow[r,"g_{\Kpp}"]
& \mX_{\Kpp}^{\dR},
\end{tikzcd}
\] where both squares are Cartesian, and \( \mX_{K^{p},\bar{\QQ}_{p}}^{\pdR/\bar{\QQ}_{p},\sm}:=\varprojlim_{K_{p}}\mX_{\Kpp,\bar{\QQ}_{p}}^{\pdR/\bar{\QQ}_{p}} \), which is an open subspace of \( (\mX^{\tor}_{K^{p},\bar{\QQ}_{p}})^{\pdR/\bar{\QQ}_{p},\sm}_{\log} \)  in Definition \ref{dfnSMratConstruction}.

(2) \( g_{K^{p}}^{*} \) is fully faithful and the natural morphism \( \mrm{id}\to (g_{K^{p}})_{*}\circ g_{K^{p}}^{*} \) is an equivalence. 
\end{lem}
\begin{proof}
(1) follows from Theorem \ref{thmAnDRStackGeneral} (3) and Lemma \ref{lemBasicLogDRSta} which implies that for \( X\to X' \) \'etale, \( X^{\dR/X'}\cong X^{\pdR/X'}\cong X' \).  
(2) follows from  Proposition \ref{propAnToAlgDRJuan}.
\end{proof}
\begin{cor}\label{corQCohOnXdR}
We have isomorphisms in \( \mrm{Pr}^{L} \)  \[
\QCoh(\mX^{\dR}_{K^{p}})\cong \varinjlim^{*}_{K_{p}}\QCoh(\mX^{\dR}_{\Kpp})
\] \[
\QCoh(\mX^{\dR}_{\Kpp})\cong \varinjlim_{E_{\p}\subset L\subset\bar{\QQ}_{p}}^{*}\QCoh(\mX^{\dR}_{\Kpp,L}),
\]
where the transition maps are given by \( (-)^{*} \),
and the second colimit is taken over all the finite extension of \( E_{\p} \). 
\end{cor}
\begin{proof}
This follows from Corollary \ref{corXdRDescendableCover} by Theorem \ref{thmVariousDescent} (3) and Lemma \ref{dfnUniformAffineProper}.
\end{proof}

Recall from Lemma \ref{lemBdRXtorKp} that we have \( \BdR^{+}(\mX_{K^{p}})\to \mX_{K^{p},\bar{\QQ}_{p}}^{\pdR/\bar{\QQ}_{p},\sm} \), where the LHS is defined by Construction \ref{constructionBdROnDRStack}.
By Lemma \ref{lemTwoBdRCoincide}, this map is the unique extension of the natural map \( \mX_{K^{p}}\to \mX_{K^{p},\bar{\QQ}_{p}}^{\pdR/\bar{\QQ}_{p},\sm}  \). 
Similarly, by Definition \ref{dfnAnDRStack}, the composition \[
 h:\BdR^{+}(\mX_{K^{p}})\to \mX_{K^{p},\bar{\QQ}_{p}}^{\pdR/\bar{\QQ}_{p},\sm} \to \mX_{K^{p}}^{\dR}
\] is the unique extension of the natural map \( \mX_{K^{p}}\to\mX_{K^{p}}^{\dR} \). 
\begin{notation}\label{notationBdROnAnalyticDR}
Let \( h_{K^{p}}:\BdR^{+}(\mX_{K^{p}})\to  \mX_{K^{p}}^{\dR} \) 
be the unique extension of the natural map \( \mX_{K^{p}}\to\mX_{K^{p}}^{\dR} \),
where \( \BdR^{+}(\mX_{K^{p}}) \)  is defined by Construction \ref{constructionBdROnDRStack}.

Let \( \BdRKp^{+}:=(h_{K^{p}})_{*}(\mO_{\BdR^{+}(\mX_{K^{p}})}) \) equipped with the \( t \)-adic filtration, 
and \( \BdRKp:=\BdRKp^{+}[1/t]\in\widehat{\Fil}^{\ZZ}(\QCoh(\mX^{\dR}_{K^{p}})) \). Let \( \BdRKp^{\la}:=\BdRKp^{R-\la}\in \widehat{\Fil}^{\ZZ}(\QCoh(\mX^{\dR}_{K^{p}})) \) as in Notation \ref{notationRlaConvention},
and \( D_{\sen}(\BdRKp^{\la})\in \widehat{\Fil}^{\ZZ}(\QCoh(\mX^{\dR}_{K^{p}})) \) 
for \( D_{\sen} \) in Definition \ref{dfnDsenFunctor}. Note that all the operations are taken in the filtered complete category.

Here \( (-)^{R-\la} \) and \( D_{\sen}(-) \) are defined in the following precise sense:
by Corollary \ref{corQCohOnXdR}, it suffices to define both functors for the \( * \)-push-forward of \( \BdRKp \) to \( \mX^{\dR}_{\Kpp,L} \) for all \( K_{p} \) and \( L \). Then the groups \( K_{p} \)  
and \( \Gal_{L} \) acts trivially on \( \mX^{\dR}_{\Kpp,L} \), and thus \( (-)^{R-K_{p}-\la} \)
and \( D_{\sen,L}(-) \) can be defined.
\end{notation}
\begin{cor}\label{corReformulateRiemmanHilbPerfdCase}
Assume that there exists a perfectoid space \( \mX_{K^{p}} \) such that 
\( \mX_{K^{p}}\sim \varprojlim_{K_{p}}\mX_{\Kpp} \) (\cite[Definition 2.4.1]{SW13}, \cite[Definition 4.4.2]{BoxerPilloni2021higherColeman}).
 Then for any \( V\in\Rep(G^{c}) \), 
we have 
\begin{align*}
(\pi^{\dR}_{K_{p}})^{*}
\mcal{\bar{G}}^{c,\an}_{\dR,\Kpp}(V)\otimes_{\bar{\QQ}_{p}}B_{\dR}&\cong \RHom_{\fg^{c}}(V^{\vee},\BdRKp^{\la}),
\\
(\pi^{\dR}_{K_{p}})^{*}
\mcal{\bar{G}}^{c,\an}_{\dR,\Kpp}(V)&\cong \RHom_{\fg^{c}}(V^{\vee},E_{0}(D_{\sen}(\BdRKp^{\la})))
\end{align*} where \( \pi^{\dR}_{K_{p}}:\mX_{K^{p}}^{\dR}\to \mX_{\Kpp,E_{\p}}^{\dR} \), \( \mcal{\bar{G}}^{c,\an}_{\dR,\Kpp} \) 
is as in Proposition \ref{propAnalyticGMTheory}, and \( -\otimes_{E_{\p}}B_{\dR} \) 
is taken in the \( t^{\bullet} \)-complete sense. 
\end{cor}
\begin{rmk}
The slogan is that the de Rham torsor \(\bar{\mcal{G}}^{c,\an}_{\dR,\Kpp}\)
can be reconstructed from \(\BdR^{+}(\mX_{K^{p}})\to \mX_{K^{p}}^{\dR}\). This is a bit surprising as one can recover information about connections from perfectoid data.
\end{rmk}
\begin{proof}
Comparing with Theorem \ref{thmReformulateRHCorr}, by Proposition \ref{propAnToAlgDRJuan} (by putting \( L=\bar{\QQ}_{p} \)) and the constructiong of \( \mcal{\hat{G}}^{c,\an}_{\dR,K} \) in the proof of Proposition \ref{propAnalyticGMTheory},
it suffices to show that \( E_{0}(D_{\sen}(\BdRKp^{\la})) \) in the setting of Theorem \ref{thmReformulateRHCorr} push-forward to \( E_{0}(D_{\sen}(\BdRKp^{\la})) \) over \( \mX_{K^{p}}^{\dR} \). Since both sides are filtered complete, it suffices to study the \( t^{\bullet} \)-graded piece. 
By Corollary \ref{corQCohOnXdR} and Lemma \ref{lemQCohSMsv}, it suffices to compare both sides along \[
g_{K,L}:\mX_{K,L}^{\pdR/L}\to \mX_{K,L}^{\dR/L},
\] and we want to show that \[ (g_{K,L})_{*}(E_{0}\circ D_{\sen,L}({\mO}_{K^{p}}^{\la}(i)))\cong E_{0}\circ D_{\sen,L}(((g_{K,L})_{*}(\hat{\mO}_{K^{p}}))^{R-\la}(i)). \] 

By projection formula and the cohomological co-smoothness of \( g_{K,L} \) (Proposition \ref{propAnToAlgDRJuan}), \( (g_{K,L})_{*}(\hat{\mO}_{K^{p}})\otimes \mO_{G^{c},1}\cong (g_{K,L})_{*}(\hat{\mO}_{K^{p}}\otimes \mO_{G^{c},1}) \). Since derived smooth vectors can be calculated by  Lazard-Serre's resolution (see \cite[Theorem 5.7]{JRC2021solid}),
and \( (g_{K,L})_{!} \) 
commutes with colimits, we know that \( ((g_{K,L})_{*}(\hat{\mO}_{K^{p}}))^{R-\la}\cong (g_{K,L})_{*}(\mO^{\la}_{K^{p}}) \). \( (g_{K,L})_{*} \) 
also commutes with limits, and thus commutes with \( R\Gamma(H_{L},-) \). The commutation with \( \Gamma_{L}^{\la} \) is proven similarly. Thus  \( D_{\sen,L}(((g_{K,L})_{*}(\hat{\mO}_{K^{p}}))^{R-\la})\cong (g_{K,L})_{*}(D_{\sen,L}(\mO^{\la}_{K^{p}})) \).
Finally again by commutation with limits, \( (g_{K,L})_{*} \) 
commutes with \( E_{0} \). 
\end{proof}

\subsection{Local Analogue}\label{subsecLocalAnalogueLASv}
The results in this section have local analogue, i.e. for local Shimura varieties. 
The set-up is as follows:
\begin{notation}
Let \((G,b,\mu)\)
be a basic local Shimura datum as in \cite[Definition 24.1.1]{ScholzeWeinstein2020berkeley}.
Let \( E/\QQ_{p} \) be the associated reflex field. 
Let \(\mcal{M}:=\mcal{M}(G,b,\mu)\)
be the associated local Shimura variety representing \[
(S\to\Spd(\CC_{p}))\mapsto \{\mcal{G}_{1}\to \mcal{G}_{b}:\text{modification supported at \( S^{\#} \) and bounded by \( \mu \)}\}
\] where \( \mcal{G}_{1} \) 
denotes the trivial \( G \)-torsor on \( \mX_{\rF}(S) \),  
and \( \mcal{G}_{b} \) denotes the \( G \)-torsor on \( \mX_{\rF}(S) \) corresponding to \( b\in B(G) \) by \cite{Fargues2020Gtorseurs}, \cite{Anschutz2019reductiveFF}. 

Then
\( \mcal{M} \) is representable by a perfectoid space over \( \Spa(\CC_{p},\mO_{\CC_{p}}) \) by \cite[Theorem D]{SW13}, which carries an action of \((G(\QQ_{p}),*_{G})\times (G_{b}(\QQ_{p}),*_{G_{b}})\times \Gal_{E}\),
and we have two Hodge-Tate period maps \[\begin{tikzcd}
& \mcal{M}\arrow[dl,"\pi_{HT,G}=\pi_{GM,G_{b}}"']\arrow[dr,"\pi_{HT,G_{b}}=\pi_{GM,G}"] & \\
\Fl_{G,\mu}=\Fl_{G_{b},\mu^{-1}}^{\std} & & \Fl_{G_{b},\mu^{-1}}=\Fl_{G,\mu}^{std},
\end{tikzcd}\]
where
\( \pi_{\HT,G}=\pi_{\GM,G_{b}} \) is a \( \unl{G(\QQ_{p})} \)-equivariant  pro\'etale \( \unl{G_{b}(\QQ_{p})} \)-torsor, 
and \( \pi_{\HT,G_{b}}=\pi_{\GM,G} \) is a \( \unl{G_{b}(\QQ_{p})} \)-equivariant  pro\'etale \( \unl{G(\QQ_{p})} \)-torsor. 

For any open compact subgroup \( K_{1}\subset G(\QQ_{p}) \) (resp. \( K_{b}\subset G_{b}(\QQ_{p}) \) ), the diamond \( \mcal{M}/K_{1}\)  (resp. \( \mcal{M}/K_{b} \) )
is represented by a smooth analytic adic space \( \mcal{M}_{K}(G,b,\mu) \) (resp. \( \mcal{M}_{K_{b}}(G_{b},1,\mu^{-1}) \) ) over \( \CC_{p} \). 
\end{notation}
\begin{dfn}\label{dfnSMRealizeLocalSV}
We define solid stacks
	\[\mcal{M}^{G-\sm}:=\varprojlim_{K_{1}\subset G(\QQ_{p})}\mcal{M}_{K_{1}}(G,b,\mu),\;\mcal{M}^{G_{b}-\sm}:=\varprojlim_{K_{b}\subset G_{b}}\mcal{M}_{K_{b}}(G_{b},1,\mu^{-1}).\]
\end{dfn}
\begin{thm}[\cite{DospinescuJuan2024jacquet}]
\label{thmDosJuanLSV}
(1) Let \( U\subset \mcal{M}_{K}(G,b,\mu) \) be an open affinoid subspace admitting an étale map to \( \mbb{T}^{d} \) that factors as a finite composition of rational localizations and finite étale maps. Let \( \ti{U}\subset \mcal{M} \) be the preimage of \( U \). Then \( \hat{\mO}_{\mcal{M}}(U)^{R-G(\QQ_{p})-\la} \) is concentrated in degree \( 0 \).  We write \( \mO^{G(\QQ_{p})-\la}_{\mcal{M}}:=\hat{\mO}_{\mcal{M}}^{R-G(\QQ_{p})-\la} \). Then we have an identification \( \mO^{G(\QQ_{p})-\la}_{\mcal{M}}= \mO^{G_{b}(\QQ_{p})-\la}_{\mcal{M}} \)  
as subsheaves of \( \hat{\mO}_{\mcal{M}} \). 

(2)
Consider the actions of \((\fg\otimes\Fl_{G,\mu},*_{G})\)
and of \((\fg\otimes \mO_{\Fl_{G_{b},\mu^{-1}}},*_{G_{b}})\) on \(\mO_{\mcal{M}}^{\la}\). 
Then
the actions of \(\pi_{\HT,G}^{*}(\fn_{\mu}^{0})\) and \(\pi_{\HT,G_{b}}^{*}(\fn_{\mu^{-1}}^{0})\) are zero on \(\mO_{\mcal{M}}^{\la}\), which induces an action of \(\pi_{\HT,G}^{*}(\mfk{m}_{\mu}^{0})\) and \(\pi_{\HT,G_{b}}^{*}(\mfk{m}_{\mu^{-1}}^{0})\) on \(\mO_{\mcal{M}}^{\la}\). Moreover,
we have an identification of the vector bundles \[\pi_{\HT,G}^{*}(\mfk{m}_{\mu}^{0})\cong \pi_{\HT,G_{b}}^{*}(\mfk{m}_{\mu^{-1}}^{0}), \]
which is compatible with their actions on \(\mO_{\mcal{M}}^{\la}\). 
\end{thm}
\begin{proof}
	These are shown in \cite[Corollary 5.1.9 \& Theorem 5.1.10]{DospinescuJuan2024jacquet}
\end{proof}



\begin{dfn}\label{dfnLArealization}
Following Definition \ref{dfnLocallyAnSV},
we define \[
\mcal{M}^{G_{b}(\QQ_{p})-\la}:=\unl{\AnSpec}_{\mcal{M}^{G_{b}-\sm}}(\mO^{\la}_{\mcal{M}}),\;\mcal{M}^{G(\QQ_{p})-\la}:=\unl{\AnSpec}_{\mcal{M}^{G-\sm}}(\mO^{\la}_{\mcal{M}}).
\]
\end{dfn}
\begin{lem}\label{lemHTPeriodFactorLocal}
(1)
We have a natural isomorphism \(\mcal{M}^{G_{b}(\QQ_{p})-\la}\cong \mcal{M}^{G(\QQ_{p})-\la}\), and
we will denote them unambiguously as \(\mcal{M}^{\la}\). 

(2) \( \pi_{HT,G}=\pi_{GM,G_{b}} \) 
and \( \pi_{HT,G_{b}}=\pi_{GM,G} \) canonically
factor through \( \mcal{M}\to\mcal{M}^{\la} \). 
\end{lem}
\begin{proof}
(1) follows from Theorem \ref{thmDosJuanLSV} and the identification of the locale defined by the open subspaces of \( \varprojlim_{K_{b}\subset G_{b}(\QQ_{p})}|\mcal{M}_{K_{b}}(G_{b},1,\mu^{-1})| \) 
and that defined by \( \varprojlim_{K_{1}\subset G(\QQ_{p})}|\mcal{M}_{K_{1}}(G,b,\mu)| \), 
as both are homeomorphic to \( |\mcal{M}| \). 
(2) follows the same proof of Lemma \ref{lemHTPeriodFactor}
\end{proof}
\begin{dfn}\label{dfnGrothendieck-Messing-Hodge-TateperioddomainLocal}
We define the \emph{Grothendieck-Messing-Hodge-Tate period domain} as \[
\Per_{G,b,\mu}:=\Per(G,b,\mu):=M_{\mu}\backslash
\left((N_{\mu}\backslash G_{\CC_{p}})(-1)\times (N_{\mu^{-1}}\backslash G_{b,\CC_{p}})\right),
\] where \((-1)\) is as in Theorem \ref{thmIdenTwoTorsorsOverLa}.

As in Definition \ref{dfnGMHTperDomain}, we have an action of \( G\times G_{b} \) on \( \Per_{G,b,\mu} \), and we have natural morphisms \[
\Fl_{G,\mu}\xleftarrow{\pi^{per}_{\GM,G_{b}}=\pi^{per}_{\HT,G}} \Per_{G,b,\mu}\xrightarrow{\pi^{per}_{\GM,G}=\pi^{per}_{\HT,G_{b}}} \Fl_{G_{b},\mu^{-1}}.
\] 
\end{dfn}
\begin{thm}\label{thmGMHTPeriMapLocal}
(1)
There is a natural \emph{Grothendieck-Messing-Hodge-Tate period map} \[
\pi_{\GM\HT}^{\la}:\mcal{M}^{\la}\to \Per_{G,b,\mu},
\] such that we have a commutative diagram \[
\begin{tikzcd}
\mcal{M}_{K_{1}}(G,b,\mu)\arrow[d,"\pi_{\GM,G,K_{1}}"]
& 
\mcal{M}^{\la}\arrow[l,"\pi_{K_{1}}^{\la}"]
\arrow[r,"\pi_{K_{b}}^{\la}"]\arrow[d,"\pi_{\GM\HT}^{\la}"]
& 
\mcal{M}_{K_{b}}(G_{b},1,\mu^{-1})\arrow[d,"\pi_{\GM,G_{b},K_{b}}"]
\\
\Fl_{G_{b},\mu^{-1}}
& 
\Per_{G,b,\mu}\arrow[l,"\pi^{per}_{\GM,G}"]\arrow[r,"\pi^{per}_{\GM,G_{b}}"]
& \Fl_{G,\mu}.
\end{tikzcd}
\]

(2) We have the following Cartesian diagram of Tate stacks \[
\begin{tikzcd}
\mcal{M}^{\la}\arrow[r,"\pi^{\la}_{\GM\HT}"]\arrow[d,"\pi^{\la}_{\sm}"]
&  \Per_{G,b,\mu}\arrow[d]\arrow[r] & \Fl_{G,\mu}\arrow[d]
\\
\mcal{M}^{G-\sm}\arrow[r,"\pi^{G-\sm}_{\GM\HT}"]\arrow[d]
&  \Per_{G,b,\mu}^{\dR/\Fl_{G_{b},\mu^{-1}}}\arrow[d]\arrow[r] & \Fl_{G,\mu}/(\fg^{0}/\fn_{\mu}^{0})^{\dagger}
\\
\mcal{M}^{\dR}\arrow[r,"\pi^{\dR}_{\GM\HT}"]
& \Per_{G,b,\mu}^{\dR}.
\end{tikzcd}
\] 
\end{thm}
\begin{proof}
(1) is shown in \cite[Proof of Theorem 5.1.10]{DospinescuJuan2024jacquet}.
The proof of (2) is identical to that of Theorem \ref{thmLAstrucGMHTper}. 
The key step is to prove the isomorphism \[
\mO^{\la}_{\mcal{M}}\otimes_{\mO_{\mcal{M}^{G-\sm}}}\mO^{\la}_{\mcal{M}}\cong \mO^{\la}_{\mcal{M}}\{(\fg^{0}/\fn_{\mu}^{0})^{\vee}\}^{\dagger},
\] as was used in the proof of Proposition \ref{propLAvsSmoothSV},
which is proven as (5.13) in the proof of Theorem 5.1.10 in \cite{DospinescuJuan2024jacquet}.
\end{proof}
We also have Riemann-Hilbert correspondence as in Corollary \ref{corReformulateRiemmanHilbPerfdCase}.
\begin{notation}\label{notationBdROnAnalyticDRLocal}
Let \( h_{K^{p}}:\BdR^{+}(\mcal{M})\to  \mcal{M}^{\dR} \) 
be the unique extension of the natural map \( \mcal{M}\to\mcal{M}^{\dR} \),
where \( \BdR^{+}(\mcal{M}) \)  is defined by Construction \ref{constructionBdROnDRStack}.

Let \( \BdRX{\mcal{M}}^{+}:=(h_{K^{p}})_{*}(\mO_{\BdR^{+}(\mcal{M})}) \) equipped with the \( t \)-adic filtration, 
and \( \BdRX{\mcal{M}}:=\BdRX{\mcal{M}}^{+}[1/t]\in\widehat{\Fil}^{\ZZ}(\QCoh(\mcal{M}^{\dR})) \). Let \( \BdRX{\mcal{M}}^{G-\la}:=\BdRX{\mcal{M}}^{R-G(\QQ_{p})-\la}\in \widehat{\Fil}^{\ZZ}(\QCoh(\mcal{M}^{\dR})) \), 
and \( D_{\sen}(\BdRX{\mcal{M}}^{G-\la})\in \widehat{\Fil}^{\ZZ}(\QCoh(\mcal{M}^{\dR})) \) 
for \( D_{\sen} \) in Definition \ref{dfnDsenFunctor}. 
We also define \( \BdRX{\mcal{M}}^{G_{b}-\la}:=\BdRX{\mcal{M}}^{R-G_{b}(\QQ_{p})-\la}\in \widehat{\Fil}^{\ZZ}(\QCoh(\mcal{M}^{\dR})) \). 
Here \( (-)^{R-\la} \) and \( D_{\sen}(-) \) are as in Notation \ref{notationBdROnAnalyticDR}.
\end{notation}
\begin{lem}\label{lemBdRLATwosides}
We have an identification \( \BdRX{\mcal{M}}^{G(\QQ_{p})-\la}= \BdRX{\mcal{M}}^{G_{b}(\QQ_{p})-\la} \)  
as subsheaves of \( \BdRX{\mcal{M}} \), and
we will denote them unambiguously as \(\BdRX{\mcal{M}}^{\la}\). 

Moreover, \( \BdRX{\mcal{M}}^{\la} \) is complete with respect to \( t^{\bullet} \), and  \( \gr^{i}_{t^{\bullet}}\BdRX{\mcal{M}}^{\la}\cong \mO_{\mcal{M}}^{\la}(i) \) for \( i\in\ZZ \). 
Here we regard \( \mO_{\mcal{M}}^{\la} \) 
as an object in \( \mcal{M}^{\dR} \) 
via \( * \)-push-forward. 
\end{lem}
\begin{proof}
The description of the  graded pieces 
follows from \( \gr^{i}_{t^{\bullet}}(\BdRX{\mcal{M}})\cong \hat{\mO}_{\mcal{M}} \). 
The rest follows from Theorem \ref{thmDosJuanLSV} (1).
\end{proof}
\begin{prop}\label{propRHCorrLocalSV}
For any \( V\in\Rep(G) \), 
we have a \( G_{b}(\QQ_{p})^{\la} \)-equivariant isomorphism 
\begin{align*}
((V,*_{G_{b}})\otimes \mO_{\mcal{M}^{\dR}})\otimes_{E_{\p}}B_{\dR}&\cong \RHom_{\fg}((V^{\vee},*_{G}),(\BdRX{\mcal{M}}^{\la},*_{G})),
\\
(V,*_{G_{b}})\otimes \mO_{\mcal{M}^{\dR}}&\cong \RHom_{\fg}((V^{\vee},*_{G}),(E_{0}(D_{\sen}(\BdRX{\mcal{M}}^{\la})),*_{G}))
\end{align*} where the factor \( V \) on the LHS is regarded as a representation of \( G_{b} \), and \( V^{\vee} \) on the RHS is regarded as a representation of \( G \), and
\( -\otimes_{E_{\p}}B_{\dR} \) 
is taken in the \( t^{\bullet} \)-complete sense. 
\end{prop}
\begin{proof}
The case where \( V \) is the trivial representation is given by Proposition \ref{propPushForwardOBdR} applied to \( \mcal{M}_{K_{1}}(G,b,\mu) \), following the proof of Theorem \ref{thmReformulateRHCorr}. For general \( V \), 
by the moduli interpretation of \( \mcal{M} \), 
we have a \( G_{b}(\QQ_{p})\times G(\QQ_{p}) \)-equivariant isomorphism \[
(V,*_{G})\otimes \BdRX{M}\cong (V,*_{G_{b}})\otimes \BdRX{M}.
\] 
Taking \( (-)^{R-G_{b}(\QQ_{p})-\la} \) on both sides, we obtain by Lemma \ref{lemBdRLATwosides}
that \[
(V,*_{G})\otimes \BdRX{M}^{\la}\cong (V,*_{G_{b}})\otimes \BdRX{M}^{\la}.
\] We then reduce to the case where \( V \) is trivial. 
\end{proof}

\printbibliography

\end{document}